\documentclass[12pt, reqno]{amsart}
\usepackage{mathrsfs}
\usepackage{amssymb,amsthm,amsmath}
\usepackage[numbers,sort&compress]{natbib}
\usepackage{amssymb,amsmath}
\usepackage{amsfonts}
\usepackage{mathrsfs}
\usepackage{latexsym}
\usepackage{amssymb}
\usepackage{amsthm}
\usepackage{bm}
\usepackage{color}
\usepackage{pdfsync}
\usepackage{indentfirst}
\usepackage{appendix}
\date{today}

\usepackage{hyperref}
\hypersetup{hypertex=true,colorlinks=true,linkcolor=blue,anchorcolor=g,citecolor=red}

\usepackage{amsmath}
\usepackage{amsthm}
\allowdisplaybreaks
\newtheorem{remark}{Remark}[section]

\newtheorem{theorem}{Theorem}[section]

\newtheorem{lemma}{Lemma}[section]

\newcommand{\R}{\mathbb R}

\newcommand{\beq}{\begin{equation}}
	\newcommand{\eeq}{\end{equation}}
\newcommand{\ben}{\begin{eqnarray}}
	\newcommand{\een}{\end{eqnarray}}
\newcommand{\beno}{\begin{eqnarray*}}
	\newcommand{\eeno}{\end{eqnarray*}}

\numberwithin{equation}{section}

\usepackage[
letterpaper,                 
textheight=8.35in,           
headsep=20pt,                
footskip=36pt,               
marginparwidth=0pt,          
marginparsep=0pt,            
left=0.75in,
right=0.75in
]{geometry}
\begin{document}
	\title[Boundary layer of 2D Chemotaxis Navier-Stokes equations]{Boundary layer analysis for the 2D chemotaxis-Navier-Stokes system with logarithmic sensitivity, Part II. viscous vanishing limit}
	\author{Hui~Wang}
	\address[Hui~Wang]{School of Mathematical Sciences, Dalian University of Technology, Dalian, 116024,  China}
	\email{whd@mail.dlut.edu.cn}
	\author{Lingling~Zhao}
	\address[Lingling~Zhao]{College of Mathematics, Taiyuan University of Technology, Taiyuan, 030024, China}
	\email{zhaolingling@tyut.edu.cn}
	\date{\today}
	\maketitle
	\renewcommand{\theequation}{\arabic{section}.\arabic{equation}}
	\catcode`@=11 \@addtoreset{equation}{section} \catcode`@=12
	\begin{abstract}
        This is the second part of a two-part work concerning boundary layer solutions to the coupled Chemotaxis-Navier-Stokes system in the two-dimensional half-space. In the present work, we address the convergence of boundary layer solutions to singular chemotaxis-fluid equations under slip boundary conditions with respect to the chemical diffusion-viscosity parameter $\varepsilon$
       in the two-dimensional half-plane. More precisely, we show that the boundary layer for $\varepsilon>0$ (viscous convection coefficient) converges to the superposition of the outer layer (solution with $\varepsilon=0$) and the inner layer as $\varepsilon\rightarrow0$. The outer and inner profiles are explicitly derived as in the first part\cite{WWZ}. Furthermore, the well-posedness results of the coupled Chemotaxis-Navier-Stokes system in conormal Sobolev spaces will be presented in Appendix. They answer the question mentioned in the first part of the two-part work. This study could help the understanding of the chemotactic movement of aerobic bacteria to the water-air surface observed experimentally in fluids, and enrich the theoretical results of boundary layer in chemotactic fluid models.
		\\\ \\
		{\bf Keywords:} Chemotaxis Navier-Stokes system; Slip boundary conditions; Boundary layer; Viscous vanishing limit.
        \\
         \textbf{2020 Mathematics Subject Classification.} {Primary 35Q92; Secondary 76N05, 35B44, 35Q30, 35Q35}
	\end{abstract}
	\section{Introduction}
    In this paper, we will consider the viscosity vanishing limits for the following Chemotaxis-Navier-Stokes system with logarithmic singularity:
	\begin{eqnarray}\label{ncu equation}
		\left\{
		\begin{split}{}
			&\partial_tn-D\Delta n+u\cdot\nabla n+\chi\nabla\cdot(\frac{n}{c}\nabla c)=0,\;\;&\Omega\times(0,T),\\
			&\partial_tc-\varepsilon\Delta c+u\cdot\nabla c+nc=0,\;\;& \Omega\times(0,T),\\
			&\partial_tu-\varepsilon\Delta u+(u\cdot\nabla)u+\nabla p=n\nabla\phi,\;\;& \Omega\times(0,T),\\
			&\nabla\cdot u=0,\;\;& \Omega\times(0,T),\\
		\end{split}
		\right.
	\end{eqnarray}	
	where $n$, $c$, $u$ and $p$ denote cell density, oxygen concentration, the fluid velocity and the associated pressure respectively, and $\Omega=\R^2_+=\{(x,y)\in\R^2|y>0\}$. $D>0$ and $\varepsilon\geq0$ are cell and oxygen diffusion coefficients, and $\chi>0$ is referred to as the chemotactic coefficient measuring the strength of the chemotactic sensitivity, $\phi(x,y)$  is a sufficiently smooth function and independent of t. For the sake of simplicity, we suppose $D=\chi=1$ and $\nabla\phi=e_2=(0,1)$.
    The boundary layer problems on the system that the chemotactic term in (\ref{ncu equation}) reduce to $\nabla\cdot(n\nabla c)$ are available such as in \cite{HO,HO1,LSW}. However, there are few results of the well-posedness and the viscosity vanishing limit problem on the system (\ref{ncu equation}). Motivated by it, the well-posedness theory related to system (\ref{ncu equation}) will be established in conormal Sobolev spaces in Appendix. Our another goal is to investigate whether there exist significant differences on the dynamical behaviors between the system (\ref{ncu equation}) at \(\varepsilon=0\) and  \(\varepsilon>0\) under chemotaxis with logarithmic sensitivity. That is, it is necessary to clarify whether the solution of equations (\ref{ncu equation}) with $\varepsilon>0$ will converge to the solution with $\varepsilon=0$ as $\varepsilon\rightarrow0$. When discussing this problem, the singularity at $c=0$ is effectively resolved by applying Cole-Hopf type transformations (see \cite{LS,LW}):
	\begin{eqnarray}\label{c transform v}
		\begin{split}{}
			v^\varepsilon=-\nabla\ln c=-\frac{\nabla c}{c}.
		\end{split}
	\end{eqnarray}
    This translates the system(\ref{ncu equation}) into a non-singular conservation law system for $\varepsilon\geq0$:
	\begin{eqnarray}\label{nvu equation}
		\left\{
		\begin{split}{}
			&\partial_tn^\varepsilon-\Delta n^\varepsilon+u^\varepsilon\cdot\nabla n^\varepsilon-\nabla\cdot(n^\varepsilon v^\varepsilon)=0,\\
			&\partial_tv^\varepsilon-\varepsilon\Delta v^\varepsilon+\nabla(u^\varepsilon\cdot v^\varepsilon)+\nabla(\varepsilon|v^\varepsilon|^2-n^\varepsilon)=0,\\
			&\partial_tu^\varepsilon-\varepsilon\Delta u^\varepsilon+(u^\varepsilon\cdot\nabla)u^\varepsilon+\nabla p^\varepsilon=n^\varepsilon{e_2},\\
			&\nabla\cdot u^\varepsilon=0,\\
			&n^\varepsilon|_{t=0}=n_{\rm{in}}, v^\varepsilon|_{t=0}=v_{\rm{in}}, u^\varepsilon|_{t=0}=u_{\rm{in}},
		\end{split}
		\right.
	\end{eqnarray}
	with slip boundary conditions
	\begin{eqnarray}\label{nvu condition}
		\left\{
		\begin{split}{}
			&({\partial_{y}} n^\varepsilon+n^\varepsilon v^\varepsilon_2)|_{y=0}=0,\;\;(v^\varepsilon_2,u^\varepsilon_2)|_{y=0}=0,\;\;(\partial_yv^\varepsilon_1,\partial_yu^\varepsilon_1)|_{y=0}=0,\;\;&\varepsilon>0,\\
            &({\partial_{y}} n^0+n^0 v^0_2)|_{y=0}=0,\;\;u^0_2|_{y=0}=0,&\varepsilon=0.\\
		\end{split}
		\right.
	\end{eqnarray}
	 From the Cole-Hopf transformation (\ref{c transform v}), the curl for $v^\varepsilon$ must be intrinsically free:
	\begin{eqnarray}\label{v curl}
		\begin{split}{}
			\nabla\times v^\varepsilon=\partial_xv^\varepsilon_2-\partial_yv^\varepsilon_1=0.
		\end{split}
	\end{eqnarray}
	By taking the curl on both sides of the second equation (\ref{nvu equation}), we can obtain $\partial_t(\nabla\times v^\varepsilon)=\varepsilon\Delta(\nabla\times v^\varepsilon)$. Therefore, in order to maintain the intrinsic curl-free condition (\ref{v curl}) for $\varepsilon>0$, one can impose $\nabla\times v_0=0$. From $(\ref{nvu condition})_1$,
     the boundary condition  $\nabla\times v^\varepsilon|_{y=0}=0$ be included.

	Currently, there are relatively few studies concerning the boundary layer solution of system (\ref{nvu equation}). Researches have be concentrated on the boundary layer problems of either Keller-Segel system or Navier-Stokes system.
    If there is no influence of fluid velocity $u$, (\ref{nvu equation}) is simplified as a viscous parabolic system:
	\begin{eqnarray}\label{nv equation}
		\left\{
		\begin{split}{}
			&\partial_tn^\varepsilon-\nabla\cdot(n^\varepsilon v^\varepsilon)=\Delta n^\varepsilon,\\
			&\partial_tv^\varepsilon+\nabla(\varepsilon(v^\varepsilon)^2-n^\varepsilon)=\varepsilon\Delta v^\varepsilon,\\
			&n^\varepsilon|_{t=0}=n_{\rm{in}}, v^\varepsilon|_{t=0}=v_{\rm{in}},
		\end{split}
		\right.
	\end{eqnarray}
    For the results about the well-posedness problems $\varepsilon\geq0$ can be referred to \cite{LL,LP, RWW} and their references. The existence of boundary layer behavior for the system(\ref{nv equation}) was verified by Li-Zhao\cite{LZ}. By imposing Dirichlet boundary conditions on the one-dimensional interval (0,1), the existence of boundary layer solutions to the system (\ref{nv equation}) was established as $\varepsilon\rightarrow0$ and the formal asymptotic analysis demonstrated that the boundary layer thickness was $\varepsilon^\frac{1}{2}$ in \cite{HWZ} by Hou-Wang-Zhao.
    Hou-Liu-Wang-Wang\cite{HL} proved that for $\varepsilon>0$, the solution to equation (\ref{nv equation}) converged to the solution corresponding to $\varepsilon=0$ with the optimal convergence rate $O(\varepsilon^\frac{1}{2})$, and the relation of outer and inner solutions could be explicitly identified. Meanwhile, Peng-Wang-Zhao\cite{PW} proved the existence of boundary layer profiles of the system (\ref{nv equation}) in (0,1) in the limit $\varepsilon\rightarrow0$, and derived uniform $\varepsilon$-estimates by weighted energy estimates combining with the effective viscous flux method. Under Dirichlet boundary, Hou-Wang\cite{HW} rigorously derived the outer and inner structures of boundary layers for the system (\ref{nv equation}) in the two-dimensional half-plane and proved that as $\varepsilon\rightarrow0$, the boundary layer solutions converged to the superposition of the outer-layer solutions (i.e., the solutions for $\varepsilon=0$) and the inner-layer solutions.
   For further researches on boundary layers, one can refer to \cite{CHW,CLW,CLW1,LWY,MXW,MXW1} and the reference therein.

    The boundary layer arises from viscous effects induced by aerobic bacteria accumulating at the water surface, which further modulates flow velocity and oxygen dissolution kinetics (see \cite{DCC,TC}), hence the hydrodynamic effects are of fundamental significance. For the viscous Navier-Stokes system
	\begin{align}\label{u equation}
		\left\{
		\begin{aligned}{}
			&\partial_tu^\varepsilon-\varepsilon\Delta u^\varepsilon+(u^\varepsilon\cdot\nabla)u^\varepsilon+\nabla p^\varepsilon=0,\\
			&\nabla\cdot u^\varepsilon=0,\\
			&u^\varepsilon|_{t=0}=u_{\rm{in}},
		\end{aligned}
		\right.
	\end{align}
    the well-posedness of the system can be referred to \cite{MR}. The boundary layer problem has been a fundamental topic in fluid mechanics due to the distortion of nonviscous flow by surrounding viscous forces, as observed by Prandtl in 1904 \cite{PRA}, and has attracted extensive studies. Under the assumption of monotonic outflow velocity, Oleinik-Samokhin\cite{OS} established the local existence of smooth solutions for the boundary layer problem of the two-dimensional Prandtl equation. In subsequent work, Xiao-Xin \cite{XX1} investigated the solvability, regularity and vanishing viscosity limit of 3D incompressible Navier-Stokes system with the slip boundary conditions. When the slip length depends on the viscosity, the detailed description for asymptotic behavior in two-dimensional and three-dimensional spaces and the $L^\infty$-convergence of the solutions to the system(\ref{u equation}) were derived by Wang-Wang-Xin\cite{WWX}. On the other hand, Iftimie-Sueur\cite{IS} studied the boundary layer problem involving weak amplitude linear behavior of two-dimensional and three-dimensional incompressible Navier-Stokes equations with fixed slip length, and obtained the $L^2$-convergence of the solutions to the system (\ref{u equation}) when $\varepsilon\rightarrow0$.  Wang-Xin-Zang\cite{WXZ} extended this result and obtained the \(L^\infty\) convergence in three dimensions. Xiao-Xin {\cite{XX}} pointed out that it is impossible for $H^2$ convergence in general three dimensional domains, and the $H^1$-convergence was obtained at a rate of $O(\varepsilon)$ for the complete slip boundary conditions. Recently, Tao-Wang-Zhang \cite{twz} obtained the zero-viscosity limit for the Navier-Stokes system in the two-dimensional half-plane with the Navier friction boundary condition. For more related results, ones can refer to \cite{CLX,FGL1,FGL2,FTZ,IP,JM,SC1,SC2,TAO,KE,WX,WYZ1,WYZ2,XY} and their referees.

    Compared with the mentioned independent studies above concerning boundary layers for chemotaxis and fluid flows, this paper further considers boundary-layer effects of the coupled chemotaxis-Navier-Stokes system (\ref{nvu equation}) in the half-plane $\Omega=\R^2_+=\{(x,y)\in\R^2|y>0\}$, where $\partial\Omega=\{(x,y)\in\R^2|y=0\}$, and the outward unit normal vector on the boundary is given by $\vec{n}=(0,-1)$. According to the boundary layer theory\cite{PRA,SG}, the solutions $(n^\varepsilon,v^\varepsilon,u^\varepsilon)$ to the system(\ref{nvu equation}) with (\ref{nvu condition}) decompose into two parts for small $\varepsilon>0$:  inner boundary-layer profile and outer profile(the solution corresponding to $\varepsilon=0$). Furthermore, the boundary conditions for $n^\varepsilon$ remain consistent from $\varepsilon>0$ to $\varepsilon=0$, so no interior boundary layer arises for the $n^\varepsilon$-component. We denote the solutions to the system(\ref{nvu equation}) with (\ref{nvu condition}) corresponding to $\varepsilon>0$ and $\varepsilon=0$ by $(n^\varepsilon,v^\varepsilon,u^\varepsilon)$ and $(n^0,v^0,u^0)$, respectively. For $0<\alpha,\gamma, \beta\leq\frac{1}{2}$, one anticipates that $(n^\varepsilon,v^\varepsilon,u^\varepsilon)$ admits the following decomposition:
    \begin{align}\label{convergence}
       n^\varepsilon(t,x,y)&=n^0(t,x,y)+o(\varepsilon^\alpha),\notag\\
       v^\varepsilon(t,x,y)&=v^0(t,x,y)+(v^{B,0}_1(t,x,\frac{y}{\sqrt{\varepsilon}}),v^{B,0}_2(t,x,\frac{y}{\sqrt{\varepsilon}}))+o(\varepsilon^\beta),\\
       u^\varepsilon(t,x,y)&=u^0(t,x,y)+o(\varepsilon^\gamma),\notag\notag
    \end{align}
     where the outer profile $(n^0,v^0,u^0)=(n^0,v^0_1,v^0_2,u^0_1,u^0_2)$ is the solution to the system (\ref{nvu equation})-(\ref{nvu condition}) with $\varepsilon=0$, and the inner profile $(v^{B,0}_1,v^{B,0}_2)$ describes the sharp transition from a value away from the boundary layer to another value on the physical boundary.

     Now we state our main results. Some notations are given for convenience firstly.

     \textbf{Notations 1.1}
        $N$ represents the set of positive integers.
	In the following sections, we use $L^p_{xy}$ and $L^p_{xz}$ to denote the Lebegue spaces $L^p(\R\times\R_+)$, with respect to $(x,y)$ and $(x,z)$, respectively, with corresponding norms $\|\cdot\|_{L^p_{xy}}$ and $\|\cdot\|_{L^p_{xz}}$ for $1\leq p\leq\infty$. $H^q_{xy}$ and $H^q_{xz}$ for $q\in N$ represent the Sobolev space $W^{q,2}(\R\times\R_+)$, with respect to $(x,y)$ and $(x,z)$, respectively, with corresponding norms $\|\cdot\|_{H^q_{xy}}$ and $\|\cdot\|_{H^q_{xz}}$.
    \begin{theorem}\label{th nvu L infty convergence}
		{Assume that the initial data satisfy $(n_{\text{in}},v_{\text{in}},u_{\text{in}})\in H^{14}_{xy}\times H^{14}_{xy}\times H^{15}_{xy}$. Let $(n^0,v^0,u^0)$ be the solution obtained from Theorem 2.1 in Part I \cite{WWZ}, and let $0<T\leq T_{\text{max}}$. Then there exists $\varepsilon_0>0$ such that for every $0<\varepsilon\leq\varepsilon_0$, the problem $(\ref{nvu equation})$-$(\ref{nvu condition})_1$ admits a unique solution $(n^\varepsilon,v^\varepsilon,u^\varepsilon)$ satisfying the regularity of Theorem \ref{th nvu regularity}, and the following holds on $[0,T]$:
			\begin{align}\label{nvu varepsilon L infty}
			&\sup_{t\in[0,T]}\|n^\varepsilon(t,x,y)-n^0(t,x,y)\|_{L^\infty_{xy}}\leq C\varepsilon^{\frac{1}{4}},\notag\\
				&\sup_{t\in[0,T]}\|v^\varepsilon(t,x,y)-v^0(t,x,y)-(0,v^{B,0}_2)(t,x,\frac{y}{\sqrt{\varepsilon}})\|_{L^\infty_{xy}}\leq C\varepsilon^{\frac{1}{4}},\\
				&\sup_{t\in[0,T]}\|u^\varepsilon(t,x,y)-u^0(t,x,y)\|_{L^\infty_{xy}}\leq C\varepsilon^{\frac{1}{2}},\notag\notag
	\end{align}		
			where $C$ is a constant independent of $\varepsilon$.}				
	\end{theorem}
    \begin{remark}
    The system $(\ref{nvu equation})$-$(\ref{nvu condition})_1$ become Keller-Segel system with logarithmic when one don't consider the influence of fluid, and compared to {\rm{\cite{HW}}} we consider the boundary layer problems under different boundary conditions, which is the slip boundary conditions. Compared to \cite{HO,HO1,LSW}, we consider the different system that is one with viscosity dissipation fluid terms and logarithmic singularity. By overcoming the difficulties from the imbalance between viscosity dissipation and non-viscous nonlinear terms, we get the $L^\infty_{xy}$-convergence results $(\ref{nvu varepsilon L infty})$ eventually. The result is new.
    \end{remark}

        Next, we turn the results obtained form Theorem $\ref{th nvu L infty convergence}$ to the original chemotaxis-Navier-Stokes system $(\ref{ncu equation})$. Note that the boundary {conditions} in $(\ref{nvu condition})$ for $v^\varepsilon$ is equivalent to $\partial_yc^\varepsilon=0$ by simple calculation. Then the corresponding initial-boundary value problem of the original chemotaxis-Navier-Stokes system $(\ref{ncu equation})$ reads as
     \begin{align}\label{ncu varepsilon equation}
		\begin{cases}
		  \partial_tn^\varepsilon=\Delta n^\varepsilon-u^\varepsilon\cdot\nabla n^\varepsilon-\chi\nabla\cdot(\frac{n^\varepsilon}{c^\varepsilon}\nabla c^\varepsilon),\\
			\partial_tc^\varepsilon=\varepsilon\Delta c^\varepsilon-u^\varepsilon\cdot\nabla c^\varepsilon-n^\varepsilon c^\varepsilon,\\
			\partial_tu^\varepsilon=\varepsilon\Delta u^\varepsilon-(u^\varepsilon\cdot\nabla)u^\varepsilon-\nabla p^\varepsilon+n^\varepsilon e_2,\\
			\nabla\cdot u^\varepsilon=0,\\
		(n^\varepsilon,c^\varepsilon,u^\varepsilon)|_{t=0}=(n_{in},c_{in},u_{in}),\\
        \partial_yn^\varepsilon+n^\varepsilon\partial_yc^\varepsilon|_{y=0}=0,\;\partial_yc^\varepsilon|_{y=0}=0,\;\partial_yu_1^\varepsilon|_{y=0}=0,\;u^\varepsilon_2|_{y=0}=0,\;\;\; if \;\; \varepsilon>0,\\
        \partial_yn^0+n^0\partial_yc^0|_{y=0}=0,\;u^0_2|_{y=0}=0,\;\;\; if \;\; \varepsilon=0,\\
		\end{cases}
	\end{align}

         By Theorem $\ref{th nvu L infty convergence}$, we get the following results for the system $(\ref{ncu varepsilon equation})$.

         \begin{theorem}\label{th ncu L infty}
		{Suppose $(n_{\rm{in}},\ln c_{\rm{in}},u_{\rm{in}})\in H^{14}_{xy}\times H^{15}_{xy}\times H^{15}_{xy}$ with $n_{\rm{in}}\geq0$, $c_{\rm{in}}>0$. Then $(\ref{ncu varepsilon equation})$ admits a unique solution $(n^\varepsilon,c^\varepsilon,u^\varepsilon)$ on $[0,T]$ such that
		\begin{equation}\label{Equ(2.22)}
		\begin{aligned}{}
        &\sup_{t\in[0,T]}\|c^\varepsilon(t,x,y)-c^0(t,x,y)\|_{L^2_{xy}}\leq C\varepsilon^{\frac{1}{2}},\\
        &\sup_{t\in[0,T]}\|c^\varepsilon(t,x,y)-c^0(t,x,y)\|_{L^\infty_{xy}}\leq C\varepsilon^{\frac{1}{4}}.\\
        \end{aligned}
	\end{equation}

        {Since} the convergence of $n^\varepsilon$ and $u^\varepsilon$ is the same as in $(\ref{nvu varepsilon L infty})$ and $(\ref{L^2-convergence})$, we omit the description here,
       {where} $C$ is a constant independent of $\varepsilon$.		
		}
	\end{theorem}
    \begin{remark}
    By directly computation, the boundary layers be presented in the derivative of $c^\varepsilon$
    ,i.e. $$\sup\limits_{t\in[0,T]}\|\nabla c^\varepsilon(t,x,y)-\nabla c^0(t,x,y)-(0,c^0(t,x,y)v^{B,0}_2)(t,x,\frac{y}{\sqrt{\varepsilon}})\|_{L^\infty_{xy}}\leq C\varepsilon^{\frac{1}{4}},$$
    which means the chemical diffusion rate $\varepsilon$ plays an important role for the system and cannot be neglected.
 \end{remark}

	The rest of the paper is organized as follows.
    First, in Sect.\ref{part2}, we derive the boundary layer remainder equations for the system $(\ref{nvu equation})$-$(\ref{nvu condition})_1$, together with the proofs of our main theorems. In Sect.\ref{part3}, we derive several prior estimates. In Sect.\ref{part4}, the uniform bound is derived for the remainder. In Sect.\ref{5}, we establish the well-posedness of the system $(\ref{nvu equation})$-$(\ref{nvu condition})_1$ in conormal Sobolev spaces.
	
	\section{Remainder results and proof of main Theorem }\label{part2}
    \subsection{Remainder Equations}
	By the detailed derivation of boundary layer profiles in Section 3 of Part I \cite{WWZ}, we expand the solutions (\(n^\varepsilon, v^\varepsilon, u^\varepsilon\)) to the system (\ref{nvu equation}) as follows:
    \begin{align*}
        \begin{cases}
        n^\varepsilon=n^0+\sqrt{\varepsilon} n^{B,1}+\varepsilon n^{B,2}+\varepsilon^\frac{1}{2}H^\varepsilon,\\
		v^\varepsilon=v^0+(0,v^{B,0}_2)+\sqrt{\varepsilon} v^{B,1}+\varepsilon(v^{B,2}_1,0)+\varepsilon^\frac{1}{2}V^\varepsilon,\\
        u^\varepsilon=u^0+\sqrt{\varepsilon}(u^{B,1}_1,0)+\varepsilon u^{B,2}+\varepsilon^\frac{3}{2}(0, u^{B,3}_2)+\varepsilon^\frac{1}{2}U^\varepsilon,\\
        p^\varepsilon=p^0+\varepsilon p^{B,2}+\varepsilon^\frac{1}{2}K^\varepsilon,\\
        \end{cases}
    \end{align*}
    where $H^\varepsilon, V^\varepsilon$, $U^\varepsilon$ and $K^\varepsilon$ denote $\varepsilon^\frac{j-1}{2}n^{I,j}+\varepsilon^\frac{j}{2}n^{B,j+1}$, $\varepsilon^\frac{j-1}{2}v^{I,j}+(\varepsilon^\frac{j}{2}v^{B,j+1}_1,0)+(0,\varepsilon^\frac{j-1}{2}v^{B,j}_2)$, $\varepsilon^\frac{j-1}{2}u^{I,j}+(\varepsilon^\frac{j}{2}u^{B,j+1}_1,0)+(0,\varepsilon^\frac{j+1}{2}u^{B,j+2}_2)$ and $\varepsilon^\frac{j-1}{2}p^{I,j}+\varepsilon^\frac{j}{2}p^{B,j+1}$, respectively, for $j\geq2$.
    These higher-order aggregated terms \((H^\varepsilon, V^\varepsilon, U^\varepsilon)\) are referred to as boundary layer remainders. For the convenience of estimating these remainders later on, we first introduce the following approximations:
         \begin{eqnarray}\label{HVUK theta}
		\left\{
		\begin{split}{}
			&H^\theta=n^0+\sqrt{\varepsilon} n^{B,1}+\varepsilon n^{B,2},\\
			&V^\theta=v^0+(0,v^{B,0}_2)+\sqrt{\varepsilon} v^{B,1}+\varepsilon(v^{B,2}_1,0),\\
                &U^\theta=u^0+\sqrt{\varepsilon}(u^{B,1}_1,0)+\varepsilon u^{B,2}+\varepsilon^\frac{3}{2}(0, u^{B,3}_2),\\
                &K^\theta=p^0+\varepsilon p^{B,2},\\
		\end{split}
		\right.
	\end{eqnarray}
        then the remainders $(H^\varepsilon,V^\varepsilon,U^\varepsilon)$ be as follows
        \begin{eqnarray}\label{HVUK varepsilon}
		\begin{split}{}
			H^\varepsilon&=\varepsilon^{-\frac{1}{2}}(n^\varepsilon-H^\theta),\;\;
			V^\varepsilon=\varepsilon^{-\frac{1}{2}}(v^\varepsilon-V^\theta),\\
                U^\varepsilon&=\varepsilon^{-\frac{1}{2}}(u^\varepsilon-U^\theta),\;\;
                K^\varepsilon=\varepsilon^{-\frac{1}{2}}(p^\varepsilon-K^\theta).\\
		\end{split}
	\end{eqnarray}
        The remaining part is directly derived from $(\ref{nvu equation})$-$(\ref{nvu condition})_1$ as follows
	\begin{align}\label{HVU equation}
		\begin{cases}
			\partial_tH^\varepsilon=\Delta H^\varepsilon-\varepsilon^\frac{1}{2} U^\varepsilon\cdot\nabla H^\varepsilon-U^\varepsilon\cdot\nabla H^\theta-U^\theta\cdot\nabla H^\varepsilon+\varepsilon^\frac{1}{2}\nabla\cdot(H^\varepsilon V^\varepsilon)\\
            \quad\quad\quad+\nabla\cdot(H^\varepsilon V^\theta)+\nabla\cdot(H^\theta V^\varepsilon)+\varepsilon^{-\frac{1}{2}}f^\varepsilon,\\
			\partial_tV^\varepsilon=\varepsilon\Delta V^\varepsilon-\varepsilon^\frac{1}{2}\nabla(U^\varepsilon\cdot V^\varepsilon)-\nabla(U^\varepsilon\cdot V^\theta)-\nabla(U^\theta\cdot V^\varepsilon)-\varepsilon^\frac{3}{2}\nabla(|V^\varepsilon|^2)\\
            \quad\quad\quad-2\varepsilon\nabla(V^\varepsilon\cdot V^\theta)+\nabla H^\varepsilon+\varepsilon^{-\frac{1}{2}}g^\varepsilon,\\
            \nabla\times V^\varepsilon=0,\\
                \partial_tU^\varepsilon=\varepsilon\Delta U^\varepsilon-\varepsilon^\frac{1}{2} U^\varepsilon\cdot\nabla U^\varepsilon-U^\varepsilon\cdot\nabla U^\theta-U^\theta\cdot\nabla U^\varepsilon-\nabla K^\varepsilon+H^\varepsilon e_2+\varepsilon^{-\frac{1}{2}}h^\varepsilon,\\
                \nabla\cdot U^\varepsilon=0,\\
                (H^\varepsilon,V^\varepsilon,U^\varepsilon)(0,x,y)=(0,0,0),\\
            \partial_y H^\varepsilon(t,x,0)=0,\;\; (\partial_yV^\varepsilon_1, \partial_yU^\varepsilon_1)(t,x,0)=(0,0),\;\; V^\varepsilon_2(t,x,0)=0,\\
            U^\varepsilon_2(t,x,0)=-\varepsilon^\frac{1}{2}u^{B,2}_2(t,x,0)-\varepsilon u^{B,3}_2(t,x,0),\\
		\end{cases}
	\end{align}
	where
	\begin{align}\label{fgh varepsilon}
        &f^\varepsilon=-\partial_tH^\theta+\Delta H^\theta-U^\theta\cdot\nabla H^\theta+\nabla\cdot(H^\theta V^\theta),\notag\\
		&g^\varepsilon=-\partial_tV^\theta+\varepsilon\Delta V^\theta-\nabla(U^\theta\cdot V^\theta)-\varepsilon\nabla(|V^\theta|^2)+\nabla H^\theta,\\
        &h^\varepsilon=-\partial_tU^\theta+\varepsilon\Delta U^\theta-U^\theta\cdot\nabla U^\theta-\nabla K^\theta+H^\theta{e_2}.\notag\notag
	\end{align}
    \subsection{Prior results}
    \begin{lemma}\label{lem n0v0u0}
		{{\rm(see Theorem 2.1 in \cite{WWZ})} Assume that the initial data satisfy
        $$(n_{in},v_{in},u_{in})\in H^{14}_{xy}\times H^{14}_{xy}\times H^{15}_{xy},\;\;n_{in}\geq0,\;\;\nabla\times v_{in}=0,$$ then there exists the time $T>0$ such that the system $(\ref{nvu equation})$-$(\ref{nvu condition})_2$ has the unique solutions $(n^0,v^0,u^0)$ on $[0,T]$ satisfying
        \begin{align*}\label{n^0v^0u^0 Hm}
           \|n^0\|_{H^{13}_{xy}} +\|v^0\|_{H^{13}_{xy}}+\|u^0\|_{H^{14}_{xy}}+ \int_0^t \|n^0(\tau)\|^2_{H^{13}_{xy}}d\tau\leq C,
        \end{align*}
        and
        \begin{align*}
          \|\partial^j_tn^0\|_{H^{14-2j}_{xy}}+\|\partial^j_tv^0\|_{H^{14-2j}_{xy}}&\leq C,\;0<2j\leq 14,\notag\\
           \|\partial^j_tu^0\|_{H^{15-2j}_{xy}}&\leq C,\;0<2j\leq 15,
        \end{align*}
          }
       where $C$ is a constant independent of the viscosity.
	\end{lemma}
    Throughout this paper, $C$ denotes a generic constant which may vary from line to line and is independent of the viscosity.

    \begin{lemma}\label{lem inner layer}
		{{\rm(see Theorem 2.2 in \cite{WWZ})} Let $(n_{in},v_{in},u_{in})$ satisfies the assumptions in Lemma \ref{lem n0v0u0}. Then there exists the time $T>0$ such that inner boundary layer profiles $$v^{B,0}_2,n^{B,1},v^{B,1}_1,v^{B,1}_2,u^{B,1}_1,n^{B,2},v^{B,2}_1,u^{B,2}_1,u^{B,2}_2,u^{B,3}_2,p^{B,2}$$ satisfy
            \begin{align*}
                 &\|\partial^j_tv^{B,0}_2\|^2_{L^\infty(0,T;\mathcal{H}^{k,10-j,0})\cap L^2(0,T;\mathcal{H}^{k,10-j,1})}\leq C,&\|\partial^j_tn^{B,1}\|^2_{L^\infty(0,T;\mathcal{H}^{k,10-j,1})\cap L^2(0,T;\mathcal{H}^{k,10-j,2})}\leq C,\\
                 &\;\;\|\partial^j_tv^{B,1}_1\|^2_{L^\infty(0,T;\mathcal{H}^{k,9-j,0})\cap L^2(0,T;\mathcal{H}^{k,9-j,1})}\leq C,&\|\partial^j_tv^{B,1}_2\|^2_{L^\infty(0,T;\mathcal{H}^{k,8-j,0})\cap L^2(0,T;\mathcal{H}^{k,8-j,1})}\leq C,\\
                 &\|\partial^j_tu^{B,1}_1\|^2_{L^\infty(0,T;\mathcal{H}^{k,10-j,0})\cap L^2(0,T;\mathcal{H}^{k,10-j,1})}\leq C,&\|\partial^j_tn^{B,2}\|^2_{L^\infty(0,T;\mathcal{H}^{k,8-j,1})\cap L^2(0,T;\mathcal{H}^{k,8-j,2})}\leq C,\\
                &\;\;\|\partial^j_tv^{B,2}_1\|^2_{L^\infty(0,T;\mathcal{H}^{k,7-j,0})\cap L^2(0,T;\mathcal{H}^{k,7-j,1})}\leq C,
                &\|\partial^j_tu^{B,2}_1\|^2_{L^\infty(0,T;\mathcal{H}^{k,8-j,0})\cap L^2(0,T;\mathcal{H}^{k,8-j,1})}\leq C,\\
                &\;\;\|\partial^j_tu^{B,2}_2\|^2_{L^\infty(0,T;\mathcal{H}^{k,9-j,1})\cap L^2(0,T;\mathcal{H}^{k,9-j,2})}\leq C,&\|\partial^j_tu^{B,3}_2\|^2_{L^\infty(0,T;\mathcal{H}^{k,7-j,1})\cap L^2(0,T;\mathcal{H}^{k,7-j,2})}\leq C,\\
                &\|\partial^j_tp^{B,2}\|^2_{L^\infty(0,T;\mathcal{H}^{k,10-j,2})\cap L^2(0,T;\mathcal{H}^{k,10-j,3})}\leq C,
			\end{align*}
            {\rm for $j=0,1,2$ and}
             \begin{align*}
               &\|\partial_tv^{B,0}_2\|^2_{L^\infty(0,T; \mathcal{H}^{k,8-r,r+1})}\leq C,&\|\partial_tn^{B,1}\|^2_{L^\infty(0,T;\mathcal{H}^{k,8-r,r+2})}\leq C,\quad&\|\partial_tv^{B,1}_1\|^2_{L^\infty(0,T;\mathcal{H}^{k,7-r,r+1})}\leq C,\\
               &\|\partial_tv^{B,1}_2\|^2_{L^\infty(0,T;\mathcal{H}^{k,6-r,r+1})}\leq C,&\|\partial_tu^{B,1}_1)\|^2_{L^\infty(0,T; \mathcal{H}^{k,8-r,r+1})}\leq C,\quad&\|\partial_tn^{B,2}\|^2_{L^\infty(0,T;\mathcal{H}^{k,6-r,r+2})}\leq C,\\
               &\|\partial_tv^{B,2}_1\|^2_{L^\infty(0,T;\mathcal{H}^{k,5-r,r+1})}\leq C,&\|\partial_tu^{B,2}_1\|^2_{L^\infty(0,T;\mathcal{H}^{k,6-r,r+1})}\leq C,\quad&\|\partial_tu^{B,2}_2\|^2_{L^\infty(0,T;\mathcal{H}^{k,7-r,r+2})}\leq C,\\
               &\|\partial_tu^{B,3}_2\|^2_{L^\infty(0,T;\mathcal{H}^{k,5-r,r+2})}\leq C,&\|\partial_tp^{B,2}\|^2_{L^\infty(0,T;\mathcal{H}^{k,8-r,r+3})}\leq C,\quad&
             \end{align*}
              {\rm for $r=0,1$ as well as}
			\begin{align*}
                &\;\;\|v^{B,0}_2\|^2_{L^\infty(0,T;\mathcal{H}^{k,10-i,i})}\leq C,\;\;&\|n^{B,1}\|^2_{L^\infty(0,T;\mathcal{H}^{k,10-i,i+1})}\leq C,\quad&\;\;\;\|v^{B,1}_1\|^2_{L^\infty(0,T;\mathcal{H}^{k,9-i,i})}\leq C,\\
                &\;\;\;\|v^{B,1}_2\|^2_{L^\infty(0,T;\mathcal{H}^{k,8-i,i})}\leq C,&\|u^{B,1}_1\|^2_{L^\infty(0,T;\mathcal{H}^{k,10-i,i})}\leq C,\quad&\|n^{B,2}\|^2_{L^\infty(0,T;\mathcal{H}^{k,8-i,i+1})}\leq C,\\
                &\;\;\;\|v^{B,2}_1\|^2_{L^\infty(0,T;\mathcal{H}^{k,7-i,i})}\leq C,&\|u^{B,2}_1\|^2_{L^\infty(0,T;\mathcal{H}^{k,8-i,i})}\leq C,\quad&\|u^{B,2}_2\|^2_{L^\infty(0,T;\mathcal{H}^{k,9-i,i+1})}\leq C,\\
                &\|u^{B,3}_2\|^2_{L^\infty(0,T;\mathcal{H}^{k,7-i,i+1})}\leq C,&\|p^{B,2}\|^2_{L^\infty(0,T;\mathcal{H}^{k,10-i,i+2})}\leq C,\quad&
			\end{align*}}
	\end{lemma}
    for $i=1,2,3,4$, where we denote $\mathcal{H}^{k,m,l}$ by the anisotropic Sobolev space as follows: for $k,m,l\in\mathbb{N}$,
	\begin{align*}
			&\mathcal{H}^{k,m,l}:=\{g(x,z)\in L^2(\mathbb{R}\times\mathbb{R}_+):(1+z^{2k})^\frac{1}{2}\partial^\alpha_x\partial^\gamma_zg(x,z)\in L^2(\mathbb{R}\times\mathbb{R}_+),|\alpha|\leq m,\gamma\in N,|\gamma|\leq l\}.
	\end{align*}

        For the proof of the main results, the local well-posedness results of $(\ref{nvu equation})$-$(\ref{nvu condition})_1$ with $\varepsilon>0$ be required. As in \cite{MR,MW1}, the conormal derivative $\partial^\alpha=\partial^{\alpha_1}_x\psi^{\alpha_2}(y)\partial^{\alpha_2}_y$ for $\alpha=(\alpha_1,\alpha_2)$ be introduced, where $\psi(y)$ is a smooth function defined by
        \begin{equation}\label{psi}
        \psi(y)=\left\{
        \begin{split}{}
        &\delta y\quad\quad\;\;{\rm for}\;y\leq\frac{1}{2},\\
        &\frac{\delta y}{1+y}\quad{\rm for}\;y\geq1,
        \end{split}
        \right.
		\end{equation}
        where $\delta>0$ is to be decided later. For $m\in N$, the conormal sobolev spaces are defined as follows:
        \begin{equation*}
        \begin{split}{}
        \widetilde{H}^m(\R^2_+)=\Big{\{}u\Big{|}\|u\|^2_{\widetilde{H}^m}=\sum_{|\alpha|\leq m}\|\partial^\alpha u\|^2_2<\infty\Big{\}}.
        \end{split}
		\end{equation*}
        Moreover, we say that $\beta\le\alpha$, provided that $\alpha=(\alpha_1,\alpha_2)$ and $\beta=(\beta_1,\beta_2)$ satisfy $\beta_1\le\alpha_1$ and $\beta_2\le\alpha_2$.

        \begin{theorem}\label{th nvu regularity}
		{Let $(n_{\rm{in}},v_{\rm{in}},u_{\rm{in}})\in H^m_{xy}\times H^m_{xy}\times H^{m+1}_{xy}$, then there exists the time $0<T< T^*$ $(T^*<\infty)$ such that $(\ref{nvu equation})$-$(\ref{nvu condition})_1$ admits the unique solution $(n^\varepsilon,v^\varepsilon,u^\varepsilon)$ on $[0,T]$ for $m\geq5$,
			\begin{equation*}
				\begin{split}{}
            &\|n^\varepsilon\|^2_{L^\infty_TL^2_{xy}}+\|v^\varepsilon\|^2_{L^\infty_TL^2_{xy}}+\|u^\varepsilon\|^2_{L^\infty_TL^2_{xy}}+\|\nabla n^\varepsilon\|^2_{L^\infty_T\widetilde{H}^m_{xy}}+\|\nabla v^\varepsilon\|^2_{L^\infty_T\widetilde{H}^m_{xy}}+\|\nabla u^\varepsilon\|^2_{L^\infty_T\widetilde{H}^{m+1}_{xy}}\\
             &+\|\nabla^2 n^\varepsilon\|^2_{L^2_T\widetilde{H}^m_{xy}}+\varepsilon\|\nabla^2 v^\varepsilon\|^2_{L^2_T\widetilde{H}^m_{xy}}+\varepsilon\|\nabla^2 u^\varepsilon\|^2_{L^2_T\widetilde{H}^{m+1}_{xy}}+\|\omega^\varepsilon\|^2_{L^\infty_TL^\infty_{xy}}+\|\omega^\varepsilon\|^2_{L^\infty_TW^{1,\infty}_{xy}}\leq C,\\
				\end{split}
			\end{equation*}
		}
	\end{theorem}
        where $\omega^\varepsilon=\partial_xu^\varepsilon_2-\partial_yu^\varepsilon_1$ is the vorticity. The detailed proof of Theorem \ref{th nvu regularity} is included in Appendix of Section \ref{5}.

        \begin{remark}\label{re n v u varepsilon regularity}
		 {We expect the system $(\ref{nvu equation})$ to own the uniform higher-order regularity in general Sobolev space when $\varepsilon>0$, while it fails because of the intricate relationship between viscous dissipation term and non-viscous nonlinear term $\nabla(u^\varepsilon\cdot v^\varepsilon)$. To overcome the difficulties, we turn to establish the poor regularity estimate of the solution in the conormal Sobolev spaces. Owing to the $H^1_{xy}$ regularity of $(n^\varepsilon,v^\varepsilon,u^\varepsilon)$ be required during the proof, we establish the regularity of $(\|\nabla n^\varepsilon\|_{\widetilde{H}^m_{xy}}, \|\nabla v^\varepsilon\|_{\widetilde{H}^m_{xy}}, \|\nabla u^\varepsilon\|_{\widetilde{H}^{m+1}_{xy}})$ and the difficult term $\|\psi(y)\partial_y\omega^\varepsilon\|_{L^\infty_{xy}}$ be dealed with by the maximum principle referred to \cite{MR}.}
	\end{remark}

   \subsection{Remainder and $L^2$-convergence results} We first establish the $L^2_{xy}$ estimate for the remainders.
	\begin{theorem}\label{th L2 convergence}
		{Let \((n_{\rm{in}}, v_{\rm{in}}, u_{\rm{in}}) \in H^{14}_{xy} \times H^{14}_{xy} \times H^{15}_{xy}\) , and \((n^\varepsilon, v^\varepsilon, u^\varepsilon)\) be the solution to the system $(\ref{nvu equation})$-$(\ref{nvu condition})_1$. Then the system {\rm(\ref{HVU equation})} admits a unique solution \((H^\varepsilon, V^\varepsilon, U^\varepsilon)\) on \([0,T]\) satisfying the following:
				\begin{equation}\label{HVU L^2xy estimate}
                \begin{split}{}
					\sup_{t\in[0,T]}(\|H^\varepsilon(t,x,y)\|_{L^2_{xy}}+\|V^\varepsilon(t,x,y)\|_{L^2_{xy}}+\|U^\varepsilon(t,x,y)\|_{L^2_{xy}}+\|U^\varepsilon(t,x,y)\|_{H^1_{xy}})&\leq C\varepsilon^{\frac{1}{4}}.\\
				\end{split}
			\end{equation}
             {Furthermore}, there exists $\varepsilon_0>0$ such that for each $0<\varepsilon\leq\varepsilon_0$, the following holds:
			\begin{equation}\label{L^2-convergence}
		\begin{split}{}
			&\sup_{t\in[0,T]}\|n^\varepsilon(t,x,y)-n^0(t,x,y)\|_{L^2_{xy}}\leq C\varepsilon^{\frac{1}{2}},\\
				&\sup_{t\in[0,T]}\|v^\varepsilon(t,x,y)-v^0(t,x,y)-(0,v^{B,0}_2)(t,x,\frac{y}{\sqrt{\varepsilon}})\|_{L^2_{xy}}\leq C\varepsilon^{\frac{1}{2}},\\
				&\sup_{t\in[0,T]}\|u^\varepsilon(t,x,y)-u^0(t,x,y)\|_{L^2_{xy}}\leq C\varepsilon^{\frac{1}{2}},\\
        \end{split}
	\end{equation}			
		}
	\end{theorem}
    The detailed proof of Theorem \ref{th L2 convergence} be given in Section \ref{L2 convergence}.

    \begin{remark}
     Theorem \ref{th L2 convergence} establishes $L^2_{xy}$-convergence  rate $O(\varepsilon^\frac{1}{2})$, which is the sharp result. The key issues are to handle the terms involving $\varepsilon^\frac{1}{2}\nabla(U^\varepsilon\cdot V^\varepsilon)$ and $\nabla(U^\varepsilon\cdot V^\theta)$. For the term involving $\varepsilon^\frac{1}{2}\nabla(U^\varepsilon\cdot V^\varepsilon)$, since the viscosity be need to assign for $V^\varepsilon$, which lead to make the $H^1_{xy}$ regularity of $U^\varepsilon$ required. To establish the $H^1_{xy}$ estimations of $U^\varepsilon$, the lower $L^2_{xy}$ regularity be required, while the unknown pressure makes the proof difficult, hence we introduce the vorticity estimation and the Leray orthogonal projection to deal with it. Besides, a cut-off function \(\varphi\) be applied to overcome the higher regularity requirement involving the boundary term. For the terms involving $\nabla(U^\varepsilon\cdot V^\theta)$, they are very difficult to be controlled because of the absence of viscosity. Hence we appeal to the conormal Sobolev spaces to get a weak regularity result of the viscous solutions to serve for the reminders and $L^2_{xy}$-convergence estimations.
    \end{remark}
    Secondly, the bounds for tangential and normal derivatives of the remainder are established.
    \begin{theorem}\label{th L infty convergence}
		{Assume that the initial data satisfy \((n_{\rm{in}}, v_{\rm{in}}, u_{\rm{in}}) \in H^{14}_{xy} \times H^{14}_{xy} \times H^{15}_{xy}\), and \((n^\varepsilon, v^\varepsilon, u^\varepsilon)\) be the solution to the system $(\ref{nvu equation})$-$(\ref{nvu condition})_1$. Then the remainders \((H^\varepsilon, V^\varepsilon, U^\varepsilon)\) satisfy the following:
				\begin{equation}\label{HVU varepsilon L infty}
                \begin{split}{}
                    \sup_{t\in[0,T]}(\|\partial_xH^\varepsilon(t,x,y)\|_{L^2_{xy}}+\|\partial_xV^\varepsilon(t,x,y)\|_{L^2_{xy}}+\|\partial_x\nabla U^\varepsilon(t,x,y)\|_{L^2_{xy}})&\leq C\varepsilon^{\frac{1}{4}},\\
                    \sup_{t\in[0,T]}(\varepsilon^{\frac{1}{4}}\|\partial_yH^\varepsilon(t,x,y)\|_{L^2_{xy}}+\varepsilon^{\frac{1}{4}}\|\partial_yV^\varepsilon(t,x,y)\|_{L^2_{xy}}+\varepsilon^{\frac{1}{4}}\|\partial_y\nabla U^\varepsilon(t,x,y)\|_{L^2_{xy}})&\leq C,\\
                    \sup_{t\in[0,T]}(\varepsilon^{\frac{3}{4}}\|\partial_x\partial_yH^\varepsilon(t,x,y)\|_{L^2_{xy}}+\varepsilon^{\frac{3}{4}}\|\partial_x\partial_yV^\varepsilon(t,x,y)\|_{L^2_{xy}})&\leq C.\\
				\end{split}
			\end{equation}
            }
	\end{theorem}
     The detailed proof of Theorem \ref{th L infty convergence} be given in Section \ref{de hvu}.

    \begin{remark}
    During the proof, to overcome difficulty from the high order normal derivative, we take advantage of the $L^2$ estimates of \((H^\varepsilon, V^\varepsilon, U^\varepsilon)\), the $L^\infty$ estimates of the higher derivative of \((H^\theta, V^\theta, U^\theta)\) and anisotropic Sobolev embedding inequalities.
    \end{remark}

    \subsection{Proof of Theorem \ref{th nvu L infty convergence} and Theorem $\ref{th ncu L infty}$} We now proceed to prove Theorem $\ref{th nvu L infty convergence}$ by applying the prior estimates from Lemma \ref{lem n0v0u0}-Lemma \ref{lem inner layer} and Theorem \ref{th L2 convergence}-Theorem \ref{th L infty convergence}.\\

    \textbf{Proof of Theorem $\ref{th nvu L infty convergence}$.} Firstly, by the fact that $(H^\varepsilon,V^\varepsilon,U^\varepsilon)$ uniquely solves problem $(\ref{HVU equation})$, one {deduces that $n^\varepsilon= \varepsilon^\frac{1}{2}H^\varepsilon+H^\theta$, $v^\varepsilon =\varepsilon^\frac{1}{2}V^\varepsilon+V^\theta$} and $u^\varepsilon =\varepsilon^\frac{1}{2}U^\varepsilon+U^\theta$ {be} the unique solution of $(\ref{nvu equation})$-$(\ref{nvu condition})$ with $\varepsilon\in(0, \varepsilon_0]$ for $0<\varepsilon_0<1$. Thus
the regularity $(n^\varepsilon,v^\varepsilon,u^\varepsilon)\in L^\infty([0, T];L^\infty_{xy})$  follows from the fact that $(H^\varepsilon,V^\varepsilon,U^\varepsilon),\;(H^\theta,V^\theta,U^\theta)\in L^\infty([0, T];L^\infty_{xy})$. Then by Theorem $\ref{th L2 convergence}$-Theorem $\ref{th L infty convergence}$, we get
  \begin{equation}\label{HV L infty}
		\begin{split}{}
             \|H^\varepsilon\|_{L^\infty_TL^\infty_{xy}}&\leq C(\|H^\varepsilon\|^\frac{1}{2}_{L^\infty_TL^2_{xy}}\|\partial_yH^\varepsilon\|^\frac{1}{2}_{L^\infty_TL^2_{xy}}+\|\partial_xH^\varepsilon\|^\frac{1}{2}_{L^\infty_TL^2_{xy}}\|\partial_x\partial_yH^\varepsilon\|^\frac{1}{2}_{L^\infty_TL^2_{xy}})\\
             &\leq C\varepsilon^{\frac{1}{8}}\cdot\varepsilon^{-\frac{3}{8}}\\
             &\leq C\varepsilon^{-\frac{1}{4}},\\
             \|V^\varepsilon\|_{L^\infty_TL^\infty_{xy}}&\leq C(\|V^\varepsilon\|^\frac{1}{2}_{L^\infty_TL^2_{xy}}\|\partial_yV^\varepsilon\|^\frac{1}{2}_{L^\infty_TL^2_{xy}}+\|\partial_xV^\varepsilon\|^\frac{1}{2}_{L^\infty_TL^2_{xy}}\|\partial_x\partial_yV^\varepsilon\|^\frac{1}{2}_{L^\infty_TL^2_{xy}})\\
             &\leq C\varepsilon^{\frac{1}{8}}\cdot\varepsilon^{-\frac{3}{8}}\\
             &\leq C\varepsilon^{-\frac{1}{4}}.\\	
		\end{split}
	\end{equation}
    Similarly, it follows that
    \begin{equation}\label{U L infty}
		\begin{split}{}
        \|U^\varepsilon\|_{L^\infty_TL^\infty_{xy}}&\leq C(\|U^\varepsilon\|^\frac{1}{2}_{L^\infty_TL^2_{xy}}\|\partial_yU^\varepsilon\|^\frac{1}{2}_{L^\infty_TL^2_{xy}}+\|\partial_xU^\varepsilon\|^\frac{1}{2}_{L^\infty_TL^2_{xy}}\|\partial_x\partial_yU^\varepsilon\|^\frac{1}{2}_{L^\infty_TL^2_{xy}})\\
        &\leq C\varepsilon^{\frac{1}{8}}\cdot\varepsilon^\frac{1}{8}\\
        &\leq C\varepsilon^\frac{1}{4}.\\
		\end{split}
	\end{equation}
    Hence, the definitions of $H^\varepsilon,V^\varepsilon$, the Sobolev embedding inequality and $(\ref{HV L infty})$ lead to
    \begin{align}\label{nv varepsilon L infty}
			&\|n^\varepsilon(t,x,y)-n^0(t,x,y)\|_{L^\infty_TL^\infty_{xy}}\notag\\
            &\leq C\varepsilon^{\frac{1}{2}}(\|n^{B,1}\|_{L^\infty_TH^2_{xz}}+\varepsilon^{\frac{1}{2}}\|n^{B,2}\|_{L^\infty_TH^2_{xz}}+\|H^\varepsilon\|_{L^\infty_TL^\infty_{xy}})\notag\\
            &\leq C\varepsilon^{\frac{1}{4}},\notag\\
		&\|v^\varepsilon(t,x,y)-v^0(t,x,y)-(0,v^{B,0}_2(t,x,\frac{y}{\sqrt{\varepsilon}}))\|_{L^\infty_TL^\infty_{xy}}\\
                &\leq C\varepsilon^{\frac{1}{2}}(\|v^{B,1}_1\|_{L^\infty_TH^2_{xz}}+\|v^{B,1}_2\|_{L^\infty_TH^2_{xz}}+\|v^{B,2}_1\|_{L^\infty_TH^2_{xz}}+\|V^\varepsilon\|_{L^\infty_TL^\infty_{xy}})\notag\\
                &\leq C\varepsilon^{\frac{1}{4}}.\notag\notag
	\end{align}
    Similarly, we have
    \begin{align}\label{u varepsilon L infty}
            &\|u^\varepsilon(t,x,y)-u^0(t,x,y)\|_{L^\infty_TL^\infty_{xy}}\notag\\
                &\leq C\varepsilon^{\frac{1}{2}}(\|u^{B,1}_1\|_{L^\infty_TH^2_{xz}}+\varepsilon^{\frac{1}{2}}\|u^{B,2}_1\|_{L^\infty_TH^2_{xz}}+\varepsilon^{\frac{1}{2}}\|u^{B,2}_2\|_{L^\infty_TH^2_{xz}}+\varepsilon\|u^{B,3}_2\|_{L^\infty_TH^2_{xz}}+\|U^\varepsilon\|_{L^\infty_TL^\infty_{xy}})\\
                &\leq C\varepsilon^{\frac{1}{2}}.\notag\notag
	\end{align}
    {Combining} $(\ref{nv varepsilon L infty})$ and $(\ref{u varepsilon L infty})$, the proof is {complete}.\\

    From Theorem \ref{th nvu L infty convergence} and Theorem \ref{th L2 convergence}, the proof of Theorem $\ref{th ncu L infty}$ proceeds as follows.\\

    \textbf{Proof of Theorem $\ref{th ncu L infty}$.}
        The convergence rates of $n^\varepsilon-n^0$ and {$u^\varepsilon-u^0$ are direct consequences} of Theorem \ref{th nvu L infty convergence} and Theorem \ref{th L2 convergence}. We are left to prove the convergence for $c^\varepsilon$ in $(\ref{ncu varepsilon equation})$.
        First, distinguishing the case $\varepsilon>0$ from $\varepsilon=0$ in $(\ref{ncu varepsilon equation})$, we arrive at the following two equations:
        \begin{align*}
           \partial_tc^\varepsilon&=\varepsilon\Delta c^\varepsilon-u^\varepsilon\cdot\nabla c^\varepsilon-n^\varepsilon c^\varepsilon,\\
           \partial_tc^0&=-u^0\cdot\nabla c^0-n^0 c^0.\\
        \end{align*}
        Subtracting them gives
       \begin{align*}
           \partial_t(c^\varepsilon-c^0)&=\varepsilon\Delta c^\varepsilon-(u^\varepsilon-u^0)\cdot\nabla c^\varepsilon-u^0\cdot\nabla (c^\varepsilon-c^0)-(n^\varepsilon-n^0) c^\varepsilon+n^0 (c^\varepsilon-c^0).
        \end{align*}
       Taking the inner product of this equation with $(c^\varepsilon-c^0)$ yields
       \begin{align*}
           \frac{1}{2}\frac{d}{dt}\|c^\varepsilon-c^0\|^2_{L^2_{xy}}&=\varepsilon\int^\infty_0\int^\infty_{-\infty}\Delta c^\varepsilon(c^\varepsilon-c^0) dxdy-\int^\infty_0\int^\infty_{-\infty}(u^\varepsilon-u^0)\cdot\nabla c^\varepsilon(c^\varepsilon-c^0) dxdy\\
           &\quad-\int^\infty_0\int^\infty_{-\infty}u^0\cdot\nabla (c^\varepsilon-c^0)(c^\varepsilon-c^0) dxdy-\int^\infty_0\int^\infty_{-\infty}(n^\varepsilon-n^0) c^\varepsilon(c^\varepsilon-c^0) dxdy\\
           &\quad+\int^\infty_0\int^\infty_{-\infty}n^0 (c^\varepsilon-c^0)(c^\varepsilon-c^0) dxdy\\
           &=\mathcal{C}_1+\mathcal{C}_2+\mathcal{C}_3+\mathcal{C}_4+\mathcal{C}_5.
        \end{align*}
        From the boundary condition $\partial_yc^\varepsilon|_{y=0}=0$ in $(\ref{ncu varepsilon equation})_6$, we see that
        \begin{align*}
            \mathcal{C}_1&=-\varepsilon\int^\infty_0\int^\infty_{-\infty}\nabla c^\varepsilon\cdot\nabla(c^\varepsilon-c^0) dxdy\geq-\frac{1}{2}\varepsilon\|\nabla c^\varepsilon\|^2_{L^2_{xy}}-\frac{1}{2}\varepsilon\|\nabla (c^\varepsilon-c^0)\|^2_{L^2_{xy}}.
        \end{align*}
        From the result in $(\ref{nvu varepsilon L infty})_3$, we obtain
        \begin{align*}
            \mathcal{C}_2&\leq \|u^\varepsilon-u^0\|_{L^\infty_{xy}}\|\nabla c^\varepsilon\|_{L^2_{xy}}\|(c^\varepsilon-c^0)\|_{L^2_{xy}}\leq \frac{1}{2}\varepsilon\|\nabla c^\varepsilon\|^2_{L^2_{xy}}+C\|(c^\varepsilon-c^0)\|^2_{L^2_{xy}}.
        \end{align*}
        By the boundary condition $u^0_2=0$ and the divergence-free condition, we have $\mathcal{C}_3=0$. In addition, using $(\ref{nvu varepsilon L infty})_1$ and $(\ref{L^2-convergence})_1$, $\mathcal{C}_4$ is estimated below:
        \begin{align*}
           \mathcal{C}_4&=-\int^\infty_0\int^\infty_{-\infty}(n^\varepsilon-n^0) (c^\varepsilon-c^0)^2 dxdy-\int^\infty_0\int^\infty_{-\infty}(n^\varepsilon-n^0)c^0 (c^\varepsilon-c^0) dxdy\\
           &\leq\|n^\varepsilon-n^0\|_{L^\infty_{xy}}\|(c^\varepsilon-c^0)\|^2_{L^2_{xy}}+\|n^\varepsilon-n^0\|_{L^2_{xy}}\|c^0\|_{L^\infty_{xy}}\|(c^\varepsilon-c^0)\|_{L^2_{xy}}\\
           &\leq C\|(c^\varepsilon-c^0)\|^2_{L^2_{xy}}+C\varepsilon.
        \end{align*}
        Similarly, we obtain $\mathcal{C}_5\leq\|n^0\|_{L^\infty_{xy}}\|(c^\varepsilon-c^0)\|^2_{L^2_{xy}}\leq C\|(c^\varepsilon-c^0)\|^2_{L^2_{xy}}$.
        Combining $\mathcal{C}_1$-$\mathcal{C}_5$ and applying Gronwall's inequality, we obtain
        \begin{align*}
           \|c^\varepsilon-c^0\|^2_{L^2_{xy}}+\varepsilon\int^T_0\|\nabla (c^\varepsilon-c^0)\|^2_{L^2_{xy}}d\tau\leq C\varepsilon.
        \end{align*}
        Therefore, we clearly get the following convergence result.
        \begin{align*}
           \|c^\varepsilon-c^0\|_{L^\infty_{T}L^2_{xy}}\leq C\varepsilon^\frac{1}{2}.
        \end{align*}
        Using the Sobolev embedding, we further obtain
        \begin{align*}
           \|c^\varepsilon-c^0\|_{L^\infty_{T}L^\infty_{xy}}&\leq C\|c^\varepsilon-c^0\|^\frac{1}{2}_{L^\infty_{T}L^2_{xy}}\|\nabla^2(c^\varepsilon-c^0)\|^\frac{1}{2}_{L^\infty_{T}L^2_{xy}}\\
           &\leq C\|c^\varepsilon-c^0\|^\frac{1}{2}_{L^\infty_{T}L^2_{xy}}(\|\nabla^2c^\varepsilon\|^\frac{1}{2}_{L^\infty_{T}L^2_{xy}}+\|\nabla^2c^0\|^\frac{1}{2}_{L^\infty_{T}L^2_{xy}})\\
           &\leq C\varepsilon^\frac{1}{4},
        \end{align*}
        where because $v^\varepsilon=-\nabla(\ln c^\varepsilon)=-\frac{\nabla c^\varepsilon}{c^\varepsilon}$, the regularity of $c^\varepsilon$ is obtained via Theorem \ref{th nvu regularity}. Similarly, owing to $v^0=-\frac{\nabla c^0}{c^0}$, the regularity of $c^0$ follows from Lemma \ref{lem n0v0u0}.
        Thus, the proof of Theorem $\ref{th ncu L infty}$ is complete.

        \section{The priori estimates for $f^\varepsilon$, $g^\varepsilon$ and $h^\varepsilon$}\label{part3}
        Before carrying out the estimates for the remainder terms, we need to handle the source terms $f^\varepsilon$, $g^\varepsilon$ and $h^\varepsilon$ in equation $(\ref{HVU equation})$. Thus, we first derive the relevant prior estimates.


         \begin{lemma}\label{lem f L2}
         {Let $(n^0,v^0,u^0)$ be the solution obtained in Lemma $\ref{lem n0v0u0}$. Then there exists a constant $C>0$ independent of $\varepsilon$ on $[0,T]$ such that
			\begin{align}\label{f L2}
            &\|f^\varepsilon\|_{L^\infty_TL^2_{xy}}\leq C\varepsilon^\frac{3}{4},\;\|\partial_xf^\varepsilon\|_{L^\infty_TL^{2}_{xy}}\leq C\varepsilon^\frac{3}{4},\notag\\
            &\|\partial_yf^\varepsilon\|_{L^\infty_TL^{2}_{xy}}\leq C\varepsilon^\frac{1}{4},\;\|\partial_x\partial_yf^\varepsilon\|_{L^\infty_TL^{2}_{xy}}\leq C\varepsilon^\frac{1}{4}.
			\end{align}
        }
	    \end{lemma}
         \noindent{\bf{Proof.}} {Firstly}, it follows from the {definitions} of $H^\theta,V^\theta,U^\theta$ and $f^\varepsilon$ that
         \begin{align*}
			f^\varepsilon&=\varepsilon^\frac{1}{2}\Big[\partial_x^2n^{B,1}+\varepsilon^\frac{1}{2}\partial_x^2n^{B,2}\Big]+\Big[-\varepsilon^\frac{1}{2}\partial_tn^{B,1}\Big]+\Big[-\varepsilon\partial_tn^{B,2}\Big]\\
            &\quad-\varepsilon^\frac{1}{2}\left[(u^0_1+\varepsilon^\frac{1}{2}u^{B,1}_1+\varepsilon u^{B,2}_1)\partial_x(n^{B,1}+\varepsilon^\frac{1}{2}n^{B,2})+(u^{B,1}_1+\varepsilon^\frac{1}{2}u^{B,2}_1)\partial_xn^0\right]\\
            &\quad-\varepsilon\Big[(\varepsilon^\frac{1}{2} u^{B,2}_2+\varepsilon u^{B,3}_2)\partial_y(n^{B,1}+\varepsilon^\frac{1}{2}n^{B,2})+(u^{B,2}_2+\varepsilon^\frac{1}{2} u^{B,3}_2)\partial_yn^0\Big]\\
            &\quad+\varepsilon^\frac{1}{2}\Big[(n^{B,1}+\varepsilon^\frac{1}{2}n^{B,2})\partial_xv^0_1+(n^0+\varepsilon^\frac{1}{2}n^{B,1}+\varepsilon n^{B,2})\partial_x(v^{B,1}_1+\varepsilon^\frac{1}{2}v^{B,2}_1)\\
            &\quad+(v^0_1+\varepsilon^\frac{1}{2}v^{B,1}_1+\varepsilon v^{B,2}_1)\partial_x(n^{B,1}+\varepsilon^\frac{1}{2}n^{B,2})+(v^{B,1}_1+\varepsilon^\frac{1}{2}v^{B,2}_1)\partial_xn^0\\
            &\quad+(n^{B,1}+\varepsilon^\frac{1}{2}n^{B,2})\partial_yv^0_2+\varepsilon^\frac{1}{2}v^{B,1}_2\partial_yn^{B,1}+v^{B,1}_2\partial_yn^0\Big]\\
            &\quad+\varepsilon\Big[(n^{B,1}+\varepsilon^\frac{1}{2}n^{B,2})\partial_yv^{B,1}_2+n^{B,2}\partial_yv^{B,0}_2+(v^0_2+v^{B,0}_2+\varepsilon^\frac{1}{2}v^{B,1}_2)\partial_yn^{B,2}\Big]\\
            &\quad+\varepsilon^\frac{1}{2}\Big[(n^0-\overline{n^{0}})\partial_yv^{B,1}_2+(v^0_2-\overline{v^0_2})\partial_yn^{B,1}+u^0_2\partial_y(n^{B,1}+\varepsilon^\frac{1}{2}n^{B,2})\Big]\\
            &\quad+\Big[(n^0-\overline{n^{0}}-y\overline{\partial_yn^{0}})\partial_yv^{B,0}_2+(\partial_yn^0-\overline{\partial_yn^{0}})v^{B,0}_2\Big]=\sum^{9}_{i=1}A_i,
	\end{align*}
    where $A_i$ denotes the entire content inside the $i$-th bracket $[\cdot]$ together  with the coefficient of $\varepsilon$ preceding this bracket in the above expression. Using Lemma \ref{lem inner layer}, we can derive that
         \begin{equation*}
		\begin{split}{}
        \|{A}_1\|_{L^\infty_TL^2_{xy}}&\leq\varepsilon^\frac{3}{4}\left(\|\partial^2_xn^{B,1}\|_{L^\infty_TL^2_{xz}}+\varepsilon^\frac{1}{2}\|\partial^2_xn^{B,2}\|_{L^\infty_TL^2_{xz}}\right)\leq C\varepsilon^\frac{3}{4}.\\	
        \end{split}
	\end{equation*}
        Similarly,
        \begin{align*}
        \|{A}_2\|_{L^\infty_TL^2_{xy}}&\leq\varepsilon^\frac{3}{4}\|\partial_tn^{B,1}\|_{L^\infty_TL^2_{xz}}\leq C\varepsilon^\frac{3}{4},\\
        \|{A}_3\|_{L^\infty_TL^2_{xy}}&\leq\varepsilon^\frac{5}{4}\|\partial_tn^{B,2}\|_{L^\infty_TL^2_{xz}}\leq C\varepsilon^\frac{5}{4}.	
		\end{align*}
        {By the} Sobolev embedding inequality and Lemma \ref{lem inner layer}, we have
        \begin{align*}
        &\|{A}_4\|_{L^\infty_TL^2_{xy}}\\
        &\leq\varepsilon^\frac{3}{4}\Big[\left(\|u^0_1\|_{L^\infty_TL^\infty_{xy}}+\varepsilon^\frac{1}{2}\|u^{B,1}_1\|_{L^\infty_TL^\infty_{xy}}+\varepsilon\|u^{B,2}_1\|_{L^\infty_TL^\infty_{xy}}\right)\left(\|\partial_xn^{B,1}\|_{L^\infty_TL^2_{xz}}+\varepsilon^\frac{1}{2}\|\partial_xn^{B,2}\|_{L^\infty_TL^2_{xz}}\right)\\
        &\quad+\left(\|u^{B,1}_1\|_{L^\infty_TL^2_{xz}}+\varepsilon^\frac{1}{2}\|u^{B,2}_1\|_{L^\infty_TL^2_{xz}}\right)\|\partial_xn^0\|_{L^\infty_TL^\infty_{xy}}\Big]\\
        &\leq C\varepsilon^\frac{3}{4}.	
	\end{align*}
        Similar arguments further give the estimates for $A_5$, $A_6$ and $A_7$ as follows:
        \begin{equation*}
		\begin{split}{}	
        ||{A}_5||_{L^2_{xy}}&\leq\varepsilon^\frac{3}{4}\Big[\left(\varepsilon^\frac{1}{2}||u^{B,2}_2||_{L^2_{xz}}+\varepsilon||u^{B,3}_2||_{L^2_{xz}}\right)\left(||\partial_zn^{B,1}||_{L^2_{xz}}+\varepsilon^\frac{1}{2}||\partial_zn^{B,2}||_{L^2_{xz}}\right)\\
        &\quad+\left(||u^{B,2}_2||_{L^2_{xz}}+\varepsilon^\frac{1}{2}||u^{B,3}_2||_{L^2_{xz}}\right)||\partial_yn^0||_{L^\infty_{xy}}\Big]\\
        &\leq C\varepsilon^\frac{3}{4},\\		
		\end{split}
	\end{equation*}
       and
        \begin{align*}	
        ||{A}_6||_{L^2_{xy}}
        &\leq C\varepsilon^\frac{3}{4}\Big[\left(||n^{B,1}||_{L^2_{xz}}+\varepsilon^\frac{1}{2}||n^{B,2}||_{L^2_{xz}}\right)||v^{0}_1||_{H^3_{xy}}\\
        &\quad+\left(||n^{0}||_{H^3_{xy}}+\varepsilon^\frac{1}{2}||n^{B,1}||_{H^3_{xz}}+\varepsilon||n^{B,2}||_{H^3_{xz}}\right)\left(||\partial_xv^{B,1}_1||_{L^2_{xz}}+\varepsilon^\frac{1}{2}||\partial_xv^{B,2}_1||_{L^2_{xz}}\right)\\
        &\quad+\left(||v^{0}_1||_{H^2_{xy}}+\varepsilon^\frac{1}{2}||v^{B,1}_1||_{H^2_{xz}}+\varepsilon||v^{B,2}_1||_{H^2_{xz}}\right)\left(||\partial_xn^{B,1}||_{L^2_{xz}}+\varepsilon^\frac{1}{2}||\partial_xn^{B,2}||_{L^2_{xz}}\right)\\
        &\quad+\left(||v^{B,1}_1||_{L^2_{xz}}+\varepsilon^\frac{1}{2}||v^{B,2}_1||_{L^2_{xz}}\right)||n^{0}||_{H^3_{xy}}+\left(||n^{B,1}||_{L^2_{xz}}+\varepsilon^\frac{1}{2}||n^{B,2}||_{L^2_{xz}}\right)||v^{0}_2||_{H^3_{xy}}\\
        &\quad+||v^{B,1}_2||_{H^2_{xz}}||\partial_zn^{B,1}||_{L^2_{xz}}+||v^{B,1}_2||_{L^2_{xz}}||n^{0}||_{H^3_{xy}}\Big]\\
        &\leq C\varepsilon^\frac{3}{4},
	\end{align*}
         and
        \begin{align*}
        ||{A}_7||_{L^2_{xy}}&\leq\varepsilon^\frac{3}{4}\Big[
       \left(||n^{B,1}||_{H^2_{xz}}+\varepsilon^\frac{1}{2}||n^{B,2}||_{H^2_{xz}}\right)||\partial_zv^{B,1}_2||_{L^2_{xz}}+||n^{B,2}||_{H^2_{xz}}||\partial_zv^{B,0}_2||_{L^2_{xz}}\\
        &\quad+\left(||v^{0}_2||_{H^2_{xy}}+||v^{B,0}_2||_{H^2_{xz}}+\varepsilon^\frac{1}{2}||v^{B,1}_2||_{H^2_{xz}}\right)||\partial_zn^{B,2}||_{L^2_{xz}}\Big]\\
        &\leq C\varepsilon^\frac{3}{4}.
	\end{align*}
        By applying the change of variables $y=\varepsilon^\frac{1}{2}z$, Taylor¡¯s formula,  Lemma \ref{lem n0v0u0} and Lemma \ref{lem inner layer}, we can obtain
        \begin{align*}		
        \|{A}_{8}\|_{L^\infty_TL^2_{xy}}&=\varepsilon\Big[\|\frac{n^{0}-\overline{n^{0}}}{y}z\partial_yv^{B,1}_2\|_{L^\infty_TL^2_{xy}}+\|\frac{v^0_2- \overline{v^0_2}}{y}z\partial_yn^{B,1}\|_{L^\infty_TL^2_{xy}}\\
        &\quad+\|\frac{u^0_2-\overline{u^0_2}}{y}z\partial_y(n^{B,1}+\varepsilon^\frac{1}{2}n^{B,2})\|_{L^\infty_TL^2_{xy}}\Big]\\
        &\leq \varepsilon\Big[\|\partial_yn^{0}\|_{L^\infty_TL^\infty_{xy}}\|z\partial_yv^{B,1}_2\|_{L^\infty_TL^2_{xy}}+\|\partial_yv^{0}_2\|_{L^\infty_TL^\infty_{xy}}\|z\partial_yn^{B,1}\|_{L^\infty_TL^2_{xy}}\\
      &\quad+\|\partial_yu^{0}_2\|_{L^\infty_TL^\infty_{xy}}\|z\partial_y(n^{B,1}+\varepsilon^\frac{1}{2}n^{B,2})\|_{L^\infty_TL^2_{xy}}\Big]\\
        &\leq \varepsilon^\frac{3}{4}\Big[\|n^{0}\|_{L^\infty_TH^3_{xy}}\|(1+z^{2k})\partial_zv^{B,1}_2\|_{L^\infty_TL^2_{xz}}+\|v^{0}_2\|_{L^\infty_TH^3_{xy}}\|(1+z^{2k})\partial_zn^{B,1}\|_{L^\infty_TL^2_{xz}}\\
        &\quad+\|u^{0}_2\|_{L^\infty_TH^3_{xy}}\|(1+z^{2k})\partial_z(n^{B,1}+\varepsilon^\frac{1}{2}n^{B,2})\|_{L^\infty_TL^2_{xz}}\Big]\\
        &\leq C\varepsilon^\frac{3}{4}.
	\end{align*}
        A similar argument as estimating ${A}_{8}$ leads to
         \begin{align*}	 		
        \|{A}_{9}\|_{L^\infty_TL^2_{xy}}&=\varepsilon\|\frac{n^{0}-\overline{n^0}-y\overline{\partial_yn^0}}{y^2}z^2\partial_yv^{B,0}_2\|_{L^\infty_TL^2_{xy}}+\varepsilon^\frac{1}{2}\|\frac{\partial_yn^{0}-\overline{\partial_yn^0}}{y}zv^{B,0}_2\|_{L^\infty_TL^2_{xy}}\\
        &\leq C\varepsilon^\frac{3}{4}\left(\|n^0\|_{L^\infty_TH^4_{xy}}\|(1+z^{2k})^2\partial_zv^{B,0}_2\|_{L^\infty_TL^2_{xz}}+\|n^0\|_{L^\infty_TH^4_{xy}}\|(1+z^{2k})v^{B,0}_2\|_{L^\infty_TL^2_{xz}}\right)\\
        &\leq C\varepsilon^\frac{3}{4}.
	\end{align*}
        Substituting the above estimates for $A_1$ to $A_{9}$, we conclude that $\|f^\varepsilon\|_{L^\infty_TL^2_{xy}}\leq C\varepsilon^\frac{3}{4}$.

        Applying $\partial_x$ to {the} equations $(\ref{fgh varepsilon})_1$, by Lemma \ref{lem n0v0u0}, Lemma \ref{lem inner layer}, and repeating the above proof process for the estimate of  $\|f^\varepsilon\|_{L^\infty_TL^{2}_{xy}}$, then $\|\partial_xf^\varepsilon\|_{L^\infty_TL^{2}_{xy}}\leq C\varepsilon^\frac{3}{4}$ can be obtained. Moreover, applying $\partial_y$ to the equations $(\ref{fgh varepsilon})_1$, {then} $\|\partial_y f^\varepsilon\|_{L^\infty_TL^2_{xy}}=\sum^9_{i=1}\|\partial_y A_i\|_{L^\infty_TL^2_{xy}}$. Since $\partial_y = \varepsilon^{-\frac{1}{2}}\partial_z$, the following norm estimate is given:
        \begin{align}\label{z derivative}					
                    \|\partial^l_yf(t,x,\frac{y}{\sqrt{\varepsilon}})\|_{H^m_xL^2_y}&=\varepsilon^{\frac{1}{4}-\frac{l}{2}} \|\partial^l_zf(t,x,z)\|_{H^m_xL^2_z}.
			\end{align}
        For example, the term $\partial_y{A}_4$ has the following estimate
        \begin{align*}
        &\|\partial_y{A}_4\|_{L^\infty_TL^2_{xy}}\\
        &\leq\varepsilon^\frac{1}{2}\Big[\|\partial_y\left(u^0_1+\varepsilon^\frac{1}{2}u^{B,1}_1+\varepsilon u^{B,2}_1\right)\|_{L^\infty_TL^2_{xy}}\|\partial_xn^{B,1}+\varepsilon^\frac{1}{2}\partial_xn^{B,2}\|_{L^\infty_TL^\infty_{xy}}\\
        &\quad+\|\partial_yu^{B,1}_1+\varepsilon^\frac{1}{2}\partial_yu^{B,2}_1\|_{L^\infty_TL^2_{xy}}\|\partial_xn^0\|_{L^\infty_TL^\infty_{xy}}\\
        &\quad+\|u^0_1+\varepsilon^\frac{1}{2}u^{B,1}_1+\varepsilon u^{B,2}_1\|_{L^\infty_TL^\infty_{xy}}\|\partial_y\left(\partial_xn^{B,1}+\varepsilon^\frac{1}{2}\partial_xn^{B,2}\right)\|_{L^\infty_TL^2_{xy}}\\
        &\quad+\|u^{B,1}_1+\varepsilon^\frac{1}{2}u^{B,2}_1\|_{L^\infty_TL^2_{xy}}\|\partial_x\partial_yn^0\|_{L^\infty_TL^\infty_{xy}}\Big]\\
        &\leq C\varepsilon^\frac{1}{2}\Big[\Big(\|u^0_1\|_{L^\infty_TH^1_{xy}}+\varepsilon^\frac{1}{4}\|u^{B,1}_1\|_{L^\infty_TH^1_{xz}}+\varepsilon^\frac{3}{4}\|u^{B,2}_1\|_{L^\infty_TH^1_{xz}}\Big)\Big(\|n^{B,1}\|_{L^\infty_TH^3_{x}H^2_{z}}+\varepsilon^\frac{1}{2}\|n^{B,2}\|_{L^\infty_TH^3_{x}H^2_{z}}\Big)\\
        &\quad+\Big(\|u^{B,1}_1\|_{L^\infty_TL^2_{xz}}+\varepsilon^\frac{1}{2}\|u^{B,2}_1\|_{L^\infty_TL^2_{xz}}\Big)\|n^0\|_{L^\infty_TH^3_{xy}}\Big]\\
        &\quad+C\varepsilon^\frac{1}{4}\Big[\Big(\|u^0_1\|_{H^2_{xy}}+\varepsilon^\frac{1}{2}\|u^{B,1}_1\|_{L^\infty_TH^2_{xz}}+\varepsilon\|u^{B,2}_1\|_{L^\infty_TH^2_{xz}}\Big)\Big(\|n^{B,1}\|_{L^\infty_TH^2_{x}H^1_{z}}+\varepsilon^\frac{1}{2}\|n^{B,2}\|_{L^\infty_TH^2_{x}H^1_{z}}\Big)\\
        &\quad+\Big(\|u^{B,1}_1\|_{L^\infty_TH^1_{xz}}+\varepsilon^\frac{1}{2}\|u^{B,2}_1\|_{L^\infty_TH^1_{xz}}\Big)\|n^0\|_{L^\infty_TH^3_{xy}}\Big]\\
        &\leq C\varepsilon^\frac{1}{4}.
	\end{align*}
         By providing other terms $(A_1-A_3,A_5-A_9)$ {with similar estimates}, it can be concluded that $\|\partial_y f^\varepsilon\|_{L^\infty_TL^2_{xy}}\leq C\varepsilon^\frac{1}{4}$. {Similarly, applying $\partial_x\partial_y$ to the equations $(\ref{fgh varepsilon})_1$ we can get}
        \begin{equation*}
		\begin{split}{}			
        \|\partial_x\partial_y f^\varepsilon\|_{L^\infty_TL^2_{xy}}\leq C\varepsilon^\frac{1}{4}.
		\end{split}
	\end{equation*}

        \begin{lemma}\label{lem g L2}
         {Let $(n^0,v^0,u^0)$ be the solution obtained in Lemma $\ref{lem n0v0u0}$. Then there exists a constant $C>0$ independent of $\varepsilon$ on $[0,T]$ such that
			\begin{align}\label{g L2}
            &\|g^\varepsilon\|_{L^\infty_TL^2_{xy}}\leq C\varepsilon^\frac{3}{4},\;\|\partial_xg^\varepsilon\|_{L^\infty_TL^{2}_{xy}}
            \leq C\varepsilon^\frac{3}{4},\notag\\
            &\|\partial_yg^\varepsilon\|_{L^\infty_TL^{2}_{xy}}
            \leq C\varepsilon^\frac{1}{4},\;\|\partial_x\partial_yg^\varepsilon\|_{L^\infty_TL^{2}_{xy}}
            \leq C\varepsilon^\frac{1}{4}.
			\end{align}
        }
	    \end{lemma}
         \noindent{\bf{Proof.}}
         By the definition of $g^\varepsilon$ in $(\ref{fgh varepsilon})_2$, we write its first component $g^\varepsilon_1$ as follows:
         \begin{align}\label{g1}
			g^\varepsilon_1&=-\varepsilon^\frac{1}{2}\Big[\partial_{t}v^{B,1}_{1}-\partial^{2}_{z}v_{1}^{B,1}-\partial_{x}n^{B,1}+\partial_{x}(u^{0}_1v^{B,1}_{1})+\partial_{x}(v^{0}_1u^{B,1}_{1})\Big]\notag\\
            &\quad+\varepsilon\Big[-\partial_{t}v^{B,2}_{1}+\Delta v^0_1+\varepsilon^{\frac{1}{2}}\partial^{2}_{x}v^{B,1}_{1}+\varepsilon\partial^{2}_{x}v^{B,2}_{1}+\varepsilon\partial^{2}_{y}v^{B,2}_{1}\Big]\notag\\
            &\quad-\varepsilon^\frac{1}{2}\partial_{x}\Big[u^{0}_2v_{2}^{B,1}+((u_{1}^{B,1},0)+\varepsilon^\frac{1}{2} u^{B,2}+\varepsilon(0,u^{B,3}_2))\cdot((0,v_{2}^{B,0})+\sqrt{\varepsilon}v^{B,1}+\varepsilon(v^{B,2}_1,0))\Big]\\
            &\quad-\varepsilon\partial_{x}\Big[(v^{0}+(0,v_{2}^{B,0})+\varepsilon^\frac{1}{2}v^{B,1}+\varepsilon(v^{B,2}_1,0))(v^{0}+(0,v_{2}^{B,0})\notag\\
            &\quad+\varepsilon^\frac{1}{2}v^{B,1}+\varepsilon(v^{B,2}_1,0))+u^{0}_1v_{1}^{B,2}+(u^{B,2}+\varepsilon^\frac{1}{2}(0,u^{B,3}_2))v^0+n^{B,2}\Big]\notag\\
            &\quad-\Big[\partial_{x}(u^{0}_2v_{2}^{B,0})\Big]=\sum^{5}_{i=1}B_i.\notag\notag
	\end{align}
    Substituting equation
    \begin{align*}\label{eq:vB11}	     &\partial_tv^{B,1}_1+\partial_x(z\overline{\partial_yu^{I,0}_2})v^{B,0}_2+\overline{\partial_xu^{I,0}_1}v^{B,1}_1  +\partial_xu^{B,1}_1\overline{v^{I,0}_1}    +\overline{\partial_xv^{I,0}_1}u^{B,1}_1+\partial_xv^{B,0}_2z\overline{\partial_yu^{I,0}_2}\\
     &+\partial_xv^{B,1}_1\overline{u^{I,0}_1}-\partial_xn^{B,1}=\partial^2_zv^{B,1}_1,
     \end{align*}
     derived from (2.7) in Part I \cite{WWZ} into $B_{1}$, we reduce $B_{1}$ to
        \begin{equation*}
		\begin{split}{}			
        {B}_{1}&=-\varepsilon^\frac{1}{2}(u^0_1-\overline{u^0_1})\partial_xv_{1}^{B,1}-\varepsilon^\frac{1}{2}(\partial_xu^0_1-\overline{\partial_xu^0_1})v_{1}^{B,1}-\varepsilon^\frac{1}{2}(v^0_1-\overline{v^0_1})\partial_xu_{1}^{B,1}\\
        &\quad-\varepsilon^\frac{1}{2}(\partial_xv^0_1-\overline{\partial_xv^0_1})u_{1}^{B,1}+\varepsilon^\frac{1}{2}z\overline{\partial_x\partial_yu^0_2}v_{2}^{B,0}+\varepsilon^\frac{1}{2}z\overline{\partial_yu^0_2}\partial_xv_{2}^{B,0}.
		\end{split}
	\end{equation*}
    Then we have the following \(L^2_{xy}\) estimation for $B_{1}$
          \begin{equation*}
		\begin{split}{}	 		
        \|{B}_{1}\|_{L^\infty_TL^2_{xy}}&=\varepsilon^\frac{3}{4}\Big[\|(u^0_1-\overline{u^0_1})\|_{L^\infty_TL^\infty_{xy}}\|\partial_xv^{B,1}_1\|_{L^\infty_TL^2_{xz}}+\|(\partial_xu^0_1-\overline{\partial_xu^0_1})\|_{L^\infty_TL^\infty_{xy}}\|v^{B,1}_1\|_{L^\infty_TL^2_{xz}}\\
        &\quad+\|(v^0_1-\overline{v^0_1})\|_{L^\infty_TL^\infty_{xy}}\|\partial_xu^{B,1}_1\|_{L^\infty_TL^2_{xz}}+\|(\partial_xv^0_1-\overline{\partial_xv^0_1})\|_{L^\infty_TL^\infty_{xy}}\|u^{B,1}_1\|_{L^\infty_TL^2_{xz}}\\ &\quad+\|\overline{\partial_x\partial_yu^0_2}\|_{L^\infty_TL^\infty_{xy}}\|zv^{B,0}_2\|_{L^\infty_TL^2_{xz}}+\|\overline{\partial_yu^0_2}\|_{L^\infty_TL^\infty_{xy}}\|(1+z^{2k})\partial_xv^{B,0}_2\|_{L^\infty_TL^2_{xz}}\Big]\\
        &\leq C\varepsilon^\frac{3}{4}.		
		\end{split}
	\end{equation*}
        By Lemma \ref{lem n0v0u0} and Lemma \ref{lem inner layer}, we have
        \begin{equation*}
		\begin{split}{}
        \|{B}_2\|_{L^\infty_TL^2_{xy}}&=\varepsilon\Big[\|\partial_t v^{B,2}_1\|_{L^\infty_TL^2_{xz}}+\|\Delta v^0_1\|_{L^\infty_TL^2_{xy}}+\varepsilon^\frac{1}{2}\|\partial^2_x v^{B,1}_1\|_{L^\infty_TL^2_{xz}}\\
        &\quad+\varepsilon\|\partial^2_x v^{B,2}_1\|_{L^\infty_TL^2_{xz}}+\|\partial^2_z v^{B,2}_1\|_{L^\infty_TL^2_{xz}}\Big]\\
        &\leq C\varepsilon.\\
		\end{split}
	\end{equation*}
        Similarly, by H{\"o}lder's inequality and the Sobolev embedding inequality, we have
        \begin{equation*}
		\begin{split}{}	
        \|{B}_3\|_{L^\infty_TL^2_{xy}}\leq C\varepsilon^\frac{3}{4}.\\		
		\end{split}
	\end{equation*}
        {For the term} ${B}_4$, we first estimate {$\|V^\theta\|_{L^\infty_TL^\infty_{xy}}$} by the Sobolev embedding inequality as follows
        \begin{equation}\label{V theta L infty}
		\begin{split}{}	 		
        \|V^\theta\|_{L^\infty_TL^\infty_{xy}}&\leq C\Big[\|v^{0}\|_{L^\infty_TH^2_{xy}}+\|v_{2}^{B,0}\|_{L^\infty_TH^2_{xz}}+\varepsilon^\frac{1}{2}\|v^{B,1}_1\|_{L^\infty_TH^2_{xz}}\\
        &\quad+\varepsilon^\frac{1}{2}\|v^{B,1}_2\|_{L^\infty_TH^2_{xz}}+\varepsilon\|v^{B,2}_1\|_{L^\infty_TH^2_{xz}}\Big]\\
        &\leq C.\\		
		\end{split}
	\end{equation}
        Similar {arguments} further yield
        \begin{align}\label{V theta estimates}
           \|\partial_xV^\theta\|_{L^\infty_TL^2_{xy}}\leq C,\;\|\partial_yV^\theta\|_{L^\infty_TL^2_{xy}}\leq C\varepsilon^{-\frac{1}{4}},\;\|\partial_x\partial_yV^\theta\|_{L^\infty_TL^2_{xy}}\leq C\varepsilon^{-\frac{1}{4}}.
        \end{align}
        Thus
        \begin{align*} 		
            \|{B}_{4}\|_{L^\infty_TL^2_{xy}}&\leq C\varepsilon\Big[2\|V^\theta\|_{L^\infty_TL^\infty_{xy}}\|\partial_xV^\theta\|_{L^\infty_TL^2_{xy}}+\Big(\|\partial_xu^{B,2}\|_{L^\infty_TL^2_{xz}}+\varepsilon^\frac{1}{2}\|\partial_xu^{B,3}_2\|_{L^\infty_TL^2_{xz}}\Big)\|v^{0}\|_{L^\infty_TH^2_{xy}}\\
            &\quad+\Big(\|u^{B,2}\|_{L^\infty_TL^2_{xz}}+\varepsilon^\frac{1}{2}\|u^{B,3}_2\|_{L^\infty_TL^2_{xz}}\Big)\|v^{0}\|_{L^\infty_TH^3_{xy}}+\|\partial_xn^{B,2}\|_{L^\infty_TL^2_{xz}}\Big]\\
            &\leq C\varepsilon.
	\end{align*}
        A similar argument as estimating ${A}_{8}$ leads to
         \begin{equation*}
		\begin{split}{}	 		
        \|{B}_{5}\|_{L^\infty_TL^2_{xy}}&=\varepsilon^\frac{1}{2}\Big(\|\partial_x\frac{u^{0}_2-\overline{u^0_2}}{y}zv^{B,0}_2\|_{L^\infty_TL^2_{xy}}+\varepsilon^\frac{1}{2}\|\frac{u^{0}_2-\overline{u^0_2}}{y}z\partial_xv^{B,0}_2\|_{L^\infty_TL^2_{xy}}\Big)\\
        &\leq C\varepsilon^\frac{3}{4}\Big(\|u^{0}_2\|_{L^\infty_TH^4_{xy}}\|v^{B,0}_2\|_{L^\infty_TL^2_{xz}}+\|u^{0}_2\|_{L^\infty_TH^3_{xy}}\|(1+z^{2k})\partial_xv^{B,0}_2\|_{L^\infty_TL^2_{xz}}\Big)\\
        &\leq C\varepsilon^\frac{3}{4}.\\		
		\end{split}
	\end{equation*}
        Hence collecting the above estimates from $B_1$ to $B_{5}$, one derives $\|g^\varepsilon_1\|_{L^\infty_TL^2_{xy}}\leq C\varepsilon^\frac{3}{4}$. By applying $\partial_x$ to the equation$(\ref{g1})$, Lemma \ref{lem n0v0u0}, Lemma \ref{lem inner layer} and repeating the above proof process for the estimate of $\|g^\varepsilon_1\|_{L^\infty_TL^{2}_{xy}}$, then $\|\partial_xg^\varepsilon_1\|_{L^\infty_TL^{2}_{xy}}\leq C\varepsilon^\frac{3}{4}$ can be obtained. By (\ref{z derivative}), we can estimate the normal derivative of $B_1$ as
        \begin{align*} 		
        \|\partial_y{B}_{1}\|_{L^\infty_TL^2_{xy}}&=\varepsilon^\frac{1}{2}\|-\partial_y((u^0_1-\overline{u^0_1})\partial_xv_{1}^{B,1})-\partial_y((\partial_xu^0_1-\overline{\partial_xu^0_1})v_{1}^{B,1})-\partial_y((v^0_1-\overline{v^0_1})\partial_xu_{1}^{B,1})\|_{L^\infty_TL^2_{xy}}\\
        &\quad+\varepsilon^\frac{1}{2}\|\partial_y((\partial_xv^0_1-\overline{\partial_xv^0_1})u_{1}^{B,1})+\overline{\partial_x\partial_yu^0_2}\partial_y(zv_{2}^{B,0})+\overline{\partial_yu^0_2}\partial_y(z\partial_xv_{2}^{B,0})\|_{L^\infty_TL^2_{xy}}\\
        &\leq C\varepsilon^\frac{1}{4}\Big[\|\partial^2_yu^0_1\|_{L^\infty_TL^\infty_{xy}}\|\partial_xv^{B,1}_1\|_{L^\infty_TL^2_{xz}}+\|\partial_yu^0_1\|_{L^\infty_TL^\infty_{xy}}\|\partial_x\partial_zv^{B,1}_1\|_{L^\infty_TL^2_{xz}}\\
        &\quad+\|\partial_x\partial^2_yu^0_1\|_{L^\infty_TL^\infty_{xy}}\|v^{B,1}_1\|_{L^\infty_TL^2_{xz}}+\|\partial_x\partial_yu^0_1\|_{L^\infty_TL^\infty_{xy}}\|\partial_zv^{B,1}_1\|_{L^\infty_TL^2_{xz}}\\
        &\quad+\|\partial^2_yv^0_1\|_{L^\infty_TL^\infty_{xy}}\|\partial_xu^{B,1}_1\|_{L^\infty_TL^2_{xz}}+\|\partial_yv^0_1\|_{L^\infty_TL^\infty_{xy}}\|\partial_x\partial_zu^{B,1}_1\|_{L^\infty_TL^2_{xz}}\\
        &\quad+\|\partial_x\partial^2_yv^0_1\|_{L^\infty_TL^\infty_{xy}}\|u^{B,1}_1\|_{L^\infty_TL^2_{xz}}+\|\partial_x\partial_yv^0_1\|_{L^\infty_TL^\infty_{xy}}\|\partial_zu^{B,1}_1\|_{L^\infty_TL^2_{xz}}\\ &\quad+\|\overline{\partial_x\partial_yu^0_2}\|_{L^\infty_TL^\infty_{xy}}\|v^{B,0}_2\|_{L^\infty_TL^2_{xz}}+\|\overline{\partial_x\partial_yu^0_2}\|_{L^\infty_TL^\infty_{xy}}\|(1+z^{2k})\partial_zv^{B,0}_2\|_{L^\infty_TL^2_{xz}}\\
        &\quad+\|\overline{\partial_yu^0_2}\|_{L^\infty_TL^\infty_{xy}}\|\partial_xv^{B,0}_2\|_{L^\infty_TL^2_{xz}}+\|\overline{\partial_yu^0_2}\|_{L^\infty_TL^\infty_{xy}}\|(1+z^{2k})\partial_x\partial_zv^{B,0}_2\|_{L^\infty_TL^2_{xz}}\Big]\\
        &\leq C\varepsilon^\frac{1}{4}.	
	\end{align*}
        Using (\ref{V theta estimates}), similar to the above estimate for $B_1$, we further conclude that         $$\|\partial_y{B}_{2}\|_{L^\infty_TL^{2}_{xy}}+\|\partial_y{B}_{3}\|_{L^\infty_TL^{2}_{xy}}+\|\partial_y{B}_{4}\|_{L^\infty_TL^{2}_{xy}}+\|\partial_y{B}_{5}\|_{L^\infty_TL^{2}_{xy}}\leq C\varepsilon^\frac{1}{4}.$$ This implies that both $\|\partial_yg^\varepsilon_1\|_{L^\infty_TL^{2}_{xy}}\leq C\varepsilon^\frac{1}{4}$ and $\|\partial_x\partial_yg^\varepsilon_1\|_{L^\infty_TL^{2}_{xy}}\leq C\varepsilon^\frac{1}{4}$ hold.
        By the definition of $g^\varepsilon$ in $(\ref{fgh varepsilon})_2$, we write its second component $g^\varepsilon_2$ as follows:
         \begin{align}\label{g2}
			g^\varepsilon_2&=-\Big[\partial_{t}v_{2}^{B,0}-\partial_{z}^{2}v_{2}^{B,0}-\partial_{z}n^{B,1}+\partial_y(u^{0}_2v_{2}^{B,0})+\varepsilon^\frac{1}{2}(\partial_yu^{B,1}_1v^0_1+\partial_yv^{B,1}_1u^0_1)\Big]\notag\\
            &\quad-\varepsilon^\frac{1}{2}\Big[\partial_{t}v_{2}^{B,1}
            -\partial_{z}^{2}v_{2}^{B,1}-\partial_{z}n^{B,2}+\partial_{y}u^{I,1}_2v_{2}^{B,0}+\varepsilon^\frac{1}{2}\partial_{y}(u_{1}^{B,1}v^{B,1}_1)+\varepsilon^\frac{1}{2}\partial_{y}(u^{B,2}_2v_{2}^{B,0})+\varepsilon^\frac{1}{2}v^{0}_2\partial_{y}u^{B,2}_2\notag\\
            &\quad+\varepsilon^\frac{1}{2}\partial_{y}(v^{B,0}_2v^{B,0}_2)+2\varepsilon^\frac{1}{2}\partial_{y}v^{B,0}_2v^{0}_2+u^{B,1}_1\partial_{y}v^0_1+v^{B,1}_1\partial_{y}u^0_1+\partial_{y}u^{0}_2v^{B,1}_2+\varepsilon^\frac{1}{2}v^{0}_1\partial_{y}u^{B,2}_1+\varepsilon^\frac{1}{2}u^{0}_1\partial_{y}v^{B,2}_1\Big]\notag\\
            &\quad-\varepsilon^\frac{1}{2}\Big[u^{0}_2\partial_{y}v^{B,1}_2+\varepsilon^\frac{1}{2}\partial_{y}u^0_1v^{B,2}_1\Big]\notag\\
            &\quad+\varepsilon\Big[\Delta v_{2}^{0}+\partial^{2}_{x}v_{2}^{B,0}+\varepsilon^\frac{1}{2}\partial^{2}_{x} v_{2}^{B,1}\Big]\\
            &\quad-\varepsilon\partial_{y}\Big[u^{B,2}\cdot(\varepsilon^\frac{1}{2} v^{B,1}+\varepsilon(v^{B,2}_1,0))\Big]\notag\\
            &\quad-\varepsilon\partial_{y}\Big[(v^{0}+(0,v_{2}^{B,0})+\varepsilon^\frac{1}{2} v^{B,1}+\varepsilon(v^{B,2}_1,0))\cdot(\varepsilon^\frac{1}{2} v^{B,1}+\varepsilon(v^{B,2}_1,0))+\varepsilon^\frac{1}{2} v^{B,1}\cdot(v^0+(0,v_{2}^{B,0}))\notag\\
            &\quad+{(v^{0}+\varepsilon(v^{B,2}_1,0))\cdot v^{0}}+\varepsilon^\frac{1}{2}u^{B,1}_1v^{B,2}_1+\varepsilon^\frac{1}{2}u^{B,3}_2(v^{0}_2+v_{2}^{B,0}+\sqrt{\varepsilon}v^{B,1}_2)\Big]\notag\\
            &\quad-\varepsilon\Big[\partial_{y}v^{0}_1u^{B,2}_1+\partial_{y}v^{0}_2u^{B,2}_2+\partial_{y}v^{0}_2v^{B,0}_2\Big]\notag\\
            &=\sum^{12}_{i=6}B_i.\notag\notag
	\end{align}
       Substituting the equation
       \begin{align*}          \partial_tv^{B,0}_2+\overline{\partial_yu^{I,0}_2}v^{B,0}_2+\partial_zu^{B,1}_1\overline{v^{I,0}_1}+\partial_{x}v^{B,0}_2\overline{u^{I,0}_1}+\partial_zv^{B,0}_2z\overline{\partial_yu^{I,0}_2}-\partial_zn^{B,1}=\partial^2_zv^{B,0}_2
       \end{align*}
       derived from (2.5) in Part I \cite{WWZ} into $B_{6}$, we can obtain
        \begin{equation*}
		\begin{split}{}			
        {B}_{6}&=-(u^0_2-\overline{u^0_2}-y\overline{\partial_yu^0_2})\partial_yv_{2}^{B,0}-(\partial_yu^0_2-\overline{\partial_yu^0_2})v_{2}^{B,0}-\varepsilon^\frac{1}{2}(u^0_1-\overline{u^0_1})\partial_yv_{1}^{B,1}\\
        &\quad-\varepsilon^\frac{1}{2}(v^0_1-\overline{v^0_1})\partial_yu_{1}^{B,1}.
		\end{split}
	\end{equation*}
        Using the same arguments to $A_{8}$ by (\ref{z regularity}), Lemma \ref{lem n0v0u0} and Lemma \ref{lem inner layer}, we get
        \begin{align*}		
        \|{B}_{6}\|_{L^\infty_TL^2_{xy}}&\leq \varepsilon^\frac{3}{4}\Big[\|\partial^2_yu^{0}_2\|_{L^\infty_TL^\infty_{xy}}\|z^2\partial_zv^{B,0}_2\|_{L^\infty_TL^2_{xz}}+\|\partial^2_yu^{0}_2\|_{L^\infty_TL^\infty_{xy}}\|zv^{B,0}_2\|_{L^\infty_TL^2_{xz}}\\
        &\quad\;+\|\partial_yu^{0}_1\|_{L^\infty_TL^\infty_{xy}}\|z\partial_zv^{B,1}_1\|_{L^\infty_TL^2_{xz}}+\|\partial_yv^{0}_1\|_{L^\infty_TL^\infty_{xy}}\|z\partial_zu^{B,1}_1\|_{L^\infty_TL^2_{xz}}\Big]\\
        &\leq C\varepsilon^\frac{3}{4}\Big[\|u^{0}_2\|_{L^\infty_TH^4_{xy}}\|(1+z^{2k})^2\partial_zv^{B,0}_2\|_{L^\infty_TL^2_{xz}}+\|u^{0}_2\|_{L^\infty_TH^4_{xy}}\|(1+z^{2k})v^{B,0}_2\|_{L^\infty_TL^2_{xz}}\\
        &\quad\;+\|u^{0}_1\|_{L^\infty_TH^3_{xy}}\|(1+z^{2k})\partial_zv^{B,1}_1\|_{L^\infty_TL^2_{xz}}+\|v^{0}_1\|_{L^\infty_TH^3_{xy}}\|(1+z^{2k})\partial_zu^{B,1}_1\|_{L^\infty_TL^2_{xz}}\Big]\\
        &\leq C\varepsilon^\frac{3}{4}.
	\end{align*}
    Substituting the equation
    \begin{align*}\label{eq:vB12}
    &\partial_tv^{B,1}_2+z\overline{\partial^2_yu^{I,0}_2}v^{B,0}_2+\overline{\partial_yu^{I,0}_1}v^{B,1}_1+\overline{\partial_yu^{I,0}_2}v^{B,1}_2+\partial_zu^{B,1}_1(z\overline{\partial_yv^{I,0}_1}+v^{B,1}_1)+\partial_zu^{B,2}_1\overline{v^{I,0}_1}\\
    &+\partial_zu^{B,2}_2(\overline{v^{I,0}_2}+v^{B,0}_2)+\overline{\partial_yv^{I,0}_1}u^{B,1}_1+\partial_zv^{B,0}_2(\frac{1}{2}z^2\overline{\partial^2_yu^{I,0}_2}+\overline{u^{I,2}_2}+u^{B,2}_2)+\partial_zv^{B,1}_1(z\overline{\partial_yu^{I,0}_1}+u^{B,1}_1)\\
    &+\partial_zv^{B,1}_2z\overline{\partial_yu^{I,0}_2}+\partial_{x}v^{B,1}_2\overline{u^{I,0}_1}+2\partial_zv^{B,0}_2(\overline{v^{I,0}_2}+v^{B,0}_2)-\partial_zn^{B,2}=\partial^2_zv^{B,1}_2,
    \end{align*}
    derived from (2.8) in Part I \cite{WWZ} into $B_{7}$, we obtain
          \begin{align*}		
        {B}_{7}&=-\varepsilon(v^0_2-\overline{v^0_2})\partial_yu_{2}^{B,2}-2\varepsilon(v^0_2-\overline{v^0_2})\partial_yv_{2}^{B,0}-\varepsilon^\frac{1}{2}(\partial_yv^0_1-\overline{\partial_yv^0_1})u_{1}^{B,1}-\varepsilon^\frac{1}{2}(\partial_yu^0_1-\overline{\partial_yu^0_1})v_{1}^{B,1}\\
        &\quad-\varepsilon^\frac{1}{2}(\partial_yu^0_2-\overline{\partial_yu^0_2})v_{2}^{B,1}-\varepsilon(v^0_1-\overline{v^0_1})\partial_yu_{1}^{B,2}-\varepsilon(u^0_1-\overline{u^0_1})\partial_yv_{1}^{B,2}+\varepsilon^\frac{1}{2}z\overline{\partial^2_yu^0_2}v^{B,0}_2+\varepsilon\partial_yu^{B,1}_1z\overline{\partial_yv^0_1}\\
        &\quad+\frac{1}{2}\varepsilon\partial_yv^{B,0}_2z^2\overline{\partial^2_yu^0_2}+\varepsilon\partial_yv^{B,0}_2\overline{u^{I,2}_2}+\varepsilon\partial_yv^{B,1}_1z\overline{\partial_yu^0_1}+\varepsilon\partial_yv^{B,1}_2z\overline{\partial_yu^0_2}.
	\end{align*}
         By Lemma \ref{lem inner layer}, {we get}
        \begin{align*} 		
        \|{B}_{7}\|_{L^\infty_TL^2_{xy}}
        &\leq C\varepsilon^\frac{3}{4}\Big[\|v^{B,0}_2\|_{L^\infty_TL^2_{xz}}+\|\partial_zu^{B,2}_2\|_{L^\infty_TL^2_{xz}}+\|(1+z^{2k})^2\partial_zv^{B,0}_2\|_{L^\infty_TL^2_{xz}}+\|u^{B,1}_1\|_{L^\infty_TL^2_{xz}}\\
        &\quad+\|v^{B,1}_1\|_{L^\infty_TL^2_{xz}}+\|v^{B,1}_2\|_{L^\infty_TL^2_{xz}}+\|\partial_zv^{B,1}_1\|_{L^\infty_TL^2_{xz}}+\|\partial_zu^{B,2}_1\|_{L^\infty_TL^2_{xz}}\\
        &\quad+\|\partial_zv^{B,2}_1\|_{L^\infty_TL^2_{xz}}+\|(1+z^{2k})v^{B,0}_2\|_{L^\infty_TL^2_{xz}}+\|(1+z^{2k})\partial_zu^{B,1}_1\|_{L^\infty_TL^2_{xz}}\\
        &\quad+\|(1+z^{2k})\partial_zv^{B,1}_1\|_{L^\infty_TL^2_{xz}}+\|\partial_zv^{B,1}_2\|_{L^\infty_TL^2_{xz}}\Big]\\
        &\leq C\varepsilon^\frac{3}{4}.
	\end{align*}
        Using the Taylor¡¯s formula, {it then indicates} that
        \begin{equation*}
		\begin{split}{}			
        \|{B}_{8}\|_{L^\infty_TL^2_{xy}}&\leq C\varepsilon^\frac{3}{4}\Big(\|\partial_yu^{0}_2\|_{L^\infty_TL^\infty_{xy}}\|(1+z^{2k})\partial_zv^{B,1}_2\|_{L^\infty_TL^2_{xz}}+\|\partial_yu^0_1\|_{L^\infty_TL^\infty_{xy}}\|v^{B,2}_1\|_{L^\infty_TL^2_{xz}}\Big)\\
        &\leq C\varepsilon^\frac{3}{4}.
		\end{split}
	\end{equation*}
        By Lemma \ref{lem n0v0u0} and Lemma \ref{lem inner layer}, we obtain
        \begin{equation*}
		\begin{split}{}			
        \|{B}_{9}\|_{L^\infty_TL^2_{xy}}&\leq\varepsilon\Big(\|\Delta v^{0}_2\|_{L^\infty_TL^2_{xy}}+\|\partial^2_x v^{B,0}_2\|_{L^\infty_TL^2_{xz}}+\varepsilon^\frac{1}{2}\|\partial^2_x v^{B,1}_2\|_{L^\infty_TL^2_{xz}}\Big)\\
        &\leq C\varepsilon.
		\end{split}
	\end{equation*}
         By the Sobolev embedding inequality, {it then indicates} that
        \begin{equation*}
		\begin{split}{}			
        \|{B}_{10}\|_{L^\infty_TL^2_{xy}}+\|{B}_{11}\|_{L^\infty_TL^2_{xy}}+\|{B}_{12}\|_{L^\infty_TL^2_{xy}}\leq C\varepsilon^\frac{3}{4}.
		\end{split}
	\end{equation*}
        Similarly, with the estimates from $B_{6}$ to $B_{12}$ mentioned above, it can be inferred that $\|g^\varepsilon_2\|_{L^\infty_TL^2_{xy}}\leq C\varepsilon^\frac{3}{4}$. A similar argument yields $\|\partial_xg^\varepsilon_2\|_{L^\infty_TL^2_{xy}}\leq C\varepsilon^\frac{3}{4}$. By (\ref{z derivative}), applying $\partial_y$, $\partial_x\partial_y$ for $B_6$ to $B_{12}$, we further derive that $\|\partial_yg^\varepsilon_2\|_{L^\infty_TL^{2}_{xy}}\leq C\varepsilon^\frac{1}{4}$ and $\|\partial_x\partial_yg^\varepsilon_2\|_{L^\infty_TL^{2}_{xy}}\leq C\varepsilon^\frac{1}{4}$.

        \begin{lemma}\label{lem h L2}
         {Let $(n^0,v^0,u^0)$ be the solution obtained in Lemma $\ref{lem n0v0u0}$. Then there exists a constant $C>0$ independent of $\varepsilon$ on $[0,T]$ such that
			\begin{align}\label{h L2}
            &\|h^\varepsilon\|_{L^\infty_TL^2_{xy}}\leq C\varepsilon,\;\|\partial_xh^\varepsilon\|_{L^\infty_TL^2_{xy}}
            \leq C\varepsilon,\\
            &\|\partial_yh^\varepsilon\|_{L^\infty_TL^{2}_{xy}}
            \leq C\varepsilon^\frac{3}{4},\;\|\partial_x\partial_yh^\varepsilon\|_{L^\infty_TL^{2}_{xy}}
            \leq C\varepsilon^\frac{3}{4},\;\;\|\partial^2_yh^\varepsilon\|_{L^\infty_TL^{2}_{xy}}
            \leq C\varepsilon^\frac{1}{4}.
			\end{align}
        }
	    \end{lemma}
         \noindent{\bf{Proof.}}
        By the definition of $h^\varepsilon$ in $(\ref{fgh varepsilon})_3$, we write its first component $h^\varepsilon_1$ as follows:
         \begin{align*}
			h^\varepsilon_1&=\varepsilon^\frac{1}{2}\Big[-\partial_tu^{B,1}_1+\partial_z^2u^{B,1}_1-u^{0}_1\partial_{x}u_{1}^{B,1}-u^{B,1}_1\partial_{x}u^{0}_{1}-u^{0}_{2}\partial_{y}u_{1}^{B,1}\Big]\\
            &\quad
            -\varepsilon\Big[\partial_tu^{B,2}_1-\partial_z^2u^{B,2}_1+u_{1}^{B,1}\partial_{x}u_{1}^{B,1}
            +\partial_{x}p^{B,2}+u^{B,2}_1\partial_{x}u^{0}_{1}+u^{B,2}_2\partial_{y}u^{0}_{1}\\
            &\quad+u^{0}_1\partial_{x} u_{1}^{B,2}+{u^{0}_{2}\partial_{y} u_{1}^{B,2}}+\varepsilon^\frac{1}{2}u^{B,2}_2\partial_{y}u_{1}^{B,1}\Big]\\
            &\quad+\varepsilon\Big[\Delta u^0_1+\varepsilon^\frac{1}{2}\partial_x^2u^{B,1}_1+\varepsilon\partial_x^2u^{B,2}_1+\varepsilon u^{B,2}_2\partial_{y} u_{1}^{B,2}\Big]\\
            &\quad-\varepsilon^\frac{3}{2}\Big[u^{B,2}_1\partial_{x}(u_{1}^{B,1}+\varepsilon^\frac{1}{2}u_{1}^{B,2})+ u^{B,1}_1\partial_{x}u_{1}^{B,2}+u^{B,3}_2\partial_y(u^{0}_1+\varepsilon^\frac{1}{2}u^{B,1}_1+\varepsilon u^{B,2}_1)\Big]\\
            &=G_1+G_2+G_3+G_4.
	\end{align*}
    Substituting the equation
    \begin{align*}        \partial_tu^{B,1}_1+u^{B,1}_1\overline{\partial_xu^{I,0}_1}+\overline{u^{I,0}_1}\partial_xu^{B,1}_1+z\overline{\partial_yu^{I,0}_2}\partial_zu^{B,1}_1=\partial^2_zu^{B,1}_1
    \end{align*}
    from the formula (2.9) in Part I\cite{WWZ} into $G_{1}$, we can obtain
        \begin{equation*}
		\begin{split}{}			
        {G}_{1}&=\varepsilon^\frac{1}{2}\Big[-(u^0_1-\overline{u^0_1})\partial_xu_{1}^{B,1}-(\partial_xu^0_1-\overline{\partial_xu^0_1})u_{1}^{B,1}-(u^0_2-\overline{u^0_2}-y\overline{\partial_yu^0_2})\partial_yu_{1}^{B,1}\Big]\\
        &\leq C\varepsilon^\frac{5}{4}\Big[\|u^{0}_1\|_{L^\infty_TH^3_{xy}}\|\partial_xu^{B,1}_1\|_{L^\infty_TL^2_{xz}}+\|u^{0}_1\|_{L^\infty_TH^4_{xy}}\|u^{B,1}_1\|_{L^\infty_TL^2_{xz}}\\
        &\quad+\|u^{0}_2\|_{L^\infty_TH^4_{xy}}\|(1+z^{2k})^2\partial_zu^{B,1}_1\|_{L^\infty_TL^2_{xz}}\Big]\\
        &\leq C\varepsilon^\frac{5}{4}
		\end{split}
	\end{equation*}
    Substituting the equation
    \begin{align*}\label{eq:uB21}
    &\partial_tu^{B,2}_1+u^{B,1}_1z\overline{\partial_x\partial_yu^{I,0}_1}+u^{B,2}_1\overline{\partial_xu^{I,0}_1}+u^{B,2}_2\overline{\partial_yu^{I,0}_1}+\overline{u^{I,0}_1}\partial_xu^{B,2}_1\\
    &+z\overline{\partial_yu^{I,0}_1}\partial_xu^{B,1}_1+u^{B,1}_1\partial_xu^{B,1}_1+z\overline{\partial_yu^{I,0}_2}\partial_zu^{B,2}_1+(\frac{1}{2}z^2\overline{\partial^2_yu^{I,0}_2}+\overline{u^{I,2}_2})\partial_zu^{B,1}_1\\
    &+u^{B,2}_2\partial_zu^{B,1}_1+\partial_xp^{B,2}=\partial^2_zu^{B,2}_1,
\end{align*}
    from the formula (2.11) in Part I\cite{WWZ} into $G_{2}$, we obtain
          \begin{equation*}
		\begin{split}{}	 		
        {G}_{2}&=\varepsilon\Big[-(\partial_xu^0_1-\overline{\partial_xu^0_1})u_{1}^{B,2}-(\partial_yu^0_1-\overline{\partial_yu^0_1})u_{2}^{B,2}-(u^0_1-\overline{u^0_1})\partial_xu_{1}^{B,2}-(u^0_2-\overline{u^0_2}-y\overline{\partial_yu^0_2})\partial_yu_{1}^{B,2}\\
        &\quad+ u_{1}^{B,1}z\overline{\partial_x\partial_yu^0_1}+\partial_xu^{B,1}_1z\overline{\partial_yu^0_1}+\frac{1}{2}\varepsilon^\frac{1}{2}\partial_yu^{B,1}_1z^2\overline{\partial^2_yu^0_2}+\varepsilon^\frac{1}{2}\overline{u^{I,2}_2}\partial_yu_{1}^{B,1}\Big].\\
		\end{split}
	\end{equation*}
         Using the Sobolev embedding inequality, {we get}
        \begin{align*}		
        \|{G}_{2}\|_{L^\infty_TL^2_{xy}}
        &\leq C\varepsilon\Big[\|(1+z^{2k})u^{B,1}_1\|_{L^\infty_TL^2_{xz}}+\|u^{B,2}_1\|_{L^\infty_TL^2_{xz}}+\|u^{B,2}_2\|_{L^\infty_TL^2_{xz}}+\|\partial_xu^{B,2}_1\|_{L^\infty_TL^2_{xz}}\\
        &\quad+\|\partial_xu^{B,1}_1\|_{L^\infty_TL^2_{xz}}+\|(1+z^{2k})\partial_zu^{B,2}_1\|_{L^\infty_TL^2_{xz}}
        +\|(1+z^{2k})\partial_xu^{B,1}_1\|_{L^\infty_TL^2_{xz}}\\
        &\quad+\|(1+z^{2k})^2\partial_zu^{B,2}_1\|_{L^\infty_TL^2_{xz}}+\frac{1}{2}\|(1+z^{2k})^2\partial_zu^{B,1}_1\|_{L^\infty_TL^2_{xz}}\\
        &\quad+\|(1+z^{2k})\partial_zu^{B,1}_1\|_{L^\infty_TL^2_{xz}}\Big]\\
        &\leq C\varepsilon.
	\end{align*}
        By Lemma \ref{lem n0v0u0} and Lemma \ref{lem inner layer}, we obtain
        \begin{align*}		
        \|{G}_3\|_{L^\infty_TL^2_{xy}}\leq C\varepsilon.
	\end{align*}
        Similar, we deduce that
        \begin{align*}		
        \|{G}_{4}\|_{L^\infty_TL^2_{xy}}&\leq C\varepsilon^\frac{5}{4}\Big[\|u^{B,2}_1\|_{L^\infty_TH^2_{xz}}\Big(+\|\partial_xu^{B,1}_1\|_{L^\infty_TL^2_{xz}}+\varepsilon^\frac{1}{2}\|\partial_xu^{B,2}_1\|_{L^\infty_TL^2_{xz}}\Big)+\|u^{B,1}_1\|_{L^\infty_TH^2_{xz}}\|\partial_xu^{B,2}_1\|_{L^\infty_TL^2_{xz}}\\
        &\quad+\|u^{B,3}_2\|_{L^\infty_TL^2_{xz}}\Big(\|\partial_yu^{0}_1\|_{L^\infty_TH^2_{xy}}+\varepsilon^\frac{1}{2}\|\partial_zu^{B,1}_1\|_{L^\infty_TH^2_{xz}}+\varepsilon\|\partial_zu^{B,2}_1\|_{L^\infty_TH^2_{xz}}\Big)\Big]\\
        &\leq C\varepsilon^\frac{5}{4}.
	\end{align*}
         Now collecting the above estimates from $G_{1}$ to $G_{4}$, we conclude that $\|h^\varepsilon_1\|_{L^\infty_TL^2_{xy}}\leq C\varepsilon$. Clearly, we can also obtain $\|\partial_xh^\varepsilon_1\|_{L^\infty_TL^2_{xy}}\leq C\varepsilon$.
         By (\ref{z derivative}), we can estimate $\partial_yG_1$ as
        \begin{equation*}
		\begin{split}{}	 		
        &\|\partial_y{G}_{1}\|_{L^\infty_TL^2_{xy}}\\
        &=\varepsilon^\frac{1}{2}\|\partial_y\left[-(u^0_1-\overline{u^0_1})\partial_xu_{1}^{B,1}-(\partial_xu^0_1-\overline{\partial_xu^0_1})u_{1}^{B,1}-(u^0_2-\overline{u^0_2}-y\overline{\partial_yu^0_2})u_{1}^{B,1}\right]\|_{L^\infty_TL^2_{xy}}\\
        &\leq C\varepsilon^\frac{3}{4}\Big[\|\partial_yu^0_1\|_{L^\infty_TL^\infty_{xy}}\|\partial_x\partial_zu^{B,1}_1\|_{L^\infty_TL^2_{xz}}+\|\partial^2_yu^0_1\|_{L^\infty_TL^\infty_{xy}}\|\partial_xu^{B,1}_1\|_{L^\infty_TL^2_{xz}}+\|\partial_x\partial_yu^0_1\|_{L^\infty_TL^\infty_{xy}}\|\partial_zu^{B,1}_1\|_{L^\infty_TL^2_{xz}}\\&\quad+\|\partial_x\partial^2_yu^0_1\|_{L^\infty_TL^\infty_{xy}}\|u^{B,1}_1\|_{L^\infty_TL^2_{xz}}+\|\partial^2_yu^0_2\|_{L^\infty_TL^\infty_{xy}}\|\partial_zu^{B,1}_1\|_{L^\infty_TL^2_{xz}}+\|\partial^3_yu^0_2\|_{L^\infty_TL^\infty_{xy}}\|u^{B,1}_1\|_{L^\infty_TL^2_{xz}}\Big]\\
        &\leq C\varepsilon^\frac{3}{4}.		
		\end{split}
	\end{equation*}
        Similar to the above estimate for $G_1$, we conclude that         $$\|\partial_y{G}_{2}\|_{L^\infty_TL^{2}_{xy}}+\|\partial_y{G}_{3}\|_{L^\infty_TL^{2}_{xy}}+\|\partial_y{G}_{4}\|_{L^\infty_TL^{2}_{xy}}\leq C\varepsilon^\frac{3}{4}.$$ This implies that both $\|\partial_yh^\varepsilon_1\|_{L^\infty_TL^{2}_{xy}}\leq C\varepsilon^\frac{3}{4}$ and $\|\partial_x\partial_yh^\varepsilon_1\|_{L^\infty_TL^{2}_{xy}}\leq C\varepsilon^\frac{3}{4}$ hold. Moreover, applying $\partial^2_y$ for $G_1$, we have
        \begin{equation*}
		\begin{split}{}	 		
        &\|\partial^2_y{G}_{1}\|_{L^\infty_TL^2_{xy}}\\
        &=\varepsilon^\frac{1}{2}\|\partial^2_y[-(u^0_1-\overline{u^0_1})\partial_xu_{1}^{B,1}-(\partial_xu^0_1-\overline{\partial_xu^0_1})u_{1}^{B,1}-(u^0_2-\overline{u^0_2}-y\overline{\partial_yu^0_2})\partial_yu_{1}^{B,1})]\|_{L^\infty_TL^2_{xy}}\\
        &\leq C\varepsilon^\frac{1}{4}\Big[\|\partial_yu^0_1\|_{L^\infty_TL^\infty_{xy}}\|\partial_x\partial^2_zu^{B,1}_1\|_{L^\infty_TL^2_{xz}}+\|\partial^2_yu^0_1\|_{L^\infty_TL^\infty_{xy}}\|\partial_x\partial_zu^{B,1}_1\|_{L^\infty_TL^2_{xz}}+\|\partial^3_yu^0_1\|_{L^\infty_TL^\infty_{xy}}\|\partial_xu^{B,1}_1\|_{L^\infty_TL^2_{xz}}\\&\quad+\|\partial_x\partial_yu^0_1\|_{L^\infty_TL^\infty_{xy}}\|\partial^2_zu^{B,1}_1\|_{L^\infty_TL^2_{xz}}+\|\partial_x\partial^2_yu^0_1\|_{L^\infty_TL^\infty_{xy}}\|\partial_zu^{B,1}_1\|_{L^\infty_TL^2_{xz}}+\|\partial_x\partial^3_yu^0_1\|_{L^\infty_TL^\infty_{xy}}\|u^{B,1}_1\|_{L^\infty_TL^2_{xz}}\\&\quad+\|\partial^2_yu^0_2\|_{L^\infty_TL^\infty_{xy}}\|\partial^3_zu^{B,1}_1\|_{L^\infty_TL^2_{xz}}+\|\partial^3_yu^0_2\|_{L^\infty_TL^\infty_{xy}}\|\partial^2_zu^{B,1}_1\|_{L^\infty_TL^2_{xz}}+\|\partial^4_yu^0_2\|_{L^\infty_TL^\infty_{xy}}\|\partial_zu^{B,1}_1\|_{L^\infty_TL^2_{xz}}
        \Big]\\
        &\leq C\varepsilon^\frac{1}{4}.		
		\end{split}
	\end{equation*}
         By (\ref{z derivative}), applying $\partial^2_y$ for $G_2$ to $G_{4}$ respectively, we derive that $\|\partial^2_yh^\varepsilon_1\|_{L^\infty_TL^{2}_{xy}}\leq C\varepsilon^\frac{1}{4}$.
         By the definition of $h^\varepsilon$ in $(\ref{fgh varepsilon})_3$ we write its second component $h^\varepsilon_2$ as follows:
         \begin{align*}
	     h^\varepsilon_2&=-\varepsilon\Big[\partial_{t}u_{2}^{B,2}+\sqrt{\varepsilon}\partial_{t}u_{2}^{B,3}\Big]+\varepsilon\Big[\Delta u_{2}^0+\varepsilon\Delta u_{2}^{B,2}+\varepsilon^\frac{3}{2}\Delta u_{2}^{B,3}\Big]\\
         &\quad-\varepsilon^\frac{1}{2}\Big[(u_{1}^{B,1}+\varepsilon^\frac{1}{2} u_{1}^{B,2})\partial_{x}(u_{2}^{0}+\varepsilon u_{2}^{B,2}+\varepsilon^\frac{3}{2} u_{2}^{B,3})\Big]\\
         &\quad-\varepsilon\Big[u_{1}^{0}(\partial_{x}u_{2}^{B,2}+\varepsilon^\frac{1}{2}\partial_{x}u_{2}^{B,3})+u_{2}^{0}(\partial_{y}u_{2}^{B,2}+\varepsilon^\frac{1}{2}\partial_{y}u_{2}^{B,3})+u_{2}^{B,2}\partial_{y}(u_{2}^{0}+\varepsilon u_{2}^{B,2}+\varepsilon^\frac{3}{2} u_{2}^{B,3})-n^{B,2}\Big]\\
         &\quad-\varepsilon^\frac{3}{2}\Big[(\partial_{y}u_{2}^{B,2}+\varepsilon^\frac{1}{2} \partial_yu_{2}^{B,3})+u^{B,3}_2\partial_y(u^{0}_2+\varepsilon u^{B,2}_2+\varepsilon^\frac{3}{2} u^{B,3}_2)\Big]\\
         &={G}_5+{G}_6+{G}_7+{G}_8+{G}_{9}.
	\end{align*}
         By Lemma \ref{lem inner layer}, we can obtain
         \begin{equation*}
		\begin{split}{}			
        \|{G}_5\|_{L^\infty_TL^2_{xy}}&\leq\varepsilon\Big(\|\partial_tu^{B,2}_2\|_{L^\infty_TL^2_{xz}}+\varepsilon^\frac{1}{2}\|\partial_tu^{B,3}_2\|_{L^\infty_TL^2_{xz}}\Big)\leq C\varepsilon^\frac{5}{4}.
		\end{split}
	\end{equation*}
        {By} the {Sobolev inequalities}, we can obtain
        \begin{align*}
        \|{G}_7\|_{L^\infty_TL^2_{xy}}&=-\varepsilon^\frac{1}{2}\Big[\|\left(u_{1}^{B,1}+\varepsilon^\frac{1}{2} u_{1}^{B,2}\right)\partial_{x}\left((u_{2}^{0}-\overline{u_{2}^{0}})+\varepsilon u_{2}^{B,2}+\varepsilon^\frac{3}{2} u_{2}^{B,3}\right)\|_{L^\infty_TL^2_{xy}}\Big{]}\\
        &\leq \varepsilon^\frac{5}{4}\Big{[}\Big(\|(1+z^{2k})u_{1}^{B,1}\|_{L^\infty_TL^2_{xy}}+\varepsilon^\frac{1}{2} \|(1+z^{2k})u_{1}^{B,2}\|_{L^\infty_TL^2_{xy}}\Big)\|\partial_{x}\partial_{y}u_{2}^{0}\|_{L^\infty_TL^\infty_{xy}}\\
        &\quad+\Big(\|u_{1}^{B,1}\|_{L^\infty_TL^\infty_{xy}}+\varepsilon^\frac{1}{2}\|u_{1}^{B,2}\|_{L^\infty_TL^\infty_{xy}}\Big)\Big(\varepsilon^\frac{3}{4} \|\partial_{x}u_{2}^{B,2}\|_{L^\infty_TL^2_{xz}}+\varepsilon^\frac{5}{4} \|\partial_{x}u_{2}^{B,3}\|_{L^\infty_TL^2_{xz}}\Big)\Big{]}\\
        &\leq C\varepsilon^\frac{5}{4}.
	\end{align*}
         Similarly, it can be inferred that
        \begin{equation*}
		\begin{split}{}			
        \|{G}_{6}\|_{L^\infty_TL^2_{xy}}+\|{G}_{8}\|_{L^\infty_TL^2_{xy}}+\|{G}_{9}\|_{L^\infty_TL^2_{xy}}\leq C\varepsilon.
		\end{split}
	\end{equation*}
         Similarly, collecting the above estimates from $G_{5}$ to $G_{9}$, we conclude that $\|h^\varepsilon_2\|_{L^\infty_TL^2_{xy}}\leq C\varepsilon$. Evidently, $\|\partial_xh^\varepsilon_2\|_{L^\infty_TL^2_{xy}}\leq C\varepsilon$ holds. Finally, from (\ref{z derivative}) and above similar estimates for $G_{5}$ to $G_{9}$ by applying $\partial_y,\partial_x\partial_y,\partial^2_y$ respectively, one deduces that $\|\partial_yh^\varepsilon_2\|_{L^\infty_TL^{2}_{xy}}\leq C\varepsilon^\frac{3}{4}$, $\|\partial_x\partial_yh^\varepsilon_2\|_{L^\infty_TL^{2}_{xy}}\leq C\varepsilon^\frac{3}{4}$ and $\|\partial^2_yh^\varepsilon_2\|_{L^\infty_TL^{2}_{xy}}\leq C\varepsilon^\frac{1}{4}$. The proof is completed.

    \section{Estimations of the remainders and $L^2$-convergence}\label{part4}

 This section is devoted to establishing the well-posedness of the system $(\ref{HVU equation})$. Firstly, several preliminary facts be prepared for subsequent applications. For $f(t,x,z)\in H^{m}_{x}H^{l}_{z}$, $m,l\in N$ with fixed $t>0$, we have that
			\begin{align}\label{z regularity}					
                    \|f(t,x,\frac{y}{\sqrt{\varepsilon}})\|_{L^\infty_{xy}}&=\|f(t,x,z)\|_{L^\infty_{xz}}\leq C\|f(t,x,z)\|_{H^2_{xz}},
			\end{align}
        and
        \begin{align}\label{z=0 regularity}					
                    \|f(t,x,0)\|_{L^\infty_x}&\leq C\|f(t,x,z)\|_{L^\infty_{xz}}\leq C\|f(t,x,z)\|_{H^2_{xy}},\notag\\
                    \|\partial_xf(t,x,0)\|_{L^\infty_{x}}&\leq C\|f(t,x,z)\|_{H^3_{x}H^2_{z}}.
			\end{align}
        It can be seen from the trace theorem
            \begin{align}\label{trace theorem}
	 				\|f(t,x,0)\|_{L^2_{x}}\leq C\|f\|^{\frac{1}{2}}_{L^2_{xy}}\|\partial_yf\|^{\frac{1}{2}}_{L^2_{xy}}.
	 		\end{align}
        Moreover, it follows from the Gagliardo-Nirenberg interpolation inequality that
        \begin{align}\label{G-N inequality}
        \|f\|_{L^\infty_{xy}}&\leq C\left(\|f\|^\frac{1}{2}_{L^2_{xy}}\|\partial_yf\|^\frac{1}{2}_{L^2_{xy}}+\|\partial_xf\|^\frac{1}{2}_{L^2_{xy}}\|\partial_x\partial_yf\|^\frac{1}{2}_{L^2_{xy}}\right),\notag\\
        \|f\|_{L^2_xL^\infty_{y}}&\leq C\|f\|^\frac{1}{2}_{L^2_{xy}}\|\partial_yf\|^\frac{1}{2}_{L^2_{xy}},\\
        \|f\|_{L^\infty_{x}L^2_y}&\leq C\|f\|^\frac{1}{2}_{L^2_{xy}}\|\partial_xf\|^\frac{1}{2}_{L^2_{xy}}.\notag\notag
	\end{align}

    Next, we derive the $L^2_{xy}$ estimates for $H^\varepsilon$, $V^\varepsilon$ and $U^\varepsilon$.
     \subsection{$L^2$ {estimations} of the remainders}
	\begin{lemma}\label{lem HVU L2}
         {Suppose the assumptions in Theorem $\ref{th L2 convergence}$ hold. Assume further that the solution $(H^\varepsilon,V^\varepsilon,U^\varepsilon)$ of the system $(\ref{HVU equation})$ on $[0,T]$ satisfying
			\begin{equation}\label{HVU assumption}				\|H^\varepsilon\|^2_{L^\infty_TL^2_{xy}}+\|V^\varepsilon\|^2_{L^\infty_TL^2_{xy}}+\|U^\varepsilon\|^2_{L^\infty_TL^2_{xy}}\leq 1.\\
			\end{equation}	
		Then there exists a constant $\varepsilon_0>0$ such that for each $\varepsilon\in(0,\varepsilon_0]$,
		\begin{equation}\label{HVUL2 varepsilon}
			\|H^\varepsilon\|^2_{L^\infty_TL^2_{xy}}+\|V^\varepsilon\|^2_{L^\infty_TL^2_{xy}}+\|U^\varepsilon\|^2_{L^\infty_TL^2_{xy}}+\|\nabla U^\varepsilon\|^2_{L^\infty_TL^2_{xy}}\leq C\varepsilon^\frac{1}{2}.\\
		\end{equation}	
		Furthermore, there exists a constant $C$ indepenfent of $\varepsilon$ such that
		\begin{equation}\label{HVUL2T varepsilon}
        \begin{split}{}
			&\|\partial_x H^\varepsilon\|^2_{L^2_TL^2_{xy}}+\|\partial_y H^\varepsilon\|^2_{L^2_TL^2_{xy}}+\varepsilon\|\partial_x V^\varepsilon\|^2_{L^2_TL^2_{xy}}+\varepsilon\|\partial_y V^\varepsilon\|^2_{L^2_TL^2_{xy}}\\
            &+\varepsilon\|\partial_x\nabla U^\varepsilon\|^2_{L^2_TL^2_{xy}}+\varepsilon\|\partial_y\nabla U^\varepsilon\|^2_{L^2_TL^2_{xy}}\leq C\varepsilon^\frac{1}{2}.\\
            \end{split}
		\end{equation}	
		}
	\end{lemma}
	{\noindent\bf{Proof.}} Firstly, we estimate $\|(H^\theta,V^\theta,U^\theta)\|_{L^\infty}$. Similar to $(\ref{V theta L infty})$, by the Sobolev embedding inequality, Lemma \ref{lem n0v0u0} and Lemma \ref{lem inner layer}, we have the following estimates
	\begin{equation}\label{HVU theta L2}
		\begin{split}{}
			&\|H^\theta\|_{L^\infty_TL^\infty_{xy}}+
            \|V^\theta\|_{L^\infty_TL^\infty_{xy}}+
            \|U^\theta\|_{L^\infty_TL^\infty_{xy}}\leq C,\\
            &\|\partial_x H^\theta\|_{L^\infty_TL^\infty_{xy}}+
            \|\partial_x V^\theta\|_{L^\infty_TL^\infty_{xy}}+\|\partial_x U^\theta\|_{L^\infty_TL^\infty_{xy}}\leq C,\\
            &\|\partial_y H^\theta\|_{L^\infty_TL^\infty_{xy}}+
            \|\partial_y V^\theta_1\|_{L^\infty_TL^\infty_{xy}}+
            \|\partial_y U^\theta\|_{L^\infty_TL^\infty_{xy}}\leq C.\\
		\end{split}
	\end{equation}
    \textbf{Step 1. The \(L^2\) estimate of $H^\varepsilon$.}
    Multiplying $(\ref{HVU equation})_1$ by $H^\varepsilon$, we obtain
	\begin{align*}
			&\frac{1}{2}\frac{d}{dt}\|H^\varepsilon\|^2_{L^2_{xy}}+\|\partial_x H^\varepsilon\|^2_{L^2_{xy}}+\|\partial_y H^\varepsilon\|^2_{L^2_{xy}}\\
            &=\varepsilon^\frac{1}{2}\Big{[}-\int^\infty_0\int^\infty_{-\infty}U^\varepsilon\cdot\nabla H^\varepsilon H^\varepsilon dxdy+\int^\infty_0\int^\infty_{-\infty}(V^\varepsilon \cdot\nabla H^\varepsilon +H^\varepsilon\nabla\cdot V^\varepsilon)H^\varepsilon dxdy\Big{]}\\
            &\quad+\Big{[}-\int^\infty_0\int^\infty_{-\infty}U^\varepsilon\cdot\nabla H^\theta H^\varepsilon dxdy-\int^\infty_0\int^\infty_{-\infty}U^\theta\cdot\nabla H^\varepsilon H^\varepsilon dxdy\\
            &\quad+\int^\infty_0\int^\infty_{-\infty} (H^\varepsilon\nabla\cdot V^\theta+V^\theta\cdot\nabla H^\varepsilon)H^\varepsilon dxdy+\int^\infty_0\int^\infty_{-\infty}(V^\varepsilon\cdot\nabla H^\theta+H^\theta\nabla\cdot V^\varepsilon) H^\varepsilon dxdy\Big{]}\\
            &\quad+\varepsilon^{-\frac{1}{2}}\int^\infty_0\int^\infty_{-\infty}f^\varepsilon H^\varepsilon dxdy\\
            &=D_1+D_2+D_3.
	\end{align*}
	Using the Gagliardo-Nirenberg interpolation inequality, we can {get} that
	\begin{align*}
        D_1&=\varepsilon^\frac{1}{2}\int^\infty_0\int^\infty_{-\infty}(-U^\varepsilon_1\partial_x H^\varepsilon-U^\varepsilon_2\partial_y H^\varepsilon+V^\varepsilon_1\partial_x H^\varepsilon+V^\varepsilon_2\partial_y H^\varepsilon+ H^\varepsilon\partial_xV^\varepsilon_1+H^\varepsilon\partial_yV^\varepsilon_2)H^\varepsilon dxdy\\
        &\leq\varepsilon^\frac{1}{2}(\frac{1}{2}\|\partial_xU^\varepsilon_1\|_{L^2_{xy}}\|H^\varepsilon\|^2_{L^4_{xy}}+\|U^\varepsilon_2\|_{L^2_{x}L^\infty_{y}}\|\partial_y H^\varepsilon\|_{L^2_{xy}}\|H^\varepsilon\|_{L^\infty_{x}L^2_{y}}\\
        &\quad+\frac{1}{2}\|\partial_xV^\varepsilon_1\|_{L^2_{xy}}\| H^\varepsilon\|^2_{L^4_{xy}}+\frac{1}{2}\|\partial_yV^\varepsilon_2\|_{L^2_{xy}}\| H^\varepsilon\|^2_{L^4_{xy}})\\
        &\leq C\varepsilon^\frac{1}{2}[\|\partial_xU^\varepsilon_1\|_{L^2_{xy}}(\|H^\varepsilon\|^\frac{1}{2}_{L^2_{xy}}\|\nabla H^\varepsilon\|^\frac{1}{2}_{L^2_{xy}}+\|H^\varepsilon\|_{L^2_{xy}})\\
        &\quad+\|U^\varepsilon_2\|^\frac{1}{2}_{L^2_{xy}}\|\partial_yU^\varepsilon_2\|^\frac{1}{2}_{L^2_{xy}}\|\partial_y H^\varepsilon\|_{L^2_{xy}}\|H^\varepsilon\|^\frac{1}{2}_{L^2_{xy}}\|\partial_xH^\varepsilon\|^\frac{1}{2}_{L^2_{xy}}\\
        &\quad+\|\partial_xV^\varepsilon_1\|_{L^2_{xy}}(\|H^\varepsilon\|^\frac{1}{2}_{L^2_{xy}}\|\nabla H^\varepsilon\|^\frac{1}{2}_{L^2_{xy}}+\|H^\varepsilon\|_{L^2_{xy}})\\
        &\quad+\|\partial_yV^\varepsilon_2\|_{L^2_{xy}}(\|H^\varepsilon\|^\frac{1}{2}_{L^2_{xy}}\|\nabla H^\varepsilon\|^\frac{1}{2}_{L^2_{xy}}+\|H^\varepsilon\|_{L^2_{xy}})],
	\end{align*}
         by $\|\nabla H^\varepsilon\|_{L^2_{xy}}\leq\|\partial_x H^\varepsilon\|_{L^2_{xy}}+\|\partial_y H^\varepsilon\|_{L^2_{xy}}$ and $(\ref{HVU assumption})$,
        \begin{equation*}
		\begin{split}{}
			{D}_1
			&\leq \frac{1}{4}\|\partial_x H^\varepsilon\|^2_{L^2_{xy}}+\frac{1}{4}\|\partial_y H^\varepsilon\|^2_{L^2_{xy}}+\frac{1}{4}\varepsilon\|\partial_x V^\varepsilon\|^2_{L^2_{xy}}+\frac{1}{4}\varepsilon\|\partial_y V^\varepsilon\|^2_{L^2_{xy}}\\
            &\quad+C\|\partial_xU^\varepsilon\|^2_{L^2_{xy}}+C\|\partial_yU^\varepsilon\|^2_{L^2_{xy}}+C\|H^\varepsilon\|^2_{L^2_{xy}}+C\varepsilon^2.\\
		\end{split}
	\end{equation*}
         Using ($\ref{HVU theta L2}$), we obtain
	\begin{align*}
        D_2&=\int^\infty_0\int^\infty_{-\infty}(-U^\varepsilon_1\partial_x H^\theta-U^\varepsilon_2\partial_y H^\theta-U^\theta_1\partial_x H^\varepsilon-U^\theta_2\partial_y H^\varepsilon+ H^\varepsilon\partial_xV^\theta_1+V^\theta_1\partial_x H^\varepsilon+V^\theta_2\partial_y H^\varepsilon\\
        &\quad+V^\varepsilon_1\partial_x H^\theta+V^\varepsilon_2\partial_y H^\theta)H^\varepsilon dxdy+\int^\infty_0\int^\infty_{-\infty}(H^\varepsilon\partial_yV^\theta_2+ H^\theta\partial_xV^\varepsilon_1+H^\theta\partial_yV^\varepsilon_2)H^\varepsilon dxdy\\
        &=D_{21}+D_{22}.
	\end{align*}
   Using the embedding inequalities for $D_{21}$,
    \begin{align*}
        D_{21}&\leq(\|U^\varepsilon_1\|_{L^2_{xy}}\|\partial_x H^\theta\|_{L^\infty_{xy}}+\|U^\varepsilon_2\|_{L^2_{xy}}\|\partial_y H^\theta\|_{L^\infty_{xy}}+\|U^\theta_1\|_{L^\infty_{xy}}\|\partial_x H^\varepsilon\|_{L^2_{xy}}\\
        &\quad+\|U^\theta_2\|_{L^\infty_{xy}}\|\partial_y H^\varepsilon\|_{L^2_{xy}}+ \|H^\varepsilon\|_{L^2_{xy}}\|\partial_xV^\theta_1\|_{L^\infty_{xy}}+\|V^\theta_1\|_{L^\infty_{xy}}\|\partial_x H^\varepsilon\|_{L^2_{xy}}\\
        &\quad+\|V^\theta_2\|_{L^\infty_{xy}}\|\partial_y H^\varepsilon\|_{L^2_{xy}}+\|V^\varepsilon_1\|_{L^2_{xy}}\|\partial_x H^\theta\|_{L^\infty_{xy}}+\|V^\varepsilon_2\|_{L^2_{xy}}\|\partial_y H^\theta\|_{L^\infty_{xy}})\|H^\varepsilon\|_{L^2_{xy}}\\
        &\leq \frac{1}{8}\|\partial_x H^\varepsilon\|^2_{L^2_{xy}}+\frac{1}{8}\|\partial_y H^\varepsilon\|^2_{L^2_{xy}}+C\|H^\varepsilon\|^2_{L^2_{xy}}+C\|V^\varepsilon\|^2_{L^2_{xy}}+C\|U^\varepsilon\|^2_{L^2_{xy}}.
	\end{align*}
     By integrating by parts, we get
     \begin{equation*}
		\begin{split}{}
        D_{22}&\leq2\|V^\theta_2\|_{L^\infty_{xy}}\|\partial_y H^\varepsilon\|_{L^2_{xy}}\|H^\varepsilon\|_{L^2_{xy}}+\|\partial_x H^\theta\|_{L^\infty_{xy}}\|V^\varepsilon_1\|_{L^2_{xy}}\|H^\varepsilon\|_{L^2_{xy}}\\
        &\quad+\|H^\theta\|_{L^\infty_{xy}}\|V^\varepsilon_1\|_{L^2_{xy}}\|\partial_x H^\varepsilon\|_{L^2_{xy}}+\|\partial_y H^\theta\|_{L^\infty_{xy}}\|V^\varepsilon_2\|_{L^2_{xy}}\|H^\varepsilon\|_{L^2_{xy}}\\
        &\quad+ \|H^\theta\|_{L^\infty_{xy}}\|V^\varepsilon_2\|_{L^2_{xy}}\|\partial_yH^\varepsilon\|_{L^2_{xy}}\\
        &\leq \frac{1}{8}\|\partial_x H^\varepsilon\|^2_{L^2_{xy}}+\frac{1}{8}\|\partial_y H^\varepsilon\|^2_{L^2_{xy}}+C\|H^\varepsilon\|^2_{L^2_{xy}}+C\|V^\varepsilon\|^2_{L^2_{xy}}.\\
		\end{split}
	\end{equation*}
    Then
    \begin{equation*}
		\begin{split}{}
        D_{2}&\leq \frac{1}{4}\|\partial_x H^\varepsilon\|^2_{L^2_{xy}}+\frac{1}{4}\|\partial_y H^\varepsilon\|^2_{L^2_{xy}}+C\|H^\varepsilon\|^2_{L^2_{xy}}+C\|V^\varepsilon\|^2_{L^2_{xy}}+C\|U^\varepsilon\|^2_{L^2_{xy}}.\\
		\end{split}
	\end{equation*}
	 Combining Lemma $\ref{lem f L2}$, we can collect items with similar estimates
	\begin{equation*}
		\begin{split}{}
        D_3&\leq\varepsilon^{-\frac{1}{2}}\|f^\varepsilon\|_{L^2_{xy}}\|H^\varepsilon\|_{L^2_{xy}}\leq C\|H^\varepsilon\|^2_{L^2_{xy}}+C\varepsilon^\frac{1}{2}.\\
		\end{split}
	\end{equation*}
    Obviously,
    \begin{equation*}
		\begin{split}{}
			&\frac{1}{2}\frac{d}{dt}\|H^\varepsilon\|^2_{L^2_{xy}}+\|\partial_x H^\varepsilon\|^2_{L^2_{xy}}+\|\partial_y H^\varepsilon\|^2_{L^2_{xy}}\\
            &\leq \frac{1}{2}\|\partial_x H^\varepsilon\|^2_{L^2_{xy}}+\frac{1}{2}\|\partial_y H^\varepsilon\|^2_{L^2_{xy}}+\frac{1}{4}\varepsilon\|\partial_x V^\varepsilon\|^2_{L^2_{xy}}+\frac{1}{4}\varepsilon\|\partial_y V^\varepsilon\|^2_{L^2_{xy}}+C\|\partial_x U^\varepsilon\|^2_{L^2_{xy}}\\
            &\quad+C\|\partial_y U^\varepsilon\|^2_{L^2_{xy}}+C\|H^\varepsilon\|^2_{L^2_{xy}}+C\|U^\varepsilon\|^2_{L^2_{xy}}+C\|V^\varepsilon\|^2_{L^2_{xy}}+C\varepsilon^\frac{1}{2}.
        \end{split}
	\end{equation*}
    \textbf{Step 2. The \(L^2\) estimate of $V^\varepsilon$.}
         Multiplying $(\ref{HVU equation})_2$ by $V^\varepsilon$, we can obtain
	\begin{align*}
			&\frac{1}{2}\frac{d}{dt}\|V^\varepsilon\|^2_{L^2_{xy}}+\varepsilon\|\partial_x V^\varepsilon\|^2_{L^2_{xy}}+\varepsilon\|\partial_y V^\varepsilon\|^2_{L^2_{xy}}\\
            &=\Big{[}-\varepsilon^\frac{1}{2}\int^\infty_0\int^\infty_{-\infty}\partial_x(U^\varepsilon\cdot V^\varepsilon) V^\varepsilon_1 dxdy-\varepsilon^\frac{1}{2}\int^\infty_0\int^\infty_{-\infty}\partial_y(U^\varepsilon\cdot V^\varepsilon) V^\varepsilon_2 dxdy\Big{]}\\
            &\quad+\Big{[}-\int^\infty_0\int^\infty_{-\infty}\partial_x(U^\varepsilon\cdot V^\theta) V^\varepsilon_1 dxdy-\int^\infty_0\int^\infty_{-\infty}\partial_y(U^\varepsilon\cdot V^\theta)V^\varepsilon_2 dxdy\Big{]}\\
            &\quad+\Big{[}-\int^\infty_0\int^\infty_{-\infty}\partial_x(U^\theta\cdot V^\varepsilon) V^\varepsilon_1 dxdy-\int^\infty_0\int^\infty_{-\infty}\partial_y(U^\theta\cdot V^\varepsilon)V^\varepsilon_2 dxdy\Big{]}\\
            &\quad-\Big{[}2\varepsilon^\frac{3}{2}\int^\infty_0\int^\infty_{-\infty}(\partial_xV^\varepsilon\cdot V^\varepsilon) V^\varepsilon_1 dxdy+2\varepsilon^\frac{3}{2}\int^\infty_0\int^\infty_{-\infty}(\partial_yV^\varepsilon\cdot V^\varepsilon) V^\varepsilon_2 dxdy\Big{]}\\
            &\quad+\Big{[}-2\varepsilon\int^\infty_0\int^\infty_{-\infty}\partial_x(V^\varepsilon\cdot V^\theta) V^\varepsilon_1 dxdy-2\varepsilon\int^\infty_0\int^\infty_{-\infty}\partial_y(V^\varepsilon\cdot V^\theta)V^\varepsilon_2 dxdy\Big{]}\\
            &\quad+\Big{[}\int^\infty_0\int^\infty_{-\infty}\partial_x H^\varepsilon V^\varepsilon_1 dxdy+\int^\infty_0\int^\infty_{-\infty}\partial_y H^\varepsilon V^\varepsilon_2 dxdy\Big{]}+\varepsilon^{-\frac{1}{2}}\int^\infty_0\int^\infty_{-\infty}g^\varepsilon\cdot V^\varepsilon dxdy\\
            &=M_1+M_2+M_3+M_4+M_5+M_6+M_7.
	\end{align*}
        Using the Gagliardo-Nirenberg interpolation inequality, $(\ref{HVU equation})_3$ and $(\ref{HVU assumption})$, we can get that
	\begin{align*}
       M_1&=-\varepsilon^\frac{1}{2}\int^\infty_0\int^\infty_{-\infty}[\partial_x(U^\varepsilon_1V^\varepsilon_1+U^\varepsilon_2V^\varepsilon_2)V^\varepsilon_1+\partial_y(U^\varepsilon_1V^\varepsilon_1+U^\varepsilon_2V^\varepsilon_2)V^\varepsilon_2]dxdy\\
       &\leq\varepsilon^\frac{1}{2}(\frac{1}{2}\|\partial_x U^\varepsilon_1\|_{L^2_{xy}}\|V^\varepsilon_1\|^2_{L^4_{xy}}+\|\partial_x U^\varepsilon_2\|_{L^2_{xy}}\|V^\varepsilon_2\|_{L^4_{xy}}\|V^\varepsilon_1\|_{L^4_{xy}}\\
       &\quad+\| U^\varepsilon_2(t,x,0)\|_{L^\infty_{xy}}\|V^\varepsilon_1(t,x,0)\|^2_{L^2_{xy}}+\frac{1}{2}\|\partial_y U^\varepsilon_2\|_{L^2_{xy}}\|V^\varepsilon_1\|^2_{L^4_{xy}}+\|\partial_y U^\varepsilon_1\|_{L^2_{xy}}\|V^\varepsilon_1\|_{L^4_{xy}}\|V^\varepsilon_2\|_{L^4_{xy}}\\
       &\quad+\frac{1}{2}\|\partial_xU^\varepsilon_1\|_{L^2_{xy}}\|V^\varepsilon_2\|^2_{L^4_{xy}}+\frac{1}{2}\|\partial_y U^\varepsilon_2\|_{L^2_{xy}}\|V^\varepsilon_2\|^2_{L^4_{xy}})\\
        &\leq\frac{1}{10}\varepsilon\|\partial_x V^\varepsilon\|^2_{L^2_{xy}}+\frac{1}{10}\varepsilon\|\partial_y V^\varepsilon\|^2_{L^2_{xy}}+C\|\partial_x U^\varepsilon\|^2_{L^2_{xy}}+C\|\partial_y U^\varepsilon\|^2_{L^2_{xy}}+C\|V^\varepsilon\|^2_{L^2_{xy}}.
	\end{align*}
         Combining with the results of Theorem \ref{th nvu regularity} and $V^\theta=v^\varepsilon-\varepsilon^\frac{1}{2}V^\varepsilon$ in $(\ref{HVUK varepsilon})$, we obtain
	\begin{align*}
        M_2&=-\int^\infty_0\int^\infty_{-\infty}\Big[\partial_x(U^\varepsilon_1V^\theta_1+U^\varepsilon_2V^\theta_2)V^\varepsilon_1+\partial_y(U^\varepsilon_1V^\theta_1+U^\varepsilon_2V^\theta_2)V^\varepsilon_2\Big]dxdy\\
        &=\int^\infty_0\int^\infty_{-\infty}\Big[\partial_xU^\varepsilon_1V^\theta_1+U^\varepsilon_1\partial_xV^\theta_1+\partial_xU^\varepsilon_2V^\theta_2+U^\varepsilon_2\partial_xV^\theta_2\Big]V^\varepsilon_1\\
        &\quad+\Big[\partial_yU^\varepsilon_1V^\theta_1+U^\varepsilon_1\partial_yV^\theta_1+\partial_yU^\varepsilon_2V^\theta_2+U^\varepsilon_2\partial_y(v^\varepsilon_2-\varepsilon^\frac{1}{2}V^\varepsilon_2)\Big]V^\varepsilon_2dxdy\\
       &\leq(\|\partial_x U^\varepsilon_1\|_{L^2_{xy}}\|V^\theta_1\|_{L^\infty_{xy}}+\| U^\varepsilon_1\|_{L^2_{xy}}\|\partial_xV^\theta_1\|_{L^\infty_{xy}}\\
       &\quad+\|\partial_x U^\varepsilon_2\|_{L^2_{xy}}\|V^\theta_2\|_{L^\infty_{xy}}+\| U^\varepsilon_2\|_{L^2_{xy}}\|\partial_xV^\theta_2\|_{L^\infty_{xy}})\|V^\varepsilon_1\|^2_{L^2_{xy}}\\
       &\quad+(\|\partial_y U^\varepsilon_1\|_{L^2_{xy}}\|V^\theta_1\|_{L^\infty_{xy}}+\|\partial_y U^\varepsilon_2\|_{L^2_{xy}}\|V^\theta_2\|_{L^\infty_{xy}})\|V^\varepsilon_2\|^2_{L^2_{xy}}+\| U^\varepsilon_1\|_{L^2_{xy}}\|\partial_yV^\theta_1\|_{L^\infty_{xy}}\|V^\varepsilon_2\|_{L^2_{xy}}\\
       &\quad+\| U^\varepsilon_2\|_{L^2_{xy}}\|\partial_yv^\varepsilon_2\|_{L^\infty_{xy}}\|V^\varepsilon_2\|_{L^2_{xy}}+\frac{1}{2}\varepsilon^\frac{1}{2}\|\partial_yU^\varepsilon_2\|_{L^2_{xy}}\| V^\varepsilon_2\|^2_{L^4_{xy}}\\
        &\leq\frac{1}{10}\varepsilon\|\partial_x V^\varepsilon\|^2_{L^2_{xy}}+\frac{1}{10}\varepsilon\|\partial_y V^\varepsilon\|^2_{L^2_{xy}}+C\|\partial_x U^\varepsilon\|^2_{L^2_{xy}}+C\|\partial_y U^\varepsilon\|^2_{L^2_{xy}}+C\|V^\varepsilon\|^2_{L^2_{xy}}+C\|U^\varepsilon\|^2_{L^2_{xy}}.
	\end{align*}
    By $(\ref{HVU theta L2})$, $(\ref{HVU equation})_3$ and $U^\theta_2(t,x,0)=-(\varepsilon u^{B,2}_2(t,x,0)+\varepsilon^\frac{3}{2} u^{B,3}_2(t,x,0))$, we have
    \begin{align*}
        M_3&=-\int^\infty_0\int^\infty_{-\infty}\partial_x(U^\theta_1V^\varepsilon_1+U^\theta_2V^\varepsilon_2)V^\varepsilon_1+\partial_y(U^\theta_1V^\varepsilon_1+U^\theta_2V^\varepsilon_2)V^\varepsilon_2dxdy\\
       &\leq\frac{1}{2}\|\partial_x U^\theta_1\|_{L^\infty_{xy}}\|V^\varepsilon_1\|^2_{L^2_{xy}}+\|\partial_x U^\theta_2\|_{L^\infty_{xy}}\|V^\varepsilon_2\|_{L^2_{xy}}\|V^\varepsilon_1\|_{L^2_{xy}}\\
       &\quad+\|U^\theta_2(t,x,0)\|_{L^2_{xy}}\|V^\varepsilon_1(t,x,0)\|^2_{L^2_{xy}}+\|\partial_y U^\theta_2\|_{L^\infty_{xy}}\|V^\varepsilon_1\|^2_{L^2_{xy}}\\
       &\quad+\|\partial_y U^\theta_1\|_{L^\infty_{xy}}\|V^\varepsilon_1\|_{L^2_{xy}}\|V^\varepsilon_2\|_{L^2_{xy}}+\frac{1}{2}\|\partial_x U^\theta_1\|_{L^\infty_{xy}}\|V^\varepsilon_2\|^2_{L^2_{xy}}+\frac{1}{2}\|\partial_y U^\theta_2\|_{L^\infty_{xy}}\|V^\varepsilon_2\|^2_{L^2_{xy}}\\
        &\leq \frac{1}{10}\varepsilon\|\partial_y V^\varepsilon\|^2_{L^2_{xy}}+C\| V^\varepsilon\|^2_{L^2_{xy}}.\\
		\end{align*}
	 Directly using the Sobolev embedding inequality, we can see
	\begin{equation*}
		\begin{split}{}
        M_4&=\int^\infty_0\int^\infty_{-\infty}2\varepsilon^\frac{3}{2}\Big[(\partial_xV^\varepsilon_1V^\varepsilon_1+\partial_xV^\varepsilon_2V^\varepsilon_2)V^\varepsilon_1+(\partial_yV^\varepsilon_1V^\varepsilon_1+\partial_yV^\varepsilon_2V^\varepsilon_2)V^\varepsilon_2\Big]dxdy\\
      &\leq\frac{1}{10}\varepsilon\|\partial_x V^\varepsilon\|^2_{L^2_{xy}}+\frac{1}{10}\varepsilon\|\partial_y V^\varepsilon\|^2_{L^2_{xy}}+C\|V^\varepsilon\|^2_{L^2_{xy}}+C_0\varepsilon^2\|V^\varepsilon\|^2_{L^2_{xy}}\|\nabla V^\varepsilon\|^2_{L^2_{xy}}\\
      &\leq\frac{1}{5}\varepsilon\|\partial_x V^\varepsilon\|^2_{L^2_{xy}}+\frac{1}{5}\varepsilon\|\partial_y V^\varepsilon\|^2_{L^2_{xy}}+C\|V^\varepsilon\|^2_{L^2_{xy}}.\\
		\end{split}
	\end{equation*}
     Similarly, by $(\ref{HVU equation})_3$ we get
    \begin{align*}
        M_5&=-\int^\infty_0\int^\infty_{-\infty}2\varepsilon\Big[\partial_x(V^\varepsilon_1V^\theta_1+V^\varepsilon_2V^\theta_2)V^\varepsilon_1+\partial_y(V^\varepsilon_1V^\theta_1+V^\varepsilon_2V^\theta_2)V^\varepsilon_2\Big]dxdy\\
        &\leq2\varepsilon(\frac{1}{2}\| V^\varepsilon_1\|^2_{L^2_{xy}}\|\partial_xV^\theta_1\|_{L^\infty_{xy}}+\| \partial_xV^\varepsilon_2\|_{L^2_{xy}}\|V^\theta_2\|_{L^\infty_{xy}}\|V^\varepsilon_2\|_{L^2_{xy}}+\|V^\varepsilon_2\|_{L^2_{xy}}\| \partial_xV^\theta_2\|_{L^\infty_{xy}}\|V^\varepsilon_1\|_{L^2_{xy}}\\
       &\quad+\frac{1}{2}\| V^\varepsilon_2\|^2_{L^2_{xy}}\|\partial_xV^\theta_1\|_{L^\infty_{xy}}+\|V^\varepsilon_1\|_{L^2_{xy}}\| \partial_yV^\theta_1\|_{L^\infty_{xy}}\|V^\varepsilon_2\|_{L^2_{xy}}+\frac{1}{2}\| V^\varepsilon_2\|^2_{L^2_{xy}}\|\partial_yV^\theta_2\|_{L^\infty_{xy}})\\
        &\leq\frac{1}{10}\varepsilon\|\partial_x V^\varepsilon\|^2_{L^2_{xy}}+C\|V^\varepsilon\|^2_{L^2_{xy}}.
	\end{align*}
    And
    \begin{equation*}
		\begin{split}{}
        M_6&\leq\|\partial_x H^\varepsilon\|_{L^2_{xy}}\| V^\varepsilon_1\|_{L^2_{xy}}+\|\partial_y H^\varepsilon\|_{L^2_{xy}}\| V^\varepsilon_2\|_{L^2_{xy}}\\
        &\leq\frac{1}{8}\|\partial_x H^\varepsilon\|^2_{L^2_{xy}}+\frac{1}{8}\|\partial_y H^\varepsilon\|^2_{L^2_{xy}}+C\|V^\varepsilon\|^2_{L^2_{xy}}.
		\end{split}
	\end{equation*}
    By Lemma $\ref{lem g L2}$, we have
    \begin{equation*}
		\begin{split}{}
        M_7&\leq\varepsilon^{-\frac{1}{2}}\|g^\varepsilon\|_{L^2_{xy}}\| V^\varepsilon\|_{L^2_{xy}}\leq C\|V^\varepsilon\|^2_{L^2_{xy}}+C\varepsilon^\frac{1}{2}.
		\end{split}
	\end{equation*}
    Obviously,
    \begin{equation*}
		\begin{split}{}
			&\frac{1}{2}\frac{d}{dt}\|V^\varepsilon\|^2_{L^2_{xy}}+\varepsilon\|\partial_x V^\varepsilon\|^2_{L^2_{xy}}+\varepsilon\|\partial_y V^\varepsilon\|^2_{L^2_{xy}}\\
            &\leq \frac{1}{4}\|\partial_x H^\varepsilon\|_{L^2_{xy}}+\frac{1}{8}\|\partial_y H^\varepsilon\|^2_{L^2_{xy}}+\frac{1}{2}\varepsilon\|\partial_x V^\varepsilon\|^2_{L^2_{xy}}+\frac{1}{2}\varepsilon\|\partial_y V^\varepsilon\|^2_{L^2_{xy}}\\
            &\quad+C\|\partial_x U^\varepsilon\|^2_{L^2_{xy}}+C\|\partial_y U^\varepsilon\|^2_{L^2_{xy}}+C\|U^\varepsilon\|^2_{L^2_{xy}}+C\|V^\varepsilon\|^2_{L^2_{xy}}+C\varepsilon^\frac{1}{2}.
        \end{split}
	\end{equation*}
    \textbf{Step 3. The \(L^2\) estimate of $U^\varepsilon$.}
     To avoid estimating the unknown pressure term $\nabla K^\varepsilon$ and obtain the estimate of $\|U^\varepsilon\|^2_{L^2_{xy}}$, we apply the Leray projector $\mathbb {P}$, which is the $L^2_{xy}$ orthogonal projection on the space of divergence free vector fields
tangent to the boundary, and decompose $U^\varepsilon=\mathbb {P}U^\varepsilon+(I-\mathbb {P})U^\varepsilon$. Next, multiplying the $U^\varepsilon$ equation given in $(\ref{HVU equation})_4$ by $\mathbb {P}U^\varepsilon$ and integrating it, we obtain
	\begin{align*}
			\frac{1}{2}\frac{d}{dt}\|\mathbb {P}U^\varepsilon\|^2_{L^2_{xy}}&=\varepsilon\int^\infty_0\int^\infty_{-\infty}\Delta U^\varepsilon\cdot\mathbb {P} U^\varepsilon dxdy-\varepsilon^\frac{1}{2}\int^\infty_0\int^\infty_{-\infty}U^\varepsilon\cdot\nabla U^\varepsilon\cdot \mathbb {P}U^\varepsilon dxdy\\
            &\quad-\Big{[}\int^\infty_0\int^\infty_{-\infty}U^\varepsilon\cdot\nabla U^\theta\cdot \mathbb {P}U^\varepsilon dxdy+\int_{\Omega}U^\theta\cdot\nabla U^\varepsilon\cdot \mathbb {P}U^\varepsilon dxdy\Big{]}\\
            &\quad+\Big{[}\int^\infty_0\int^\infty_{-\infty}H^\varepsilon e_2\cdot \mathbb {P}U^\varepsilon dxdy+\varepsilon^{-\frac{1}{2}}\int^\infty_0\int^\infty_{-\infty}h^\varepsilon\cdot \mathbb {P}U^\varepsilon dxdy\Big{]}\\
            &=N_1+N_2+N_3+N_4.
	\end{align*}
    Using the Lemma 6 in \cite{IS}, we can get that
    \begin{align}\label{(I-mathbb P)U varepsilon}
       \|(I-\mathbb {P})U^\varepsilon\|_{H^1_{xy}} \leq C\varepsilon^\frac{1}{2}.
    \end{align}

    By the trace theorem, it can be verified that
        \begin{equation*}
		\begin{split}{}
		    N_1&=\varepsilon\int^\infty_{-\infty}\partial_y U^\varepsilon(t,x,0)\cdot\mathbb {P} U^\varepsilon(t,x,0)dx-\varepsilon\int^\infty_0\int^\infty_{-\infty}\nabla (I-\mathbb {P})U^\varepsilon\cdot\nabla\mathbb {P} U^\varepsilon dxdy-\varepsilon\|\nabla\mathbb {P} U^\varepsilon\|^2_{L^2_{xy}}\\
            &\leq\varepsilon\|\partial_yU^\varepsilon\|^\frac{1}{2}_{L^2_{xy}}\|\partial^2_yU^\varepsilon\|^\frac{1}{2}_{L^2_{xy}}\|\mathbb{P}U^\varepsilon\|^\frac{1}{2}_{L^2_{xy}}\|\nabla\mathbb{P}U^\varepsilon\|^\frac{1}{2}_{L^2_{xy}}+\varepsilon\|\nabla (I-\mathbb {P})U^\varepsilon\|_{L^2_{xy}}||\nabla\mathbb{P}U^\varepsilon\|_{L^2_{xy}}-\varepsilon\|\nabla\mathbb {P} U^\varepsilon\|^2_{L^2_{xy}}\\
            &\leq \frac{1}{4}\varepsilon\|\partial^2_yU^\varepsilon\|^2_{L^2_{xy}}+\frac{1}{4}\varepsilon\|\partial_yU^\varepsilon\|^2_{L^2_{xy}}+C\|\mathbb{P}U^\varepsilon\|^2_{L^2_{xy}}+C\varepsilon-\frac{1}{2}\varepsilon\|\nabla\mathbb {P} U^\varepsilon\|^2_{L^2_{xy}}.
		\end{split}
	\end{equation*}
    By $\nabla\cdot U^\varepsilon=0$ and integrating by parts, we obtain
	\begin{equation*}
		\begin{split}{}
       N_2&=-\varepsilon^\frac{1}{2}\int^\infty_0\int^\infty_{-\infty}(I-\mathbb {P})U^\varepsilon\cdot\nabla \mathbb {P}U^\varepsilon\cdot \mathbb {P}U^\varepsilon dxdy-\varepsilon^\frac{1}{2}\int^\infty_0\int^\infty_{-\infty}\mathbb {P}U^\varepsilon\cdot\nabla (I-\mathbb {P})U^\varepsilon\cdot \mathbb {P}U^\varepsilon dxdy\\
       &\quad-\varepsilon^\frac{1}{2}\int^\infty_0\int^\infty_{-\infty}(I-\mathbb {P})U^\varepsilon\cdot\nabla (I-\mathbb {P})U^\varepsilon\cdot \mathbb {P}U^\varepsilon dxdy\\
       &\leq \frac{1}{4}\varepsilon||\nabla\mathbb {P} U^\varepsilon||^2_{L^2}+C||\mathbb {P}U^\varepsilon||^2_{L^2}+C\varepsilon.
		\end{split}
	\end{equation*}
        Similarly, by the free divergence condition, combining with $(\ref{HVU theta L2})$-$(\ref{(I-mathbb P)U varepsilon})$, we get
	\begin{equation*}
		\begin{split}{}
        N_3&\leq||(I-\mathbb {P})U^\varepsilon||_{L^2_{xy}}||\nabla U^\theta||_{L^\infty_{xy}}||\mathbb {P}U^\varepsilon||_{L^2_{xy}}+||\mathbb {P}U^\varepsilon||^2_{L^2_{xy}}||\nabla U^\theta||_{L^\infty_{xy}}\\
        &\;\;\;\;+||U^\theta||_{L^\infty_{xy}}||\nabla(I-\mathbb {P})U^\varepsilon||_{L^2_{xy}}||\mathbb {P}U^\varepsilon||_{L^2_{xy}}\\
        &\leq C||\mathbb {P}U^\varepsilon||^2_{L^2_{xy}}+C\varepsilon.\\
		\end{split}
	\end{equation*}
	 Combining Lemma $\ref{lem h L2}$, we can obtain
	\begin{equation*}
		\begin{split}{}
        N_4&\leq||H^\varepsilon||_{L^2_{xy}}||\mathbb {P}U^\varepsilon||_{L^2_{xy}}+\varepsilon^{-\frac{1}{2}}||h^\varepsilon||_{L^2_{xy}}||\mathbb {P}U^\varepsilon||_{L^2_{xy}}\leq C||\mathbb {P}U^\varepsilon||^2_{L^2_{xy}}+C||H^\varepsilon||^2_{L^2_{xy}}+C\varepsilon.\\
		\end{split}
	\end{equation*}
         Then it yields
    \begin{equation*}
		\begin{split}{}
			\frac{d}{dt}||\mathbb {P}U^\varepsilon||^2_{L^2}
            &\leq \frac{1}{2}\varepsilon\|\partial_y\omega^{\varepsilon*}\|^2_{L^2_{xy}}+C||H^\varepsilon||^2_{L^2_{xy}}+C||\mathbb {P}U^\varepsilon||^2_{L^2_{xy}}+C\varepsilon.
        \end{split}
	\end{equation*}
     \textbf{Step 4. Estimate of $\nabla U^\varepsilon$.}
     We note that
     \begin{equation*}
		\begin{split}{}
        \|\nabla U^\varepsilon\|_{L^2_{xy}}\leq\|\nabla\times U^\varepsilon\|_{L^2_{xy}}+\|\nabla\cdot U^\varepsilon\|_{L^2_{xy}}=\|\nabla\times U^\varepsilon\|_{L^2_{xy}}.
		\end{split}
	\end{equation*}
      Therefore, it suffices to estimate $\|\nabla\times U^\varepsilon\|_{L^2_{xy}}$. Set $\omega^{\varepsilon*}=\nabla\times U^\varepsilon$, then $(\ref{HVU equation})_4$ implies that
      \begin{equation}\label{omega varepsilon}
		\begin{split}{}
        \partial_t\omega^{\varepsilon*}&=\varepsilon\Delta \omega^{\varepsilon*}-\varepsilon^\frac{1}{2}\nabla\times(U^\varepsilon\cdot\nabla U^\varepsilon)-\nabla\times(U^\varepsilon\cdot\nabla U^\theta)-\nabla\times(U^\theta\cdot\nabla U^\varepsilon)\\
        &\quad+\nabla\times(H^\varepsilon e_2)+\varepsilon^{-\frac{1}{2}}\nabla\times h^\varepsilon,
		\end{split}
	\end{equation}
    with the boundary conditions as
    \begin{equation}\label{omega varepsilon y=0}
		\begin{split}{}
        \omega^{\varepsilon*}|_{y=0}=-(\varepsilon^\frac{1}{2}\partial_xu^{B,2}_2(t,x,0)+\varepsilon\partial_xu^{B,3}_2(t,x,0)).
		\end{split}
	\end{equation}
    We introduce a smooth cut-off function \(\varphi\) defined on \((0,\infty)\) as follows:
    $$\varphi(0)=1, \varphi(y)=1,\;{\rm{for}}\;y>1.$$
    Denote $\tilde{\omega}=\omega^{\varepsilon*}+\varphi(y)(\varepsilon^\frac{1}{2}\partial_xu^{B,2}_2(t,x,0)+\varepsilon\partial_xu^{B,3}_2(t,x,0))$. Then
    one deduces from $(\ref{omega varepsilon})$ that
    \begin{equation}\label{tilde omega}
    \left\{
		\begin{split}{}
        \partial_t\tilde{\omega}&=\varepsilon\partial^2_x\tilde{\omega}+\varepsilon\partial^2_y\tilde{\omega}-\varepsilon^\frac{1}{2}U^\varepsilon\cdot\nabla \tilde{\omega}-\nabla\times(U^\varepsilon\cdot\nabla U^\theta)-U^\theta\cdot\nabla \tilde{\omega}\\
        &\quad-(\partial_xU^\theta_1\partial_xU^\varepsilon_2+\partial_xU^\theta_2\partial_yU^\varepsilon_2-\partial_yU^\theta_1\partial_xU^\varepsilon_1-\partial_yU^\theta_2\partial_yU^\varepsilon_1)\\
        &\quad+\varepsilon^\frac{1}{2}U^\varepsilon_1\partial_x[(\varepsilon^\frac{1}{2}\partial_xu^{B,2}_2+\varepsilon\partial_xu^{B,3}_2)(t,x,0)]\varphi(y)\\
        &\quad+\varepsilon^\frac{1}{2}U^\varepsilon_2[(\varepsilon^\frac{1}{2}\partial_xu^{B,2}_2+\varepsilon\partial_xu^{B,3}_2)(t,x,0)]\partial_y\varphi(y)+\partial_xH^\varepsilon+\varepsilon^{-\frac{1}{2}}\nabla\times h^\varepsilon+\rho^\varepsilon,\\
        \tilde{\omega}(t&,x,0)=0,\\
        \tilde{\omega}(0&,x,y)=0,
		\end{split}
        \right.
	\end{equation}
    where
    \begin{align}\label{rho varepsilon}
        \rho^\varepsilon&=\partial_t[(\varepsilon^\frac{1}{2}\partial_xu^{B,2}_2+\varepsilon\partial_xu^{B,3}_2)(t,x,0)]\varphi(y)-\varepsilon\partial^2_x[(\varepsilon^\frac{1}{2}\partial_xu^{B,2}_2+\varepsilon\partial_xu^{B,3}_2)(t,x,0)]\varphi(y)\notag\\
        &\quad-\varepsilon[(\varepsilon^\frac{1}{2}\partial_xu^{B,2}_2+\varepsilon\partial_xu^{B,3}_2)(t,x,0)]\partial^2_y\varphi(y)+U^\theta_1\partial_x[(\varepsilon^\frac{1}{2}\partial_xu^{B,2}_2+\varepsilon\partial_xu^{B,3}_2)(t,x,0)]\varphi(y)\\
        &\quad+U^\theta_2[(\varepsilon^\frac{1}{2}\partial_xu^{B,2}_2+\varepsilon\partial_xu^{B,3}_2)(t,x,0)]\partial_y\varphi(y).\notag\\
        \rho^\varepsilon|_{y=0}&=\partial_t[(\varepsilon^\frac{1}{2}\partial_xu^{B,2}_2+\varepsilon\partial_xu^{B,3}_2)(t,x,0)]-\varepsilon\partial^2_x[(\varepsilon^\frac{1}{2}\partial_xu^{B,2}_2+\varepsilon\partial_xu^{B,3}_2)(t,x,0)]\notag\\
        &\quad+U^\theta_1(t,x,0)\partial_x[(\varepsilon^\frac{1}{2}\partial_xu^{B,2}_2+\varepsilon\partial_xu^{B,3}_2)(t,x,0)].\notag\notag
    \end{align}
    By Lemma \ref{lem inner layer} and (\ref{z=0 regularity}), we can obtain
    \begin{align}\label{partialx uB22uB32}
        \|(\varepsilon^\frac{1}{2}\partial_xu^{B,2}_2+\varepsilon\partial_xu^{B,3}_2)(t,x,0)\|_{L^2_{x}}&\leq\varepsilon^\frac{1}{2}\|(u^{B,2}_2+u^{B,3}_2)(t,x,z)\|_{H^1_{x}H^1_z}\leq C\varepsilon^\frac{1}{2}\notag\\
        \|(\varepsilon^\frac{1}{2}\partial_xu^{B,2}_2+\varepsilon\partial_xu^{B,3}_2)(t,x,0)\|_{L^\infty_{x}}&\leq\|(\varepsilon^\frac{1}{2}\partial_xu^{B,2}_2+\varepsilon\partial_xu^{B,3}_2)(t,x,z)\|_{L^\infty_{xz}}\\
        &\leq\varepsilon^\frac{1}{2}\|(u^{B,2}_2+u^{B,3}_2)(t,x,z)\|_{H^3_{x}H^2_z}\notag\\
        &\leq C\varepsilon^\frac{1}{2}.\notag\notag
    \end{align}

    Thus, combining the definition of $U^\theta$ and (\ref{partialx uB22uB32}), we obtain
    \begin{align}\label{rho varepsilon frac12}
      \|\rho^\varepsilon\|_{L^2_{xy}}&\leq C\varepsilon^\frac{1}{2}(\|U^\theta_1\|_{L^\infty_{xy}}+\|U^\theta_2\|_{L^\infty_{xy}})+C\varepsilon^\frac{1}{2}\leq C\varepsilon^\frac{1}{2}\notag,\\
      \|\rho^\varepsilon(t,x,0)\|_{L^2_{xy}}&\leq C\varepsilon^\frac{1}{2}(\|U^\theta_1\|_{H^1_{xy}}+1)\leq C\varepsilon^\frac{1}{2}.
    \end{align}
      Multiplying $(\ref{tilde omega})_1$ by $\tilde{\omega}$, we can obtain
       \begin{align*}
        &\frac{1}{2}\frac{d}{dt}\|\tilde{\omega}\|^2_{L^2_{xy}}+\varepsilon\|\partial_x \tilde{\omega}\|^2_{L^2_{xy}}+\varepsilon\|\partial_y \tilde{\omega}\|^2_{L^2_{xy}}\\
        &=-\varepsilon^\frac{1}{2}\int^\infty_0\int^\infty_{-\infty}U^\varepsilon\cdot\nabla \tilde{\omega}\tilde{\omega} dxdy\\
        &\quad-\int^\infty_0\int^\infty_{-\infty}\nabla\times(U^\varepsilon\cdot\nabla U^\theta)\tilde{\omega} dxdy-\int^\infty_0\int^\infty_{-\infty}U^\theta\cdot\nabla \tilde{\omega}\tilde{\omega} dxdy\\
        &\quad-\int^\infty_0\int^\infty_{-\infty}(\partial_xU^\theta_1\partial_xU^\varepsilon_2+\partial_xU^\theta_2\partial_yU^\varepsilon_2-\partial_yU^\theta_1\partial_xU^\varepsilon_1-\partial_yU^\theta_2\partial_yU^\varepsilon_1)\tilde{\omega} dxdy\\
        &\quad +\varepsilon^\frac{1}{2}\int^\infty_0\int^\infty_{-\infty}U^\varepsilon_1\partial_x[(\varepsilon^\frac{1}{2}\partial_xu^{B,2}_2+\varepsilon\partial_xu^{B,3}_2)(t,x,0)]\varphi(y)\tilde{\omega}dxdy\\
        &\quad+\varepsilon^\frac{1}{2}\int^\infty_0\int^\infty_{-\infty}U^\varepsilon_2[(\varepsilon^\frac{1}{2}\partial_xu^{B,2}_2+\varepsilon\partial_xu^{B,3}_2)(t,x,0)]\partial_y\varphi(y)\tilde{\omega}dxdy\\
        &\quad+\int^\infty_0\int^\infty_{-\infty}\partial_xH^\varepsilon\tilde{\omega} dxdy+\varepsilon^{-\frac{1}{2}}\int^\infty_0\int^\infty_{-\infty}(\nabla\times h^\varepsilon)\tilde{\omega}dxdy+\int^\infty_0\int^\infty_{-\infty}\rho^\varepsilon\tilde{\omega} dxdy\\
        &=\sum^{9}_{i=1}E_i.
	\end{align*}
    Since $U^\varepsilon_2(t,x,0)=-(\varepsilon^\frac{1}{2}u^{B,2}_2(t,x,0)+\varepsilon u^{B,3}_2(t,x,0))$, {with} the divergence condition and $(\ref{partialx uB22uB32})$, it can be verified that
\begin{equation*}
		\begin{split}{}
       E_1&=-\frac{1}{2}\varepsilon^\frac{1}{2}\int^\infty_{-\infty}U^\varepsilon_2(t,x,0)|\tilde{\omega}(t,x,0)|^2dx\\
            &\leq \varepsilon(\|u^{B,2}_2(t,x,0)\|_{L^\infty_{xy}}+\varepsilon^\frac{1}{2}\|u^{B,3}_2(t,x,0)\|_{L^\infty_{xy}})\|\tilde{\omega}\|^2_{L^2,y=0}\\
            &\leq \frac{1}{6}\varepsilon\|\partial_y \tilde{\omega}\|^2_{L^2_{xy}}+C\| \tilde{\omega}\|^2_{L^2_{xy}}.
		\end{split}
	\end{equation*}

    For the term $E_2$,
	\begin{equation*}
		\begin{split}{}
       E_2&=\int^\infty_0\int^\infty_{-\infty}(\partial_y U^\varepsilon_1\partial_x U^\theta_1+ U^\varepsilon_1\partial_x\partial_y U^\theta_1+\partial_y U^\varepsilon_2\partial_y U^\theta_1-\partial_x U^\varepsilon_1\partial_x U^\theta_2\\
       &\quad-U^\varepsilon_1\partial^2_x U^\theta_2-\partial_x U^\varepsilon_2\partial_y U^\theta_2-U^\varepsilon_2\partial_x\partial_y U^\theta_2)\tilde{\omega}dxdy+\int^\infty_0\int^\infty_{-\infty}U^\varepsilon_2\partial^2_yU^\theta_1\tilde{\omega}dxdy\\
       &=E_{21}+E_{22}.\\
		\end{split}
	\end{equation*}
    By the definition of $\tilde{\omega}$ and the estimates in $(\ref{HVU theta L2})$, we can get that
    \begin{align*}
       E_{21}&\leq(\|\partial_y U^\varepsilon_1\|_{L^2_{xy}}\|\partial_x U^\theta_1\|_{L^\infty_{xy}}+\| U^\varepsilon_1\|_{L^2_{xy}}\|\partial_x \partial_yU^\theta_1\|_{L^\infty_{xy}}+\|\partial_y U^\varepsilon_2\|_{L^2_{xy}}\|\partial_y U^\theta_1\|_{L^\infty_{xy}}\\
            &\quad+\|\partial_x U^\varepsilon_1\|_{L^2_{xy}}\|\partial_x U^\theta_2\|_{L^\infty_{xy}}+\| U^\varepsilon_1\|_{L^2_{xy}}\|\partial^2_x U^\theta_1\|_{L^\infty_{xy}}+\|\partial_x U^\varepsilon_2\|_{L^2_{xy}}\|\partial_y U^\theta_2\|_{L^\infty_{xy}}\\
            &\quad+\|U^\varepsilon_2\|_{L^2_{xy}}\|\partial_x\partial_y U^\theta_2\|_{L^\infty_{xy}})\|\tilde{\omega}\|_{L^2_{xy}}\\
            &\leq C\|\tilde{\omega}\|^2_{L^2_{xy}}+C\|\omega^\varepsilon\|^2_{L^2_{xy}}+C\| U^\varepsilon\|^2_{L^2_{xy}}\\
            &\leq C\|\tilde{\omega}\|^2_{L^2_{xy}}+C\| U^\varepsilon\|^2_{L^2_{xy}}+C\varepsilon.
	\end{align*}
    It can be obtained by the equation $U^\theta$ in $(\ref{HVUK varepsilon})$ and integration by parts,
    \begin{equation*}
		\begin{split}{}
       E_{22}&=\int^\infty_0\int^\infty_{-\infty}U^\varepsilon_2(\partial^2_y u^\varepsilon_1-\varepsilon^\frac{1}{2}\partial^2_y U^\varepsilon_1)\tilde{\omega}dxdy\\
       &=\int^\infty_0\int^\infty_{-\infty}U^\varepsilon_2\partial^2_y u^\varepsilon_1\tilde{\omega}dxdy-\varepsilon^\frac{1}{2}\int^\infty_0\int^\infty_{-\infty}U^\varepsilon_2\partial^2_y U^\varepsilon_1\tilde{\omega}dxdy\\
       &=E_{22_1}+E_{22_2}.
		\end{split}
	\end{equation*}
    By Theorem $\ref{th nvu regularity}$ and Hardy inequality, we can get that
    \begin{equation*}
		\begin{split}{}
       E_{22_1}&=\int^\infty_0\int^\infty_{-\infty}\frac{U^\varepsilon_2}{\psi} \psi\partial^2_y u^\varepsilon_1\tilde{\omega}dxdy\leq C(\|\partial_y U^\varepsilon_2\|_{L^2_{xy}}+\| U^\varepsilon_2\|_{L^2_{xy}})\|\partial_y u^\varepsilon_1\|_{W^{1,\infty}_{xy}}\|\tilde{\omega}\|_{L^2_{xy}}\\&\leq C\|\tilde{\omega}\|^2_{L^2_{xy}}+C\| U^\varepsilon\|^2_{L^2_{xy}}+C\varepsilon.\\
		\end{split}
	\end{equation*}
    By the divergence condition, substituting the definition of $\tilde{\omega}$ into $E_{22_2}$, using $\omega^{\varepsilon*}=\partial_xU^\varepsilon_2-\partial_yU^\varepsilon_1$ and (\ref{G-N inequality}), we obtain
    \begin{equation*}
		\begin{split}{}
       E_{22_2}&=\varepsilon^\frac{1}{2}\int^\infty_0\int^\infty_{-\infty}\partial_yU^\varepsilon_2\partial_y U^\varepsilon_1\tilde{\omega}dxdy+\varepsilon^\frac{1}{2}\int^\infty_0\int^\infty_{-\infty}U^\varepsilon_2\partial_y U^\varepsilon_1\partial_y\tilde{\omega}dxdy\\
       &=-\varepsilon^\frac{1}{2}\int^\infty_0\int^\infty_{-\infty}\partial_xU^\varepsilon_1\partial_y U^\varepsilon_1\tilde{\omega}dxdy\\
       &\quad+\varepsilon^\frac{1}{2}\int^\infty_0\int^\infty_{-\infty}U^\varepsilon_2\partial_y U^\varepsilon_1(\partial_y\omega^\varepsilon+[(\varepsilon^\frac{1}{2}\partial_xu^{B,2}_2+\varepsilon\partial_xu^{B,3}_2)(t,x,0)]\partial_y\varphi(y))dxdy\\
       &\leq\varepsilon^\frac{1}{2}\|\partial_x U^\varepsilon_1\|_{ L^2_{x}L^\infty_y}\|\partial_y U^\varepsilon_1\|_{L^\infty_{x}L^2_y}\|\tilde{\omega}\|_{L^2_{xy}}+C\varepsilon^\frac{1}{2}\|U^\varepsilon_2\|_{L^2_{xy}}\|\partial_y U^\varepsilon_1\|_{L^2_{xy}}+\varepsilon^\frac{1}{2}\int^\infty_0\int^\infty_{-\infty}U^\varepsilon_2\partial_y U^\varepsilon_1\partial_y\omega^{\varepsilon*} dxdy\\
       &\leq \frac{1}{4}\varepsilon\|\partial_x\tilde{\omega}\|^2_{L^2_{xy}}+C\|\tilde{\omega}\|^2_{L^2_{xy}}\|\tilde{\omega}\|^2_{L^2_{xy}}+C\| U^\varepsilon\|^2_{L^2_{xy}}+C\|\tilde{\omega}\|^2_{L^2_{xy}}\\&
       \quad+\varepsilon^\frac{1}{2}\int^\infty_0\int^\infty_{-\infty}U^\varepsilon_2\partial_y U^\varepsilon_1(\partial_x\partial_yU^\varepsilon_2-\partial^2_yU^\varepsilon_1)dxdy+C\varepsilon
		\end{split}
	\end{equation*}
    where
    \begin{align*}
    &\varepsilon^\frac{1}{2}\int^\infty_0\int^\infty_{-\infty}U^\varepsilon_2\partial_y U^\varepsilon_1\partial_x\partial_yU^\varepsilon_2dxdy-\varepsilon^\frac{1}{2}\int^\infty_0\int^\infty_{-\infty}U^\varepsilon_2\partial_y U^\varepsilon_1\partial^2_yU^\varepsilon_1dxdy\\
       &=-\varepsilon^\frac{1}{2}\int^\infty_0\int^\infty_{-\infty}U^\varepsilon_2\partial_y U^\varepsilon_1\partial^2_xU^\varepsilon_1dxdy-\varepsilon^\frac{1}{2}\int^\infty_0\int^\infty_{-\infty}U^\varepsilon_2\partial_y U^\varepsilon_1\partial^2_yU^\varepsilon_1dxdy\\
       &=\varepsilon^\frac{1}{2}\int^\infty_0\int^\infty_{-\infty}\partial_xU^\varepsilon_2\partial_y U^\varepsilon_1\partial_xU^\varepsilon_1dxdy-\frac{1}{2}\varepsilon^\frac{1}{2}\int^\infty_0\int^\infty_{-\infty}\partial_yU^\varepsilon_2|\partial_xU^\varepsilon_1|^2dxdy\\
       &\quad+\frac{1}{2}\varepsilon^\frac{1}{2}\int^\infty_{-\infty}U^\varepsilon_2(t,x,0)|\partial_xU^\varepsilon_1(t,x,0)|^2dx+\frac{1}{2}\varepsilon^\frac{1}{2}\int^\infty_0\int^\infty_{-\infty}\partial_yU^\varepsilon_2|\partial_y U^\varepsilon_1|^2dxdy\\
       &\leq \varepsilon^\frac{1}{2}\|\partial_xU^\varepsilon_2\|_{L^2_{xy}}\|\partial_y U^\varepsilon_1\|_{L^\infty_{x}L^2_y}\|\partial_x U^\varepsilon_1\|_{ L^2_{x}L^\infty_y}\\&\quad+\frac{1}{2}\varepsilon^\frac{1}{2}\|\partial_yU^\varepsilon_2\|_{L^2_{xy}}\|\partial_x U^\varepsilon_1\|_{L^\infty_{x}L^2_y}\|\partial_x U^\varepsilon_1\|_{ L^2_{x}L^\infty_y}\\
       &\quad+C\varepsilon(\|u^{B,2}_2(t,x,0)\|_{L^\infty_{xy}}+\varepsilon^\frac{1}{2}\|u^{B,3}_2(t,x,0)\|_{L^\infty_{xy}})\|\partial_xU^\varepsilon_1\|_{L^2_{xy}}\|\partial_x\partial_yU^\varepsilon_1\|_{L^2_{xy}}\\
       &\quad+\frac{1}{2}\varepsilon^\frac{1}{2}\|\partial_yU^\varepsilon_1\|_{L^2_{xy}}\|\partial_y U^\varepsilon_1\|_{L^\infty_{x}L^2_y}\|\partial_x U^\varepsilon_1\|_{L^2_{x}L^\infty_y}\\
       &\leq \frac{1}{4}\varepsilon\|\partial_x\tilde{\omega}\|^2_{L^2_{xy}}+\frac{1}{6}\varepsilon\|\partial_y\tilde{\omega}\|^2_{L^2_{xy}}+C\|\tilde{\omega}\|^2_{L^2_{xy}}\|\tilde{\omega}\|^2_{L^2_{xy}}+C\|\tilde{\omega}\|^2_{L^2_{xy}}+C\varepsilon.
	\end{align*}
    Then combining the estimates of $E_{21}$ and $E_{22}$, we can get
    \begin{equation*}
		\begin{split}{}
       E_{2}&\leq \frac{1}{2}\varepsilon\|\partial_x\tilde{\omega}\|^2_{L^2_{xy}}+\frac{1}{6}\varepsilon\|\partial_y\tilde{\omega}\|^2_{L^2_{xy}}+C\|\tilde{\omega}\|^2_{L^2_{xy}}\|\tilde{\omega}\|^2_{L^2_{xy}}+C\|\tilde{\omega}\|^2_{L^2_{xy}}+C\| U^\varepsilon\|^2_{L^2_{xy}}+C\varepsilon.
		\end{split}
	\end{equation*}
    Since $U^\theta_2(t,x,0)=\varepsilon u^{B,2}_2(t,x,0)+\varepsilon^\frac{3}{2} u^{B,3}_2(t,x,0)$, {with} the divergence condition $\nabla\cdot U^\theta=0$ and (\ref{partialx uB22uB32}), it can be verified that
\begin{equation*}
		\begin{split}{}
       E_3&=-\int^\infty_{-\infty}U^\theta_2(t,x,0)|\tilde{\omega}(t,x,0)|^2dx\\
            &\leq \varepsilon(\|u^{B,2}_2(t,x,0)\|_{L^\infty_{xy}}+\varepsilon^\frac{1}{2}\|u^{B,3}_2(t,x,0)\|_{L^\infty_{xy}})\|\tilde{\omega}\|^2_{L^2,y=0}\\
            &\leq \frac{1}{6}\varepsilon\|\partial_y \tilde{\omega}\|^2_{L^2_{xy}}+C\|\tilde{\omega}\|^2_{L^2_{xy}}.
		\end{split}
	\end{equation*}
     We use (\ref{HVU theta L2}) to get
     \begin{equation*}
		\begin{split}{}
       E_{4}&\leq(\|\partial_x U^\theta_1\|_{L^\infty_{xy}}\|\partial_x U^\varepsilon_2\|_{L^2_{xy}}+\|\partial_x U^\theta_2\|_{L^\infty_{xy}}\|\partial_y U^\varepsilon_2\|_{L^2_{xy}}\\
            &\quad+\|\partial_y U^\theta_1\|_{L^\infty_{xy}}\|\partial_x U^\varepsilon_1\|_{L^2_{xy}}+\|\partial_y U^\theta_2\|_{L^\infty_{xy}}\|\partial_y U^\varepsilon_1\|_{L^2_{xy}})\|\tilde{\omega}\|^2_{L^2_{xy}}\\
            &\leq C\|\tilde{\omega}\|^2_{L^2_{xy}}+C\|\omega^\varepsilon\|^2_{L^2_{xy}}\\
            &\leq C\|\tilde{\omega}\|^2_{L^2_{xy}}+C\varepsilon.
		\end{split}
	\end{equation*}
      Applying the estimates of (\ref{partialx uB22uB32}), it follows that
      \begin{equation*}
		\begin{split}{}
        E_{5}+E_{6}&\leq C\varepsilon^\frac{1}{2}(\|U^\varepsilon_1\|_{L^2_{xy}}\|\tilde{\omega}\|_{L^2_{xy}}+\|U^\varepsilon_2\|_{L^2_{xy}}\|\tilde{\omega}\|_{L^2_{xy}})\leq  C\|\tilde{\omega}\|^2_{L^2_{xy}}+C\|U^\varepsilon\|^2_{L^2_{xy}}.
		\end{split}
	\end{equation*}
    By H\"{o}lder's inequality, we obtain
	\begin{equation*}
		\begin{split}{}        E_{7}&\leq\|\partial_xH^\varepsilon\|_{L^2_{xy}}\|\tilde{\omega}\|_{L^2_{xy}}\leq \frac{1}{8}\|\partial_x H^\varepsilon\|^2_{L^2_{xy}}+C\|\tilde{\omega}\|^2_{L^2_{xy}},\\        E_{9}&\leq\|\rho^\varepsilon\|_{L^2_{xy}}\|\tilde{\omega}\|_{L^2_{xy}}\leq C\|\rho^\varepsilon\|^2_{L^2_{xy}}+C\|\tilde{\omega}\|^2_{L^2_{xy}}.\\
		\end{split}
	\end{equation*}
        By Lemma $\ref{lem h L2}$, we have
        \begin{equation*}
		\begin{split}{}
        E_{8}&\leq\varepsilon^{-\frac{1}{2}}\|\partial_x h^\varepsilon_2-\partial_y h^\varepsilon_1\|_{L^2_{xy}}\|\tilde{\omega}\|_{L^2_{xy}}\leq C\|\tilde{\omega}\|^2_{L^2_{xy}}+C\varepsilon^\frac{1}{2}.\\
		\end{split}
	\end{equation*}
        {It then yields}
       \begin{align*}
			&\frac{d}{dt}\|\tilde{\omega}\|^2_{L^2_{xy}}+\varepsilon\|\partial_x\tilde{\omega}\|^2_{L^2_{xy}}+\varepsilon\|\partial_y\tilde{\omega}\|^2_{L^2_{xy}}\\
            &\leq \frac{1}{4}\|\partial_x H^\varepsilon\|^2_{L^2_{xy}}+C\|\tilde{\omega}\|^2_{L^2_{xy}}\|\tilde{\omega}\|^2_{L^2_{xy}}+C\|\tilde{\omega}\|^2_{L^2_{xy}}+C\| U^\varepsilon\|^2_{L^2_{xy}}+C\varepsilon^\frac{1}{2}.
	\end{align*}
	Therefore, combining Step 1-Step 4, we can obtain
	\begin{equation}\label{fracddt HVU varepsilon frac12}
		\begin{split}{}
			&\frac{d}{dt}(\|H^\varepsilon\|^2_{L^2_{xy}}+\|V^\varepsilon\|^2_{L^2_{xy}}+\|\mathbb {P}U^\varepsilon\|^2_{L^2_{xy}}+\|\tilde{\omega}\|^2_{L^2_{xy}})+\|\partial_xH^\varepsilon\|^2_{L^2_{xy}}+\|\partial_yH^\varepsilon\|^2_{L^2_{xy}}\\
            &\quad+\varepsilon\|\partial_xV^\varepsilon\|^2_{L^2_{xy}}+\varepsilon\|\partial_yV^\varepsilon\|^2_{L^2_{xy}}+\varepsilon\|\partial_x\tilde{\omega}\|^2_{L^2_{xy}}+\varepsilon\|\partial_y\tilde{\omega}\|^2_{L^2_{xy}}\\
			&\leq C\|H^\varepsilon\|^2_{L^2_{xy}}+C\|V^\varepsilon\|^2_{L^2_{xy}}+C\|\mathbb {P}U^\varepsilon\|^2_{L^2_{xy}}+C\|\tilde{\omega}\|^2_{L^2_{xy}}+C\|\tilde{\omega}\|^2_{L^2_{xy}}\|\tilde{\omega}\|^2_{L^2_{xy}}+C\varepsilon^\frac{1}{2},\\
		\end{split}
	\end{equation}
    along with Gronwall's inequality for $[0,T]$, which yields
	 \begin{equation*}
	 \begin{split}{}
	 	&\|H^\varepsilon\|^2_{L^2_{xy}}+\|V^\varepsilon\|^2_{L^2_{xy}}+\| U^\varepsilon\|^2_{L^2_{xy}}+\|\tilde{\omega}\|^2_{L^2_{xy}}\leq C\varepsilon^\frac{1}{2}.\\			
	 \end{split}
	 \end{equation*}
	 Hence by the definition of $\tilde{\omega}$ and integrating $(\ref{fracddt HVU varepsilon frac12})$ over the time, we obtain $(\ref{HVUL2 varepsilon})$-$(\ref{HVUL2T varepsilon})$. The proof is complete. \\

     \subsection{Proof of Theorem $\ref{th L2 convergence}$}\label{L2 convergence}

    Next, we prove Theorem $\ref{th L2 convergence}$ by the Lemma $\ref{lem HVU L2}$.\\

     \textbf{Proof of $(\ref{HVU L^2xy estimate})$.} Because of the conclusion {$(\ref{HVUL2 varepsilon})$} being stronger than hypothesis {$(\ref{HVU assumption})$} in Lemma $\ref{lem HVU L2}$, it can be inferred that all conclusions in Lemma $\ref{lem HVU L2}$ hold true.\\

     \textbf{Proof of $(\ref{L^2-convergence})$.} By the fact that $(H^\varepsilon,V^\varepsilon,U^\varepsilon)$ uniquely solves problem $(\ref{HVU equation})$. Then for any $\varepsilon\in(0, \varepsilon_0]$ by $(\ref{HVUL2 varepsilon})$, we get
    \begin{align*}
        &\|n^\varepsilon(t,x,y)-n^0(t,x,y)\|_{L^\infty_TL^2_{xy}}\\
        &\leq C\varepsilon^{\frac{1}{2}}(\|n^{B,1}\|_{L^\infty_TL^2_{xz}}+\varepsilon^{\frac{1}{2}}\|n^{B,2}\|_{L^\infty_TL^2_{xz}}+\|H^\varepsilon\|_{L^\infty_TL^2_{xy}})\\
            &\leq C\varepsilon^{\frac{1}{2}},\\
				&\|v^\varepsilon(t,x,y)-v^0(t,x,y)-(0,v^{B,0}_2(t,x,\frac{y}{\sqrt{\varepsilon}}))\|_{L^\infty_TL^2_{xy}}\\
                &\leq C\varepsilon^{\frac{1}{2}}(\|v^{B,1}_1\|_{L^\infty_TL^2_{xz}}+\|v^{B,1}_2\|_{L^\infty_TL^2_{xz}}+\varepsilon^{\frac{1}{2}}\|v^{B,2}_1\|_{L^\infty_TL^2_{xz}}+\|V^\varepsilon\|_{L^\infty_TL^2_{xy}})\\
                &\leq C\varepsilon^{\frac{1}{2}},\\
				&\|u^\varepsilon(t,x,y)-u^0(t,x,y)\|_{L^\infty_TL^2_{xy}}\\
                &\leq C\varepsilon^{\frac{1}{2}}(\|u^{B,1}_1\|_{L^\infty_TL^2_{xz}}+\varepsilon^{\frac{1}{2}}\|u^{B,2}_1\|_{L^\infty_TL^2_{xz}}+\varepsilon^{\frac{1}{2}}\|u^{B,2}_2\|_{L^\infty_TL^2_{xz}}+\varepsilon\|u^{B,3}_2\|_{L^\infty_TL^2_{xz}}+\|U^\varepsilon\|_{L^\infty_TL^2_{xy}})\\
                &\leq C\varepsilon^{\frac{1}{2}}.
	\end{align*}

\subsection{Tangential and normal derivatives estimations of remainders}\label{de hvu}

We first make the following estimate of $(H^\varepsilon,V^\varepsilon,U^\varepsilon)$.

\begin{lemma}\label{lem HVUx L2}
         {Let the assumptions in Lemma $\ref{lem HVU L2}$ hold. Then there exists $\varepsilon_0>0$ such that for any $\varepsilon\in(0,\varepsilon_0]$, the following estimates hold:
		\begin{equation}\label{partialx HVU varepsilon frac12}
			\|\partial_xH^\varepsilon\|^2_{L^\infty_TL^2_{xy}}+\|\partial_xV^\varepsilon\|^2_{L^\infty_TL^2_{xy}}+\|\partial_x\nabla U^\varepsilon\|^2_{L^\infty_TL^2_{xy}}\leq C\varepsilon^\frac{1}{2}.\\
		\end{equation}
        Furthermore, there exists a constant $C$ indepenfent of $\varepsilon$ such that
            \begin{equation}\label{partialx L2T HVU varepsilon frac12}
			\begin{split}{}&\|\partial^2_xH^\varepsilon\|^2_{L^2_TL^2_{xy}}+\|\partial_x\partial_yH^\varepsilon\|^2_{L^2_TL^2_{xy}}+\varepsilon\|\partial^2_xV^\varepsilon\|^2_{L^2_TL^2_{xy}}+\varepsilon\|\partial^2_x\nabla U^\varepsilon\|^2_{L^2_TL^2_{xy}}\leq C\varepsilon^\frac{1}{2}.\\
            \end{split}
		\end{equation}
		}
	\end{lemma}
	\noindent{\bf{Proof.}} Owing to $x\in(-\infty,\infty)$, it is evident that the proof process of Lemma \ref{lem HVUx L2} is similar to that of Lemma \ref{lem HVU L2}, and we omit it.

    \begin{lemma}\label{lem HVUy L2}
         {Let the assumptions in Lemma $\ref{lem HVU L2}$ and Lemma $\ref{lem HVUx L2}$ hold. Then there exists $\varepsilon_0>0$ such that for any $\varepsilon\in(0,\varepsilon_0]$, the following estimates hold:
		\begin{equation}\label{partialy HVU varepsilon frac12}
        \varepsilon^{\frac{1}{2}}\|\partial_yH^\varepsilon\|^2_{L^\infty_TL^2_{xy}}+\varepsilon^{\frac{1}{2}}\|\partial_yV^\varepsilon\|^2_{L^\infty_TL^2_{xy}}+\varepsilon^{\frac{1}{2}}\|\partial_y\nabla U^\varepsilon\|^2_{L^\infty_TL^2_{xy}}\leq C.\\
		\end{equation}
        Furthermore, there exists a constant $C$ indepenfent of $\varepsilon$ such that
            \begin{equation}\label{partialx L2T HVU varepsilon frac12}
			\begin{split}
            &\varepsilon^{\frac{1}{2}}\|\partial^2_yH^\varepsilon\|^2_{L^2_TL^2_{xy}}+\varepsilon^{\frac{3}{2}}\|\partial^2_yV^\varepsilon\|^2_{L^2_TL^2_{xy}}+\varepsilon^{\frac{3}{2}}\|\partial^2_y\nabla U^\varepsilon\|^2_{L^2_TL^2_{xy}}\leq C.\\
            \end{split}
		\end{equation}
		}
	\end{lemma}
	\noindent{\bf{Proof.}}
    \textbf{Step 1. The \(L^2\) estimate of $\partial_yH^\varepsilon$.}
    Multiplying $(\ref{HVU equation})_1$ by $-\partial^2_yH^\varepsilon$, we obtain

        \begin{align*}
			&\frac{1}{2}\frac{d}{dt}\|\partial_yH^\varepsilon\|^2_{L^2_{xy}}+\|\partial_x\partial_y H^\varepsilon\|^2_{L^2_{xy}}+\|\partial^2_y H^\varepsilon\|^2_{L^2_{xy}}\\
            &=\varepsilon^\frac{1}{2}\Big{[}\int^\infty_0\int^\infty_{-\infty}U^\varepsilon\cdot\nabla H^\varepsilon\partial^2_y H^\varepsilon dxdy-\int^\infty_0\int^\infty_{-\infty}(V^\varepsilon \cdot\nabla H^\varepsilon +H^\varepsilon\nabla\cdot V^\varepsilon)\partial^2_yH^\varepsilon dxdy\Big{]}\\
            &\quad+\Big{[}\int^\infty_0\int^\infty_{-\infty}U^\varepsilon\cdot\nabla H^\theta\partial^2_y H^\varepsilon dxdy+\int^\infty_0\int^\infty_{-\infty}U^\theta\cdot\nabla H^\varepsilon\partial^2_y H^\varepsilon dxdy\Big{]}\\
            &\quad+\Big{[}-\int^\infty_0\int^\infty_{-\infty} (H^\varepsilon\nabla\cdot V^\theta+V^\theta\cdot\nabla H^\varepsilon)\partial^2_yH^\varepsilon dxdy\\
            &\quad-\int^\infty_0\int^\infty_{-\infty}(V^\varepsilon\cdot\nabla H^\theta+H^\theta\nabla\cdot V^\varepsilon)\partial^2_y H^\varepsilon dxdy\Big{]}-\varepsilon^{-\frac{1}{2}}\int^\infty_0\int^\infty_{-\infty}f^\varepsilon\partial^2_y H^\varepsilon dxdy\\
            &=\mathbb{D}_1+\mathbb{D}_2+\mathbb{D}_3+\mathbb{D}_4.
	\end{align*}
	Using the Gagliardo-Nirenberg interpolation inequality, we can get that
	\begin{align*}
        \mathbb{D}_1&=\varepsilon^\frac{1}{2}\int^\infty_0\int^\infty_{-\infty}(U^\varepsilon_1\partial_x H^\varepsilon+U^\varepsilon_2\partial_y H^\varepsilon-V^\varepsilon_1\partial_x H^\varepsilon-V^\varepsilon_2\partial_y H^\varepsilon- H^\varepsilon\partial_xV^\varepsilon_1-H^\varepsilon\partial_yV^\varepsilon_2)\partial^2_yH^\varepsilon dxdy\\
        &\leq\varepsilon^\frac{1}{2}(\|U^\varepsilon_1\|_{L^\infty_{x}L^2_{y}}\|\partial_x H^\varepsilon\|_{L^2_{x}L^\infty_{y}}+\|U^\varepsilon_2\|_{L^2_{x}L^\infty_{y}}\|\partial_y H^\varepsilon\|_{L^\infty_{x}L^2_{y}}+\|V^\varepsilon_1\|_{L^\infty_{x}L^2_{y}}\|\partial_x H^\varepsilon\|_{L^2_{x}L^\infty_{y}}\\
        &\quad+\|V^\varepsilon_2\|_{L^2_{x}L^\infty_{y}}\|\partial_y H^\varepsilon\|_{L^\infty_{x}L^2_{y}}+\|H^\varepsilon\|_{L^\infty_{x}L^2_{y}}\|\partial_x V^\varepsilon_1\|_{L^2_{x}L^\infty_{y}}+\|H^\varepsilon\|_{L^2_{x}L^\infty_{y}}\|\partial_y V^\varepsilon_2\|_{L^\infty_{x}L^2_{y}})\|\partial^2_yH^\varepsilon\|_{L^2_{xy}}\\
        &\leq \frac{1}{4}\|\partial_x\partial_y H^\varepsilon\|^2_{L^2_{xy}}+\frac{1}{8}\|\partial^2_y H^\varepsilon\|^2_{L^2_{xy}}+\frac{1}{4}\varepsilon\|\partial_x\partial_y V^\varepsilon\|^2_{L^2_{xy}}+C(1+\|\partial_y V^\varepsilon\|^2_{L^2_{xy}})\|\partial_yH^\varepsilon\|^2_{L^2_{xy}}+C\varepsilon^\frac{3}{2}.
	\end{align*}
         Using $(\ref{HVU theta L2})$, we obtain
	\begin{align*}
        \mathbb{D}_2&=\int^\infty_0\int^\infty_{-\infty}(U^\varepsilon_1\partial_x H^\theta+U^\varepsilon_2\partial_y H^\theta+U^\theta_1\partial_x H^\varepsilon+U^\theta_2\partial_y H^\varepsilon)\partial^2_yH^\varepsilon dxdy\\
        &\leq(\|U^\varepsilon_1\|_{L^2_{xy}}\|\partial_x H^\theta\|_{L^\infty_{xy}}+\|U^\varepsilon_2\|_{L^2_{xy}}\|\partial_y H^\theta\|_{L^\infty_{xy}}+\|U^\theta_1\|_{L^\infty_{xy}}\|\partial_x H^\varepsilon\|_{L^2_{xy}}\\
        &\quad+\|U^\theta_2\|_{L^\infty_{xy}}\|\partial_y H^\varepsilon\|_{L^2_{xy}})\|\partial^2_yH^\varepsilon\|_{L^2_{xy}}\\
        &\leq \frac{1}{8}\|\partial^2_y H^\varepsilon\|^2_{L^2_{xy}}+C\|\partial_yH^\varepsilon\|^2_{L^2_{xy}}+C\varepsilon^\frac{1}{2}.
	\end{align*}
   Using embedding inequalities for $\mathbb{D}_{3}$,
    \begin{align*}
        \mathbb{D}_{3}&=-\int^\infty_0\int^\infty_{-\infty}H^\varepsilon\partial_xV^\theta_1+H^\varepsilon\partial_yV^\theta_2+V^\theta_1\partial_x H^\varepsilon+V^\theta_2\partial_y H^\varepsilon\\
        &\quad+V^\varepsilon_1\partial_x H^\theta+V^\varepsilon_2\partial_y H^\theta+ H^\theta\partial_xV^\varepsilon_1+H^\theta\partial_yV^\varepsilon_2)\partial^2_yH^\varepsilon dxdy\\
        &\leq( \|H^\varepsilon\|_{L^2_{xy}}\|\partial_xV^\theta_1\|_{L^\infty_{xy}}+\|H^\varepsilon\|_{L^2_{xy}}\|\partial_yV^\theta_2\|_{L^\infty_{xy}}+\|V^\theta_1\|_{L^\infty_{xy}}\|\partial_x H^\varepsilon\|_{L^2_{xy}}\\
        &\quad+\|V^\theta_2\|_{L^\infty_{xy}}\|\partial_y H^\varepsilon\|_{L^2_{xy}}+\|V^\varepsilon_1\|_{L^2_{xy}}\|\partial_x H^\theta\|_{L^\infty_{xy}}+\|V^\varepsilon_2\|_{L^2_{xy}}\|\partial_y H^\theta\|_{L^\infty_{xy}}\\
        &\quad+\|H^\theta\|_{L^\infty_{xy}}\|\partial_xV^\varepsilon_1\|_{L^2_{xy}}+\|H^\theta\|_{L^\infty_{xy}}\|\partial_yV^\varepsilon_2\|_{L^2_{xy}})\|\partial^2_yH^\varepsilon\|_{L^2_{xy}}\\
        &\leq \frac{1}{8}\|\partial^2_y H^\varepsilon\|^2_{L^2_{xy}}+C\|\partial_yH^\varepsilon\|^2_{L^2_{xy}}+C\|\partial_yV^\varepsilon\|^2_{L^2_{xy}}+C\varepsilon^{-\frac{1}{2}}.
	\end{align*}
	 Combining Lemma $\ref{lem f L2}$, we can collect items with similar estimates
	\begin{align*}
        \mathbb{D}_4&\leq\varepsilon^{-\frac{1}{2}}\|f^\varepsilon\|_{L^2_{xy}}\|\partial^2_yH^\varepsilon\|_{L^2_{xy}}\leq \frac{1}{8}\|\partial^2_yH^\varepsilon\|^2_{L^2_{xy}}+C\varepsilon^\frac{1}{2}.
	\end{align*}
    {Obviously,}
    \begin{align*}
			&\frac{d}{dt}\|\partial_yH^\varepsilon\|^2_{L^2_{xy}}+\|\partial_x\partial_y H^\varepsilon\|^2_{L^2_{xy}}+\|\partial^2_y H^\varepsilon\|^2_{L^2_{xy}}\\
            &\leq \frac{1}{4}\|\partial_x\partial_y H^\varepsilon\|^2_{L^2_{xy}}+\frac{1}{2}\|\partial^2_y H^\varepsilon\|^2_{L^2_{xy}}+\frac{1}{4}\varepsilon\|\partial_x\partial_y V^\varepsilon\|^2_{L^2_{xy}}\\
            &\quad+C(1+\|\partial_y H^\varepsilon\|^2_{L^2_{xy}})(1+\|\partial_y V^\varepsilon\|^2_{L^2_{xy}})+C\varepsilon^{-\frac{1}{2}}.
	\end{align*}
    \textbf{Step 2. The \(L^2\) estimate of $\partial_yV^\varepsilon$.}
      Multiplying $(\ref{HVU equation})_2$ by $-\partial^2_yV^\varepsilon$, we can obtain
	\begin{align*}
			&\frac{1}{2}\frac{d}{dt}\|\partial_yV^\varepsilon\|^2_{L^2_{xy}}+\varepsilon\|\partial_x\partial_y V^\varepsilon\|^2_{L^2_{xy}}+\varepsilon\|\partial^2_y V^\varepsilon\|^2_{L^2_{xy}}\\
            &=\Big{[}\varepsilon^\frac{1}{2}\int^\infty_0\int^\infty_{-\infty}\partial_x(U^\varepsilon\cdot V^\varepsilon)\partial^2_y V^\varepsilon_1 dxdy+\varepsilon^\frac{1}{2}\int^\infty_0\int^\infty_{-\infty}\partial_y(U^\varepsilon\cdot V^\varepsilon)\partial^2_y V^\varepsilon_2 dxdy\Big{]}\\
            &\quad+\Big{[}\int^\infty_0\int^\infty_{-\infty}\partial_x(U^\varepsilon\cdot V^\theta)\partial^2_yV^\varepsilon_1 dxdy+\int^\infty_0\int^\infty_{-\infty}\partial_y(U^\varepsilon\cdot V^\theta)\partial^2_yV^\varepsilon_2 dxdy\Big{]}\\
            &\quad+\Big{[}\int^\infty_0\int^\infty_{-\infty}\partial_x(U^\theta\cdot V^\varepsilon)\partial^2_yV^\varepsilon_1 dxdy+\int^\infty_0\int^\infty_{-\infty}\partial_y(U^\theta\cdot V^\varepsilon)\partial^2_yV^\varepsilon_2 dxdy\Big{]}\\
            &\quad+\Big{[}2\varepsilon^\frac{3}{2}\int^\infty_0\int^\infty_{-\infty}(\partial_xV^\varepsilon\cdot V^\varepsilon)\partial^2_yV^\varepsilon_1 dxdy+2\varepsilon^\frac{3}{2}\int^\infty_0\int^\infty_{-\infty}(\partial_yV^\varepsilon\cdot V^\varepsilon)\partial^2_yV^\varepsilon_2 dxdy\Big{]}\\
            &\quad+\Big{[}2\varepsilon\int^\infty_0\int^\infty_{-\infty}\partial_x(V^\varepsilon\cdot V^\theta)\partial^2_yV^\varepsilon_1 dxdy+2\varepsilon\int^\infty_0\int^\infty_{-\infty}\partial_y(V^\varepsilon\cdot V^\theta)\partial^2_yV^\varepsilon_2 dxdy\Big{]}\\
            &\quad-\Big{[}\int^\infty_0\int^\infty_{-\infty}\partial_x H^\varepsilon \partial^2_yV^\varepsilon_1 dxdy+\int^\infty_0\int^\infty_{-\infty}\partial_y H^\varepsilon\partial^2_yV^\varepsilon_2 dxdy\Big{]}-\varepsilon^{-\frac{1}{2}}\int^\infty_0\int^\infty_{-\infty}g^\varepsilon\cdot\partial^2_yV^\varepsilon dxdy\\
            &=\mathbb{M}_1+\mathbb{M}_2+\mathbb{M}_3+\mathbb{M}_4+\mathbb{M}_5+\mathbb{M}_6+\mathbb{M}_7.
	\end{align*}
        Using the Gagliardo-Nirenberg interpolation inequality, we can {get} that
	\begin{align*}
       \mathbb{M}_1&=\varepsilon^\frac{1}{2}\int^\infty_0\int^\infty_{-\infty}[\partial_x(U^\varepsilon_1V^\varepsilon_1+U^\varepsilon_2V^\varepsilon_2)\partial^2_yV^\varepsilon_1+\partial_y(U^\varepsilon_1V^\varepsilon_1+U^\varepsilon_2V^\varepsilon_2)\partial^2_yV^\varepsilon_2]dxdy\\
       &\leq\varepsilon^{\frac{1}{2}}(\|\partial_x U^\varepsilon_1\|_{L^2_{xy}}\|V^\varepsilon_1\|_{L^\infty_{xy}}+\|U^\varepsilon_1\|_{L^\infty
       _{x}L^2_{y}}\|\partial_x V^\varepsilon_1\|_{L^2_{x}L^\infty_{y}}+\|\partial_x U^\varepsilon_2\|_{L^2_{xy}}\|V^\varepsilon_2\|_{L^\infty_{xy}}\\
       &\quad+\|U^\varepsilon_2\|_{L^\infty
       _{x}L^2_{y}}\|\partial_x V^\varepsilon_2\|_{L^2_{x}L^\infty_{y}})\|\partial^2_y V^\varepsilon_1\|_{L^2_{xy}}+\varepsilon^\frac{1}{2}(\|\partial_y U^\varepsilon_1\|_{L^2_{xy}}\|V^\varepsilon_1\|_{L^\infty_{xy}}+\|U^\varepsilon_1\|_{L^2_{x}L^\infty_{y}}\|\partial_y V^\varepsilon_1\|_{L^\infty
       _{x}L^2_{y}}\\
       &\quad+\|\partial_y U^\varepsilon_2\|_{L^2_{xy}}\|V^\varepsilon_2\|_{L^\infty_{xy}}+\|U^\varepsilon_2\|_{L^2_{x}L^\infty_{y}}\|\partial_y V^\varepsilon_2\|_{L^\infty
       _{x}L^2_{y}})\|\partial^2_y V^\varepsilon_2\|_{L^2_{xy}}\\
        &\leq\frac{1}{6}\varepsilon\|\partial_x\partial_y V^\varepsilon\|^2_{L^2_{xy}}+\frac{1}{12}\varepsilon\|\partial^2_y V^\varepsilon\|^2_{L^2_{xy}}+C\|\partial_y V^\varepsilon\|^2_{L^2_{xy}}+C\varepsilon^{\frac{1}{2}}.
	\end{align*}
         Combining Theorem \ref{th nvu regularity}, $V^\theta=v^\varepsilon-\varepsilon^\frac{1}{2}V^\varepsilon$ in $(\ref{HVUK varepsilon})$, $(\ref{HVU theta L2})$ and Lemma \ref{lem HVU L2}, we obtain
	\begin{align*}
        \mathbb{M}_2&=\int^\infty_0\int^\infty_{-\infty}[\partial_x(U^\varepsilon_1V^\theta_1+U^\varepsilon_2V^\theta_2)\partial^2_yV^\varepsilon_1+\partial_y(U^\varepsilon_1V^\theta_1+U^\varepsilon_2V^\theta_2)\partial^2_yV^\varepsilon_2]dxdy\\
        &=\int^\infty_0\int^\infty_{-\infty}(\partial_xU^\varepsilon_1V^\theta_1+U^\varepsilon_1\partial_xV^\theta_1+\partial_xU^\varepsilon_2V^\theta_2+U^\varepsilon_2\partial_xV^\theta_2)\partial^2_yV^\varepsilon_1\\
        &\quad+[\partial_yU^\varepsilon_1V^\theta_1+U^\varepsilon_1\partial_yV^\theta_1+\partial_yU^\varepsilon_2V^\theta_2+U^\varepsilon_2\partial_y(v^\varepsilon_2-\varepsilon^\frac{1}{2}V^\varepsilon_2)]\partial^2_yV^\varepsilon_2dxdy\\
       &\leq\varepsilon^{-\frac{1}{2}}(\|\partial_x U^\varepsilon_1\|_{L^2_{xy}}\|V^\theta_1\|_{L^\infty_{xy}}+\| U^\varepsilon_1\|_{L^2_{xy}}\|\partial_xV^\theta_1\|_{L^\infty_{xy}}\\
       &\quad+\|\partial_x U^\varepsilon_2\|_{L^2_{xy}}\|V^\theta_2\|_{L^\infty_{xy}}+\| U^\varepsilon_2\|_{L^2_{xy}}\|\partial_xV^\theta_2\|_{L^\infty_{xy}})\varepsilon^{\frac{1}{2}}\|\partial^2_yV^\varepsilon_1\|^2_{L^2_{xy}}\\
       &\quad+\varepsilon^{-\frac{1}{2}}(\|\partial_y U^\varepsilon_1\|_{L^2_{xy}}\|V^\theta_1\|_{L^\infty_{xy}}+\| U^\varepsilon_1\|_{L^2_{xy}}\|\partial_yV^\theta_1\|_{L^\infty_{xy}}+\|\partial_y U^\varepsilon_2\|_{L^2_{xy}}\|V^\theta_2\|_{L^\infty_{xy}})\varepsilon^{\frac{1}{2}}\|\partial^2_yV^\varepsilon_2\|^2_{L^2_{xy}}\\
       &\quad+\varepsilon^{-\frac{1}{2}}\| U^\varepsilon_2\|_{L^2_{xy}}\|\partial_yv^\varepsilon_2\|_{L^\infty_{xy}}\varepsilon^{\frac{1}{2}}\|\partial^2_yV^\varepsilon_2\|_{L^2_{xy}}+\varepsilon^\frac{1}{2}\|U^\varepsilon_2\|_{L^2_{x}L^\infty_{y}}\|\partial_y V^\varepsilon_2\|_{L^\infty
       _{x}L^2_{y}}\|\partial^2_y V^\varepsilon_2\|_{L^2_{xy}}\\
        &\leq\frac{1}{6}\varepsilon\|\partial_x\partial_y V^\varepsilon\|^2_{L^2_{xy}}+\frac{1}{12}\varepsilon\|\partial^2_y V^\varepsilon\|^2_{L^2_{xy}}+C\|\partial_y V^\varepsilon\|^2_{L^2_{xy}}+C\varepsilon^{-\frac{1}{2}}.
	\end{align*}
    By $(\ref{HVU theta L2})$ and $U^\theta_2(t,x,0)=\varepsilon u^{B,2}_2(t,x,0)+\varepsilon^\frac{3}{2} u^{B,3}_2(t,x,0)$ we have
    \begin{align*}
        \mathbb{M}_3&=\int^\infty_0\int^\infty_{-\infty}[\partial_x(U^\theta_1V^\varepsilon_1+U^\theta_2V^\varepsilon_2)\partial^2_yV^\varepsilon_1+\partial_y(U^\theta_1V^\varepsilon_1+U^\theta_2V^\varepsilon_2)\partial^2_yV^\varepsilon_2]dxdy\\
       &\leq\varepsilon^{-\frac{1}{2}}(\|\partial_x U^\theta_1\|_{L^\infty_{xy}}\|V^\varepsilon_1\|_{L^2_{xy}}+\| U^\theta_1\|_{L^\infty_{xy}}\|\partial_xV^\varepsilon_1\|_{L^2_{xy}}\\
       &\quad+\|\partial_x U^\theta_2\|_{L^\infty_{xy}}\|V^\varepsilon_2\|_{L^2_{xy}}+\| U^\theta_2\|_{L^\infty_{xy}}\|\partial_xV^\varepsilon_2\|_{L^2_{xy}})\varepsilon^\frac{1}{2}\|\partial^2_yV^\varepsilon_1\|_{L^2_{xy}}\\
       &\quad+\varepsilon^{-\frac{1}{2}}(\|\partial_y U^\theta_1\|_{L^\infty_{xy}}\|V^\varepsilon_1\|_{L^2_{xy}}+\| U^\theta_1\|_{L^\infty_{xy}}\|\partial_xV^\varepsilon_2\|_{L^2_{xy}}+\|\partial_y U^\theta_2\|_{L^\infty_{xy}}\|V^\varepsilon_2\|_{L^2_{xy}})\varepsilon^\frac{1}{2}\|\partial^2_yV^\varepsilon_2\|_{L^2_{xy}}\\
       &\quad+\frac{1}{2}\|U^\theta_2(t,x,0)\|_{L^\infty_{xy}}\|\partial_yV^\varepsilon_2(t,x,0)\|^2_{L^2_{xy}}+\frac{1}{2}\|\partial_y U^\theta_2\|_{L^\infty_{xy}}\|\partial_yV^\varepsilon_2\|^2_{L^2_{xy}}\\
        &\leq\frac{1}{12}\varepsilon\|\partial^2_y V^\varepsilon\|^2_{L^2_{xy}}+C\|\partial_y V^\varepsilon\|^2_{L^2_{xy}}+C\varepsilon^{-\frac{1}{2}}.
		\end{align*}
	 Using the Sobolev embedding inequality, we can get that
	\begin{equation*}
		\begin{split}{}
        \mathbb{M}_4&=\int^\infty_0\int^\infty_{-\infty}2\varepsilon^\frac{3}{2}[(\partial_xV^\varepsilon_1V^\varepsilon_1+\partial_xV^\varepsilon_2V^\varepsilon_2)\partial^2_yV^\varepsilon_1+(\partial_yV^\varepsilon_1V^\varepsilon_1+\partial_yV^\varepsilon_2V^\varepsilon_2)\partial^2_yV^\varepsilon_2]dxdy\\
      &\leq\frac{1}{6}\varepsilon\|\partial_x\partial_y V^\varepsilon\|^2_{L^2_{xy}}+\frac{1}{12}\varepsilon\|\partial^2_y V^\varepsilon\|^2_{L^2_{xy}}+C\|\partial_yV^\varepsilon\|^4_{L^2_{xy}}+C\varepsilon.\\
		\end{split}
	\end{equation*}
    And
    \begin{equation*}
		\begin{split}{}
        \mathbb{M}_5&=-\int^\infty_0\int^\infty_{-\infty}2\varepsilon[\partial_x(V^\varepsilon_1V^\theta_1+V^\varepsilon_2V^\theta_2)\partial^2_yV^\varepsilon_1+\partial_y(V^\varepsilon_1V^\theta_1+V^\varepsilon_2V^\theta_2)\partial^2_yV^\varepsilon_2]dxdy\\
        &\leq\frac{1}{12}\varepsilon\|\partial^2_y V^\varepsilon\|^2_{L^2_{xy}}+C\|\partial_yV^\varepsilon\|^2_{L^2_{xy}}+C\varepsilon.\\
		\end{split}
	\end{equation*}
    Using integrating by parts, we have
    \begin{equation*}
		\begin{split}{}
        \mathbb{M}_6&\leq\|\partial_x\partial_y H^\varepsilon\|_{L^2_{xy}}\|\partial_y V^\varepsilon_1\|_{L^2_{xy}}+\|\partial^2_y H^\varepsilon\|_{L^2_{xy}}\|\partial_y V^\varepsilon_2\|_{L^2_{xy}}\\
        &\leq\frac{1}{4}\|\partial_x\partial_y H^\varepsilon\|^2_{L^2_{xy}}+\frac{1}{8}\|\partial^2_y H^\varepsilon\|^2_{L^2_{xy}}+C\|\partial_yV^\varepsilon\|^2_{L^2_{xy}}.
		\end{split}
	\end{equation*}
    By Lemma \ref{lem g L2}, we have
    \begin{equation*}
		\begin{split}{}
        \mathbb{M}_7&\leq\varepsilon^{-\frac{1}{2}}\|g^\varepsilon\|_{L^2_{xy}}\|\partial^2_y V^\varepsilon\|_{L^2_{xy}}\leq\frac{1}{12} \varepsilon\|\partial^2_yV^\varepsilon\|^2_{L^2_{xy}}+C\varepsilon^{-\frac{1}{2}}.
		\end{split}
	\end{equation*}
    Obviously,
    \begin{align*}
			&\frac{d}{dt}\|V^\varepsilon\|^2_{L^2_{xy}}+\varepsilon\|\partial_x\partial_y V^\varepsilon\|^2_{L^2_{xy}}+\varepsilon\|\partial^2_y V^\varepsilon\|^2_{L^2_{xy}}\\
            &\leq \frac{1}{4}\|\partial_x\partial_y H^\varepsilon\|_{L^2_{xy}}+\frac{1}{8}\|\partial^2_y H^\varepsilon\|^2_{L^2_{xy}}+\frac{1}{2}\varepsilon\|\partial_x\partial_y V^\varepsilon\|^2_{L^2_{xy}}+\frac{1}{2}\varepsilon\|\partial^2_y V^\varepsilon\|^2_{L^2_{xy}}+C\|\partial_yV^\varepsilon\|^2_{L^2_{xy}}+C\varepsilon^{-\frac{1}{2}}.
	\end{align*}
    \textbf{Step 3. The \(L^2\) estimate of $\partial_y\tilde{\omega}$.}
    Multiplying $(\ref{tilde omega})_1$ by $-\partial^2_y\tilde{\omega}$, we can obtain
       \begin{align*}
        &\frac{1}{2}\frac{d}{dt}\|\partial_y\tilde{\omega}\|^2_{L^2_{xy}}+\varepsilon\|\partial_x\partial_y \tilde{\omega}\|^2_{L^2_{xy}}+\varepsilon\|\partial^2_y \tilde{\omega}\|^2_{L^2_{xy}}\\
        &=\varepsilon^\frac{1}{2}\int^\infty_0\int^\infty_{-\infty}U^\varepsilon\cdot\nabla \tilde{\omega}\partial^2_y\tilde{\omega} dxdy\\
        &\quad+\int^\infty_0\int^\infty_{-\infty}\nabla\times(U^\varepsilon\cdot\nabla U^\theta)\partial^2_y\tilde{\omega} dxdy+\int^\infty_0\int^\infty_{-\infty}U^\theta\cdot\nabla \tilde{\omega}\partial^2_y\tilde{\omega} dxdy\\
        &\quad
        +\int^\infty_0\int^\infty_{-\infty}(\partial_xU^\theta_1\partial_xU^\varepsilon_2+\partial_xU^\theta_2\partial_yU^\varepsilon_2-\partial_yU^\theta_1\partial_xU^\varepsilon_1-\partial_yU^\theta_2\partial_yU^\varepsilon_1)\partial^2_y\tilde{\omega} dxdy\\
        &\quad -\varepsilon^\frac{1}{2}\int^\infty_0\int^\infty_{-\infty}U^\varepsilon_1\partial_x[(\varepsilon^\frac{1}{2}\partial_xu^{B,2}_2+\varepsilon\partial_xu^{B,3}_2)(t,x,0)]\varphi(y)\partial^2_y\tilde{\omega}dxdy\\
        &\quad-\varepsilon^\frac{1}{2}\int^\infty_0\int^\infty_{-\infty}U^\varepsilon_2[(\varepsilon^\frac{1}{2}\partial_xu^{B,2}_2+\varepsilon\partial_xu^{B,3}_2)(t,x,0)]\partial_y\varphi(y)\partial^2_y\tilde{\omega}dxdy\\
        &\quad-\int^\infty_0\int^\infty_{-\infty}\partial_xH^\varepsilon\partial^2_y\tilde{\omega} dxdy-\varepsilon^{-\frac{1}{2}}\int^\infty_0\int^\infty_{-\infty}(\nabla\times h^\varepsilon)\partial^2_y\tilde{\omega}dxdy-\int^\infty_0\int^\infty_{-\infty}\rho^\varepsilon\partial^2_y\tilde{\omega} dxdy dxdy\\
        &=\sum^{9}_{i=1}\mathbb{E}_i.
	\end{align*}
    It can be verified that
\begin{equation*}
		\begin{split}{}
       \mathbb{E}_1&=\varepsilon^\frac{1}{2}\int^\infty_0\int^\infty_{-\infty}U^\varepsilon_1\partial_x\tilde{\omega}\partial^2_y\tilde{\omega} dxdy+\varepsilon^\frac{1}{2}\int^\infty_0\int^\infty_{-\infty}U^\varepsilon_2\partial_y\tilde{\omega}\partial^2_y\tilde{\omega} dxdy\\
       &=-\varepsilon^\frac{1}{2}\int^\infty_0\int^\infty_{-\infty}\partial_yU^\varepsilon_1\partial_x\tilde{\omega}\partial_y\tilde{\omega} dxdy-\varepsilon^\frac{1}{2}\int^\infty_0\int^\infty_{-\infty}U^\varepsilon_1\partial_x\partial_y\tilde{\omega}\partial_y\tilde{\omega} dxdy\\
       &\quad-\varepsilon^\frac{1}{2}\int^\infty_0\int^\infty_{-\infty}U^\varepsilon_2\partial_y\tilde{\omega}\partial^2_y\tilde{\omega} dxdy\\
            &\leq \varepsilon^\frac{1}{2}\|\partial_yU^\varepsilon_1\|_{L^2_{xy}}\|\partial_x\tilde{\omega}\|_{L^2_xL^\infty_y}\|\partial_y\tilde{\omega}\|_{L^\infty_xL^2_y}+\varepsilon^\frac{1}{2}\|U^\varepsilon_1\|_{L^2_xL^\infty_y}\|\partial_x\partial_y\tilde{\omega}\|_{L^2_{xy}}\|\partial_y\tilde{\omega}\|_{L^\infty_xL^2_y}\\
            &\quad+\varepsilon^\frac{1}{2}\|U^\varepsilon_2\|_{L^2_xL^\infty_y}\|\partial_y\tilde{\omega}\|_{L^\infty_xL^2_y}\|\partial^2_y\tilde{\omega}\|_{L^2_{xy}}\\
            &\leq \frac{1}{2}\varepsilon\|\partial_x\partial_y \tilde{\omega}\|^2_{L^2_{xy}}+\frac{1}{16}\varepsilon\|\partial^2_y \tilde{\omega}\|^2_{L^2_{xy}}+C\|\partial_y \tilde{\omega}\|^2_{L^2_{xy}}+C\varepsilon.
		\end{split}
	\end{equation*}

    For the term $\mathbb{E}_2$, we obtain
	\begin{align*}
       \mathbb{E}_2&=-\int^\infty_0\int^\infty_{-\infty}(\partial_y U^\varepsilon_1\partial_x U^\theta_1+ U^\varepsilon_1\partial_x\partial_y U^\theta_1+\partial_y U^\varepsilon_2\partial_y U^\theta_1-\partial_x U^\varepsilon_1\partial_x U^\theta_2\\
       &\quad-U^\varepsilon_1\partial^2_x U^\theta_2-\partial_x U^\varepsilon_2\partial_y U^\theta_2-U^\varepsilon_2\partial_x\partial_y U^\theta_2)\partial^2_y\tilde{\omega}dxdy-\int^\infty_0\int^\infty_{-\infty}U^\varepsilon_2\partial^2_yU^\theta_1\partial^2_y\tilde{\omega}dxdy\\
       &=\mathbb{E}_{21}+\mathbb{E}_{22}.
	\end{align*}
    By the estimates in $(\ref{HVU theta L2})$, we can get that
    \begin{align*}
       \mathbb{E}_{21}&\leq\varepsilon^{-\frac{1}{2}}(\|\partial_y U^\varepsilon_1\|_{L^2_{xy}}\|\partial_x U^\theta_1\|_{L^\infty_{xy}}+\| U^\varepsilon_1\|_{L^2_{xy}}\|\partial_x \partial_yU^\theta_1\|_{L^\infty_{xy}}+\|\partial_y U^\varepsilon_2\|_{L^2_{xy}}\|\partial_y U^\theta_1\|_{L^\infty_{xy}}\\
            &\quad+\|\partial_x U^\varepsilon_1\|_{L^2_{xy}}\|\partial_x U^\theta_2\|_{L^\infty_{xy}}+\| U^\varepsilon_1\|_{L^2_{xy}}\|\partial^2_x U^\theta_1\|_{L^\infty_{xy}}+\|\partial_x U^\varepsilon_2\|_{L^2_{xy}}\|\partial_y U^\theta_2\|_{L^\infty_{xy}}\\
            &\quad+\|U^\varepsilon_2\|_{L^2_{xy}}\|\partial_x\partial_y U^\theta_2\|_{L^\infty_{xy}})\varepsilon^\frac{1}{2}\|\partial^2_y\tilde{\omega}\|_{L^2_{xy}}\\
            &\leq \frac{1}{48}\varepsilon\|\partial^2_y\tilde{\omega}\|_{L^2_{xy}}+C\varepsilon^{-\frac{1}{2}}.
	\end{align*}
    It can be obtained by the definition of $U^\theta$ in $(\ref{HVUK varepsilon})$ and integration by parts
    \begin{equation*}
		\begin{split}{}
       \mathbb{E}_{22}&=-\int^\infty_0\int^\infty_{-\infty}U^\varepsilon_2(\partial^2_y u^\varepsilon_1-\varepsilon^\frac{1}{2}\partial^2_y U^\varepsilon_1)\partial^2_y\tilde{\omega}dxdy\\
       &=-\int^\infty_0\int^\infty_{-\infty}U^\varepsilon_2\partial^2_y u^\varepsilon_1\partial^2_y\tilde{\omega}dxdy+\varepsilon^\frac{1}{2}\int^\infty_0\int^\infty_{-\infty}U^\varepsilon_2\partial^2_y U^\varepsilon_1\partial^2_y\tilde{\omega}dxdy\\
       &=\mathbb{E}_{22_1}+\mathbb{E}_{22_2}.
		\end{split}
	\end{equation*}
    From Theorem $\ref{th nvu regularity}$ and Hardy inequality, we can get that
    \begin{equation*}
		\begin{split}{}
       \mathbb{E}_{22_1}&=\int^\infty_0\int^\infty_{-\infty}\frac{U^\varepsilon_2}{\psi} \psi\partial^2_y u^\varepsilon_1\partial^2_y\tilde{\omega}dxdy\leq C\varepsilon^{-\frac{1}{2}}(\|\partial_y U^\varepsilon_2\|_{L^2_{xy}}+\|U^\varepsilon_2\|_{L^2_{xy}})\|\partial_y u^\varepsilon_1\|_{W^{1,\infty}_{xy}}\varepsilon^{\frac{1}{2}}\|\partial^2_y\tilde{\omega}\|_{L^2_{xy}}\\&\leq \frac{1}{48}\varepsilon\|\partial^2_y\tilde{\omega}\|^2_{L^2_{xy}}+C\varepsilon^{-\frac{1}{2}}.\\
		\end{split}
	\end{equation*}
    By the divergence condition, the definition of $\tilde{\omega}$ into $\mathbb{E}_{22_2}$, and using $\omega^{\varepsilon*}=\partial_xU^\varepsilon_2-\partial_yU^\varepsilon_1$, we obtain
    \begin{equation*}
		\begin{split}{}
       \mathbb{E}_{22_2}&\leq\varepsilon^\frac{1}{2}\|U^\varepsilon_2\|_{L^\infty_{xy}}\|\partial^2_y U^\varepsilon_1\|_{L^2_{xy}}\|\partial^2_y\tilde{\omega}\|_{L^2_{xy}}\\
       &\leq\frac{1}{48} \varepsilon\|\partial^2_y\tilde{\omega}\|^2_{L^2_{xy}}+C(\|U^\varepsilon_2\|_{L^2_{xy}}\|\partial_yU^\varepsilon_2\|_{L^2_{xy}}+\|\partial_xU^\varepsilon_2\|_{L^2_{xy}}\|\partial_x\partial_yU^\varepsilon_2\|_{L^2_{xy}})(\|\partial_x \omega^{\varepsilon*}\|^2_{L^2_{xy}}+\|\partial_y \omega^{\varepsilon*}\|^2_{L^2_{xy}})\\
       &\leq \frac{1}{48}\varepsilon\|\partial^2_y\tilde{\omega}\|^2_{L^2_{xy}}+C\|\partial_y\tilde{\omega}\|^2_{L^2_{xy}}\|\partial_y\tilde{\omega}\|^2_{L^2_{xy}}+C\varepsilon.
		\end{split}
	\end{equation*}
    Then combining the estimates of $\mathbb{E}_{21}$ and $\mathbb{E}_{22}$,  we can get
    \begin{equation*}
		\begin{split}{}
       \mathbb{E}_{2}&\leq \frac{1}{16}\varepsilon\|\partial^2_y\tilde{\omega}\|^2_{L^2_{xy}}+C\|\partial_y\tilde{\omega}\|^2_{L^2_{xy}}\|\partial_y\tilde{\omega}\|^2_{L^2_{xy}}+C\varepsilon^{-\frac{1}{2}}.
		\end{split}
	\end{equation*}
    Since $U^\theta_2(t,x,0)=\varepsilon u^{B,2}_2(t,x,0)+\varepsilon^\frac{3}{2} u^{B,3}_2(t,x,0)$, (\ref{HVU theta L2}) and (\ref{partialx uB22uB32}), it can be verified that
\begin{equation*}
		\begin{split}{}
       \mathbb{E}_3&\leq\varepsilon^{-\frac{1}{2}}\|U^\theta_1\|_{L^\infty_{xy}}\|\partial_x \tilde{\omega}\|_{L^2_{xy}}\varepsilon^\frac{1}{2}\|\partial^2_y\tilde{\omega}\|_{L^2_{xy}}-\frac{1}{2}\int^\infty_{-\infty}U^\theta_2(t,x,0)|\partial_y\tilde{\omega}(t,x,0)|^2dx+\frac{1}{2}\|\partial_y U^\theta_2\|_{L^\infty_{xy}}\|\partial_y\tilde{\omega}\|^2_{L^2_{xy}}\\
            &\leq \frac{1}{16}\varepsilon\|\partial^2_y \tilde{\omega}\|^2_{L^2_{xy}}+C\|\partial_y\tilde{\omega}\|^2_{L^2_{xy}}+C\varepsilon^{-\frac{1}{2}}.
		\end{split}
	\end{equation*}
     Using the Sobolev embedding inequality, we get
     \begin{equation*}
		\begin{split}{}
       \mathbb{E}_{4}&\leq\varepsilon^{-\frac{1}{2}}(\|\partial_x U^\theta_1\|_{L^\infty_{xy}}\|\partial_x U^\varepsilon_2\|_{L^2_{xy}}+\|\partial_x U^\theta_2\|_{L^\infty_{xy}}\|\partial_y U^\varepsilon_2\|_{L^2_{xy}}\\
            &\quad-\|\partial_y U^\theta_1\|_{L^\infty_{xy}}\|\partial_x U^\varepsilon_1\|_{L^2_{xy}}-\|\partial_y U^\theta_2\|_{L^\infty_{xy}}\|\partial_y U^\varepsilon_1\|_{L^2_{xy}})\varepsilon^{\frac{1}{2}}\|\partial^2_y\tilde{\omega}\|^2_{L^2_{xy}}\\
            &\leq \frac{1}{16}\varepsilon\|\partial^2_y\tilde{\omega}\|^2_{L^2_{xy}}+C\varepsilon^{-\frac{1}{2}}.
		\end{split}
	\end{equation*}
      Similarly, it follows that
      \begin{equation*}
		\begin{split}{}
        \mathbb{E}_{5}+\mathbb{E}_{6}&\leq C\varepsilon^\frac{1}{2}(\|U^\varepsilon_1\|_{L^2_{xy}}\|\partial^2_y\tilde{\omega}\|_{L^2_{xy}}+\|U^\varepsilon_2\|_{L^2_{xy}}\|\partial^2_y\tilde{\omega}\|_{L^2_{xy}})\leq\frac{1}{16}\varepsilon\|\partial^2_y\tilde{\omega}\|^2_{L^2_{xy}}+C\varepsilon^{\frac{1}{2}}.
		\end{split}
	\end{equation*}

	 Applying H\"{o}lder's inequality and the estimate of (\ref{rho varepsilon frac12}), we obtain
	\begin{equation*}
		\begin{split}{}
        \mathbb{E}_{7}&\leq\varepsilon^{-\frac{1}{2}}\|\partial_xH^\varepsilon\|_{L^2_{xy}}\varepsilon^\frac{1}{2}\|\partial^2_y\tilde{\omega}\|_{L^2_{xy}}\leq \frac{1}{16}\varepsilon\|\partial^2_y\tilde{\omega}\|^2_{L^2_{xy}}+\varepsilon^{-\frac{1}{2}},\\
        \mathbb{E}_{9}&\leq\|\rho^\varepsilon\|_{L^2_{xy}}\|\partial^2_y\tilde{\omega}\|_{L^2_{xy}}\leq \frac{1}{16}\varepsilon\|\partial^2_y\tilde{\omega}\|^2_{L^2_{xy}}+C.\\
		\end{split}
	\end{equation*}
        By Lemma $\ref{lem h L2}$, we have
        \begin{equation*}
		\begin{split}{}
        \mathbb{E}_{8}&\leq\varepsilon^{-\frac{1}{2}}\|\partial_x h^\varepsilon_2-\partial_y h^\varepsilon_1\|_{L^2_{xy}}\|\partial^2_y\tilde{\omega}\|_{L^2_{xy}}\leq \frac{1}{16}\varepsilon\|\partial^2_y\tilde{\omega}\|^2_{L^2_{xy}}+C\varepsilon^{-\frac{1}{2}}.\\
		\end{split}
	\end{equation*}
        It then yields
       \begin{equation*}
		\begin{split}{}
			\frac{d}{dt}\|\partial_y\tilde{\omega}\|^2_{L^2_{xy}}+\varepsilon\|\partial^2_y\tilde{\omega}\|^2_{L^2_{xy}}&\leq C\|\partial_y\tilde{\omega}\|^2_{L^2_{xy}}\|\partial_y\tilde{\omega}\|^2_{L^2_{xy}}+C\|\partial_y\tilde{\omega}\|^2_{L^2_{xy}}+C\varepsilon^{-\frac{1}{2}}.\\
        \end{split}
	\end{equation*}
	Therefore, from Step 1-Step 3, we can obtain
	\begin{equation*}
		\begin{split}{}
			&\frac{d}{dt}(\|\partial_yH^\varepsilon\|^2_{L^2_{xy}}+\|\partial_yV^\varepsilon\|^2_{L^2_{xy}}+\|\partial_y\tilde{\omega}\|^2_{L^2_{xy}})+\|\partial^2_yH^\varepsilon\|^2_{L^2_{xy}}+\varepsilon\|\partial^2_yV^\varepsilon\|^2_{L^2_{xy}}+\varepsilon\|\partial^2_y\tilde{\omega}\|^2_{L^2_{xy}}\\
			&\leq C(\|\partial_yH^\varepsilon\|^2_{L^2_{xy}}+\|\partial_yV^\varepsilon\|^2_{L^2_{xy}}+\|\partial_y\tilde{\omega}\|^2_{L^2_{xy}}+1)^2+C\varepsilon^{-\frac{1}{2}},\\
		\end{split}
	\end{equation*}
    along with Gronwall's inequality for $(0,T)$, which yields
	 \begin{equation*}
	 \begin{split}{}
	 	&\|\partial_yH^\varepsilon\|^2_{L^2_{xy}}+\|\partial_yV^\varepsilon\|^2_{L^2_{xy}}+\|\partial_y\tilde{\omega}\|^2_{L^2_{xy}}+\int^t_0(\|\partial^2_yH^\varepsilon\|^2_{L^2_{xy}}+\varepsilon\|\partial^2_yV^\varepsilon\|^2_{L^2_{xy}}+\varepsilon\|\partial^2_y\tilde{\omega}\|^2_{L^2_{xy}})\leq C\varepsilon^{-\frac{1}{2}}.\\			
	 \end{split}
	 \end{equation*}
	  The proof is {complete}.

    \begin{lemma}\label{lem HVUxy L2}
         {Let the assumptions in Lemma $\ref{lem HVU L2}$-Lemma $\ref{lem HVUy L2}$ hold. Then there exists $\varepsilon_0>0$ such that for any $\varepsilon\in(0,\varepsilon_0]$, the following estimates hold:
		\begin{equation}\label{Equ(4.26)}
			\varepsilon^{\frac{3}{2}}\|\partial_x\partial_yH^\varepsilon\|^2_{L^\infty_TL^2_{xy}}+\varepsilon^{\frac{3}{2}}\|\partial_x\partial_yV^\varepsilon\|^2_{L^\infty_TL^2_{xy}}\leq C.\\
		\end{equation}
		}
	\end{lemma}
	\noindent{\bf{Proof.}} We first estimate $\|(H^\theta,V^\theta,U^\theta)\|_{L^\infty}$. Similar to $(\ref{HVU theta L2})$, by the Sobolev embedding inequality, the following can be obtained,
	\begin{align}\label{HVU theta Linfty}
            &\|\partial_x\partial_y H^\theta\|_{L^\infty_TL^\infty_{xy}}+\|\partial^2_x\partial_y H^\theta\|_{L^\infty_TL^\infty_{xy}}+\|\partial^2_x H^\theta\|_{L^\infty_TL^\infty_{xy}}+
            \|\partial_x\partial_y V^\theta_1\|_{L^\infty_TL^\infty_{xy}}+\|\partial_x\partial_y U^\theta\|_{L^\infty_TL^\infty_{xy}}\leq C,\notag\\
            &\|\partial_x\partial^2_y H^\theta\|_{L^\infty_TL^\infty_{xy}}+
            \|\partial_x\partial_y V^\theta_2\|_{L^\infty_TL^\infty_{xy}}+
            \|\partial_x\partial^2_y V^\theta_1\|_{L^\infty_TL^\infty_{xy}}+\|\partial_x\partial^2_y U^\theta\|_{L^\infty_TL^\infty_{xy}}\leq C\varepsilon^{-\frac{1}{2}},\\
            &\|\partial_x\partial^2_y V^\theta_2\|_{L^\infty_TL^\infty_{xy}}\leq C\varepsilon^{-1}.\notag\notag
	\end{align}
    From Lemma $\ref{lem HVU L2}$-Lemma $\ref{lem HVUy L2}$, we can get that
    \begin{align}\label{U Linfty}
        \|U^\varepsilon\|_{L^\infty_TL^\infty_{xy}}\leq C(\|U^\varepsilon\|^\frac{1}{2}_{L^\infty_TL^2_{xy}}\|\partial_yU^\varepsilon\|^\frac{1}{2}_{L^\infty_TL^2_{xy}}+\|\partial_x U^\varepsilon\|^\frac{1}{2}_{L^\infty_TL^2_{xy}}\|\partial_x\partial_y U^\varepsilon\|^\frac{1}{2}_{L^\infty_TL^2_{xy}})\leq C\varepsilon^\frac{1}{4}.
    \end{align}
    \textbf{Step 1. The \(L^2\) estimate of $\partial_x\partial_yH^\varepsilon$.}
    Applying $-\partial_{x}$ to $(\ref{HVU equation})_1$, multiplying by $\partial_{x}\partial^2_{y} H^\varepsilon$, and using integration by parts, we obtain
	\begin{align*}
			&\frac{1}{2}\frac{d}{dt}\|\partial_{x}\partial_{y} H^\varepsilon\|^2_{L^2_{xy}}+\|\partial^2_{x}\partial_y H^\varepsilon\|^2_{L^2_{xy}}+\|\partial_x\partial^2_{y} H^\varepsilon\|^2_{L^2_{xy}}\\
            &=-\varepsilon^\frac{1}{2}\int^\infty_0\int^\infty_{-\infty}\partial_{x}\partial_{y}(U^\varepsilon_1\partial_x H^\varepsilon+U^\varepsilon_2\partial_y H^\varepsilon)\partial_{x}\partial_{y}H^\varepsilon dxdy\\
            &\quad-\Big{[}\int^\infty_0\int^\infty_{-\infty}\partial_{x}\partial_{y}(U^\varepsilon_1\partial_x H^\theta+U^\varepsilon_2\partial_y H^\theta)\partial_{x}\partial_{y}H^\varepsilon dxdy\\
            &\quad+\int^\infty_0\int^\infty_{-\infty}\partial_{x}\partial_{y}(U^\theta_1\partial_x H^\varepsilon+U^\theta_2\partial_y H^\varepsilon)\partial_{x}\partial_{y}H^\varepsilon dxdy\Big{]}\\
            &\quad+\varepsilon^\frac{1}{2}\Big{[}\int^\infty_0\int^\infty_{-\infty}\partial_{x}\partial_{y}(V^\varepsilon_1\partial_x H^\varepsilon+V^\varepsilon_2\partial_y H^\varepsilon)\partial_{x}\partial_{y}H^\varepsilon dxdy\\
            &\quad+\int^\infty_0\int^\infty_{-\infty}\partial_{x}\partial_{y}(H^\varepsilon\partial_x V^\varepsilon_1+H^\varepsilon\partial_y V^\varepsilon_2)\partial_{x}\partial_{y}H^\varepsilon dxdy\Big{]}\\
            &\quad+\Big{[}\int^\infty_0\int^\infty_{-\infty}\partial_{x}\partial_{y}(V^\theta_1\partial_x H^\varepsilon+V^\theta_2\partial_y H^\varepsilon)\partial_{x}\partial_{y}H^\varepsilon dxdy\\
            &\quad+\int^\infty_0\int^\infty_{-\infty}\partial_{x}\partial_{y}(H^\varepsilon\partial_x V^\theta_1+H^\varepsilon\partial_y V^\theta_2)\partial_{x}\partial_{y}H^\varepsilon dxdy\Big{]}\\
            &\quad+\Big{[}\int^\infty_0\int^\infty_{-\infty}\partial_{x}\partial_{y}(V^\varepsilon_1\partial_x H^\theta+V^\varepsilon_2\partial_y H^\theta)\partial_{x}\partial_{y}H^\varepsilon dxdy\\
            &\quad+\int^\infty_0\int^\infty_{-\infty}\partial_{x}\partial_{y}(H^\theta\partial_x V^\varepsilon_1+H^\theta\partial_y V^\varepsilon_2)\partial_{x}\partial_{y}H^\varepsilon dxdy\Big{]}+\varepsilon^{-\frac{1}{2}}\int^\infty_0\int^\infty_{-\infty}\partial_{x}\partial_{y}f^\varepsilon\partial_{x}\partial_{y}H^\varepsilon dxdy=\sum^{6}_{i=1}\mathcal{D}_i.
        \end{align*}
	Using the Gagliardo-Nirenberg interpolation {inequality,}
	\begin{align*}
        \mathcal{D}_1&=\varepsilon^\frac{1}{2}\int^\infty_0\int^\infty_{-\infty}(\partial_x\partial_{y}U^\varepsilon_1\partial_x H^\varepsilon+\partial_{y}U^\varepsilon_1\partial^2_x H^\varepsilon+\partial_{x}U^\varepsilon_1\partial_x\partial_{y} H^\varepsilon+U^\varepsilon_1\partial^2_x\partial_{y} H^\varepsilon\\
        &\quad+\partial_x\partial_{y}U^\varepsilon_2\partial_y H^\varepsilon+\partial_{y}U^\varepsilon_2\partial_x\partial_{y} H^\varepsilon+\partial_{x}U^\varepsilon_2\partial^2_y H^\varepsilon+U^\varepsilon_2\partial_x\partial^2_{y} H^\varepsilon)\partial_x\partial_{y} H^\varepsilon dxdy\\
        &\leq\varepsilon^\frac{1}{2}(\|\partial_x\partial_{y}U^\varepsilon_1\|_{L^2_{xy}}\|\partial_x H^\varepsilon\|_{L^\infty_{xy}}+\|\partial_{y}U^\varepsilon_1\|_{L^\infty_{x}L^2_{y}}\|\partial^2_x H^\varepsilon\|_{L^2_{x}L^\infty_{y}}+\|\partial_{x}U^\varepsilon_1\|_{L^2_{x}L^\infty_{y}}\|\partial_x\partial_y H^\varepsilon\|_{L^\infty_{x}L^2_{y}}\\
        &\quad+\|U^\varepsilon_1\|_{L^\infty_{xy}}\|\partial^2_x\partial_y H^\varepsilon\|_{L^2_{xy}}+\|\partial_x\partial_{y}U^\varepsilon_2\|_{L^2_{xy}}\|\partial_y H^\varepsilon\|_{L^\infty_{xy}}+\|\partial_{y}U^\varepsilon_2\|_{L^\infty_{x}L^2_{y}}\|\partial_x\partial_{y} H^\varepsilon\|_{L^2_{x}L^\infty_{y}}\\
        &\quad+\|\partial_{x}U^\varepsilon_2\|_{L^2_{x}L^\infty_{y}}\|\partial^2_y H^\varepsilon\|_{L^\infty_{x}L^2_{y}}+\|U^\varepsilon_2\|_{L^\infty_{xy}}\|\partial_x\partial^2_y H^\varepsilon\|_{L^2_{xy}})\|\partial_x\partial_{y} H^\varepsilon\|_{L^2_{xy}}\\
        &\leq\frac{1}{10}\|\partial^2_{x}\partial_y H^\varepsilon\|^2_{L^2_{xy}}+\frac{1}{10}\|\partial_x\partial^2_{y} H^\varepsilon\|^2_{L^2_{xy}}+C\|\partial^2_{x} H^\varepsilon\|^2_{L^2_{xy}}+C\|\partial^2_{y} H^\varepsilon\|^2_{L^2_{xy}}+C\|\partial_x\partial_{y} H^\varepsilon\|^2_{L^2_{xy}}+C.
	\end{align*}
         Using integration by parts, we obtain
	\begin{align*}
        \mathcal{D}_2&=-\int^\infty_0\int^\infty_{-\infty}(\partial_x\partial_{y}U^\varepsilon_1\partial_x H^\theta+\partial_{y}U^\varepsilon_1\partial^2_x H^\theta+\partial_{x}U^\varepsilon_1\partial_x\partial_{y} H^\theta+U^\varepsilon_1\partial^2_x\partial_{y} H^\theta\\
        &\quad+\partial_x\partial_{y}U^\varepsilon_2\partial_y H^\theta+\partial_{y}U^\varepsilon_2\partial_x\partial_{y} H^\theta+\partial_{x}U^\varepsilon_2\partial^2_y H^\theta+U^\varepsilon_2\partial_x\partial^2_{y} H^\theta\\
        &\quad+\partial_x\partial_{y}U^\theta_1\partial_x H^\varepsilon+\partial_{y}U^\theta_1\partial^2_x H^\varepsilon+\partial_{x}U^\theta_1\partial_x\partial_{y} H^\varepsilon+U^\theta_1\partial^2_x\partial_{y} H^\varepsilon\\
        &\quad+\partial_x\partial_{y}U^\theta_2\partial_y H^\varepsilon+\partial_{y}U^\theta_2\partial_x\partial_{y} H^\varepsilon+\partial_{x}U^\theta_2\partial^2_y H^\varepsilon+U^\theta_2\partial_x\partial^2_{y} H^\varepsilon)\partial_x\partial_{y} H^\varepsilon dxdy\\
        &\leq(\|\partial_x\partial_{y}U^\varepsilon_1\|_{L^2_{xy}}\|\partial_x H^\theta\|_{L^\infty_{xy}}+\|\partial_{y}U^\varepsilon_1\|_{L^2_{xy}}\|\partial^2_x H^\theta\|_{L^\infty_{xy}}+\|\partial_{x}U^\varepsilon_1\|_{L^2_{xy}}\|\partial_x\partial_{y} H^\theta\|_{L^\infty_{xy}}\\
        &\quad+\|U^\varepsilon_1\|_{L^\infty_{xy}}\|\partial^2_x\partial_y H^\theta\|_{L^2_{xy}}+\|\partial_x\partial_{y}U^\varepsilon_2\|_{L^2_{xy}}\|\partial_y H^\theta\|_{L^\infty_{xy}}+\|\partial_{y}U^\varepsilon_2\|_{L^2_{xy}}\|\partial_x\partial_{y} H^\theta\|_{L^\infty_{xy}}\\
        &\quad+\|\partial_{x}U^\varepsilon_2\|_{L^2_{xy}}\|\partial^2_y H^\theta\|_{L^\infty_{xy}}+\|U^\varepsilon_2\|_{L^2_{xy}}\|\partial_x\partial^2_y H^\theta\|_{L^\infty_{xy}}\\
        &\quad+\|\partial_x\partial_{y}U^\theta_1\|_{L^\infty_{xy}}\|\partial_x H^\varepsilon\|_{L^2_{xy}}+\|\partial_{y}U^\theta_1\|_{L^\infty_{xy}}\|\partial^2_x H^\varepsilon\|_{L^2_{xy}}+\|\partial_{x}U^\theta_1\|_{L^\infty_{xy}}\|\partial_x\partial_{y} H^\varepsilon\|_{L^2_{xy}}\\
        &\quad+\|U^\theta_1\|_{L^\infty_{xy}}\|\partial^2_x\partial_y H^\varepsilon\|_{L^2_{xy}}+\|\partial_x\partial_{y}U^\theta_2\|_{L^\infty_{xy}}\|\partial_y H^\varepsilon\|_{L^2_{xy}}+\|\partial_{y}U^\theta_2\|_{L^\infty_{xy}}\|\partial_x\partial_{y} H^\varepsilon\|_{L^2_{xy}}\\
        &\quad+\|\partial_{x}U^\theta_2\|_{L^\infty_{xy}}\|\partial^2_y H^\varepsilon\|_{L^2_{xy}}+\|U^\theta_2\|_{L^\infty_{xy}}\|\partial_x\partial^2_y H^\varepsilon\|_{L^2_{xy}})\|\partial_x\partial_{y} H^\varepsilon\|_{L^2_{xy}}\\
        &\leq \frac{1}{10}\|\partial^2_{x}\partial_y H^\varepsilon\|^2_{L^2_{xy}}+\frac{1}{10}\|\partial_x\partial^2_{y} H^\varepsilon\|^2_{L^2_{xy}}+C\|\partial^2_{x} H^\varepsilon\|^2_{L^2_{xy}}+C\|\partial^2_{y} H^\varepsilon\|^2_{L^2_{xy}}+C\|\partial_x\partial_y H^\varepsilon\|^2_{L^2_{xy}}+C\varepsilon^{-\frac{1}{2}}.
	\end{align*}
        Using the Gagliardo-Nirenberg inequality, we obtain
        \begin{align*}
        \mathcal{D}_3&=\varepsilon^\frac{1}{2}\int^\infty_0\int^\infty_{-\infty}(2\partial_x\partial_yV^\varepsilon_1\partial_x H^\varepsilon+\partial_{y}V^\varepsilon_1\partial^2_x H^\varepsilon+2\partial_{x}V^\varepsilon_1\partial_x\partial_y H^\varepsilon+V^\varepsilon_1\partial^2_x\partial_{y} H^\varepsilon\\
        &\quad+2\partial_x\partial_yV^\varepsilon_2\partial_y H^\varepsilon+2\partial_{y}V^\varepsilon_2\partial_x\partial_y H^\varepsilon+\partial_{x}V^\varepsilon_2\partial^2_y H^\varepsilon+V^\varepsilon_2\partial_x\partial^2_{y} H^\varepsilon\\
        &\quad+\partial_{y}H^\varepsilon\partial^2_x V^\varepsilon_1+H^\varepsilon\partial^2_x\partial_{y} V^\varepsilon_1+\partial_{x}H^\varepsilon\partial^2_y V^\varepsilon_2+H^\varepsilon\partial_x\partial^2_{y} V^\varepsilon_2)\partial_x\partial_y H^\varepsilon dxdy\\
        &\leq \varepsilon^\frac{1}{2}\Big{(}2\|\partial_{x}\partial_y V^\varepsilon_1\|_{L^2_{xy}}\|\partial_{x} H^\varepsilon\|_{L^2_{x}L^\infty_y}\|\partial_{x}\partial_y H^\varepsilon\|_{L^\infty_{x}L^2_y}+\|\partial_y V^\varepsilon_1\|_{L^\infty_{x}L^2_y}\|\partial^2_{x} H^\varepsilon\|_{L^2_{xy}}\|\partial_{x}\partial_y H^\varepsilon\|_{L^2_{x}L^\infty_y}\\
        &\quad+2\|\partial_{x}V^\varepsilon_1\|_{L^2_{x}L^\infty_y}\|\partial_{x}\partial_y H^\varepsilon\|_{L^\infty_{x}L^2_y}\|\partial_{x}\partial_y H^\varepsilon\|_{L^2_{xy}}+\|V^\varepsilon_1\|_{L^\infty_{x}L^2_y}\|\partial^2_{x}\partial_y H^\varepsilon\|_{L^2_{xy}}\|\partial_{x}\partial_y H^\varepsilon\|_{L^2_{x}L^\infty_y}\\
        &\quad+2\|\partial_{x}\partial_y V^\varepsilon_2\|_{L^2_{xy}}\|\partial_{y} H^\varepsilon\|_{L^\infty_{x}L^2_y}\|\partial_{x}\partial_y H^\varepsilon\|_{L^2_{x}L^\infty_y}+2\|\partial_{y}V^\varepsilon_2\|_{L^\infty_{x}L^2_y}\|\partial_{x}\partial_y H^\varepsilon\|_{L^2_{x}L^\infty_y}\|\partial_{x}\partial_y H^\varepsilon\|_{L^2_{xy}}\\
        &\quad+\|\partial_{x}V^\varepsilon_2\|_{L^2_{x}L^\infty_y}\|\partial^2_y H^\varepsilon\|_{L^2_{xy}}\|\partial_{x}\partial_y H^\varepsilon\|_{L^\infty_{x}L^2_y}+\|V^\varepsilon_2\|_{L^\infty_{x}L^2_y}\|\partial_{x}\partial^2_y H^\varepsilon\|_{L^2_{xy}}\|\partial_{x}\partial_y H^\varepsilon\|_{L^2_{x}L^\infty_y}\\
        &\quad+\|\partial_{y}H^\varepsilon\|_{L^\infty_{x}L^2_y}\|\partial^2_x V^\varepsilon_1\|_{L^2_{x}L^\infty_y}\|\partial_{x}\partial_y H^\varepsilon\|_{L^2_{xy}}+\|H^\varepsilon\|_{L^\infty_{x}L^2_y}\|\partial^2_x\partial_{y} V^\varepsilon_1\|_{L^2_{xy}}\|\partial_{x}\partial_y H^\varepsilon\|_{L^2_{x}L^\infty_y}\\
        &\quad+\|\partial_{x}H^\varepsilon\|_{L^2_{xy}}\|\partial^2_y V^\varepsilon_2\|_{L^\infty_{x}L^2_y}\|\partial_{x}\partial_y H^\varepsilon\|_{L^2_{x}L^\infty_y}+\|H^\varepsilon\|_{L^\infty_{x}L^2_y}\|\partial_x\partial^2_{y} V^\varepsilon_2\|_{L^2_{xy}}\|\partial_{x}\partial_y H^\varepsilon\|_{L^2_{x}L^\infty_y}
        \Big{)}
        \\
        &\leq\frac{1}{10}\|\partial^2_{x}\partial_y H^\varepsilon\|^2_{L^2_{xy}}+\frac{1}{10}\|\partial_x\partial^2_{y} H^\varepsilon\|^2_{L^2_{xy}}+\frac{1}{4}\varepsilon\|\partial^2_{x}\partial_y V^\varepsilon_1\|^2_{L^2_{xy}}+\frac{1}{4}\varepsilon\|\partial_x\partial^2_{y} V^\varepsilon_2\|^2_{L^2_{xy}}\\
        &\quad+C\|\partial_x\partial_y H^\varepsilon\|^2_{L^2_{xy}}\|\partial_x\partial_y H^\varepsilon\|^2_{L^2_{xy}}+C\|\partial_x\partial_y V^\varepsilon\|^2_{L^2_{xy}}\|\partial_x\partial_y H^\varepsilon\|^2_{L^2_{xy}}+C\|\partial^2_y V_2^\varepsilon\|^2_{L^2_{xy}}\\
        &\quad+C\|\partial_x\partial_y H^\varepsilon\|^2_{L^2_{xy}}+C\|\partial_x\partial_y V^\varepsilon\|^2_{L^2_{xy}}+C\|\partial^2_{x} H^\varepsilon\|^4_{L^2_{xy}}+C\|\partial^2_{y} H^\varepsilon\|^2_{L^2_{xy}}+C\varepsilon^{-\frac{1}{2}}.
	\end{align*}
        By $(\ref{HVU theta Linfty})$, we can obtain
        \begin{align*}
        \mathcal{D}_4&=\int^\infty_0\int^\infty_{-\infty}(2\partial_x\partial_yV^\theta_1\partial_x H^\varepsilon+\partial_{y}V^\theta_1\partial^2_x H^\varepsilon+2\partial_{x}V^\theta_1\partial_x\partial_y H^\varepsilon+V^\theta_1\partial^2_x\partial_{y} H^\varepsilon\\
        &\quad+2\partial_x\partial_yV^\theta_2\partial_y H^\varepsilon+2\partial_{y}V^\theta_2\partial_x\partial_y H^\varepsilon+\partial_{x}V^\theta_2\partial^2_y H^\varepsilon+V^\theta_2\partial_x\partial^2_{y} H^\varepsilon\\
        &\quad+\partial_{y}H^\varepsilon\partial^2_x V^\theta_1+H^\varepsilon\partial^2_x\partial_{y}V^\theta_1+\partial_{x}H^\varepsilon\partial^2_y V^\theta_2+H^\varepsilon\partial_x\partial^2_{y} V^\theta_2)\partial_x\partial_y H^\varepsilon dxdy\\
        &\leq2\|\partial_{x}\partial_yV^\theta_1\|_{L^\infty_{xy}}\|\partial_x H^\varepsilon\|_{L^2_{xy}}\|\partial_{x}\partial_y H^\varepsilon\|_{L^2_{xy}}+\|\partial_yV^\theta_1\|_{L^\infty_{xy}}\|\partial^2_x H^\varepsilon\|_{L^2_{xy}}\|\partial_{x}\partial_y H^\varepsilon\|_{L^2_{xy}}\\
        &\quad+2\|\partial_xV^\theta_1\|_{L^\infty_{xy}}\|\partial_x\partial_y H^\varepsilon\|^2_{L^2_{xy}}+\|V^\theta_1\|_{L^\infty_{xy}}\|\partial^2_x\partial_y H^\varepsilon\|_{L^2_{xy}}\|\partial_{x}\partial_y H^\varepsilon\|_{L^2_{xy}}\\
        &\quad+2\|\partial_{x}\partial_y V^\theta_2\|_{L^2_{xy}}\|\partial_{y} H^\varepsilon\|_{L^\infty_{x}L^2_y}\|\partial_{x}\partial_y H^\varepsilon\|_{L^2_{x}L^\infty_y}+2\|\partial_yV^\theta_2\|_{L^\infty_{xy}}\|\partial_x\partial_y H^\varepsilon\|^2_{L^2_{xy}}\\
        &\quad+\|\partial_xV^\theta_2\|_{L^\infty_{xy}}\|\partial^2_y H^\varepsilon\|_{L^2_{xy}}\|\partial_{x}\partial_y H^\varepsilon\|_{L^2_{xy}}
+\|V^\theta_2\|_{L^\infty_{xy}}\|\partial_x\partial^2_y H^\varepsilon\|_{L^2_{xy}}\|\partial_{x}\partial_y H^\varepsilon\|_{L^2_{xy}}
        \\
        &\quad+\|\partial_y H^\varepsilon\|_{L^2_{xy}}\|\partial^2_{x}V^\theta_1\|_{L^\infty_{xy}}\|\partial_{x}\partial_y H^\varepsilon\|_{L^2_{xy}}+\|H^\varepsilon\|_{L^2_{xy}}\|\partial^2_{x}\partial_yV^\theta_1\|_{L^\infty_{xy}}\|\partial_{x}\partial_y H^\varepsilon\|_{L^2_{xy}}
        \\
        &\quad
        +\|\partial_{x} H^\varepsilon\|_{L^2_xL^\infty_{y}}\|\partial^2_y V^\theta_2\|_{L^2_{xy}}\|\partial_{x}\partial_y H^\varepsilon\|_{L^\infty_xL^2_{y}}+\| H^\varepsilon\|_{L^\infty_xL^2_{y}}\|\partial_x\partial^2_y V^\theta_2\|_{L^2_{xy}}\|\partial_{x}\partial_y H^\varepsilon\|_{L^2_xL^\infty_{y}}\\
        &\leq\frac{1}{10}\|\partial^2_{x}\partial_y H^\varepsilon\|^2_{L^2_{xy}}+\frac{1}{10}\|\partial_x\partial^2_{y} H^\varepsilon\|^2_{L^2_{xy}}+C\|\partial^2_{x} H^\varepsilon\|^2_{L^2_{xy}}+C\|\partial^2_{y} H^\varepsilon\|^2_{L^2_{xy}}\\
        &\quad+C\|\partial_x\partial_y H^\varepsilon\|^2_{L^2_{xy}}(1+\|\partial_x\partial_y H^\varepsilon\|^2_{L^2_{xy}})+C\varepsilon^{-1}.
		\end{align*}
        Similarly, it can be concluded that
        \begin{equation*}
		\begin{split}{}
        \mathcal{D}_5&=\int^\infty_0\int^\infty_{-\infty}(2\partial_x\partial_yV^\varepsilon_1\partial_x H^\theta+\partial_{y}V^\varepsilon_1\partial^2_x H^\theta+2\partial_{x}V^\varepsilon_1\partial_x\partial_y H^\theta+V^\varepsilon_1\partial^2_x\partial_{y} H^\theta\\
        &\quad+2\partial_x\partial_yV^\varepsilon_2\partial_y H^\theta+2\partial_{y}V^\varepsilon_2\partial_x\partial_y H^\theta+\partial_{x}V^\varepsilon_2\partial^2_y H^\theta+V^\varepsilon_2\partial_x\partial^2_{y} H^\theta \\
        &\quad+\partial_{y}H^\theta\partial^2_x V^\varepsilon_1+H^\theta\partial^2_x\partial_{y}V^\varepsilon_1+\partial_{x}H^\theta\partial^2_y V^\varepsilon_2+H^\theta\partial_x\partial^2_{y} V^\varepsilon_2)\partial_x\partial_y H^\varepsilon dxdy,\\
		\end{split}
	\end{equation*}
       where using integration by parts for the last four terms, we have
       \begin{align*}
           &\int^\infty_0\int^\infty_{-\infty}\partial_{y}H^\theta\partial^2_x V^\varepsilon_1\partial_x\partial_y H^\varepsilon dxdy=-\int^\infty_0\int^\infty_{-\infty}(\partial_x\partial_yH^\theta\partial_x V^\varepsilon_1\partial_x\partial_y H^\varepsilon+\partial_{y}H^\theta\partial_x V^\varepsilon_1\partial^2_x\partial_{y} H^\varepsilon) dxdy\\
           &\int^\infty_0\int^\infty_{-\infty}H^\theta\partial^2_x\partial_{y}V^\varepsilon_1\partial_x\partial_y H^\varepsilon dxdy=-\int^\infty_0\int^\infty_{-\infty}(\partial_{x}H^\theta\partial_x\partial_y V^\varepsilon_1\partial_x\partial_y H^\varepsilon+H^\theta\partial_x\partial_y V^\varepsilon_1\partial^2_x\partial_{y} H^\varepsilon) dxdy\\
           &\int^\infty_0\int^\infty_{-\infty}\partial_{x}H^\theta\partial^2_y V^\varepsilon_2\partial_x\partial_y H^\varepsilon dxdy=-\int^\infty_0\int^\infty_{-\infty}(\partial_x\partial_yH^\theta\partial_y V^\varepsilon_2\partial_x\partial_y H^\varepsilon+\partial_{x}H^\theta\partial_y V^\varepsilon_2\partial_x\partial^2_{y} H^\varepsilon) dxdy\\
           &\int^\infty_0\int^\infty_{-\infty}H^\theta\partial_x\partial^2_{y} V^\varepsilon_2\partial_x\partial_y H^\varepsilon dxdy=-\int^\infty_0\int^\infty_{-\infty}(\partial_{y}H^\theta\partial_x\partial_y V^\varepsilon_2\partial_x\partial_y H^\varepsilon+H^\theta\partial_x\partial_y V^\varepsilon_2\partial_x\partial^2_{y} H^\varepsilon) dxdy.
       \end{align*}
       Therefore, using the Sobolev embedding inequality, we obtain
    \begin{align*}        \mathcal{D}_5&\leq3\|\partial_{x}\partial_yV^\varepsilon_1\|_{L^2_{xy}}\|\partial_x H^\theta\|_{L^\infty_{xy}}\|\partial_{x}\partial_y H^\varepsilon\|_{L^2_{xy}}+\|\partial_yV^\varepsilon_1\|_{L^2_{xy}}\|\partial^2_x H^\theta\|_{L^\infty_{xy}}\|\partial_{x}\partial_y H^\varepsilon\|_{L^2_{xy}}\\
        &\quad+3\|\partial_{x}V^\varepsilon_1\|_{L^2_{xy}}\|\partial_x\partial_y H^\theta\|_{L^\infty_{xy}}\|\partial_{x}\partial_y H^\varepsilon\|_{L^2_{xy}}+\|V^\varepsilon_1\|_{L^2_{xy}}\|\partial^2_x\partial_y H^\theta\|_{L^\infty_{xy}}\|\partial_{x}\partial_y H^\varepsilon\|_{L^2_{xy}}\\
        &\quad+\|\partial_{x}\partial_yV^\varepsilon_2\|_{L^2_{xy}}\|\partial_y H^\theta\|_{L^\infty_{xy}}\|\partial_{x}\partial_y H^\varepsilon\|_{L^2_{xy}}+2\|\partial_yV^\varepsilon_2\|_{L^2_{xy}}\|\partial_{x}\partial_y H^\theta\|_{L^\infty_{xy}}\|\partial_{x}\partial_y H^\varepsilon\|_{L^2_{xy}}\\
        &\quad+\|\partial_xV^\varepsilon_2\|_{L^2_{xy}}\|\partial^2_y H^\theta\|_{L^\infty_{xy}}\|\partial_{x}\partial_y H^\varepsilon\|_{L^2_{xy}}+\|V^\varepsilon_2\|_{L^2_{xy}}\|\partial_x\partial^2_y H^\theta\|_{L^\infty_{xy}}\|\partial_{x}\partial_y H^\varepsilon\|_{L^2_{xy}}\\
        &\quad+\|\partial_y H^\theta\|_{L^\infty_{xy}}\|\partial_{x}V^\varepsilon_1\|_{L^2_{xy}}\|\partial^2_{x}\partial_y H^\varepsilon\|_{L^2_{xy}}+\|H^\theta\|_{L^\infty_{xy}}\|\partial_{x}\partial_yV^\varepsilon_1\|_{L^2_{xy}}\|\partial^2_{x}\partial_y H^\varepsilon\|_{L^2_{xy}}\\
        &\quad+\|\partial_{x}\partial_y H^\theta\|_{L^\infty_{xy}}\|\partial_{y}V^\varepsilon_2\|_{L^2_{xy}}\|\partial_{x}\partial_y H^\varepsilon\|_{L^2_{xy}}+\|\partial_{x}H^\theta\|_{L^\infty_{xy}}\|\partial_yV^\varepsilon_2\|_{L^2_{xy}}\|\partial_{x}\partial^2_y H^\varepsilon\|_{L^2_{xy}}\\
        &\quad+\|\partial_y H^\theta\|_{L^\infty_{xy}}\|\partial_{x}\partial_{y}V^\varepsilon_2\|_{L^2_{xy}}\|\partial_{x}\partial_y H^\varepsilon\|_{L^2_{xy}}+\|H^\theta\|_{L^\infty_{xy}}\|\partial_{x}\partial_yV^\varepsilon_2\|_{L^2_{xy}}\|\partial_{x}\partial^2_y H^\varepsilon\|_{L^2_{xy}}\\
        &\leq\frac{1}{10}\|\partial^2_{x}\partial_y H^\varepsilon\|^2_{L^2_{xy}}+\frac{1}{10}\|\partial_x\partial^2_{y} H^\varepsilon\|^2_{L^2_{xy}}+C\|\partial_x\partial_y V^\varepsilon_2\|^2_{L^2_{xy}}\\
        &\quad+C\|\partial_x\partial_y H^\varepsilon\|^2_{L^2_{xy}}(1+\|\partial_x\partial_y H^\varepsilon\|^2_{L^2_{xy}})+C\|\partial_x\partial_y V^\varepsilon_1\|^2_{L^2_{xy}}+C\varepsilon^{-\frac{1}{2}}.
	\end{align*}
	 Combining Lemma $\ref{lem f L2}$, we can collect items with similar estimates
	\begin{equation*}
		\begin{split}{}
        \mathcal{D}_6&\leq\varepsilon^{-\frac{1}{2}}\|\partial_x\partial_y f^\varepsilon\|_{L^2_{xy}}\|\partial_x\partial_y H^\varepsilon\|_{L^2_{xy}}\leq C\|\partial_x\partial_y H^\varepsilon\|^2_{L^2_{xy}}+C\varepsilon^{-\frac{1}{2}}.
		\end{split}
	\end{equation*}
       It then yields
       \begin{align}\label{frac{d}{dt} partial xy H varepsilon}
			&\frac{1}{2}\frac{d}{dt}\|\partial_x\partial_y H^\varepsilon\|^2_{L^2_{xy}}+\|\partial^2_{x}\partial_{y} H^\varepsilon\|^2_{L^2_{xy}}+\|\partial_{x}\partial^2_{y} H^\varepsilon\|^2_{L^2_{xy}}\notag\\
            &\leq\frac{1}{2}\|\partial^2_{x}\partial_{y} H^\varepsilon\|^2_{L^2_{xy}}+\frac{1}{2}\|\partial_{x}\partial^2_{y} H^\varepsilon\|^2_{L^2_{xy}}+\frac{1}{4}\varepsilon\|\partial^2_{x}\partial_{y} V^\varepsilon\|^2_{L^2_{xy}}+\frac{1}{4}\varepsilon\|\partial_{x}\partial^2_{y} V^\varepsilon\|^2_{L^2_{xy}}\\
            &\;\;\;+C\|\partial_x\partial_y H^\varepsilon\|^2_{L^2_{xy}}(\|\partial_x\partial_y H^\varepsilon\|^2_{L^2_{xy}}+1)+C\|\partial^2_xH^\varepsilon\|^2_{L^2_{xy}}+C\|\partial^2_yH^\varepsilon\|^2_{L^2_{xy}}\notag\\
            &\quad+C\|\partial_x\partial_y V^\varepsilon\|^2_{L^2_{xy}}\|\partial_x\partial_y H^\varepsilon\|^2_{L^2_{xy}}+C\|\partial_x\partial_y H^\varepsilon\|^2_{L^2_{xy}}+C\|\partial_x\partial_y V^\varepsilon\|^2_{L^2_{xy}}+C\varepsilon^{-1}.\notag\notag
	\end{align}
    \textbf{Step 2. The \(L^2\) estimate of $\partial_x\partial_yV^\varepsilon$.}
    Taking $-\partial_{x}$ to $(\ref{HVU equation})_2$ and multiplying by $\partial_{x}\partial^{2}_{y} V^\varepsilon$, we obtain
            \begin{align*}
			&\frac{1}{2}\frac{d}{dt}\|\partial_{x}\partial_{y} V^\varepsilon\|^2_{L^2_{xy}}+\varepsilon\| \partial_{y}\partial^{2}_{x}V^\varepsilon\|^2_{L^2_{xy}}+\varepsilon\| \partial_{x}\partial^{2}_{y}V^\varepsilon\|^2_{L^2_{xy}}\\
            &=\varepsilon^\frac{1}{2}\int^\infty_0\int^\infty_{-\infty}\partial_{x}\nabla(U^\varepsilon\cdot V^\varepsilon)\cdot\partial_{x}\partial^{2}_{y} V^\varepsilon dxdy+\int^\infty_0\int^\infty_{-\infty}\partial_{x}\nabla(U^\varepsilon\cdot V^\theta)\cdot\partial_{x}\partial^{2}_{y} V^\varepsilon dxdy\\
            &\quad+\int^\infty_0\int^\infty_{-\infty}\partial_{x}\nabla(U^\theta\cdot V^\varepsilon)\cdot\partial_{x}\partial^{2}_{y} V^\varepsilon dxdy+\varepsilon^\frac{3}{2}\int^\infty_0\int^\infty_{-\infty}\partial_{x}\nabla(V^\varepsilon\cdot V^\varepsilon)\cdot\partial_{x}\partial^{2}_{y} V^\varepsilon dxdy\\
            &\quad+2\varepsilon\int^\infty_0\int^\infty_{-\infty}\partial_{x}\nabla(V^\varepsilon\cdot V^\theta)\cdot\partial_{x}\partial^{2}_{y} V^\varepsilon dxdy-\int^\infty_0\int^\infty_{-\infty}\partial_{x}\nabla H^\varepsilon\cdot\partial_{x}\partial^{2}_{y} V^\varepsilon dxdy\\
            &\quad-\varepsilon^{-\frac{1}{2}}\int^\infty_0\int^\infty_{-\infty}\partial_{x}g^\varepsilon\cdot\partial_{x}\partial^{2}_{y} V^\varepsilon dxdy=\sum^7_{i=1}\mathcal{M}_i.
	\end{align*}	
        For the term $\mathcal{M}_1$,
        \begin{align*}
       \mathcal{M}_1
            &=\varepsilon^\frac{1}{2}\int^\infty_0\int^\infty_{-\infty}\partial^{2}_{x}(U^\varepsilon\cdot V^\varepsilon)\cdot\partial_{x}\partial^{2}_{y} V_{1}^\varepsilon dxdy+\varepsilon^\frac{1}{2}\int^\infty_0\int^\infty_{-\infty}\partial_{x}\partial_{y}(U^\varepsilon\cdot V^\varepsilon)\cdot\partial_{x}\partial^{2}_{y} V_{2}^\varepsilon dxdy\\
            &=\mathcal{M}_{11}+\mathcal{M}_{12},
	\end{align*}
        Where
        \begin{equation*}
		\begin{split}{}
            \mathcal{M}_{11}
            &=\varepsilon^\frac{1}{2}\int^\infty_0\int^\infty_{-\infty}(\partial^{2}_{x}U^\varepsilon_1V^\varepsilon_1+2\partial_{x}U^\varepsilon_1\partial_{x}V^\varepsilon_1+U^\varepsilon_1 \partial^{2}_{x}V^\varepsilon_1+\partial^{2}_{x}U^\varepsilon_2 V^\varepsilon_2\\
           &\quad+2\partial_{x}U^\varepsilon_2 \partial_{x}V^\varepsilon_2+U^\varepsilon_2\partial^{2}_{x}V^\varepsilon_2)\partial_{x}\partial^{2}_{y}V_{1}^\varepsilon dxdy\\
           &\leq\varepsilon^\frac{1}{2}(\|\partial^{2}_{x}U^\varepsilon_1\|_{L^2_{xy}}\|V^\varepsilon_1\|_{L^\infty_{xy}}+2\|\partial_{x}U^\varepsilon_1\|_{L^{\infty}_{x}L^{2}_{y}}\|\partial_{x}V^\varepsilon_1\|_{L^{2}_{x}L^{\infty}_{y}}+\|U^\varepsilon_1\|_{L^\infty_{xy}}\|\partial^{2}_{x}V^\varepsilon_1\|_{L^2_{xy}}\\
           &\quad+\|\partial^{2}_{x}U^\varepsilon_2\|_{L^2_{xy}}\|V^\varepsilon_2\|_{L^\infty_{xy}}+2\|\partial_{x}U^\varepsilon_2\|_{L^{\infty}_{x}L^{2}_{y}}\|\partial_{x}V^\varepsilon_2\|_{L^{2}_{x}L^{\infty}_{y}}+\|U^\varepsilon_2\|_{L^\infty_{xy}}\|\partial^{2}_{x}V^\varepsilon_2\|_{L^2_{xy}})\|\partial_{x}\partial^{2}_{y}V_{1}^\varepsilon\|_{L^2_{xy}}\\
           &\leq \frac{1}{8}\varepsilon\|\partial_{x}\partial^{2}_{y}V_{1}^\varepsilon\|^2_{L^2_{xy}}+C\|\partial_{x}\partial_{y}V_{1}^\varepsilon\|^2_{L^2_{xy}}+C\|\partial^2_{x}V^\varepsilon\|^2_{L^2_{xy}}+C
        \end{split}
	\end{equation*}		
For the term $\mathcal{M}_{12}$,
\begin{align*}
       \mathcal{M}_{12}&=\varepsilon^\frac{1}{2}\int^\infty_0\int^\infty_{-\infty}(\partial_{x}\partial_{y}U^\varepsilon_1 V^\varepsilon_1+\partial_{y}U^\varepsilon_1 \partial_{x}V^\varepsilon_1+\partial_{x}U^\varepsilon_1\partial_{y}V^\varepsilon_1+U^\varepsilon_1 \partial_{x}\partial_{y}V^\varepsilon_1\\
       &\quad+\partial_{x}\partial_{y}U^\varepsilon_2 V^\varepsilon_2+\partial_{y}U^\varepsilon_2 \partial_{x}V^\varepsilon_2+\partial_{x}U^\varepsilon_2\partial_{y}V^\varepsilon_2+U^\varepsilon_2 \partial_{x}\partial_{y}V^\varepsilon_2)\partial_{x}\partial^{2}_{y} V_{2}^\varepsilon dxdy\\
       &\leq\varepsilon^\frac{1}{2}(\|\partial_{x}\partial_{y}U^\varepsilon_1\|_{L^{2}_{xy}}\|V^\varepsilon_1\|_{L^{\infty}_{xy}}       +\|\partial_{y}U^\varepsilon_1\|_{L^{\infty}_{x}L^{2}_{y}}\|\partial_{x}V^\varepsilon_1\|_{L^{2}_{x}L^{\infty}_{y}}       +\|\partial_{x}U^\varepsilon_1\|_{L^{2}_{x}L^{\infty}_{y}}\|\partial_{y}V^\varepsilon_1\|_{L^{\infty}_{x}L^{2}_{y}}\\       &\quad+\|U^\varepsilon_1\|_{L^{\infty}_{xy}}\|\partial_{x}\partial_{y}V^\varepsilon_1\|_{L^{2}_{xy}}+\|\partial_{x}\partial_{y}U^\varepsilon_2\|_{L^{2}_{xy}}\|V^\varepsilon_2\|_{L^{\infty}_{xy}}       +\|\partial_{y}U^\varepsilon_2\|_{L^{\infty}_{x}L^{2}_{y}}\|\partial_{x}V^\varepsilon_2\|_{L^{2}_{x}L^{\infty}_{y}}\\       &\quad+\|\partial_{x}U^\varepsilon_2\|_{L^{2}_{x}L^{\infty}_{y}}\|\partial_{y}V^\varepsilon_2\|_{L^{\infty}_{x}L^{2}_{y}}+\|U^\varepsilon_2\|_{L^{\infty}_{xy}}\|\partial_{x}\partial_{y}V^\varepsilon_2\|_{L^{2}_{xy}})
       \|\partial_{x}\partial^{2}_{y} V_{2}^\varepsilon\|_{L^{2}_{xy}}\\
       &\leq\frac{1}{8}\varepsilon\|\partial_{x}\partial^{2}_{y} V_{2}^\varepsilon\|^{2}_{L^{2}_{xy}}+C\|\partial_{x}\partial_{y}V^\varepsilon\|^{2}_{L^{2}_{xy}}+C\varepsilon^{-\frac{1}{2}}.
	\end{align*}
Note that
\begin{align*}
       \mathcal{M}_{2}&=
\int^\infty_0\int^\infty_{-\infty}\partial^{2}_{x}(U^\varepsilon\cdot V^\theta)\partial_{x}\partial^{2}_{y} V^\varepsilon_{1} dxdy
+\int^\infty_0\int^\infty_{-\infty}\partial_{x}\partial_{y}(U^\varepsilon\cdot V^\theta)\partial_{x}\partial^{2}_{y} V^\varepsilon_{2} dxdy\\
&=\mathcal{M}_{21}+\mathcal{M}_{22}.
	\end{align*}
By Sobolev embedding inequality and integration by parts, we can obtain
\begin{align*}
       \mathcal{M}_{21}&=-\int^\infty_0\int^\infty_{-\infty}\partial_y(\partial^{2}_{x}U^\varepsilon_1 V^\theta_1+2\partial_{x}U^\varepsilon_1 \partial_{x}V^\theta_1+U^\varepsilon_1 \partial^{2}_{x}V^\theta_1+\partial^{2}_{x}U^\varepsilon_2 V^\theta_2\\
       &\quad+2\partial_{x}U^\varepsilon_2 \partial_{x}V^\theta_2+U^\varepsilon_2 \partial^{2}_{x}V^\theta_2)\partial_{x}\partial_{y} V^\varepsilon_{1} dxdy\\       &\leq(\|\partial^{2}_{x}\partial_{y}U^\varepsilon_1\|_{L^{2}_{xy}}\|V^\theta_1\|_{L^{\infty}_{xy}}       +\|\partial^{2}_{x}U^\varepsilon_1\|_{L^{2}_{xy}}\|\partial_{y}V^\theta_1\|_{L^{\infty}_{xy}}       +2\|\partial_{x}\partial_{y}U^\varepsilon_1\|_{L^{2}_{xy}}\|\partial_{x}V^\theta_1\|_{L^{\infty}_{xy}}\\       &\quad+2\|\partial_{x}U^\varepsilon_1\|_{L^{2}_{xy}}\|\partial_{x}\partial_{y}V^\theta_1\|_{L^{\infty}_{xy}}       +\|\partial_{y}U^\varepsilon_1\|_{L^{2}_{xy}}\|\partial^{2}_{x}V^\theta_1\|_{L^{\infty}_{xy}}+\|U^\varepsilon_1\|_{L^{\infty}_{xy}}\|\partial^{2}_{x}\partial_{y}V^\theta_1\|_{L^{2}_{xy}}\\
       &\quad+\|\partial^{2}_{x}\partial_{y}U^\varepsilon_2\|_{L^{2}_{xy}}\|V^\theta_2\|_{L^{\infty}_{xy}}       +\|\partial^{2}_{x}U^\varepsilon_2\|_{L^{2}_{xy}}\|\partial_{y}V^\theta_2\|_{L^{\infty}_{xy}}       +2\|\partial_{x}\partial_{y}U^\varepsilon_2\|_{L^{2}_{xy}}\|\partial_{x}V^\theta_2\|_{L^{\infty}_{xy}}\\       &\quad+2\|\partial_{x}U^\varepsilon_2\|_{L^{2}_{xy}}\|\partial_{x}\partial_{y}V^\theta_2\|_{L^{\infty}_{xy}}       +\|\partial_{y}U^\varepsilon_2\|_{L^{2}_{xy}}\|\partial^{2}_{x}V^\theta_2\|_{L^{\infty}_{xy}}\\
       &\quad+\|U^\varepsilon_2\|_{L^{\infty}_{xy}}\|\partial^{2}_{x}\partial_{y}V^\theta_2\|_{L^{2}_{xy}})\|\partial_{x}\partial_{y} V^\varepsilon_1\|_{L^{2}_{xy}}\\
       &\leq C\|\partial_{x}\partial_{y} V^\varepsilon_1\|^2_{L^{2}_{xy}}+C\|\partial^{2}_{x}\omega^\varepsilon\|^2_{L^{2}_{xy}}+C\varepsilon^{-\frac{1}{2}}.
\end{align*}

For the term $\mathcal{M}_{22}$,
\begin{align*}
       \mathcal{M}_{22}&=\int^\infty_0\int^\infty_{-\infty}\varepsilon^{-\frac{1}{2}}(\partial_{x}\partial_{y}U^\varepsilon_1 V^\theta_1+\partial_{y}U^\varepsilon_1\partial_{x}V^\theta_1
+\partial_{x}\partial_{y}U^\varepsilon_2 V^\theta_2+\partial_{y}U^\varepsilon_2\partial_{x}V^\theta_2\\
&\quad+
\partial_{x}U^\varepsilon_{2} \partial_{y}V^\theta_{2}+U^\varepsilon_{2} \partial_{x}\partial_{y}V^\theta_{2})\varepsilon^{\frac{1}{2}}\partial_{x}\partial^2_{y} V^\varepsilon_{2} dxdy\notag\\
&\quad-\int^\infty_{-\infty}(\partial_{x}\partial_{y}U^\varepsilon_{1} \partial_{y}V^\theta_{1}+\partial_{x}U^\varepsilon_1\partial^2_{y}V^\theta_1+\partial_{y}U^\varepsilon_1\partial_{x}\partial_{y}V^\theta_1+U^\varepsilon_{1}\partial_{x}\partial^2_{y}V^\theta_{1})\partial_{x}\partial_{y} V^\varepsilon_{2}dxdy\notag\\
&\leq\varepsilon^{-\frac{1}{2}}(\|\partial_{x}\partial_{y}U^\varepsilon_1\|_{L^{2}_{xy}} \|V^\theta_1\|_{L^{\infty}_{xy}}
+\|\partial_{y}U^\varepsilon_1\|_{L^{2}_{xy}}\|\partial_{x}V^\theta_1\|_{L^{\infty}_{xy}}+\|\partial_{x}\partial_{y}U^\varepsilon_2\|_{L^{2}_{xy}} \|V^\theta_2\|_{L^{\infty}_{xy}}\\
&\quad+\|\partial_{y}U^\varepsilon_2\|_{L^{2}_{xy}}\|\partial_{x}V^\theta_2\|_{L^{\infty}_{xy}}+\|\partial_{x}U^\varepsilon_2\|_{L^{2}_{xy}}\|\partial_{y}V^\theta_2\|_{L^{\infty}_{xy}}
+\|U^\varepsilon_2\|_{L^2_{xy}}\|\partial_{x}\partial_{y}V^\theta_2\|_{L^{\infty}_{xy}})\varepsilon^{\frac{1}{2}}\|\partial_{x}\partial^2_{y} V^\varepsilon_{2}\|_{L^{2}_{xy}}\\
&\quad+(\|\partial_{x}\partial_{y}U^\varepsilon_{1}\|_{L^{2}_{xy}} \|\partial_{y}V^\theta_{1}\|_{L^{\infty}_{xy}}+\|\partial_{x}U^\varepsilon_1\|_{L^{2}_{xy}}\|\partial^2_{y}V^\theta_1\|_{L^{\infty}_{xy}}+\|\partial_{y}U^\varepsilon_1\|_{L^{2}_{xy}}\|\partial_{x}\partial_{y}V^\theta_1\|_{L^{\infty}_{xy}}\\
&\quad+\|U^\varepsilon_{1}\|_{L^{2}_{xy}}\|\partial_{x}\partial^2_{y}V^\theta_{1}\|_{L^{\infty}_{xy}})\|\partial_{x}\partial_{y} V^\varepsilon_{2}\|_{L^{2}_{xy}}\\
&\leq \frac{1}{8}\varepsilon\|\partial_{x}\partial^2_{y} V^\varepsilon_{2}\|^2_{L^{2}_{xy}}+C\|\partial_{x}\partial_{y} V^\varepsilon_{2}\|^2_{L^{2}_{xy}}+C\varepsilon^{-\frac{3}{2}}.
\end{align*}
Then
\begin{equation*}
		\begin{split}{}
\mathcal{M}_{2}&\leq\frac{1}{8}\varepsilon\|\partial_{x}\partial^2_{y} V^\varepsilon_{2}\|^2_{L^{2}_{xy}}+C\|\partial_{x}\partial_{y} V^\varepsilon\|^2_{L^{2}_{xy}}+C\|\partial^{2}_{x}\omega^\varepsilon\|_{L^{2}_{xy}}+C\varepsilon^{-\frac{3}{2}}.
\end{split}
	\end{equation*}
Next, we estimate the term $\mathcal{M}_{3}$,
\begin{equation*}
		\begin{split}{}
       \mathcal{M}_{3}&=
\int^\infty_0\int^\infty_{-\infty}\partial^{2}_{x}(U^\theta\cdot V^\varepsilon)\cdot\partial_{x}\partial^{2}_{y} V^\varepsilon_{1} dxdy+
\int^\infty_0\int^\infty_{-\infty}\partial_{x}\partial_{y}(U^\theta\cdot V^\varepsilon)\cdot\partial_{x}\partial^{2}_{y} V^\varepsilon_{2} dxdy\\
&
=\mathcal{M}_{31}+\mathcal{M}_{32}.
\end{split}
	\end{equation*}
Here by the curl conditions,
\begin{align*}
       \mathcal{M}_{31}&=\int^\infty_0\int^\infty_{-\infty}\varepsilon^{-\frac{1}{2}}(\partial^{2}_{x}U^\theta_1 V^\varepsilon_1+U^\theta_1 \partial^2_{x}V^\varepsilon_1+\partial^2_{x}U^\theta_2 V^\varepsilon_2)\varepsilon^{\frac{1}{2}}\partial_{x}\partial^2_{y} V_{1}^\varepsilon dxdy\\
       &\quad-\int^\infty_0\int^\infty_{-\infty}(2\partial_{x}\partial_{y}U^\theta_1\partial_{x} V^\varepsilon_1+2\partial_{x}U^\theta_1 \partial_{x}\partial_{y}V^\varepsilon_1+2\partial_{x}\partial_{y}U^\theta_2V^\varepsilon_2+2\partial_{x}U^\theta_2 \partial_{x}\partial_{y}V^\varepsilon_2\\
       &\quad+\partial_{y}U^\theta_2\partial^{2}_{x}V^\varepsilon_2+U^\theta_2\partial^{2}_{x}\partial_{y}V^\varepsilon_2)\partial_{x}\partial_{y} V_{1}^\varepsilon dxdy\\
       &\leq\varepsilon^{-\frac{1}{2}}(\|\partial^{2}_{x}U^\theta_1\|_{L^{\infty}_{xy}}\|V^\varepsilon_1\|_{L^{2}_{xy}}+\|U^\theta_1\|_{L^{\infty}_{xy}} \|\partial^2_{x}V^\varepsilon_1\|_{L^{2}_{xy}}+\|\partial^2_{x}U^\theta_2\|_{L^{\infty}_{xy}}\|V^\varepsilon_2\|_{L^{2}_{xy}})\varepsilon^{\frac{1}{2}}\|\partial_{x}\partial^2_{y}V^\varepsilon_1\|_{L^{2}_{xy}}\\
       &\quad+(2\|\partial_{x}\partial_{y}U^\theta_1\|_{L^{\infty}_{xy}}\|\partial_{x}V^\varepsilon_1\|_{L^{2}_{xy}}
        +2\|\partial_{x}U^\theta_1\|_{L^{\infty}_{xy}}\| \partial_{x}\partial_{y}V^\varepsilon_1\|_{L^{2}_{xy}}+2\|\partial_{x}\partial_{y}U^\theta_2\|_{L^{\infty}_{xy}}\|V^\varepsilon_2\|_{L^{2}_{xy}}\\
        &\quad+2\|\partial_{x}U^\theta_2\|_{L^{\infty}_{xy}}\| \partial_{x}\partial_{y}V^\varepsilon_2\|_{L^{2}_{xy}}+\|\partial_yU^\theta_2\|_{L^{\infty}_{xy}}\|\partial^2_{x}V^\varepsilon_2\|_{L^{2}_{xy}}+\|U^\theta_2\|_{L^{\infty}_{xy}}\|\partial^2_{x}\partial_yV^\varepsilon_2\|_{L^{2}_{xy}})\|\partial_{x}\partial_{y} V_{1}^\varepsilon\|_{L^{2}_{xy}}\\
        &\leq \frac{1}{8}\varepsilon\|\partial_{x}\partial^2_{y}V^\varepsilon_1\|^2_{L^{2}_{xy}}+C\|\partial_{x}\partial_{y}V^\varepsilon\|^{2}_{L^{2}_{xy}}+C\|\partial^{2}_{x}V^\varepsilon\|^{2}_{L^{2}_{xy}}+C\varepsilon^{-1}.
	\end{align*}
    Similarly combining the curl conditions,
    \begin{align*}
           \mathcal{M}_{32}&=
    \int^\infty_0\int^\infty_{-\infty}\varepsilon^{-\frac{1}{2}}(\partial_{x}\partial_{y}U^\theta_1 V^\varepsilon_1+\partial_{y}U^\theta_1 \partial_{x}V^\varepsilon_1+\partial_{x}U^\theta_1 \partial_{x}V^\varepsilon_2+U^\theta_1\partial^2_{x}V^\varepsilon_2\\
    &\quad+\partial_{x}\partial_{y}U^\theta_2 V^\varepsilon_2+\partial_{y}U^\theta_2 \partial_{x}V^\varepsilon_2)\varepsilon^{\frac{1}{2}}\partial_{x}\partial^2_{y} V_{2}^\varepsilon dxdy\\
    &\quad-\int^\infty_{0}\int^\infty_{-\infty}(\partial_{x}\partial_{y}U^\theta_{2}\partial_{y}V^\varepsilon_{2}+\partial_{x}U^\theta_{2}\partial^2_{y}V^\varepsilon_{2})\partial_{x}\partial_{y}V^\varepsilon_{2} dxdy+\int^\infty_{-\infty}\partial_{x}U^\theta_{2}\partial_{y}V^\varepsilon_{2}\partial_{x}\partial_{y}V^\varepsilon_{2} dx\\
    &\quad-\int^\infty_{0}\int^\infty_{-\infty}\frac{1}{2}\partial_{y}U^\theta_{2}|\partial_{x}\partial_{y}V^\varepsilon_{2}|^2 dxdy+\int^\infty_{-\infty}U^\theta_{2}|\partial_{x}\partial_{y}V^\varepsilon_{2}|^2 dx\\
    &\leq\varepsilon^{-\frac{1}{2}}(\|\partial_{x}\partial_{y}U^\theta_1\|_{L^{\infty}_{xy}} \|V^\varepsilon_1\|_{L^{2}_{xy}}+\|\partial_{y}U^\theta_1\|_{L^{\infty}_{xy}} \|\partial_{x}V^\varepsilon_1\|_{L^{2}_{xy}}+\|\partial_{x}U^\theta_1\|_{L^{\infty}_{xy}}\| \partial_{x}V^\varepsilon_2\|_{L^{2}_{xy}}\\
    &\quad+\|U^\theta_1\|_{L^{\infty}_{xy}}\|\partial^2_{x}V^\varepsilon_2\|_{L^{2}_{xy}}+\|\partial_{x}\partial_{y}U^\theta_2\|_{L^{\infty}_{xy}}\| V^\varepsilon_2\|_{L^{2}_{xy}}+\|\partial_{y}U^\theta_2\|_{L^{\infty}_{xy}} \|\partial_{x}V^\varepsilon_2\|_{L^{2}_{xy}})\varepsilon^{\frac{1}{2}}\|\partial_{x}\partial^2_{y} V_{2}^\varepsilon\|_{L^{2}_{xy}}\\
    &\quad+(\|\partial_{x}\partial_{y}U^\theta_{2}\|_{L^{\infty}_{xy}}\|\partial_{y}V^\varepsilon_{2}\|_{L^{2}_{xy}}+\|\partial_{x}U^\theta_{2}\|_{L^{\infty}_{xy}}\|\partial^2_{y}V^\varepsilon_{2}\|_{L^{2}_{xy}})\|\partial_{x}\partial_{y}V^\varepsilon_{2}\|_{L^{2}_{xy}}\\
    &\quad+\|\partial_{x}U^\theta_{2}(t,x,0)\|_{L^{\infty}_{xy}}\|\partial_{y}V^\varepsilon_{2}\|^{\frac{1}{2}}_{L^2_{xy}}\|\partial^2_{y}V^\varepsilon_{2}\|^{\frac{1}{2}}_{L^2_{xy}}\|\partial_{x}\partial_{y}V^\varepsilon_{2}\|^{\frac{1}{2}}_{L^2_{xy}}\|\partial_{x}\partial^2_{y}V^\varepsilon_{2}\|^{\frac{1}{2}}_{L^2_{xy}}\\
    &\quad+\frac{1}{2}\|\partial_{y}U^\theta_{2}\|_{L^{\infty}_{xy}}\|\partial_{x}\partial_{y}V^\varepsilon_{2}\|^2_{L^2_{xy}}+\|U^\theta_{2}(t,x,0)\|_{L^{\infty}_{xy}}\|\partial_{x}\partial_{y}V^\varepsilon_{2}\|_{L^2_{xy}}\|\partial_{x}\partial^2_{y}V^\varepsilon_{2}\|_{L^2_{xy}}\\
    &\leq\frac{1}{8}\varepsilon\|\partial_{x}\partial^{2}_{y} V^\varepsilon_2\|^{2}_{L^{2}_{xy}}+C\|\partial_{x}\partial_{y} V^\varepsilon_2\|^{2}_{L^{2}_{xy}}+C\|\partial^2_{x}V^\varepsilon_2\|_{L^{2}_{xy}}+C\|\partial^2_{y}V^\varepsilon_{2}\|^2_{L^{2}_{xy}}+C\varepsilon^{-\frac{1}{2}}.
    	\end{align*}
    Combining $M_{31}$ and $M_{32}$, we deduce
    \begin{align*}
           \mathcal{M}_{3}&\leq\frac{1}{4}\varepsilon\|\partial_{x}\partial^{2}_{y} V^\varepsilon_2\|^{2}_{L^{2}_{xy}}+C\|\partial_{x}\partial_{y} V^\varepsilon_2\|^{2}_{L^{2}_{xy}}+C\|\partial^2_{x}V^\varepsilon_2\|_{L^{2}_{xy}}+C\|\partial^2_{y}V^\varepsilon_{2}\|^2_{L^{2}_{xy}}+C\varepsilon^{-1}.
    	\end{align*}
    For the term $\mathcal{M}_{4}$, by Sobolev inequality, we have
    \begin{align*}
       \mathcal{M}_{4}&=\varepsilon^\frac{3}{2}\int^\infty_0\int^\infty_{-\infty}\partial^{2}_{x}(V^\varepsilon\cdot V^\varepsilon)\cdot\partial_{x}\partial^{2}_{y} V^\varepsilon_{1} dxdy+\varepsilon^\frac{3}{2}\int^\infty_0\int^\infty_{-\infty}\partial_{x}\partial_{y}(V^\varepsilon\cdot V^\varepsilon)\cdot\partial_{x}\partial^{2}_{y} V^\varepsilon_{2} dxdy\\&
       =\varepsilon^\frac{3}{2}\int^\infty_0\int^\infty_{-\infty}(2\partial^{2}_{x}V^\varepsilon_1 V^\varepsilon_1+2|\partial_{x}V^\varepsilon_1|^{2}+2\partial^{2}_{x}V^\varepsilon_2 V^\varepsilon_2+2|\partial_{x}V^\varepsilon_2|^{2})\partial_{x}\partial^{2}_{y} V^\varepsilon_{1} dxdy\\
       &\quad+\varepsilon^\frac{3}{2}\int^\infty_0\int^\infty_{-\infty}(2\partial_{x}\partial_{y}V^\varepsilon_1 V^\varepsilon_1+2\partial_{y}V^\varepsilon_1\partial_{x}V^\varepsilon_1+2\partial_{x}\partial_{y}V^\varepsilon_2 V^\varepsilon_2+2\partial_{y}V^\varepsilon_2\partial_{x}V^\varepsilon_2)\partial_{x}\partial^{2}_{y} V^\varepsilon_{2} dxdy\\
       &\leq2\varepsilon^\frac{3}{2}(\|\partial^{2}_{x}V^\varepsilon_1\|_{L^{2}_{x}L^{\infty}_{y}}\|V^\varepsilon_1\|_{L^{\infty}_{x}L^{2}_{y}}
       +\|\partial_{x}V^\varepsilon_1\|_{L^{2}_{x}L^{\infty}_{y}}
       \|\partial_{x}V^\varepsilon_1\|_{L^{\infty}_{x}L^{2}_{y}}+\|\partial^{2}_{x}V^\varepsilon_2\|_{L^{2}_{x}L^{\infty}_{y}}\|V^\varepsilon_2\|_{L^{\infty}_{x}L^{2}_{y}}\\
       &\quad+\|\partial_{x}V^\varepsilon_2\|_{L^{2}_{x}L^{\infty}_{y}}
       \|\partial_{x}V^\varepsilon_2\|_{L^{\infty}_{x}L^{2}_{y}})\|\partial_{x}\partial^{2}_{y}V^\varepsilon\|_{L^{2}_{xy}}\\
       &\quad+2\varepsilon^\frac{3}{2}(\|\partial_{x}\partial_{y}V^\varepsilon_1\|_{L^{\infty}_{x}L^{2}_{y}}
       \|V^\varepsilon_1\|_{L^{2}_{x}L^{\infty}_{y}}+\|\partial_{y}V^\varepsilon_1\|_{L^{\infty}_{x}L^{2}_{y}}\|\partial_{x}V^\varepsilon_1\|_{L^{2}_{x}L^{\infty}_{y}}+\|\partial_{x}\partial_{y}V^\varepsilon_2\|_{L^{\infty}_{x}L^{2}_{y}}
       \|V^\varepsilon_2\|_{L^{2}_{x}L^{\infty}_{y}}\\
       &\quad+\|\partial_{y}V^\varepsilon_2\|_{L^{\infty}_{x}L^{2}_{y}}\|\partial_{x}V^\varepsilon_2\|_{L^{2}_{x}L^{\infty}_{y}})
       \|\partial_{x}\partial^{2}_{y}V^\varepsilon\|_{L^{2}_{xy}}\\
       &\leq\frac{1}{8}\varepsilon\|\partial_{x}\partial^{2}_{y}V^\varepsilon\|^{2}_{L^{2}_{xy}}+\frac{1}{2}\varepsilon\|\partial^{2}_{x}\partial_{y}V^\varepsilon\|^{2}_{L^{2}_{xy}}+C\|\partial^2_{x}V^\varepsilon\|^{2}_{L^{2}_{xy}}+C
       \|\partial_{x}\partial_{y}V^\varepsilon\|^{2}_{L^{2}_{xy}}.
	\end{align*}
    For the term $\mathcal{M}_{5}$,
    \begin{align*}
       \mathcal{M}_{5}&=2\varepsilon\int^\infty_0\int^\infty_{-\infty}\partial^{2}_{x}(V^\varepsilon\cdot V^\theta)\partial_{x}\partial^{2}_{y} V_{1}^\varepsilon dxdy+2\varepsilon\int^\infty_0\int^\infty_{-\infty}\partial_{x}\partial_{y}(V^\varepsilon\cdot V^\theta)\partial_{x}\partial^{2}_{y} V_{2}^\varepsilon dxdy\\
       &=2\varepsilon\int^\infty_0\int^\infty_{-\infty}(\partial^{2}_{x}V^\varepsilon_1 V^\theta_1+2\partial_{x}V^\varepsilon_1 \partial_{x}V^\theta_1+V^\varepsilon_1\partial^{2}_{x}V^\theta_1+\partial^{2}_{x}V^\varepsilon_2 V^\theta_2+2\partial_{x}V^\varepsilon_2 \partial_{x}V^\theta_2\\
       &\quad+V^\varepsilon_2\partial^{2}_{x}V^\theta_2)\partial_{x}\partial^{2}_{y} V_{1}^\varepsilon dxdy+2\varepsilon\int^\infty_0\int^\infty_{-\infty}(\partial_{x}\partial_{y}V^\varepsilon_1 V^\theta_1+\partial_{y}V^\varepsilon_1 \partial_{x}V^\theta_1+\partial_{x}V^\varepsilon_1 \partial_{y}V^\theta_1+V^\varepsilon_1 \partial_{x}\partial_{y}V^\theta_1\\
       &\quad+\partial_{x}\partial_{y}V^\varepsilon_2 V^\theta_2+\partial_{y}V^\varepsilon_2 \partial_{x}V^\theta_2+\partial_{x}V^\varepsilon_2 \partial_{y}V^\theta_2+V^\varepsilon_2 \partial_{x}\partial_{y}V^\theta_2)\partial_{x}\partial^{2}_{y} V_{2}^\varepsilon dxdy\\
       &\leq 2\varepsilon\left(\|\partial^2_{x}V^\varepsilon_1\|_{L^{2}_{xy}}\|V^\theta_1\|_{L^{\infty}_{xy}}+2\|\partial_{x}V^\varepsilon_1\|_{L^{2}_{xy}}\|\partial_{x}V^\theta_1\|_{L^{\infty}_{xy}}+\|V^\varepsilon_1\|_{L^{2}_{xy}}\|\partial^2_{x}V^\theta_1\|_{L^{\infty}_{xy}}\right)\|\partial_{x}\partial^{2}_{y}V^\varepsilon_1\|_{L^{2}_{xy}}\\
       &\quad+2\varepsilon\left(\|\partial^2_{x}V^\varepsilon_2\|_{L^{2}_{xy}}\|V^\theta_2\|_{L^{\infty}_{xy}}+2\|\partial_{x}V^\varepsilon_2\|_{L^{2}_{xy}}\|\partial_{x}V^\theta_2\|_{L^{\infty}_{xy}}+\|V^\varepsilon_2\|_{L^{2}_{xy}}\|\partial^2_{x}V^\theta_2\|_{L^{\infty}_{xy}}\right)\|\partial_{x}\partial^{2}_{y}V^\varepsilon_1\|_{L^{2}_{xy}}\\
       &\quad+2\varepsilon\left(\|\partial_{x}\partial_{y}V^\varepsilon_1\|_{L^{2}_{xy}}\|V^\theta_1\|_{L^{\infty}_{xy}}+\|\partial_{y}V^\varepsilon_1\|_{L^{2}_{xy}}\|\partial_{x}V^\theta_1\|_{L^{\infty}_{xy}}\right)\|\partial_{x}\partial^{2}_{y}V^\varepsilon_1\|_{L^{2}_{xy}}\\
       &\quad+2\varepsilon\left(\|\partial_{x}V^\varepsilon_1\|_{L^{2}_{xy}}\|\partial_{y}V^\theta_1\|_{L^{\infty}_{xy}}+\|V^\varepsilon_1\|_{L^{2}_{xy}}\|\partial_{x}\partial_{y}V^\theta_1\|_{L^{\infty}_{xy}}\right)\|\partial_{x}\partial^{2}_{y}V^\varepsilon_1\|_{L^{2}_{xy}}\\
       &\quad+2\varepsilon\left(\|\partial_{x}\partial_{y}V^\varepsilon_2\|_{L^{2}_{xy}}\|V^\theta_2\|_{L^{\infty}_{xy}}+\|\partial_{y}V^\varepsilon_2\|_{L^{2}_{xy}}\|\partial_{x}V^\theta_2\|_{L^{\infty}_{xy}}\right)\|\partial_{x}\partial^{2}_{y}V^\varepsilon_1\|_{L^{2}_{xy}}\\
       &\quad+2\varepsilon\left(\|\partial_{x}V^\varepsilon_2\|_{L^{2}_{xy}}\|\partial_{y}V^\theta_2\|_{L^{\infty}_{xy}}+\|V^\varepsilon_2\|_{L^{2}_{xy}}\|\partial_{x}\partial_{y}V^\theta_2\|_{L^{\infty}_{xy}}\right)\|\partial_{x}\partial^{2}_{y}V^\varepsilon_1\|_{L^{2}_{xy}}\\
       &\leq\frac{1}{8}\varepsilon\|\partial_{x}\partial^{2}_{y}V^\varepsilon\|^{2}_{L^{2}_{xy}}+C\|\partial_{x}\partial_{y}V^\varepsilon\|^{2}_{L^{2}_{xy}}+C\|\partial^2_{x}V^\varepsilon\|^{2}_{L^{2}_{xy}}+C\varepsilon^{\frac{1}{2}}.
    	\end{align*}
   Next, applying integration by parts and {$\partial_yH^\varepsilon|_{y=0}=0$,}
	\begin{equation*}
		\begin{split}{}
       \mathcal{M}_{6}&=
\int^\infty_0\int^\infty_{-\infty}(\partial^2_{x}\partial_{y} H^\varepsilon\partial_{x}\partial_{y} V^\varepsilon_1+\partial_{x}\partial^2_{y} H^\varepsilon\partial_{x}\partial_{y} V^\varepsilon_2) dxdy\\
&\leq\frac{1}{4}\|\partial^{2}_{x}\partial_{y}H^\varepsilon\|^{2}_{L^{2}_{xy}}+\frac{1}{4}\|\partial_{x}\partial^{2}_{y}H^\varepsilon\|^{2}_{L^{2}_{xy}}+C\|\partial_{x}\partial_{y} V^\varepsilon\|^{2}_{L^{2}_{xy}}.
\end{split}
	\end{equation*}
        Using integration by parts, $(\ref{trace theorem})$ and Lemma $\ref{lem g L2}$, we can obtain the following
\begin{align*}
       \mathcal{M}_{7}&=
\varepsilon^{-\frac{1}{2}}\int^\infty_0\int^\infty_{-\infty}\partial_{x}\partial_{y}g^\varepsilon_1\partial_{x}\partial_{y} V^\varepsilon_1 dxdy
-\varepsilon^{-\frac{1}{2}}\int^\infty_{-\infty}\partial_{x}g^\varepsilon_2(t,x,0)\partial_{x}\partial_{y} V^\varepsilon_2(t,x,0)dx\\
&\quad+
\varepsilon^{-\frac{1}{2}}\int^\infty_0\int^\infty_{-\infty}\partial_{x}\partial_{y}g^\varepsilon_2\partial_{x}\partial_{y} V^\varepsilon_2 dxdy\\
&\leq\varepsilon^{-\frac{1}{2}}(\|\partial_{x}\partial_{y}g^\varepsilon_1\|_{L^{2}_{xy}}\|\partial_{x}\partial_{y} V^\varepsilon_1\|_{L^{2}_{xy}}+\|\partial_{x}\partial_{y}g^\varepsilon_2\|_{L^{2}_{xy}}\|\partial_{x}\partial_{y} V^\varepsilon_2\|_{L^{2}_{xy}})
+\varepsilon^{-\frac{1}{2}}\|\partial_{x}g^\varepsilon\|_{L^{2}_{x}H^1_y}\|\partial_{x}\partial_{y} V^\varepsilon\|_{L^{2}_{x}H^1_y}\\
&\leq\varepsilon^{-\frac{1}{2}}\cdot\varepsilon^{\frac{1}{4}}(\|\partial_{x}\partial_{y} V^\varepsilon_1\|_{L^{2}_{xy}}+\|\partial_{x}\partial_{y} V^\varepsilon_2\|_{L^{2}_{xy}})
+C\varepsilon^{-\frac{1}{2}}\cdot\varepsilon^{\frac{1}{4}}\|\partial_{x}\partial_{y} V^\varepsilon_2\|_{L^{2}_{x}H^1_y}
\\
&\leq\frac{1}{8}\varepsilon\|\partial_{x}\partial^{2}_{y} V^\varepsilon_2\|^{2}_{L^{2}_{xy}}+
C\|\partial_{x}\partial_{y} V^\varepsilon\|^{2}_{L^{2}_{xy}}+C\varepsilon^{-\frac{3}{2}}.
	\end{align*}
    Finally, by combining $M_1$-$M_7$, we arrive at
       \begin{align}\label{frac{d}{dt} partial xy V varepsilon}
			&\frac{1}{2}\frac{d}{dt}\|\partial_{x}\partial_{y} V^\varepsilon\|^2_{L^2_{xy}}+\varepsilon\|\partial^{2}_{x}\partial_{y} V^\varepsilon\|^{2}_{L^{2}_{xy}}+\varepsilon\|\partial_{x}\partial^{2}_{y} V^\varepsilon\|^{2}_{L^{2}_{xy}}\notag\\
            &\leq\frac{1}{2}(\varepsilon\|\partial^{2}_{x}\partial_{y} V^\varepsilon\|^{2}_{L^{2}_{xy}}+\varepsilon\|\partial_{x}\partial^{2}_{y} V^\varepsilon\|^{2}_{L^{2}_{xy}})+\frac{1}{4}(\|\partial^{2}_{x}\partial_{y} H^\varepsilon\|^{2}_{L^{2}_{xy}}+\|\partial_{x}\partial^{2}_{y} H^\varepsilon\|^{2}_{L^{2}_{xy}})\\
            &\quad+C\|\partial_{x}\partial_{y}V^\varepsilon\|^2_{L^2_{xy}}+C\|\partial^2_{x} V^\varepsilon\|^2_{L^2_{xy}}+C\|\partial^2_{y} V^\varepsilon\|^2_{L^2_{xy}}+C\|\partial^2_{x} \omega^\varepsilon\|^2_{L^2_{xy}}+C\varepsilon^{-\frac{3}{2}}.\notag\notag
	\end{align}
         Combining $(\ref{frac{d}{dt} partial xy H varepsilon})$-$(\ref{frac{d}{dt} partial xy V varepsilon})$, we can obtain
        \begin{align*}
			&\frac{d}{dt}(\|\partial_{x}\partial_{y} H^\varepsilon\|^2_{L^2_{xy}}+\|\partial_{x}\partial_{y} V^\varepsilon\|^2_{L^2_{xy}})+(\|\partial_x\partial^2_y H^\varepsilon\|^2_{L^2_{xy}}+\|\partial^2_x\partial_y H^\varepsilon\|^2_{L^2_{xy}}\\
            &\quad+\varepsilon\|\partial_x\partial^2_y V^\varepsilon\|^2_{L^2_{xy}}+\varepsilon\|\partial^2_x\partial_y V^\varepsilon\|^2_{L^2_{xy}})\\
            &\leq C(\|\partial_{x}\partial_{y} H^\varepsilon\|^2_{L^2_{xy}}+\|\partial_{x}\partial_{y} V^\varepsilon\|^2_{L^2_{xy}})+C\|\partial^2_{x} V^\varepsilon\|^2_{L^2_{xy}}+C\|\partial^2_{y} V^\varepsilon\|^2_{L^2_{xy}}+C\|\partial^2_{x} \omega^\varepsilon\|^2_{L^2_{xy}}+C\varepsilon^{-\frac{3}{2}}.
	\end{align*}

        {By the Gronwall's inequality, then it} implies that
         \begin{equation*}
		\begin{split}{}
         &\|\partial_{x}\partial_{y} H^\varepsilon\|^2_{L^2_{xy}}+\|\partial_{x}\partial_{y} V^\varepsilon\|^2_{L^2_{xy}}+\int^T_0(\|\partial_x\partial^2_y H^\varepsilon\|^2_{L^2_{xy}}+\|\partial^2_x\partial_y H^\varepsilon\|^2_{L^2_{xy}}+\varepsilon\|\partial_x\partial^2_y V^\varepsilon\|^2_{L^2_{xy}}+\varepsilon\|\partial^2_x\partial_y V^\varepsilon\|^2_{L^2_{xy}})\\
            &\leq C\varepsilon^{-\frac{3}{2}},
         \end{split}
	\end{equation*}

    \textbf{Proof of Theorem $\ref{th L infty convergence}$.}  {We can prove $(\ref{HVU varepsilon L infty})$ holds by the results of Lemma $\ref{lem HVUx L2}$-Lemma $\ref{lem HVUxy L2}$.\\

    \section{Appendix: Well-posedness of Chemotaxis Navier-Stokes system}\label{5}
         In this section, we prove Theorem $\ref{th nvu regularity}$ in conormal Sobolev spaces. For this purpose, we now present several facts that will be used subsequently.\\
         (i) Commutator estimate:
        \begin{equation}\label{i}
		\begin{split}{}
        \|\partial_y\partial^\alpha u\|_{L^2_{xy}}\leq \|\partial^\alpha\partial_y u\|_{L^2_{xy}}+\|[\partial_y,\partial^\alpha] u\|_{L^2_{xy}},	
		\end{split}
	\end{equation}
    where $[\partial_y,\partial^\alpha]=\partial_y\partial^\alpha-\partial^\alpha\partial_y$.\\
    (ii) Embedding inequality:
        \begin{equation}\label{ii}
		\begin{split}{}
        \|\frac{u}{\psi}\|_{L^\infty_{xy}}\leq C(\delta)\|(u,\partial_y u)\|^\frac{1}{2}_{\widetilde{H}^1_{xy}}\|(\partial_yu,\partial_{yy} u)\|^\frac{1}{2}_{\widetilde{H}^1_{xy}}. 	
		\end{split}
	\end{equation}\\
        (iii) Hardy inequality: Let
        \begin{equation*}
        \varphi(y)=\left\{
        \begin{split}{}
        &1\quad{\rm for}\;y\leq\frac{1}{2},\\
        &0\quad{\rm for}\;y\geq1,
        \end{split}
        \right.
		\end{equation*}
        then
        \begin{equation}\label{iii}
		\begin{split}{}
        \|\frac{\partial^\alpha u}{\psi}\|_{L^2_{xy}}&\leq \|\varphi(y)\frac{\partial^\alpha u}{\psi}\|_{L^2_{xy}}+\|(1-\varphi(y))\frac{\partial^\alpha u}{\psi}\|_{L^2_{xy}}\\
        &\leq\|\varphi(y)\frac{y}{\psi}\|_{L^\infty_{xy}}\|\frac{\partial^\alpha u}{y}\|_{L^2_{xy}}+\|\frac{1-\varphi(y)}{\psi}\|_{L^\infty_{xy}}\|\partial^\alpha u\|_{L^2_{xy}}\\
        &\leq C(\delta)(\|\partial_y\partial^\alpha u\|_{L^2_{xy}}+\|\partial^\alpha u\|_{L^2_{xy}}).
		\end{split}
	\end{equation}
        \begin{lemma}\label{lem4.1}
         {Under the initial data assumptions $(n_{in},v_{in},u_{in})$ of Theorem  $\ref{th nvu regularity}$, the following estimate holds for the solution $(n^\varepsilon, v^\varepsilon, u^\varepsilon)$ to the system $(\ref{nvu equation})$-$(\ref{nvu condition})_1$ on $[0, T]$,
		\begin{equation}\label{Equ(4.1)}
        \begin{split}{}        &\|n^\varepsilon\|^2_{L^\infty_TL^2_{xy}}+\|v^\varepsilon\|^2_{L^\infty_TL^2_{xy}}+\|u^\varepsilon\|^2_{L^\infty_TL^2_{xy}}+\|\nabla n^\varepsilon\|^2_{L^2_TL^2_{xy}}+\varepsilon\|\nabla v^\varepsilon\|^2_{L^2_TL^2_{xy}}+\varepsilon\|\nabla u^\varepsilon\|^2_{L^2_TL^2_{xy}}\\
        &\leq C(\|n^\varepsilon\|^2_{L^\infty_TH^1_{xy}}+\| v^\varepsilon\|^2_{L^\infty_TH^1_{xy}}+\| u^\varepsilon\|^2_{L^\infty_TH^1_{xy}})^2.\\
        \end{split}
		\end{equation}
        where $C>0$ independent of $\varepsilon$, which be a constant depending only on $T$ and the initial data $(n_{in},v_{in},u_{in})$.
		}
	\end{lemma}
	\noindent{\bf{Proof.}} Taking the $L^2$ inner product of $(\ref{nvu equation})_{1}$ with $n^\varepsilon$ and integrating by parts, together with the boundary condition $\partial_yn^\varepsilon|_{y=0}=0$ and the divergence condition $\nabla\cdot u^\varepsilon=0$, then we obtain
    \begin{align*}
            \frac{1}{2}\frac{d}{dt}\| n^\varepsilon\|^2_{L^2_{xy}}+\|\nabla n^\varepsilon\|^2_{L^2_{xy}}&=-\int^\infty_0\int^\infty_{-\infty}u^\varepsilon\cdot\nabla n^\varepsilon n^\varepsilon dxdy+\int^\infty_0\int^\infty_{-\infty}\nabla\cdot(n^\varepsilon v^\varepsilon) n^\varepsilon dxdy\\
            &\leq \|n^\varepsilon\|_{L^4_{xy}} \|v^\varepsilon\|_{L^4_{xy}} \|\nabla n^\varepsilon\|_{L^2_{xy}}\\
            &\leq \frac12\|\nabla n^\varepsilon\|^2_{L^2_{xy}}+C\|n^\varepsilon\|^2_{H^1_{xy}} \|v^\varepsilon\|^2_{H^1_{xy}}.
        \end{align*}

        Similarly, we take the $L^2$ inner product of $(\ref{nvu equation})_{2}$ with $v^\varepsilon$ and use the boundary condition $v^\varepsilon_2=0$, which yields
    \begin{align*}
            &\frac{1}{2}\frac{d}{dt}\| v^\varepsilon\|^2_{L^2_{xy}}+\varepsilon\|\nabla v^\varepsilon\|^2_{L^2_{xy}}\\
            &=-\int^\infty_0\int^\infty_{-\infty}\nabla(u^\varepsilon\cdot v^\varepsilon)\cdot v^\varepsilon dxdy-\varepsilon\int^\infty_0\int^\infty_{-\infty}\nabla(v^\varepsilon\cdot v^\varepsilon)\cdot v^\varepsilon dxdy+\int^\infty_0\int^\infty_{-\infty}\nabla n^\varepsilon\cdot v^\varepsilon dxdy\\
            &\leq \|\nabla u^\varepsilon\|_{L^2_{xy}}\|v^\varepsilon\|^2_{L^4_{xy}}+\| u^\varepsilon\|_{L^4_{xy}}\|\nabla v^\varepsilon\|_{L^2_{xy}}\|v^\varepsilon\|_{L^4_{xy}}+2\varepsilon\|\nabla v^\varepsilon\|_{L^2_{xy}}\|v^\varepsilon\|^2_{L^4_{xy}} +\|\nabla n^\varepsilon\|_{L^2_{xy}}\|v^\varepsilon\|_{L^2_{xy}}\\
            &\leq \frac 12\varepsilon\|\nabla v^\varepsilon\|^2_{L^2_{xy}}+C(\|u^\varepsilon\|^2_{H^1_{xy}}\|v^\varepsilon\|^2_{H^1_{xy}}+\|v^\varepsilon\|^4_{H^1_{xy}} +\|\nabla n^\varepsilon\|^2_{H^1_{xy}}).
        \end{align*}

        Furthermore, we take the $L^2$ inner product of $(\ref{nvu equation})_{3}$ with $u^\varepsilon$. Using $u^\varepsilon_2|_{y=0}=0$ and $\nabla\cdot u^\varepsilon=0$, we infer that
     \begin{align*}
            &\frac{1}{2}\frac{d}{dt}\|u^\varepsilon\|^2_{L^2_{xy}}+\varepsilon\|\nabla u^\varepsilon\|^2_{L^2_{xy}}\\
            &=-\int^\infty_0\int^\infty_{-\infty}u^\varepsilon\cdot\nabla
        u^\varepsilon\cdot u^\varepsilon dxdy-\int^\infty_0\int^\infty_{-\infty}\nabla p^\varepsilon\cdot u^\varepsilon dxdy+\int^\infty_0\int^\infty_{-\infty} n^\varepsilon e_2\cdot u^\varepsilon dxdy\\
        &\leq \|n^\varepsilon\|^2_{L^2_{xy}}+\|u^\varepsilon\|^2_{L^2_{xy}}.
        \end{align*}
    Integrating the above  $L^2_{xy}$ energy estimates and using Gronwall's inequality, we establish (\ref{Equ(4.1)}).\\

    In the next Lemmas, we shall perform energy estimates in conormal Sobolev spaces to derive higher-order derivative regularity for the system $(\ref{nvu equation})$. We first estimate $\|\nabla n^\varepsilon\|^2_{\widetilde{H}^m_{xy}}$ as follows:

        \begin{lemma}\label{lem4.2}
         {Assume that the initial data $(n_{in},v_{in},u_{in})$ satisfy the hypotheses of Theorem $\ref{th nvu regularity}$ and the integer $m \geq 5$. Then there exist constants $1>\delta > 0$ and $\varepsilon_0>0$ such that for any $0<\varepsilon\leq\varepsilon_0$ the solution $n^\varepsilon$ satisfies
		\begin{align}\label{nabla n}
        &\frac{d}{dt}\sum_{|\alpha|\leq m}\|\partial^\alpha\nabla n^\varepsilon\|^2_{L^2_{xy}}+2(1-C\delta)\sum_{|\alpha|\leq m}\|\partial^\alpha\nabla^2 n^\varepsilon\|^2_{L^2_{xy}}\notag\\
        &\leq C_0(\delta)\Big{(}1+\|u^\varepsilon\|^2_{L^2_{xy}}+\sum_{|\alpha|\leq m}\|\partial^\alpha v^\varepsilon\|^2_{L^2_{xy}}+\sum_{|\alpha|\leq m}\|\partial^\alpha\nabla n^\varepsilon\|^2_{L^2_{xy}}+\sum_{|\alpha|\leq m}\|\partial^\alpha\nabla v^\varepsilon\|^2_{L^2_{xy}}+\sum_{|\alpha|\leq m+1}\|\partial^{\alpha}\omega\|^2_{L^2_{xy}}\Big{)}\notag\\
        &\quad\times\sum_{|\alpha|\leq m}\|\partial^\alpha\nabla n^\varepsilon\|^2_{L^2_{xy}}+C_0(\delta)\Big{(}1+\sum_{|\alpha|\leq m}\|\partial^{\alpha}n^\varepsilon\|^2_{L^2_{xy}}+\sum_{|\alpha|\leq m}\|\partial^{\alpha}\nabla n^\varepsilon\|^2_{L^2_{xy}}\Big{)}\sum_{|\alpha|\leq m}\|\partial^\alpha \nabla v^\varepsilon\|^2_{L^2_{xy}}\\
        &\quad+C(\delta)\sum_{|\alpha|\leq m+1}\|\partial^\alpha\omega\|^2_{L^2_{xy}}+8\nu\sum_{|\alpha|\leq m}\|\partial^\alpha\nabla^2 n^\varepsilon\|^2_{L^2_{xy}},\notag\notag
		\end{align}
        where positive constant $C$ is dependent of $T$.
		}
	\end{lemma}
	\noindent{\bf{Proof.}} For $|\alpha|\leq m$, we apply $\partial^\alpha\nabla$ to $(\ref{nvu equation})_1$ and multiply by $\partial^\alpha\nabla n^\varepsilon$, which yields
    \begin{align*}
        &\sum_{|\alpha|\leq m}\langle \partial^\alpha\nabla\partial_tn^\varepsilon, \partial^\alpha\nabla n^\varepsilon \rangle-\sum_{|\alpha|\leq m}\langle \partial^\alpha\nabla\Delta n^\varepsilon, \partial^\alpha\nabla n^\varepsilon \rangle\\
        &\leq\sum_{|\alpha|\leq m}\Big{|}\langle \partial^\alpha\nabla(u^\varepsilon\cdot\nabla n^\varepsilon), \partial^\alpha\nabla n^\varepsilon \rangle\Big{|}+\sum_{|\alpha|\leq m}\Big{|}\langle \partial^\alpha\nabla(\nabla\cdot (n^\varepsilon v^\varepsilon)), \partial^\alpha\nabla n^\varepsilon \rangle\Big{|}\\
        &=:S_1+S_2.
		\end{align*}
        From the definition in $(\ref{psi})$, we have $\psi(y)|_{y=0}=0$. For the Laplacian terms, combining integration by parts, Remark 4.1(i) and $\partial_yn^\varepsilon|_{y=0}=0$, we derive
        \begin{align}\label{Delta n}
        &-\sum_{|\alpha|\leq m}\langle \partial^\alpha\nabla\Delta n^\varepsilon, \partial^\alpha\nabla n^\varepsilon \rangle\notag\\
        &=-\sum_{|\alpha|\leq m}\langle \partial_x\partial^\alpha\nabla\partial_x n^\varepsilon, \partial^\alpha\nabla n^\varepsilon \rangle-\sum_{|\alpha|\leq m}\langle \partial_y\partial^\alpha\nabla\partial_y n^\varepsilon, \partial^\alpha\nabla n^\varepsilon \rangle-\sum_{|\alpha|\leq m}\langle [\partial^\alpha,\partial_y]\nabla\partial_y n^\varepsilon, \partial^\alpha\nabla n^\varepsilon \rangle\notag\\
        &=\sum_{|\alpha|\leq m}\|\partial^\alpha\nabla^2 n^\varepsilon\|^2_{L^2_{xy}}+\sum_{|\alpha|\leq m}\langle \partial^\alpha\nabla\partial_y n^\varepsilon, [\partial_y,\partial^\alpha]\nabla n^\varepsilon \rangle-\sum_{|\alpha|\leq m}\langle [\partial^\alpha,\partial_y]\nabla\partial_y n^\varepsilon, \partial^\alpha\nabla n^\varepsilon \rangle\\
        &=\sum_{|\alpha|\leq m}\|\partial^\alpha\nabla^2 n^\varepsilon\|^2_{L^2_{xy}}+\sum_{|\alpha|\leq m}\alpha_2\langle \partial^\alpha\nabla\partial_y n^\varepsilon, \psi'\partial^{\alpha-(0,1)}\partial_y\nabla n^\varepsilon \rangle\notag\\
        &\quad+\sum_{|\alpha|\leq m}\alpha_2\langle \psi'\partial^{\alpha-(0,1)}\partial_y\nabla\partial_y n^\varepsilon, \partial^\alpha\nabla n^\varepsilon \rangle\notag\\
        &\geq\sum_{|\alpha|\leq m}\|\partial^\alpha\nabla^2 n^\varepsilon\|^2_{L^2_{xy}}-C\delta\Big{(}\sum_{|\alpha|\leq m}\|\partial^\alpha\nabla\partial_y n^\varepsilon\|^2_{L^2_{xy}}+\|\partial^{\alpha-(0,1)}\partial_y\nabla n^\varepsilon\|^2_{L^2_{xy}}\Big{)}-C\sum_{|\alpha|\leq m}\|\partial^\alpha\nabla n^\varepsilon\|^2_{L^2_{xy}}\notag\\
        &\geq(1-C\delta)\sum_{|\alpha|\leq m}\|\partial^\alpha\nabla^2 n^\varepsilon\|^2_{L^2_{xy}}-C\sum_{|\alpha|\leq m}\|\partial^\alpha\nabla n^\varepsilon\|^2_{L^2_{xy}}.\notag
		\end{align}
        For convenience, we decompose $S_1$ into two terms as follows:
        \begin{align*}
        \begin{split}{}
        S_1&=\sum_{|\alpha|\leq m}\Big{|}\langle \partial^\alpha (\nabla u^\varepsilon\cdot\nabla n^\varepsilon+(u^\varepsilon\cdot\nabla)(\nabla n^\varepsilon)), \partial^\alpha\nabla n^\varepsilon \rangle\Big{|}\\
        &=\sum_{|\alpha|\leq m}\sum_{\beta\leq\alpha}C^\beta_\alpha\Big(\Big{|}\langle \partial^\beta \nabla u^\varepsilon\cdot\partial^{\alpha-\beta}\nabla n^\varepsilon, \partial^\alpha\nabla n^\varepsilon \rangle\Big{|}+\Big{|}\langle \partial^\beta u^\varepsilon\cdot\partial^{\alpha-\beta}\nabla(\nabla n^\varepsilon), \partial^\alpha\nabla n^\varepsilon \rangle\Big{|}\Big)\\
        &=:S_{11}+S_{12},
        \end{split}
		\end{align*}
        where $S_{11}$ is estimated as follows
        \begin{align*}
        S_{11}&=\sum_{|\alpha|\leq m}\sum_{\beta\leq\alpha}C^\beta_\alpha\Big{|}\langle \partial^\beta \nabla u^\varepsilon\partial^{\alpha-\beta}\nabla n^\varepsilon, \partial^\alpha\nabla n^\varepsilon \rangle\Big{|}\\
        &=\sum_{|\alpha|\leq m}\Big{(}\Big{|}\langle \nabla u^\varepsilon\partial^{\alpha}\nabla n^\varepsilon, \partial^\alpha\nabla n^\varepsilon \rangle\Big{|}+\sum_{1\leq|\beta|\leq m-1}C^\beta_{m-1}\Big{|}\langle \partial^\beta\nabla u^\varepsilon\partial^{\alpha-\beta}\nabla n^\varepsilon, \partial^\alpha\nabla n^\varepsilon \rangle\Big{|}\\
        &\quad+\Big{|}\langle \partial^\alpha\nabla u^\varepsilon\nabla n^\varepsilon, \partial^\alpha\nabla n^\varepsilon \rangle\Big{|}\Big{)}\\
        &\leq \sum_{|\alpha|\leq m}\Big{(}\|\psi\nabla u^\varepsilon\|_{L^\infty_{xy}}\|\partial^\alpha\nabla n^\varepsilon\|_{L^2_{xy}}\|\frac{\partial^\alpha\nabla n^\varepsilon}{\psi}\|_{L^2_{xy}}\\
        &\quad+\sum_{1\leq|\beta|\leq m-1}C^\beta_{m-1}\|\psi\partial^\beta\nabla u^\varepsilon\|_{L^\infty_{xy}}\|\partial^{\alpha-\beta}\nabla n^\varepsilon\|_{L^2_{xy}}\|\frac{\partial^\alpha\nabla n^\varepsilon}{\psi}\|_{L^2_{xy}}\\
        &\quad+\|\partial^\alpha\nabla u^\varepsilon\|_{L^2_{xy}}\|\psi\nabla n^\varepsilon\|_{L^\infty_{xy}}\|\frac{\partial^\alpha\nabla n^\varepsilon}{\psi}\|_{L^2_{xy}}\Big{)}.
		\end{align*}
        To estimate $\|\partial^\beta\nabla u^\varepsilon\|_{L^\infty_{xy}}$, we note that the Sobolev embedding inequality gives
        \begin{align}\label{infty}
        \|\partial^\beta u^\varepsilon\|_{L^\infty_{xy}}\leq C\|\partial^\beta u^\varepsilon\|^\frac{1}{2}_{L^2_{xy}}\|\partial_y\partial^\beta u^\varepsilon\|^\frac{1}{2}_{L^2_{xy}}+C\|\partial_x\partial^\beta u^\varepsilon\|^\frac{1}{2}_{L^2_{xy}}\|\partial_x\partial_y\partial^\beta u^\varepsilon\|^\frac{1}{2}_{L^2_{xy}}.
	\end{align}
        Thus, by Young's inequality, we conclude that
        \begin{align*}
        S_{11}&\leq \frac{C(\delta)}{\nu}\Big{(}\sum_{|\alpha|\leq2}\|\partial^\alpha\nabla u^\varepsilon\|^2_{L^2_{xy}}+\sum_{1\leq|\beta|\leq m-1}\sum_{|\gamma|\leq2}\|\partial^{\beta+\gamma}\nabla u^\varepsilon\|^2_{L^2_{xy}}+1\Big{)}\sum_{|\alpha|\leq m}\|\partial^\alpha\nabla n^\varepsilon\|^2_{L^2_{xy}}\\
        &\quad+\frac{\nu}{2}\sum_{|\alpha
        |\leq m}\|\partial^{\alpha} \nabla^2 n^\varepsilon\|^2_{L^2_{xy}}+\frac{C(\delta)}{\nu}\sum_{|\alpha|\leq 2}\|\partial^\alpha\nabla n^\varepsilon\|^2_{L^2_{xy}}\sum_{|\alpha|\leq m}\|\partial^\alpha\nabla u^\varepsilon\|^2_{L^2_{xy}},
		\end{align*}
        where we used the Hardy inequality from (\ref{iii}),
        \begin{align*}
          \|\frac{\partial^\alpha\nabla n^\varepsilon}{\psi}\|_{L^2_{xy}}\leq C(\delta)(\|\partial_y\partial^\alpha\nabla n^\varepsilon\|_{L^2_{xy}}+\|\partial^\alpha\nabla n^\varepsilon\|_{L^2_{xy}})\leq C(\delta)(\|\partial^\alpha\nabla^2 n^\varepsilon\|_{L^2_{xy}}+\|\partial^\alpha\nabla n^\varepsilon\|_{L^2_{xy}}).
        \end{align*}
        Similarly, we estimate $S_{12}$ using inequality $(\ref{infty})$, which yields
        \begin{align*}
        S_{12}&=\sum_{|\alpha|\leq m}\sum_{\beta\leq\alpha}C^\beta_\alpha\Big{|}\langle \partial^\beta u^\varepsilon\cdot\partial^{\alpha-\beta}\nabla(\nabla n^\varepsilon), \partial^\alpha\nabla n^\varepsilon \rangle\Big{|}\\
        &=\sum_{|\alpha|\leq m}\Big{(}\Big{|}\langle u^\varepsilon\cdot\partial^{\alpha}\nabla(\nabla n^\varepsilon), \partial^\alpha\nabla n^\varepsilon \rangle\Big{|}+\sum_{1\leq|\beta|\leq m-1}C^\beta_{m-1}\Big{|}\langle \partial^\beta u^\varepsilon\cdot\partial^{\alpha-\beta}\nabla(\nabla n^\varepsilon), \partial^\alpha\nabla n^\varepsilon \rangle\Big{|}\\
        &\quad+\Big{|}\langle \partial^\alpha u^\varepsilon\cdot\nabla(\nabla n^\varepsilon), \partial^\alpha\nabla n^\varepsilon \rangle\Big{|}\Big{)}\\
        &\leq \sum_{|\alpha|\leq m}\Big{(}\|u^\varepsilon\|_{L^\infty_{xy}}\|\partial^{\alpha}\nabla^2 n^\varepsilon\|_{L^2_{xy}}\|\partial^\alpha\nabla n^\varepsilon\|_{L^2_{xy}}+\sum_{1\leq|\beta|\leq m-1}C^\beta_{m-1}\|\partial^\beta u^\varepsilon\|_{L^\infty_{xy}}\\
        &\quad\times\|\partial^{\alpha-\beta}\nabla (\nabla n^\varepsilon)\|_{L^2_{xy}}\|\partial^\alpha\nabla n^\varepsilon\|_{L^2_{xy}}+\|\partial^\alpha u^\varepsilon\|_{L^\infty_{xy}}\|\nabla^2 n^\varepsilon\|_{L^2_{xy}}\|\partial^\alpha\nabla n^\varepsilon\|_{L^2_{xy}}\Big{)}\\
        &\leq \frac{C}{\nu}\Big{(}\|u^\varepsilon\|^2_{L^2_{xy}}+\sum_{|\alpha|\leq1}\|\partial^\alpha\nabla u^\varepsilon\|^2_{L^2_{xy}}+\sum_{1\leq|\beta|\leq m-1}\|\partial^{\beta+(1,0)}\nabla u^\varepsilon\|^2_{L^2_{xy}}\\
        &\quad+\sum_{|\alpha|\leq m}\|\partial^{\alpha+(1,0)}\nabla u^\varepsilon\|^2_{L^2_{xy}}\Big{)}\sum_{|\alpha|\leq m}\|\partial^\alpha\nabla n^\varepsilon\|^2_{L^2_{xy}}+\frac{\nu}{2}\sum_{|\alpha|\leq m}\|\partial^{\alpha} \nabla^2 n^\varepsilon\|^2_{L^2_{xy}}.
		\end{align*}
        By the same argument as for $S_1$, we expand $S_2$ as follows:
        \begin{align*}
        S_2&=\sum_{|\alpha|\leq m}\Big{(}\Big{|}\langle \partial^\alpha \nabla(\nabla n^\varepsilon\cdot v^\varepsilon), \partial^\alpha\nabla n^\varepsilon \rangle\Big{|}+\Big{|}\langle \partial^\alpha \nabla(n^\varepsilon \partial_xv^\varepsilon_1), \partial^\alpha\nabla n^\varepsilon \rangle\Big{|}+\Big{|}\langle \partial^\alpha\nabla (n^\varepsilon \partial_yv^\varepsilon_2), \partial^\alpha\nabla n^\varepsilon \rangle\Big{|}\Big{)}\\
        &=:S_{21}+S_{22}+S_{23}.
		\end{align*}
        The term $S_{21}$ can be further expanded by
        \begin{equation*}
        \begin{split}{}
        S_{21}&=\sum_{|\alpha|\leq m}\Big{(}\Big{|}\langle \partial^\alpha(\nabla^2 n^\varepsilon\cdot v^\varepsilon), \partial^\alpha\nabla n^\varepsilon \rangle\Big{|}+\Big{|}\langle \partial^\alpha (\nabla n^\varepsilon\cdot(\nabla v^\varepsilon)), \partial^\alpha\nabla n^\varepsilon \rangle\Big{|}\Big{)}=S_{211}+S_{212}.\\
        \end{split}
		\end{equation*}
        For $S_{211}$, we apply inequality $(\ref{infty})$ and get
        \begin{align*}
        S_{211}&=\sum_{|\alpha|\leq m}\Big{(}\Big{|}\langle \nabla^2 n^\varepsilon\cdot\partial^\alpha v^\varepsilon, \partial^\alpha\nabla n^\varepsilon \rangle\Big{|}+\sum_{1\leq|\beta|\leq m-1}C^\beta_{m-1}\Big{|}\langle \partial^\beta\nabla^2 n^\varepsilon\cdot\partial^{\alpha-\beta} v^\varepsilon, \partial^\alpha\nabla n^\varepsilon \rangle\Big{|}\\
        &\quad+\Big{|}\langle \partial^\alpha \nabla^2 n^\varepsilon\cdot v^\varepsilon, \partial^\alpha\nabla n^\varepsilon \rangle\Big{|}\Big{)}\\
        &\leq\sum_{|\alpha|\leq m}\Big{(}\|\psi\nabla^2 n^\varepsilon\|_{L^\infty_{xy}}\|\frac{\partial^\alpha v^\varepsilon}{\psi}\|_{L^2_{xy}}\|\partial^\alpha\nabla n^\varepsilon\|_{L^2_{xy}}+\sum_{1\leq|\beta|\leq m-1}C^\beta_{m-1}\|\partial^\beta\nabla^2 n^\varepsilon\|_{L^2_{xy}}\\
        &\quad\times\|\partial^{\alpha-\beta}v^\varepsilon\|_{L^\infty_{xy}}\|\partial^\alpha\nabla n^\varepsilon\|_{L^2_{xy}}+\|\partial^\alpha\nabla^2 n^\varepsilon\|_{L^2_{xy}}\|v^\varepsilon\|_{L^\infty_{xy}}\|\partial^\alpha\nabla n^\varepsilon\|_{L^2_{xy}}\Big{)}\\
        &\leq\sum_{|\alpha|\leq m}C\Big{[}\Big(\sum_{|\alpha|\leq 2}\|\partial^\alpha\nabla^2 n^\varepsilon\|_{L^2_{xy}}\Big)\Big(\|\partial_y\partial^\alpha v^\varepsilon\|_{L^2_{xy}}+\|\partial^\alpha v^\varepsilon\|_{L^2_{xy}}\Big)\|\partial^\alpha\nabla n^\varepsilon\|_{L^2_{xy}}\\
        &\quad+\sum_{1\leq|\beta|\leq m-1}\|\partial^\beta\nabla^2 n^\varepsilon\|_{L^2_{xy}}\|\partial^{\alpha-\beta+(1,0)}\nabla v^\varepsilon\|_{L^2_{xy}}\|\partial^\alpha\nabla n^\varepsilon\|_{L^2_{xy}}\\
        &\quad+\|\partial^\alpha\nabla^2 n^\varepsilon\|_{L^2_{xy}}\Big(\|v^\varepsilon\|_{L^2_{xy}}+\sum_{|\alpha|\leq 1}\|\partial^\alpha\nabla v^\varepsilon\|_{L^2_{xy}}\Big)\|\partial^\alpha\nabla n^\varepsilon\|_{L^2_{xy}}\Big{]}\\
        &\leq \frac{C(\delta)}{\nu}\Big{(}\sum_{|\alpha|\leq m}\|\partial^\alpha \nabla v^\varepsilon\|^2_{L^2_{xy}}+\sum_{|\alpha|\leq m}\|\partial^\alpha v^\varepsilon\|^2_{L^2_{xy}}+\sum_{1\leq|\beta|\leq m-1}\|\partial^{\alpha-\beta+(1,0)}\nabla v^\varepsilon\|^2_{L^2_{xy}}\\
        &\quad+\|v^\varepsilon\|^2_{L^2_{xy}}+\sum_{|\alpha|\leq 1}\|\partial^{\alpha}\nabla v^\varepsilon\|^2_{L^2_{xy}}\Big{)}\sum_{|\alpha|\leq m}\|\partial^\alpha\nabla n^\varepsilon\|^2_{L^2_{xy}}+\frac{\nu}{2}\sum_{|\alpha|\leq m}\|\partial^{\alpha} \nabla^2 n^\varepsilon\|^2_{L^2_{xy}}.
		\end{align*}
        The estimate for $S_{212}$ is as follows:
        \begin{align*}
        S_{212}&=\sum_{|\alpha|\leq m}\Big{(}\Big{|}\langle  \nabla n^\varepsilon\cdot(\partial^\alpha\nabla v^\varepsilon), \partial^\alpha\nabla n^\varepsilon \rangle\Big{|}+\sum_{1\leq|\beta|\leq m-1}C^\beta_{m-1}\Big{|}\langle \partial^\beta\nabla n^\varepsilon\cdot (\partial^{\alpha-\beta}\nabla v^\varepsilon), \partial^\alpha\nabla n^\varepsilon \rangle\Big{|}\\
        &\quad+\Big{|}\langle \partial^\alpha \nabla n^\varepsilon\cdot(\nabla v^\varepsilon), \partial^\alpha\nabla n^\varepsilon \rangle\Big{|}\Big{)}\\
        &\leq\sum_{|\alpha|\leq m}\Big{(}\|\psi\nabla n^\varepsilon\|_{L^\infty_{xy}}\|\partial^\alpha\nabla v^\varepsilon\|_{L^2_{xy}}\|\frac{\partial^\alpha\nabla n^\varepsilon}{\psi}\|_{L^2_{xy}}+\sum_{1\leq|\beta|\leq m-1}C^\beta_{m-1}\|\partial^\beta \nabla n^\varepsilon\|_{L^\infty_{xy}}\\
        &\quad\times\|\partial^{\alpha-\beta}\nabla v^\varepsilon\|_{L^2_{xy}}\|\partial^\alpha\nabla n^\varepsilon\|_{L^2_{xy}}+\|\partial^\alpha\nabla n^\varepsilon\|_{L^2_{xy}}\|\psi\nabla v^\varepsilon\|_{L^\infty_{xy}}\|\frac{\partial^\alpha\nabla n^\varepsilon}{\psi}\|_{L^2_{xy}}\Big{)}\\
        &\leq \frac{\nu}{2}\sum_{|\alpha|\leq m}\|\partial^{\alpha} \nabla^2 n^\varepsilon\|^2_{L^2_{xy}}+\frac{C(\delta)}{\nu}\sum_{|\alpha|\leq 2}\|\partial^\alpha\nabla n^\varepsilon\|^2_{L^2_{xy}}\sum_{|\alpha|\leq m}\|\partial^\alpha\nabla v^\varepsilon\|^2_{L^2_{xy}}\\
        &\quad+\frac{C(\delta)}{\nu}\Big{(}\sum_{|\alpha|\leq m}\|\partial^{\alpha}\nabla v^\varepsilon\|^2_{L^2_{xy}}+\sum_{|\alpha|\leq 2}\|\partial^{\alpha}\nabla v^\varepsilon\|^2_{L^2_{xy}}+1\Big{)}\sum_{|\alpha|\leq m}\|\partial^\alpha\nabla n^\varepsilon\|^2_{L^2_{xy}},
		\end{align*}
        where by applying the embedding inequality $(\ref{infty})$, we have used the fact that
        \begin{align*}
          \sum_{1\leq|\beta|\leq m-1}\|\partial^\beta \nabla n^\varepsilon\|_{L^\infty_{xy}}&\leq C\sum_{1\leq|\beta|\leq m-1}(\|\partial^{\beta+(1,0)} \nabla^2 n^\varepsilon\|_{L^2_{xy}}+\|\partial^{\beta} \nabla n^\varepsilon\|_{L^2_{xy}})\\&\leq C\sum_{1\leq|\alpha|\leq m}(\|\partial^{\alpha} \nabla^2 n^\varepsilon\|_{L^2_{xy}}+\|\partial^{\alpha} \nabla n^\varepsilon\|_{L^2_{xy}}).
        \end{align*}
        Similar to the estimates for $S_{21}$, by integration by parts in the $x$-direction for $S_{22}$, it follows
        \begin{align*}
        S_{22}&=\sum_{|\alpha|\leq m}\Big{(}\Big{|}\langle \partial^\alpha(\nabla n^\varepsilon \partial_xv^\varepsilon_1), \partial^\alpha\nabla n^\varepsilon \rangle\Big{|}+\Big{|}\langle \partial^\alpha(n^\varepsilon \nabla\partial_xv^\varepsilon_1), \partial^\alpha\nabla n^\varepsilon \rangle\Big{|}\Big{)}\\
        &\leq\sum_{|\alpha|\leq m}\Big{(}\Big{|}\int^\infty_0\int^\infty_{-\infty}\nabla n^\varepsilon\partial^\alpha\partial_xv^\varepsilon_1
        \partial^\alpha\nabla n^\varepsilon dxdy\Big{|}\\
        &\quad
        +\Big{|}\int^\infty_0\int^\infty_{-\infty}\sum_{1\leq|\beta|\leq m-1}C^\beta_{m-1}\partial^\beta\nabla n^\varepsilon\partial^{\alpha-\beta} \partial_xv^\varepsilon_1\partial^\alpha\nabla n^\varepsilon dxdy\Big{|}
        \\
        &\quad+\Big{|}\int^\infty_0\int^\infty_{-\infty}\partial^\alpha\nabla n^\varepsilon\partial_xv^\varepsilon_1\partial^\alpha\nabla n^\varepsilon dxdy\Big{|}+\Big{|}\int^\infty_0\int^\infty_{-\infty} n^\varepsilon\partial^\alpha\nabla\partial_xv^\varepsilon_1\partial^\alpha\nabla n^\varepsilon dxdy\Big{|}\\
        &\quad+\Big{|}\int^\infty_0\int^\infty_{-\infty}\sum_{1\leq|\beta|\leq m-1}C^\beta_{m-1}\partial^\beta n^\varepsilon\partial^{\alpha-\beta}\nabla\partial_xv^\varepsilon_1\partial^\alpha\nabla n^\varepsilon dxdy\Big{|}\\
        &\quad+\Big{|}\int^\infty_0\int^\infty_{-\infty}\partial^\alpha n^\varepsilon\nabla\partial_xv^\varepsilon_1\partial^\alpha\nabla n^\varepsilon dxdy\Big{|}\Big{)}
         \\
        &\leq\sum_{|\alpha|\leq m}\Big{(}\|\psi\nabla n^\varepsilon\|_{L^\infty_{xy}}\|\partial^\alpha\partial_x v^\varepsilon_1\|_{L^2_{xy}}\|\frac{\partial^\alpha\nabla n^\varepsilon}{\psi}\|_{L^2_{xy}}\\
        &\quad+\sum_{1\leq|\beta|\leq m-1}C^\beta_{m-1}\Big{[}\|\partial^\beta\nabla\partial_x n^\varepsilon\|_{L^2_{xy}}\|\partial^{\alpha-\beta} v^\varepsilon_1\|_{L^\infty_{xy}}\|\partial^\alpha\nabla n^\varepsilon\|_{L^2_{xy}}\\
        &\quad+\|\partial^\beta\nabla n^\varepsilon\|_{L^2_{xy}}\|\partial^{\alpha-\beta} v^\varepsilon_1\|_{L^\infty_{xy}}\|\partial^\alpha\nabla \partial_xn^\varepsilon\|_{L^2_{xy}}\Big{]}+\|\partial^\alpha\nabla n^\varepsilon\|_{L^2_{xy}}\|\partial_x v^\varepsilon_1\|_{L^\infty_{xy}}\|\partial^\alpha\nabla n^\varepsilon\|_{L^2_{xy}}\\
        &\quad+\Big{[}\|\partial_x n^\varepsilon\|_{L^\infty_{xy}}\|\partial^\alpha\nabla v^\varepsilon_1\|_{L^2_{xy}}\|\partial^\alpha\nabla n^\varepsilon\|_{L^2_{xy}}+\|n^\varepsilon\|_{L^\infty_{xy}}\|\partial^\alpha\nabla v^\varepsilon_1\|_{L^2_{xy}}\|\partial^\alpha\nabla\partial_x n^\varepsilon\|_{L^2_{xy}}\Big{]}\\
        &\quad+\sum_{1\leq|\beta|\leq m-1}C^\beta_{m-1}\|\partial^\beta n^\varepsilon\|_{L^\infty_{xy}}\|\partial^{\alpha-\beta}\nabla\partial_x v^\varepsilon_1\|_{L^2_{xy}}\|\partial^\alpha\nabla n^\varepsilon\|_{L^2_{xy}}\\
        &\quad+\|\frac{\partial^\alpha n^\varepsilon}{\psi}\|_{L^2_{xy}}\|\psi\partial_x \nabla v^\varepsilon_1\|_{L^\infty_{xy}}\|\partial^\alpha\nabla n^\varepsilon\|_{L^2_{xy}}\Big{)}\\
     &\leq C\Big{(}\|n^\varepsilon\|^2_{L^2_{xy}}+\sum_{|\alpha|\leq 2}\|\partial^{\alpha}\nabla n^\varepsilon\|^2_{L^2_{xy}}+\sum_{1\leq|\beta|\leq m-1}\|\partial^{\beta+(1,0)}\nabla n^\varepsilon\|^2_{L^2_{xy}}\Big{)}\sum_{|\alpha|\leq m}\|\partial^\alpha \nabla v^\varepsilon_1\|^2_{L^2_{xy}}\\
        &\quad+\frac{C(\delta)}{\nu}\Big{(}\sum_{1\leq|\beta|\leq m-1}\|\partial^{\alpha-\beta+(1,0)}\nabla v^\varepsilon_1\|^2_{L^2_{xy}}+\sum_{|\alpha|\leq 3}\|\partial^{\alpha}\nabla v^\varepsilon_1\|^2_{L^2_{xy}}\Big{)}\sum_{|\alpha|\leq m}\|\partial^\alpha\nabla n^\varepsilon\|^2_{L^2_{xy}}\\
        &\quad+\frac{\nu}{2}\sum_{|\alpha|\leq m}\|\partial^{\alpha} \nabla^2 n^\varepsilon\|^2_{L^2_{xy}}+C\sum_{|\alpha|\leq m}\|\partial^{\alpha} \nabla n^\varepsilon\|^2_{L^2_{xy}}+C\sum_{|\alpha|\leq m}\|\partial^\alpha n^\varepsilon\|^2_{L^2_{xy}}.
		\end{align*}
        For the term $S_{23}$, we perform the corresponding estimate:
        \begin{equation*}
        \begin{split}{}
        S_{23}&=\sum_{|\alpha|\leq m}\Big{(}\Big{|}\langle \partial^\alpha(\nabla n^\varepsilon \partial_yv^\varepsilon_2), \partial^\alpha\nabla n^\varepsilon \rangle\Big{|}+\Big{|}\langle \partial^\alpha (n^\varepsilon\partial_y\nabla v^\varepsilon_2), \partial^\alpha\nabla n^\varepsilon \rangle\Big{|}\Big{)}=:S_{231}+S_{232}.\\
        \end{split}
		\end{equation*}
        By an analogous argument as estimating the term $S_{212}$, one deduces
        \begin{align*}
        S_{231}&\leq \sum_{|\alpha|\leq m}\Big{(}\|\psi\nabla n^\varepsilon\|_{L^\infty_{xy}}\|\partial^\alpha\partial_yv^\varepsilon_2\|_{L^2_{xy}}\|\frac{\partial^\alpha\nabla n^\varepsilon}{\psi}\|_{L^2_{xy}}\\
        &\quad+\sum_{1\leq|\beta|\leq m-1}C^\beta_{m-1}\|\partial^\beta\nabla n^\varepsilon\|_{L^\infty_{xy}}\|\partial^{\alpha-\beta}\partial_yv^\varepsilon_2\|_{L^2_{xy}}\|\partial^\alpha\nabla n^\varepsilon\|_{L^2_{xy}}\\
        &\quad+\|\partial^\alpha\nabla n^\varepsilon\|_{L^2_{xy}}\|\psi\partial_y v^\varepsilon_2\|_{L^\infty_{xy}}\|\frac{\partial^\alpha\nabla n^\varepsilon}{\psi}\|_{L^2_{xy}}\Big{)}
        \\
        &\leq \frac{\nu}{2}\sum_{|\alpha|\leq m}\|\partial^{\alpha} \nabla^2 n^\varepsilon\|^2_{L^2_{xy}}+\frac{C(\delta)}{\nu}\sum_{|\alpha|\leq2}\|\partial^\alpha\nabla n^\varepsilon\|^2_{L^2_{xy}}\sum_{|\alpha|\leq m}\|\partial^{\alpha}\nabla v^\varepsilon_2\|^2_{L^2_{xy}}\\
        &\quad+\frac{C(\delta)}{\nu}\Big{(}1+\sum_{|\alpha|\leq m}\|\partial^\alpha\nabla v^\varepsilon_2\|^2_{L^2_{xy}}+\sum_{|\alpha|\leq 2}\|\partial^{\alpha}\nabla v^\varepsilon_2\|^2_{L^2_{xy}}\Big{)}\sum_{|\alpha|\leq m}\|\partial^\alpha\nabla n^\varepsilon\|^2_{L^2_{xy}}.
		\end{align*}
     Furthermore, integrating by parts for $S_{232}$ and appealing to (\ref{i}), we obtain
        \begin{align*}
        S_{232}&\leq\sum_{|\alpha|\leq m}\Big{(}\Big{|}\langle n^\varepsilon[\partial^{\alpha},\partial_y]\nabla v^\varepsilon_2, \partial^\alpha\nabla n^\varepsilon \rangle\Big{|}+\Big{|}\langle n^\varepsilon\partial_y\partial^{\alpha}\nabla v^\varepsilon_2, \partial^\alpha\nabla n^\varepsilon \rangle\Big{|}\\
        &\quad+\sum_{1\leq|\beta|\leq m-1}C^\beta_{m-1}\Big{|}\langle \partial^\beta n^\varepsilon\partial^{\alpha-\beta}\partial_y\nabla v^\varepsilon_2, \partial^\alpha\nabla n^\varepsilon \rangle\Big{|}+\Big{|}\langle \partial^\alpha n^\varepsilon\partial_y\nabla v^\varepsilon_2,\partial^\alpha\nabla n^\varepsilon \rangle\Big{|}\Big{)}\\
        &\leq \sum_{|\alpha|\leq m}\Big{(}\|n^\varepsilon\|_{L^\infty_{xy}}\|\psi\psi'\partial^{\alpha-(0,1)}\partial_y\nabla v^\varepsilon_2\|_{L^2_{xy}}\|\frac{\partial^{\alpha}\nabla n^\varepsilon}{\psi}\|_{L^2_{xy}}\\
        &\quad+\Big{(}\|\psi \partial_y n^\varepsilon\|_{L^\infty_{xy}}\|\partial^{\alpha}\nabla v^\varepsilon_2\|_{L^2_{xy}}\|\frac{\partial^{\alpha}\nabla n^\varepsilon}{\psi}\|_{L^2_{xy}}+\| n^\varepsilon\|_{L^\infty_{xy}}\|\partial^{\alpha}\nabla v^\varepsilon_2\|_{L^2_{xy}}\|\partial_y\partial^{\alpha}\nabla n^\varepsilon\|_{L^2_{xy}}\Big{)}\\
        &\quad+\sum_{1\leq|\beta|\leq m-1}C^\beta_{m-1}\|\partial^{\beta} n^\varepsilon\|_{L^\infty_{xy}}\|\psi\partial^{\alpha-\beta}\partial_y\nabla v^\varepsilon_2\|_{L^2_{xy}}\|\frac{\partial^{\alpha}\nabla n^\varepsilon}{\psi}\|_{L^2_{xy}}\\
        &\quad+\|\frac{\partial^{\alpha} n^\varepsilon}{\psi}\|_{L^2_{xy}}\|\psi^2\partial_y\nabla v^\varepsilon_2\|_{L^\infty_{xy}}\|\frac{\partial^{\alpha}\nabla n^\varepsilon}{\psi}\|_{L^2_{xy}}\Big{)}\\
        &\leq \frac{C(\delta)}{\nu}\Big{(}\|n^\varepsilon\|^2_{L^2_{xy}}+\sum_{|\alpha|\leq 2}\|\partial^{\alpha}\nabla n^\varepsilon\|^2_{L^2_{xy}}+\sum_{1\leq|\beta|\leq m-1}\|\partial^{\beta+(1,0)}\nabla n^\varepsilon\|^2_{L^2_{xy}}\Big{)}\sum_{|\alpha|\leq m}\|\partial^\alpha\nabla v^\varepsilon_2\|^2_{L^2_{xy}}\\
        &\quad+\frac{C(\delta)}{\nu}\sum_{|\alpha|\leq3}\|\partial^{\alpha} \nabla v^\varepsilon_2\|^2_{L^2_{xy}}\sum_{|\alpha|\leq m}(\|\partial^{\alpha}\nabla n^\varepsilon\|^2_{L^2_{xy}}+\|\partial^{\alpha} n^\varepsilon\|^2_{L^2_{xy}})\\
        &\quad+\frac{\nu}{2}\sum_{|\alpha|\leq m}\|\partial^{\alpha} \nabla^2 n^\varepsilon\|^2_{L^2_{xy}}+C\sum_{|\alpha|\leq m}\|\partial^{\alpha}\nabla n^\varepsilon\|^2_{L^2_{xy}}.
		\end{align*}
        Combining all estimates for $(\ref{Delta n})$, $S_{1}$ and $S_{2}$, we immediately get $(\ref{nabla n})$.
        \begin{lemma}\label{lem4.3}
         {Under the assumptions of Theorem $\ref{th nvu regularity}$ for the initial data $(n_{in},v_{in},u_{in})$ and the integer $m\geq5$, then there exist constants $1>\delta > 0$ and $\varepsilon_0>0$ such that for any $0<\varepsilon\leq\varepsilon_0$, we obtain the following estimate for $v^\varepsilon$ on $[0,T]$:
		\begin{align}\label{nabla v}
        &\frac{d}{dt}\sum_{|\alpha|\leq m}\|\partial^\alpha\nabla v^\varepsilon\|^2_{L^2_{xy}}+\varepsilon\sum_{|\alpha|\leq m}\|\partial^\alpha\nabla^2 v^\varepsilon\|^2_{L^2_{xy}}\notag\\
        &\leq C_0(\delta)\Big{(}\|v^\varepsilon\|^2_{L^2_{xy}}+\|u^\varepsilon\|^2_{L^2_{xy}}+\|\omega^\varepsilon\|^2_{L^\infty_{xy}}+\sum_{|\alpha|\leq m}\|\partial^\alpha\nabla v^\varepsilon\|^2_{L^2_{xy}}+\sum_{|\alpha|\leq m+1}\|\partial^{\alpha} \omega^\varepsilon\|^2_{L^2_{xy}}\Big{)}\notag\\
        &\quad\times\Big{(}\sum_{|\alpha|\leq m}\|\partial^\alpha\nabla v^\varepsilon\|^2_{L^2_{xy}}+\|\psi\partial_y\omega^\varepsilon\|^2_{L^\infty_{xy}}\Big{)}+C_0(\delta)\sum_{|\alpha|\leq m}\|\partial^\alpha\nabla v^\varepsilon\|^2_{L^2_{xy}}\\
        &\quad+C_0(\delta)\sum_{|\alpha|\leq m+1}\|\partial^{\alpha}\omega^\varepsilon\|^2_{L^2_{xy}}+\nu\sum_{|\alpha|\leq m}\|\partial^\alpha \nabla^2 n^\varepsilon\|^2_{L^2_{xy}}.\notag\notag
		\end{align}
        where positive constant $C$ is dependent of $T$.
		}
	\end{lemma}
	\noindent{\bf{Proof.}}
        To obtain an estimate of $\partial^\alpha\nabla v^\varepsilon$, we write $(\ref{nvu equation})_2$ in component form as follows
        \begin{equation}\label{v1v2}
       \left\{
		\begin{split}{}
	&\partial_tv^\varepsilon_1-\varepsilon\Delta v^\varepsilon_1+\partial_x(u^\varepsilon\cdot
        v^\varepsilon)+\varepsilon\partial_x(v^\varepsilon\cdot v^\varepsilon)-\partial_xn^\varepsilon=0,\\
        &\partial_tv^\varepsilon_2-\varepsilon\Delta v^\varepsilon_2+\partial_y(u^\varepsilon\cdot v^\varepsilon)+\varepsilon\partial_y(v^\varepsilon\cdot v^\varepsilon)-\partial_yn^\varepsilon=0.
		\end{split}
		\right.
       \end{equation}
     Taking $\partial^\alpha\nabla$ on $(\ref{v1v2})_1$ and multiplying by $\partial^\alpha\nabla v^\varepsilon_1$, it can be shown
    \begin{equation*}
        \begin{split}{}
        &\sum_{|\alpha|\leq m}\langle \partial^\alpha\nabla\partial_tv^\varepsilon_1, \partial^\alpha\nabla v^\varepsilon_1 \rangle-\varepsilon\sum_{|\alpha|\leq m}\langle \partial^\alpha\nabla\Delta v^\varepsilon_1, \partial^\alpha\nabla v^\varepsilon_1 \rangle\\
        &\leq\sum_{|\alpha|\leq m}\Big{|}\langle \partial^\alpha\nabla\partial_x(u^\varepsilon\cdot v^\varepsilon), \partial^\alpha\nabla v^\varepsilon_1 \rangle\Big{|}+\varepsilon\sum_{|\alpha|\leq m}\Big{|}\langle \partial^\alpha\nabla\partial_x(v^\varepsilon\cdot v^\varepsilon), \partial^\alpha\nabla v^\varepsilon_1 \rangle\Big{|}\\
        &\quad+\sum_{|\alpha|\leq m}\Big{|}\langle \partial^\alpha\nabla\partial_xn^\varepsilon, \partial^\alpha\nabla v^\varepsilon_1 \rangle\Big{|}\\
        &=:S_3+S_4+S_5.\\
        \end{split}
		\end{equation*}
        Since $\psi(y)|_{y=0}=0$ and the boundary condition $\partial_yv^\varepsilon_1|_{y=0}=0$, the estimate for the Laplacian term is analogous to that in $(\ref{Delta n})$,
        \begin{align*}
        &-\varepsilon\sum_{|\alpha|\leq m}\langle \partial^\alpha\nabla\Delta v^\varepsilon_1, \partial^\alpha\nabla v^\varepsilon_1 \rangle\\
        &=-\varepsilon\sum_{|\alpha|\leq m}\langle \partial_x\partial^\alpha\partial_x\nabla v^\varepsilon_1, \partial^\alpha\nabla v^\varepsilon_1 \rangle-\varepsilon\sum_{|\alpha|\leq m}\langle \partial_y\partial^\alpha\partial_y\nabla v^\varepsilon_1, \partial^\alpha\nabla v^\varepsilon_1 \rangle\\
        &\quad-\varepsilon\sum_{|\alpha|\leq m}\langle [\partial^\alpha,\partial_y]\partial_y\nabla v^\varepsilon_1, \partial^\alpha\nabla v^\varepsilon_1 \rangle\\
        &\geq\varepsilon\sum_{|\alpha|\leq m}\|\partial^\alpha\nabla^2 v^\varepsilon_1\|^2_{L^2_{xy}}-C\delta\varepsilon\Big{(}\sum_{|\alpha|\leq m}\|\partial^\alpha\partial_y\nabla v^\varepsilon_1\|^2_{L^2_{xy}}+\|\partial^{\alpha-(0,1)}\partial_y\nabla v^\varepsilon_1\|^2_{L^2_{xy}}\Big{)}-C\sum_{|\alpha|\leq m}\|\partial^\alpha\nabla v^\varepsilon_1\|^2_{L^2_{xy}}\\
        &\geq(1-C\delta)\varepsilon\sum_{|\alpha|\leq m}\|\partial^\alpha\nabla^2 v^\varepsilon_1\|^2_{L^2_{xy}}-C\sum_{|\alpha|\leq m}\|\partial^\alpha\nabla v^\varepsilon_1\|^2_{L^2_{xy}}.
		\end{align*}
        To expand the term $S_3$, we proceed as follows:
        \begin{align*}
        S_3&=\sum_{|\alpha|\leq m}\Big{(}\Big{|}\langle \partial^\alpha \nabla(\partial_xu^\varepsilon_1 v^\varepsilon_1), \partial^\alpha\nabla v^\varepsilon_1 \rangle\Big{|}+\Big{|}\langle \partial^\alpha \nabla(u^\varepsilon_1 \partial_xv^\varepsilon_1), \partial^\alpha\nabla v^\varepsilon_1 \rangle\Big{|}\\
        &\quad+\Big{|}\langle \partial^\alpha\nabla (\partial_xu^\varepsilon_2 v^\varepsilon_2), \partial^\alpha\nabla v^\varepsilon_1 \rangle\Big{|}+\Big{|}\langle \partial^\alpha\nabla (u^\varepsilon_2 \partial_xv^\varepsilon_2), \partial^\alpha\nabla v^\varepsilon_1 \rangle\Big{|}\Big{)}\\
        &=:S_{31}+S_{32}+S_{33}+S_{34}.
		\end{align*}
        Using the Sobolev embedding inequality $(\ref{infty})$, {we obtain}
        \begin{align*}
        S_{31}&=\sum_{|\alpha|\leq m}\Big{(}\Big{|}\langle\partial^\alpha(\nabla\partial_x u^\varepsilon_1 v^\varepsilon_1), \partial^\alpha\nabla v^\varepsilon_1 \rangle\Big{|}+\Big{|}\langle \partial^\alpha (\partial_xu^\varepsilon_1\nabla v^\varepsilon_1), \partial^\alpha\nabla v^\varepsilon_1 \rangle\Big{|}\Big{)}\\
        &\leq\sum_{|\alpha|\leq m}\Big{(}\|\psi\nabla\partial_x u^\varepsilon_1\|_{L^\infty_{xy}}\|\frac{\partial^\alpha v^\varepsilon_1}{\psi}\|_{L^2_{xy}}\|\partial^\alpha\nabla v^\varepsilon_1\|_{L^2_{xy}}+\sum_{1\leq|\beta|\leq m-1}C^\beta_{m-1}\|\partial^\beta\nabla\partial_x u^\varepsilon_1\|_{L^2_{xy}}\\
        &\quad\times\|\partial^{\alpha-\beta}v^\varepsilon_1\|_{L^\infty_{xy}}\|\partial^\alpha\nabla v^\varepsilon_1\|_{L^2_{xy}}+\|\partial^\alpha\nabla \partial_xu^\varepsilon_1\|_{L^2_{xy}}\| v^\varepsilon_1\|_{L^\infty_{xy}}\|\partial^\alpha\nabla v^\varepsilon_1\|_{L^2_{xy}}\\
        &\quad+\|\partial_x u^\varepsilon_1\|_{L^\infty_{xy}}\|\partial^\alpha\nabla v^\varepsilon_1\|^2_{L^2_{xy}}+\sum_{1\leq|\beta|\leq m-1}C^\beta_{m-1}\|\partial^\beta\partial_x u^\varepsilon_1\|_{L^\infty_{xy}}\\
        &\quad\times\|\partial^{\alpha-\beta}\nabla v^\varepsilon_1\|_{L^2_{xy}}\|\partial^\alpha\nabla v^\varepsilon_1\|_{L^2_{xy}}+\|\frac{\partial^\alpha\partial_x u^\varepsilon_1}{\psi}\|_{L^2_{xy}}\|\psi\nabla v^\varepsilon_1\|_{L^\infty_{xy}}\|\partial^\alpha\nabla v^\varepsilon_1\|_{L^2_{xy}}\Big{)}\\
        &\leq C\Big{(}\sum_{|\alpha|\leq 3}\|\partial^\alpha \nabla u^\varepsilon_1\|^2_{L^2_{xy}}+\sum_{1\leq|\beta|\leq m-1}\|\partial^{\alpha-\beta+(1,0)}\nabla v^\varepsilon_1\|^2_{L^2_{xy}}+\|v^\varepsilon_1\|^2_{L^2_{xy}}\\
        &\quad+\sum_{|\alpha|\leq 2}\|\partial^{\alpha}\nabla v^\varepsilon_1\|^2_{L^2_{xy}}+\sum_{1\leq|\beta|\leq m-1}\sum_{|\gamma|\leq2}\|\partial^{\beta+\gamma}\nabla u^\varepsilon_1\|^2_{L^2_{xy}}\Big{)}\sum_{|\alpha|\leq m}\|\partial^\alpha\nabla v^\varepsilon_1\|^2_{L^2_{xy}}\\
        &\quad+C\sum_{|\alpha|\leq m}\|\partial^{\alpha+(1,0)}\nabla u^\varepsilon_1\|^2_{L^2_{xy}}+C(\delta)\sum_{|\alpha|\leq m}\|\partial^\alpha\nabla v^\varepsilon_1\|^2_{L^2_{xy}}+C(\delta)\sum_{|\alpha|\leq m}\|\partial^\alpha v^\varepsilon_1\|^2_{L^2_{xy}}.
		\end{align*}
        By similar arguments, we obtain the estimates for $S_{32}$ and $S_{33}$ as below:
        \begin{align*}
        S_{32}&=\sum_{|\alpha|\leq m}\Big{(}\Big{|}\langle\partial^\alpha(\nabla u^\varepsilon_1 \partial_xv^\varepsilon_1), \partial^\alpha\nabla v^\varepsilon_1 \rangle\Big{|}+\Big{|}\langle \partial^\alpha (u^\varepsilon_1\partial_x\nabla v^\varepsilon_1), \partial^\alpha\nabla v^\varepsilon_1 \rangle\Big{|}\Big{)}\\
        &\leq\sum_{|\alpha|\leq m}\Big{(}\|\nabla u^\varepsilon_1\|_{L^\infty_{xy}}\|\partial^\alpha\partial_x v^\varepsilon_1\|_{L^2_{xy}}\|\partial^\alpha\nabla v^\varepsilon_1\|_{L^2_{xy}}+\sum_{1\leq|\beta|\leq m-1}C^\beta_{m-1}\|\psi\partial^\beta\nabla u^\varepsilon_1\|_{L^\infty_{xy}}\\
        &\quad\times\|\frac{\partial^{\alpha-\beta}\partial_x v^\varepsilon_1}{\psi}\|_{L^2_{xy}}\|\partial^\alpha\nabla v^\varepsilon_1\|_{L^2_{xy}}+\|\partial^\alpha\nabla u^\varepsilon_1\|_{L^2_{xy}}\|\partial_x v^\varepsilon_1\|_{L^\infty_{xy}}\|\partial^\alpha\nabla v^\varepsilon_1\|_{L^2_{xy}}\\
        &\quad+\frac{1}{2}\|\partial_x u^\varepsilon_1\|_{L^\infty_{xy}}\|\partial^\alpha\nabla v^\varepsilon_1\|^2_{L^2_{xy}}+\sum_{1\leq|\beta|\leq m-1}C^\beta_{m-1}\|\partial^\beta u^\varepsilon_1\|_{L^\infty_{xy}}\\
        &\quad\times\|\partial^{\alpha-\beta}\nabla\partial_x v^\varepsilon_1\|_{L^2_{xy}}\|\partial^\alpha\nabla v^\varepsilon_1\|_{L^2_{xy}}+\|\frac{\partial^\alpha u^\varepsilon_1}{\psi}\|_{L^2_{xy}}\|\psi\nabla \partial_xv^\varepsilon_1\|_{L^\infty_{xy}}\|\partial^\alpha\nabla v^\varepsilon_1\|_{L^2_{xy}}\Big{)}\\
        &\leq C\Big{(}\|\omega^\varepsilon\|^2_{L^\infty_{xy}}+\sum_{1\leq|\beta|\leq m-1}\sum_{|\gamma|\leq2}\|\partial^{\beta+\gamma}\nabla u^\varepsilon_1\|^2_{L^2_{xy}}+\sum_{|\alpha|\leq 3}\|\partial^\alpha\nabla v^\varepsilon_1\|^2_{L^2_{xy}}+1\Big{)}\\
        &\quad\times\sum_{|\alpha|\leq m}\|\partial^\alpha\nabla v^\varepsilon_1\|^2_{L^2_{xy}}+C(\delta)\sum_{|\alpha|\leq m}\|\partial^{\alpha}\nabla u^\varepsilon_1\|^2_{L^2_{xy}}+C\sum_{|\alpha|\leq m}\|\partial^\alpha u^\varepsilon_1\|^2_{L^2_{xy}},
		\end{align*}
        and
        \begin{align*}
        S_{33}&=\sum_{|\alpha|\leq m}\Big{(}\Big{|}\langle \partial^\alpha(\nabla \partial_xu^\varepsilon_2 v^\varepsilon_2), \partial^\alpha\nabla v^\varepsilon_1 \rangle\Big{|}+\Big{|}\langle \partial^\alpha (\partial_xu^\varepsilon_2 \nabla v^\varepsilon_2), \partial^\alpha\nabla v^\varepsilon_1 \rangle\Big{|}\Big{)}\\
        &\leq\sum_{|\alpha|\leq m}\Big{(}\|\psi\nabla\partial_x u^\varepsilon_2\|_{L^\infty_{xy}}\|\frac{\partial^\alpha v^\varepsilon_2}{\psi}\|_{L^2_{xy}}\|\partial^\alpha\nabla v^\varepsilon_1\|_{L^2_{xy}}+\sum_{1\leq|\beta|\leq m-1}C^\beta_{m-1}\|\partial^\beta\nabla\partial_x u^\varepsilon_2\|_{L^2_{xy}}\\
        &\quad\times\|\partial^{\alpha-\beta}v^\varepsilon_2\|_{L^\infty_{xy}}\|\partial^\alpha\nabla v^\varepsilon_1\|_{L^2_{xy}}+\|\partial^\alpha\nabla \partial_xu^\varepsilon_2\|_{L^2_{xy}}\| v^\varepsilon_2\|_{L^\infty_{xy}}\|\partial^\alpha\nabla v^\varepsilon_1\|_{L^2_{xy}}\\
        &\quad+\|\partial_x u^\varepsilon_2\|_{L^\infty_{xy}}\|\partial^\alpha\nabla v^\varepsilon_2\|_{L^2_{xy}}\|\partial^\alpha\nabla v^\varepsilon_1\|_{L^2_{xy}}+\sum_{1\leq|\beta|\leq m-1}C^\beta_{m-1}\|\partial^\beta\partial_x u^\varepsilon_2\|_{L^\infty_{xy}}\\
        &\quad\times\|\partial^{\alpha-\beta}\nabla v^\varepsilon_2\|_{L^2_{xy}}\|\partial^\alpha\nabla v^\varepsilon_1\|_{L^2_{xy}}+\|\frac{\partial^\alpha\partial_x u^\varepsilon_2}{\psi}\|_{L^2_{xy}}\|\psi\nabla v^\varepsilon_2\|_{L^\infty_{xy}}\|\partial^\alpha\nabla v^\varepsilon_1\|_{L^2_{xy}}\Big{)}\\
        &\leq C\Big{(}\sum_{|\alpha|\leq 3}\|\partial^{\alpha}\nabla u^\varepsilon_2\|^2_{L^2_{xy}}+\sum_{1\leq|\beta|\leq m-1}\|\partial^{\alpha-\beta+(1,0)}\nabla v^\varepsilon_2\|^2_{L^2_{xy}}+\|v^\varepsilon_2\|^2_{L^2_{xy}}\\
        &\quad+\sum_{|\alpha|\leq 2}\|\partial^{\alpha}\nabla v^\varepsilon_2\|^2_{L^2_{xy}}+\sum_{1\leq|\beta|\leq m-1}\sum_{|\gamma|\leq2}\|\partial^{\beta+\gamma}\nabla u^\varepsilon_2\|^2_{L^2_{xy}}\Big{)}\sum_{|\alpha|\leq m}\|\partial^\alpha\nabla v^\varepsilon_1\|^2_{L^2_{xy}}\\
        &\quad+C\sum_{|\alpha|\leq m}\|\partial^\alpha\nabla v^\varepsilon_2\|^2_{L^2_{xy}}+C(\delta)\sum_{|\alpha|\leq m}\|\partial^{\alpha+(1,0)}\nabla u^\varepsilon_2\|^2_{L^2_{xy}}+C\sum_{|\alpha|\leq m}\|\partial^\alpha v^\varepsilon_2\|^2_{L^2_{xy}}.
		\end{align*}
        We will {divide} $S_{34}$ into the following two items
        \begin{equation*}
        \begin{split}{}
        S_{34}&=\sum_{|\alpha|\leq m}\Big{(}\Big{|}\langle\partial^\alpha(\nabla u^\varepsilon_2 \partial_xv^\varepsilon_2), \partial^\alpha\nabla v^\varepsilon_1 \rangle\Big{|}+\Big{|}\langle \partial^\alpha (u^\varepsilon_2 \partial_x\nabla v^\varepsilon_2), \partial^\alpha\nabla v^\varepsilon_1 \rangle\Big{|}\Big{)}=:S_{341}+S_{342}.\\
        \end{split}
		\end{equation*}
        The estimate for $S_{341}$ is similar to that for $S_{32}$, by condition $(\ref{nvu equation})_4$, then it yields that
        \begin{align*}
        S_{341}
        &\leq\sum_{|\alpha|\leq m}\Big{(}\|\nabla u^\varepsilon_2\|_{L^\infty_{xy}}\|\partial^\alpha\partial_x v^\varepsilon_2\|_{L^2_{xy}}\|\partial^\alpha\nabla v^\varepsilon_1\|_{L^2_{xy}}+\sum_{1\leq|\beta|\leq m-1}C^\beta_{m-1}\|\psi\partial^\beta\nabla u^\varepsilon_2\|_{L^\infty_{xy}}\\
        &\quad\times\|\frac{\partial^{\alpha-\beta}\partial_x v^\varepsilon_2}{\psi}\|_{L^2_{xy}}\|\partial^\alpha\nabla v^\varepsilon_1\|_{L^2_{xy}}+\|\partial^\alpha\nabla u^\varepsilon_2\|_{L^2_{xy}}\|\partial_x v^\varepsilon_2\|_{L^\infty_{xy}}\|\partial^\alpha\nabla v^\varepsilon_1\|_{L^2_{xy}}\Big{)}\\
        &\leq C\Big{(}\|\nabla u^\varepsilon_2\|^2_{L^\infty_{xy}}+\sum_{1\leq|\beta|\leq m-1}\sum_{|\gamma|\leq2}\|\partial^{\beta+\gamma}\nabla u^\varepsilon_2\|^2_{L^2_{xy}}+\sum_{|\alpha|\leq 2}\|\partial^{\alpha}\nabla v^\varepsilon_2\|^2_{L^2_{xy}}\Big{)}\\
        &\quad\times\sum_{|\alpha|\leq m}\|\partial^\alpha\nabla v^\varepsilon_1\|^2_{L^2_{xy}}+C\sum_{|\alpha|\leq m}\|\partial^\alpha\nabla u^\varepsilon_2\|^2_{L^2_{xy}}+C\sum_{|\alpha|\leq m}\|\partial^\alpha\nabla v^\varepsilon_2\|^2_{L^2_{xy}}\\
        &\leq C\Big{(}\|\omega^\varepsilon\|^2_{L^\infty_{xy}}+\sum_{1\leq|\beta|\leq m-1}\sum_{|\gamma|\leq2}\|\partial^{\beta+\gamma}\nabla u^\varepsilon_2\|^2_{L^2_{xy}}+\sum_{|\alpha|\leq 2}\|\partial^{\alpha}\nabla v^\varepsilon_2\|^2_{L^2_{xy}}\Big{)}\\
        &\quad\times\sum_{|\alpha|\leq m}\|\partial^\alpha\nabla v^\varepsilon_1\|^2_{L^2_{xy}}+C\sum_{|\alpha|\leq m}\|\partial^\alpha\nabla u^\varepsilon\|^2_{L^2_{xy}}+C\sum_{|\alpha|\leq m}\|\partial^\alpha\nabla v^\varepsilon_2\|^2_{L^2_{xy}}.
		\end{align*}
        By substituting the curl condition $(\ref{v curl})$, we can obtain $\partial_xv^\varepsilon_2=\partial_yv^\varepsilon_1$, and then
        \begin{align*}
         \langle  u^\varepsilon_2\partial^\alpha \nabla\partial_xv^\varepsilon_2,
        \partial^\alpha\nabla v^\varepsilon_1 \rangle=\langle  u^\varepsilon_2\partial^\alpha \nabla\partial_yv^\varepsilon_1, \partial^\alpha\nabla v^\varepsilon_1 \rangle.
        \end{align*}
        With $u^\varepsilon_2|_{y=0}=0$ and the divergence condition $\partial_xu^\varepsilon_1+\partial_yu^\varepsilon_2=0$, the estimation of $S_{342}$  is as follows:
        \begin{align*}
        S_{342}&=\sum_{|\alpha|\leq m}\Big{(}\Big{|}\langle  u^\varepsilon_2[\partial^\alpha, \partial_y]\nabla v^\varepsilon_1, \partial^\alpha\nabla v^\varepsilon_1 \rangle\Big{|}+\Big{|}\langle \partial_y u^\varepsilon_2\partial^\alpha\nabla v^\varepsilon_1, \partial^\alpha\nabla v^\varepsilon_1 \rangle\Big{|}\\
        &\quad+\sum_{1\leq|\beta|\leq m-1}C^\beta_{m-1}\Big{|}\langle \partial^\beta u^\varepsilon_2 \partial^{\alpha-\beta} \partial_x\nabla v^\varepsilon_2, \partial^\alpha\nabla v^\varepsilon_1 \rangle\Big{|}+\Big{|}\langle \partial^\alpha u^\varepsilon_2\partial_x\nabla v^\varepsilon_2, \partial^\alpha\nabla v^\varepsilon_1 \rangle\Big{|}\Big{)}\\
        &\leq\sum_{|\alpha|\leq m}\Big{(}\|\frac{u^\varepsilon_2}{\psi}\|_{L^\infty_{xy}}\|\psi\psi'\partial^{\alpha-(0,1)}\partial_y\nabla v^\varepsilon_1\|_{L^2_{xy}}\|\partial^\alpha\nabla v^\varepsilon_1\|_{L^2_{xy}}+\|\partial_xu^\varepsilon_1\|_{L^\infty_{xy}}\|\partial^\alpha\nabla v^\varepsilon_1\|^2_{L^2_{xy}}\\
        &\quad+\sum_{1\leq|\beta|\leq m-1}C^\beta_{m-1}\|\partial^\beta u^\varepsilon_2\|_{L^\infty_{xy}}\|\partial^{\alpha-\beta+(1,0)}\nabla v^\varepsilon_2\|_{L^2_{xy}}\|\partial^\alpha\nabla v^\varepsilon_1\|_{L^2_{xy}}\\
        &\quad+\|\frac{\partial^\alpha u^\varepsilon_2}{\psi}\|_{L^2_{xy}}\|\psi \partial_x\nabla v^\varepsilon_2\|_{L^\infty_{xy}}\|\partial^\alpha\nabla v^\varepsilon_1\|_{L^2_{xy}}\Big{)}\\
        &\leq C(\delta)\Big{(}\|u^\varepsilon_2\|^2_{L^2_{xy}}+\sum_{|\alpha|\leq 2}\|\partial^{\alpha}\nabla u^\varepsilon_1\|^2_{L^2_{xy}}+\sum_{1\leq|\beta|\leq m-1}\|\partial^{\beta+(1,0)}\nabla u^\varepsilon_2\|^2_{L^2_{xy}}\\
        &\quad+\sum_{|\alpha|\leq 3}\|\partial^{\alpha}\nabla v^\varepsilon_2\|^2_{L^2_{xy}}\Big{)}\sum_{|\alpha|\leq m}\|\partial^\alpha\nabla v^\varepsilon_1\|^2_{L^2_{xy}}+C\sum_{|\alpha|\leq m}\|\partial^{\alpha}\nabla v^\varepsilon_2\|^2_{L^2_{xy}}\\
        &\quad+C(\delta)\sum_{|\alpha|\leq m}\|\partial^\alpha\nabla u^\varepsilon_2\|^2_{L^2_{xy}}+C(\delta)\sum_{|\alpha|\leq m}\|\partial^\alpha u^\varepsilon_2\|^2_{L^2_{xy}}.
		\end{align*}
        For any $\varepsilon>0$, we directly derive the estimate of the term $S_4$
        \begin{align*}
        S_4&\leq \frac{1}{4}\varepsilon\sum_{|\alpha|\leq m}\|\partial^\alpha\nabla^2 v^\varepsilon_1\|^2_{L^2_{xy}}+\frac{1}{4}\varepsilon\sum_{|\alpha|\leq m}\|\partial^\alpha\nabla^2 v^\varepsilon_2\|^2_{L^2_{xy}}+C\Big{(}\|v^\varepsilon_1\|^2_{L^2_{xy}}+\|v^\varepsilon_2\|_{L^2_{xy}}+\sum_{|\alpha|\leq 3}\|\partial^\alpha\nabla v^\varepsilon_1\|^2_{L^2_{xy}}\\
        &\quad+\sum_{|\alpha|\leq 3}\|\partial^\alpha\nabla v^\varepsilon_2\|_{L^2_{xy}}+\sum_{1\leq|\beta|\leq m-1}\|\partial^{\alpha-\beta+(1,0)}\nabla v^\varepsilon_1\|^2_{L^2_{xy}}+\sum_{1\leq|\beta|\leq m-1}\|\partial^{\alpha-\beta+(1,0)}\nabla v^\varepsilon_2\|^2_{L^2_{xy}}\\
        &\quad+C\sum_{|\alpha|\leq m}\|\partial^\alpha\nabla v^\varepsilon_2\|^2_{L^2_{xy}}\Big{)}\sum_{|\alpha|\leq m}\|\partial^\alpha\nabla v^\varepsilon_1\|^2_{L^2_{xy}}+C\sum_{|\alpha|\leq m}\|\partial^\alpha\nabla v^\varepsilon_1\|^2_{L^2_{xy}}+C\sum_{|\alpha|\leq m}\|\partial^\alpha\nabla v^\varepsilon_2\|^2_{L^2_{xy}}\\
        &\quad+C\sum_{|\alpha|\leq m}\|\partial^\alpha v^\varepsilon_1\|^2_{L^2_{xy}}+C\sum_{|\alpha|\leq m}\|\partial^\alpha v^\varepsilon_2\|^2_{L^2_{xy}}.
		\end{align*}
        Using the H\"{o}lder's inequality, we can obtain
        \begin{equation*}
        \begin{split}{}
        S_5=\sum_{|\alpha|\leq m}\Big{|}\langle \partial^\alpha \partial_x\nabla n^\varepsilon , \partial^\alpha\nabla v^\varepsilon_1 \rangle\Big{|}\leq
        \frac{1}{4}\nu\sum_{|\alpha|\leq m}\|\partial^\alpha\nabla^2 n^\varepsilon\|^2_{L^2_{xy}}+{\frac{C(\delta
        )}{\nu}\sum_{|\alpha|\leq m}\|\partial^\alpha\nabla v^\varepsilon_1\|_{L^2_{xy}}^2}.
        \end{split}
		\end{equation*}
        Thus, combining the estimates from $S_3$ to $S_5$, we can conclude the regularity estimate of $v^\varepsilon_1$ as follows
        \begin{align}\label{v1}
        &\frac{1}{2}\frac{d}{dt}\sum_{|\alpha|\leq m}\|\partial^\alpha\nabla v^\varepsilon_1\|^2_{L^2_{xy}}+\varepsilon\sum_{|\alpha|\leq m}\|\partial^\alpha\nabla^2 v^\varepsilon_1\|^2_{L^2_{xy}}\notag\\
        &\leq \frac{1}{4}\nu\sum_{|\alpha|\leq m}\|\partial^\alpha\nabla^2 n^\varepsilon\|^2_{L^2_{xy}}+\frac{1}{4}\varepsilon\sum_{|\alpha|\leq m}\|\partial^\alpha\nabla^2 v^\varepsilon_1\|^2_{L^2_{xy}}+\frac{1}{4}\varepsilon\sum_{|\alpha|\leq m}\|\partial^\alpha\nabla^2 v^\varepsilon_2\|^2_{L^2_{xy}}\\
        &\quad +C(\delta)\Big{(}\|v^\varepsilon\|^2_{L^2_{xy}}+\|u^\varepsilon\|^2_{L^2_{xy}}+\|\omega^\varepsilon\|^2_{L^\infty_{xy}}+\sum_{|\alpha|\leq m}\|\partial^\alpha\nabla v^\varepsilon\|^2_{L^2_{xy}}+\sum_{|\alpha|\leq m+1}\|\partial^{\alpha}\nabla u^\varepsilon\|^2_{L^2_{xy}}\Big{)}\notag\\
    &\quad\times\sum_{|\alpha|\leq m}\|\partial^\alpha\nabla v^\varepsilon_1\|^2_{L^2_{xy}}+C(\delta)\Big{(}\sum_{|\alpha|\leq m}\|\partial^\alpha\nabla v^\varepsilon\|^2_{L^2_{xy}}+\sum_{|\alpha|\leq m+1}\|\partial^{\alpha}\nabla u^\varepsilon\|^2_{L^2_{xy}}\notag\\
        &\quad+\sum_{|\alpha|\leq m}\|\partial^\alpha v^\varepsilon\|^2_{L^2_{xy}}+\sum_{|\alpha|\leq m}\|\partial^\alpha u^\varepsilon\|^2_{L^2_{xy}}\Big{)}.\notag\notag
		\end{align}
        When estimating $v^\varepsilon_2$, we divide the estimation into two cases: $\alpha=(0,0)$ and $\alpha=(\alpha_1,\alpha_2)$ with $\alpha_1\geq1$,$\alpha_2\geq1$. First, for the case $\alpha=(0,0)$, multiplying equation $(\ref{v1v2})_2$ by $-\Delta v^\varepsilon_2$, we obtain
        \begin{align*}
            &\frac{1}{2}\frac{d}{dt}\|\nabla v^\varepsilon_2\|^2_{L^2_{xy}}+\varepsilon\|\nabla^2 v^\varepsilon_2\|^2_{L^2_{xy}}\\
            &=-\int^\infty_0\int^\infty_{-\infty}\nabla\partial_y(u^\varepsilon\cdot v^\varepsilon)\cdot\nabla v^\varepsilon_2dxdy-\varepsilon\int^\infty_0\int^\infty_{-\infty}\nabla\partial_y(v^\varepsilon\cdot v^\varepsilon)\cdot\nabla v^\varepsilon_2dxdy\\
            &\quad+\int^\infty_0\int^\infty_{-\infty}\nabla\partial_yn^\varepsilon\cdot\nabla v^\varepsilon_2dxdy\\
            &\leq\|\psi\nabla\partial_yu^\varepsilon\|_{L^\infty_{xy}}\|\frac{v^\varepsilon}{\psi}\|_{L^2_{xy}}\|\nabla v^\varepsilon_2\|_{L^2_{xy}}+\|\nabla u^\varepsilon\|_{L^\infty_{xy}}\|\nabla v^\varepsilon\|_{L^2_{xy}}\|\nabla v^\varepsilon_2\|_{L^2_{xy}}\\
            &\quad+\|\partial_xu^\varepsilon_1\|_{L^\infty_{xy}}\|\nabla v^\varepsilon_2\|^2_{L^2_{xy}}+\|\partial_yu^\varepsilon_2\|_{L^\infty_{xy}}\|\nabla v^\varepsilon_2\|^2_{L^2_{xy}}+2\varepsilon(\|\nabla\partial_yv^\varepsilon\|_{L^2_{xy}}\|v^\varepsilon\|_{L^\infty_{xy}}\\
            &\quad+\|\nabla v^\varepsilon\|^2_{L^4_{xy}})\|\nabla v^\varepsilon_2\|_{L^2_{xy}}+\|\nabla\partial_yn^\varepsilon\|_{L^2_{xy}}\|\nabla v^\varepsilon_2\|_{L^2_{xy}}
            \\
            &\leq\frac{1}{8}\varepsilon\|\nabla^2 v^\varepsilon\|^2_{L^2_{xy}}+\frac{1}{8}\nu\|\nabla^2 n^\varepsilon\|^2_{L^2_{xy}}+C\Big{(}1+\sum_{|\alpha|= 1}\|\partial^{\alpha}\nabla v^\varepsilon_2\|^2_{L^2_{xy}}\Big{)}\Big{(}\|u^\varepsilon_1\|^2_{L^2_{xy}}+C\sum_{|\alpha|\leq 2}\|\partial^\alpha\nabla u^\varepsilon_1\|^2_{L^2_{xy}}\Big{)}\\
            &\quad+C\|\psi\partial_y\nabla u^\varepsilon\|^2_{L^\infty_{xy}}\Big{(}\|v^\varepsilon\|^2_{L^2_{xy}}+\|\nabla v^\varepsilon\|^2_{L^2_{xy}}\Big{)}\\
            &\quad+C\Big{(}\|v^\varepsilon_1\|^2_{L^2_{xy}}+\|v^\varepsilon_2\|^2_{L^2_{xy}}+\sum_{|\alpha|\leq1}\big{(}\|\partial^{\alpha}\nabla v^\varepsilon_1\|^2_{L^2_{xy}}+\|\partial^{\alpha}\nabla v^\varepsilon_2\|^2_{L^2_{xy}}\big{)}\Big{)}\|\nabla v^\varepsilon_2\|^2_{L^2_{xy}}\\
            &\quad+C\|\omega^\varepsilon\|^2_{L^\infty_{xy}}+C\|\nabla v^\varepsilon_2\|^2_{L^2_{xy}},
        \end{align*}
        in the second last step, by integrating by parts and using the curl condition $(\ref{v curl})$, we obtain
        \begin{align*}
           \int^\infty_0\int^\infty_{-\infty}u^\varepsilon_1\nabla \partial_yv^\varepsilon_1\nabla v^\varepsilon_2dxdy&=\int^\infty_0\int^\infty_{-\infty}u^\varepsilon_1\nabla \partial_xv^\varepsilon_2\nabla v^\varepsilon_2dxdy=-\frac{1}{2}\int^\infty_0\int^\infty_{-\infty}\partial_xu^\varepsilon_1|\nabla v^\varepsilon_2|^2dxdy, \\
           \int^\infty_0\int^\infty_{-\infty}u^\varepsilon_2\nabla \partial_yv^\varepsilon_2\nabla v^\varepsilon_2dxdy&=-\frac{1}{2}\int^\infty_0\int^\infty_{-\infty}\partial_yu^\varepsilon_2|\nabla v^\varepsilon_2|^2dxdy
        \end{align*}
          by $u^\varepsilon_2|_{y=0}=0$. Moreover, for $\alpha=(\alpha_1,\alpha_2)$ at $\alpha_1\geq1$, $\alpha_2\geq1$, the derivation process is similar to that of $v^\varepsilon_1$, so it can be obtained
        \begin{align*}
        &\frac{1}{2}\frac{d}{dt}\sum_{|\alpha|\leq m}\|\partial^\alpha\nabla v^\varepsilon_2\|^2_{L^2_{xy}}+\sum_{|\alpha|\leq m}\|\partial^\alpha \nabla^2 v^\varepsilon_2\|^2_{L^2_{xy}}\\
        &\leq \frac{1}{8}\nu\sum_{|\alpha|\leq m}\|\partial^\alpha\nabla^2 n^\varepsilon\|^2_{L^2_{xy}}+\frac{1}{16}\varepsilon\sum_{|\alpha|\leq m}\|\partial^\alpha\nabla^2 v^\varepsilon_1\|^2_{L^2_{xy}}+\frac{1}{16}\varepsilon\sum_{|\alpha|\leq m}\|\partial^\alpha\nabla^2 v^\varepsilon_2\|^2_{L^2_{xy}}\\
        &\quad+C(\delta)\Big{(}\|v^\varepsilon\|^2_{L^2_{xy}}+\|u^\varepsilon\|^2_{L^2_{xy}}+\|\omega^\varepsilon\|^2_{L^\infty_{xy}}+\sum_{|\alpha|\leq m}\|\partial^\alpha\nabla v^\varepsilon\|^2_{L^2_{xy}}+\sum_{|\alpha|\leq m+1}\|\partial^{\alpha}\nabla u^\varepsilon\|^2_{L^2_{xy}}\Big{)}\\
        &\quad\cdot\sum_{|\alpha|\leq m}\|\partial^\alpha\nabla v^\varepsilon_2\|^2_{L^2_{xy}}+C(\delta)\Big{(}\sum_{|\alpha|\leq m}\|\partial^\alpha\nabla v^\varepsilon\|^2_{L^2_{xy}}+\sum_{|\alpha|\leq m+1}\|\partial^{\alpha}\nabla u^\varepsilon\|^2_{L^2_{xy}}+\sum_{|\alpha|\leq m}\|\partial^{\alpha} v^\varepsilon\|^2_{L^2_{xy}}\\&
        \quad+\sum_{|\alpha|\leq m+1}\|\partial^{\alpha} u^\varepsilon\|^2_{L^2_{xy}}\Big{)}.
		\end{align*}

        Note that for any $\alpha=(\alpha_1,0)$ ($\alpha_1\geq1$), the estimation process is the same as that of $\alpha=(0,0)$.
        Hence, combing $(\ref{v1})$, we finish the argument for $(\ref{nabla v})$.

        \begin{lemma}\label{lem4.4}
         {Assume that the initial data $(n_{in},v_{in},u_{in})$ satisfy the hypothesis conditions of Theorem $\ref{th nvu regularity}$ and the integer $m\geq5$, then there exist constants $1>\delta > 0$ and $\varepsilon_0>0$ such that for any $0<\varepsilon\leq\varepsilon_0$, the solution $\omega^\varepsilon=\partial_xu^\varepsilon_2-\partial_yu^\varepsilon_1$ is estimated on $[0,T]$ as follows:
		\begin{equation}\label{omega varepsilon Hm}
        \begin{split}{}
        &\frac{d}{dt}\sum_{|\alpha|\leq m+1}\|\partial^{\alpha}\omega^\varepsilon\|^2_{L^2_{xy}}+\varepsilon\sum_{|\alpha|\leq m+1}\|\partial^\alpha \nabla \omega^\varepsilon\|^2_{L^2_{xy}}\\
        &\leq C_0(\delta)\Big{(}1+\|u^\varepsilon\|^2_{L^2_{xy}}+\sum_{|\alpha|\leq m+1}\|\partial^\alpha\omega^\varepsilon\|^2_{L^2_{xy}}+\|\psi\partial_y \omega^\varepsilon\|^2_{L^\infty_{xy}}\Big{)}\sum_{|\alpha|\leq m+1}\|\partial^\alpha\omega^\varepsilon\|^2_{L^2_{xy}}\\
        &\quad+C\sum_{|\alpha|\leq m}\|\partial^\alpha \nabla n^\varepsilon\|^2_{L^2_{xy}}
        \end{split}
		\end{equation}
        where positive constant $C$ is dependent of $T$.
		}
	\end{lemma}
	\noindent{\bf{Proof.}}
        Applying the equation $(\ref{nvu equation})_3$ and $\nabla\cdot u^\varepsilon=0$, we {get} that
         \begin{equation}\label{eq:omega varepsilon}
		\begin{split}{}
	&\partial_t\omega^\varepsilon-\varepsilon\Delta \omega^\varepsilon+u^\varepsilon\cdot
        \nabla\omega^\varepsilon=\partial_xn^\varepsilon,\;\;\omega^\varepsilon|_{y=0}=0.\\
		\end{split}
       \end{equation}
       Taking $\partial^\alpha$ {on $(\ref{eq:omega varepsilon})$} and multiplying by $\partial^\alpha\omega^\varepsilon$ for $|\alpha|\leq m+1$, we can show that
    \begin{equation*}
        \begin{split}{}
        &\sum_{|\alpha|\leq m+1}\langle \partial^\alpha\partial_t\omega^\varepsilon, \partial^\alpha\omega^\varepsilon \rangle-{\varepsilon\sum_{|\alpha|\leq m+1}\langle \partial^\alpha\Delta \omega^\varepsilon, \partial^\alpha\omega^\varepsilon \rangle}\\
        &\leq\sum_{|\alpha|\leq m+1}\Big{|}\langle \partial^\alpha(u^\varepsilon\cdot\nabla \omega^\varepsilon), \partial^\alpha\omega^\varepsilon \rangle\Big{|}+\sum_{|\alpha|\leq m+1}\Big{|}\langle \partial^\alpha\partial_x n^\varepsilon, \partial^\alpha\omega^\varepsilon \rangle\Big{|}\\
        &=:S_6+S_7.\\
        \end{split}
		\end{equation*}
        {Similar to $(\ref{Delta n})$, this implies that}
        \begin{align}\label{Delta omega varepsilon}
        &-\varepsilon\sum_{|\alpha|\leq m+1}\langle \partial^\alpha\Delta \omega^\varepsilon, \partial^\alpha\omega^\varepsilon \rangle\notag\\
        &=-\varepsilon\sum_{|\alpha|\leq m+1}\langle \partial_x\partial^\alpha\partial_x\omega^\varepsilon, \partial^\alpha\omega^\varepsilon \rangle-\varepsilon\sum_{|\alpha|\leq m}\langle \partial_y\partial^\alpha\partial_y\omega^\varepsilon, \partial^\alpha\omega^\varepsilon \rangle\notag\\
        &\quad-\varepsilon\sum_{|\alpha|\leq m+1}\langle [\partial^\alpha,\partial_y]\partial_y\omega^\varepsilon, \partial^\alpha\omega^\varepsilon \rangle\notag\\
        &\geq\varepsilon\sum_{|\alpha|\leq m+1}\|\partial^\alpha\nabla \omega^\varepsilon\|^2_{L^2_{xy}}-C\delta\varepsilon\Big{(}\sum_{|\alpha|\leq m+1}\|\partial^\alpha\partial_y \omega^\varepsilon\|^2_{L^2_{xy}}+\|\partial^{\alpha-(0,1)}\partial_y\omega^\varepsilon\|^2_{L^2_{xy}}\Big{)}\\
        &\quad-C\sum_{|\alpha|\leq m+1}\|\partial^\alpha \omega^\varepsilon\|^2_{L^2_{xy}}\notag\\
        &\geq(1-C\delta)\varepsilon\sum_{|\alpha|\leq m+1}\|\partial^\alpha\nabla \omega^\varepsilon\|^2_{L^2_{xy}}-C\sum_{|\alpha|\leq m+1}\|\partial^\alpha \omega^\varepsilon\|^2_{L^2_{xy}}.\notag\notag
		\end{align}
        Applying the Sobolev inequality $(\ref{infty})$ and {the} condition $\nabla\cdot u^\varepsilon=0$, we obtain the estimation of $S_{6}$ is as follows
        \begin{align}\label{S6}
        S_{6}&=\sum_{|\alpha|\leq m+1}\sum_{\beta\leq\alpha}C^\beta_\alpha\Big{|}\langle \partial^\beta u^\varepsilon\cdot\partial^{\alpha-\beta}\nabla \omega^\varepsilon, \partial^\alpha\omega^\varepsilon \rangle\Big{|}\notag\\
        &\leq\sum_{|\alpha|\leq m+1}\Big{(}\Big{|}\langle u^\varepsilon_2[\partial^{\alpha},\partial_y] \omega^\varepsilon, \partial^\alpha\omega^\varepsilon \rangle\Big{|}+\sum_{1\leq|\beta|\leq m}{C^\beta_{\alpha}}\Big[\Big{|}\langle \partial^\beta u_1^\varepsilon\partial^{\alpha-\beta}\partial_x \omega^\varepsilon,\partial^\alpha\omega^\varepsilon \rangle\Big{|}\notag\\
        &\quad +\Big{|}\langle \partial^\beta u_2^\varepsilon\partial^{\alpha-\beta}\partial_y \omega^\varepsilon, \partial^\alpha\omega^\varepsilon \rangle\Big{|}\Big]+\Big{|}\langle \partial^\alpha u^\varepsilon_1\partial_x\omega^\varepsilon, \partial^\alpha\omega^\varepsilon \rangle\Big{|}+\Big{|}\langle \partial^\alpha u^\varepsilon_2\partial_y\omega^\varepsilon, \partial^\alpha\omega^\varepsilon \rangle\Big{|}\Big{)}\\
        &\leq C(\delta)\Big{(}\| u_2^\varepsilon\|^2_{L^2_{xy}}+\sum_{|\alpha|\leq2}\|\partial^\alpha\nabla u^\varepsilon\|^2_{L^2_{xy}}+\sum_{1\leq|\beta|\leq m}\|\partial^{\beta+(1,0)}\nabla u^\varepsilon\|^2_{L^2_{xy}}+\sum_{|\alpha|\leq3}\|\partial^\alpha\omega^\varepsilon\|^2_{L^2_{xy}}+\|\psi\partial_y \omega^\varepsilon\|^2_{L^\infty_{xy}}\Big{)}\notag\\
        &\quad\times\sum_{|\alpha|\leq m+1}\|\partial^\alpha\omega^\varepsilon\|^2_{L^2_{xy}}+C\sum_{|\alpha|\leq m+1}\|\partial^{\alpha}\omega^\varepsilon\|^2_{L^2_{xy}}+C(\delta)\sum_{|\alpha|\leq m+1}\|\partial^\alpha\nabla u^\varepsilon\|^2_{L^2_{xy}}+C\sum_{|\alpha|\leq m+1}\|\partial^\alpha u^\varepsilon\|^2_{L^2_{xy}}.\notag\notag
		\end{align}
        By the H{\"o}lder's inequality, $S_{7}$ can be estimated as follows:
        \begin{equation}\label{S7}
        \begin{split}{}
        S_{7}&=\sum_{|\alpha|\leq m+1}\Big{(}\Big{|}\langle  \partial^\alpha\partial_xn^\varepsilon , \partial^\alpha\omega^\varepsilon\rangle\Big{|}\Big{)}\leq\frac{\nu}{2}\sum_{|\alpha|\leq m+1}\|\partial^{\alpha} \nabla n^\varepsilon\|^2_{L^2_{xy}}+\frac{C}{\nu}\sum_{|\alpha|\leq m+1}\|\partial^\alpha\omega^\varepsilon\|^2_{L^2_{xy}}.
        \end{split}
		\end{equation}
        Combining $(\ref{Delta omega varepsilon})$-$(\ref{S7})$, we obtain that $(\ref{omega varepsilon Hm})$ holds. Hence, we finish our argument.

        \begin{lemma}\label{lem4.5}
         {Under the assumptions of Theorem $\ref{th nvu regularity}$ for $(n_{in},v_{in},u_{in})$, it can be concluded that for all $m \ge 5$, $\omega^\varepsilon$ has the following properties on $[0,T]$:
		\begin{equation}\label{omega maximum principle}
        \begin{split}{}
        \|\omega^\varepsilon\|^2_{L^\infty_{xy}}&\leq \|\omega^\varepsilon_0\|^2_{L^\infty_{xy}}+C\int^t_0\sum_{|\alpha|\leq2}\|\partial^\alpha\nabla n^\varepsilon\|^2_{L^2_{xy}}ds,\\
        \|\omega^\varepsilon\|^2_{W^{1,\infty}_{xy}}&\leq \|\omega^\varepsilon_0\|^2_{W^{1,\infty}_{xy}}+C\int^t_0\Big{(}\|\omega^\varepsilon\|^2_{W^{1,\infty}_{xy}}+\varepsilon\sum_{|\alpha|\leq m}\|\partial^\alpha\nabla\omega^\varepsilon\|^2_{L^2_{xy}}\\
        &\quad+\sum_{|\alpha|\leq m}\|\partial^\alpha\omega^\varepsilon\|^2_{L^2_{xy}}+\sum_{|\alpha|\leq3}\|\partial^\alpha\nabla n^\varepsilon\|^2_{L^2_{xy}}\Big{)}ds,\\
        \end{split}
		\end{equation}
        where positive constant $C$ is independent of $T$.
		}
	\end{lemma}
	\noindent{\bf{Proof.}} Applying the maximum principle to the vorticity equation
        $(\ref{eq:omega varepsilon})$, we see that it is evident that
         \begin{equation*}
        \begin{split}{}
        \|\omega^\varepsilon\|_{L^\infty_{xy}}&\leq C\|\omega^\varepsilon_0\|_{L^\infty_{xy}}+C\int^t_0\|\partial_xn^\varepsilon\|_{L^\infty_{xy}}ds.\\
        \end{split}
		\end{equation*}
        By the Cauchy-Schwarz inequality, we obtain  $(\ref{omega maximum principle})_1$.\\
        Next, we hope to obtain a similar estimate for $\partial_i\omega^\varepsilon(i=1,2,\;\partial_1=\partial_x,\;\partial_2=\psi(y)\partial_y)$. The main difficulty lies in the estimation of $\partial_2\omega^\varepsilon$, since the commutator of this vector field with the Laplacian operator involves two derivatives in the normal variable.

        To solve this difficulty, we set $\chi(y)$ as a smooth compactly supported function with value $1$ near $0$ and support in $[0, 1]$. Then the vorticity equation can be expressed as
       \begin{equation}\label{omega expressed}
		\begin{split}{}
	       &\omega^\varepsilon=\chi\omega^\varepsilon+(1-\chi)\omega^\varepsilon=\omega^b+\omega^{int},\\
		\end{split}
       \end{equation}
        where $\omega^{int}$  is supported away from the boundary and $\omega^b$  is compactly supported in $y$.

       Since $1-\chi$ and $\partial_y\chi$ vanish near the boundary, and the common normal $\widetilde{H}^m$ norm is equivalent to the conventional $H^s$ norm far away from the boundary, we can use the conventional Sobolev embedding to obtain
       \begin{equation*}
		\begin{split}{}
        \|\omega^{int}\|_{W^{1,\infty}_{xy}}\leq C\|\omega^{int}\|_{H^{s_0}_{xy}}=C\|(1-\chi)\omega^\varepsilon\|_{H^{s_0}_{xy}}\leq C\|\kappa u^\varepsilon\|_{H^{s_0+1}_{xy}},\;\;s_0\geq3
		\end{split}
	\end{equation*}
        where $\kappa$ is supported away from the boundary. Therefore, we get
        \begin{equation}\label{omega int}
		\begin{split}{}
        \|\omega^{int}\|_{W^{1,\infty}_{xy}}\leq  C\| u^\varepsilon\|_{\widetilde{H}^{m}_{xy}},\;\;m\geq4.
		\end{split}
	\end{equation}
       We only need to estimate $\omega^b$. Note that $\omega^b$ is the solution of the following equation
       \begin{equation}\label{omega b}
		\begin{split}{}	&\partial_t\omega^b+u^\varepsilon\cdot\nabla\omega^b=\varepsilon\Delta\omega^b+\chi\partial_xn^\varepsilon+C^b,\\
		\end{split}
       \end{equation}
       in a half-space $y>0$, where $C^b$ is a commutator
       \begin{equation*}
		\begin{split}{}
        C^b=u^\varepsilon_2\partial_y\chi\omega^\varepsilon-2\varepsilon\partial_y\chi\partial_y\omega^\varepsilon-\varepsilon\partial_{yy}\chi\omega^\varepsilon.
		\end{split}
	\end{equation*}
         Since the support for the values of $\partial_y\chi$ and $\partial_{yy}\chi$ is far from the boundary, according to the usual Sobolev embedding, we have
         \begin{equation}\label{Cb}
		\begin{split}{}
        \|C^b\|_{W^{1,\infty}_{xy}}&\leq \|u^\varepsilon_2\partial_y\chi\omega^\varepsilon\|_{W^{1,\infty}_{xy}}+2\varepsilon\|\partial_y\chi\partial_y\omega^\varepsilon\|_{W^{1,\infty}_{xy}}+\varepsilon\|\partial_{yy}\chi\omega^\varepsilon\|_{W^{1,\infty}_{xy}}\\
        &\leq C\|u^\varepsilon\|_{W^{1,\infty}_{xy}}\|\omega^\varepsilon\|_{W^{1,\infty}_{xy}}+2\varepsilon\|\omega^\varepsilon\|_{W^{2,\infty}_{xy}}+\varepsilon\|\omega^\varepsilon\|_{W^{1,\infty}_{xy}}\\
        &\leq C(1+\|\omega^\varepsilon\|_{\widetilde{H}^m_{xy}})\|\omega^\varepsilon\|_{\widetilde{H}^m_{xy}},\;\;m\geq4.
		\end{split}
	\end{equation}
     Next, in order to obtain $(\ref{omega maximum principle})_2$, by applying Lemma 14 in \cite{MR} to the equation $(\ref{omega b})$ and using $(\ref{Cb})$, we immediately get
       \begin{equation*}
        \begin{split}{}
        \|\omega^b\|_{W^{1,\infty}_{xy}}&\leq \|\omega^\varepsilon_0\|_{W^{1,\infty}_{xy}}+\int^t_0\Big{[}(\| u^\varepsilon\|_{W^{2,\infty}_{xy}}+\|\partial_yu^\varepsilon\|_{W^{1,\infty}_{xy}})(\| \omega^b\|_{W^{1,\infty}_{xy}}+\|\omega^b\|_{\widetilde{H}^{m_0+2}_{xy}})\\
        &\quad+\varepsilon\| \partial_{xx}\omega^b\|_{W^{1,\infty}_{xy}}+\|\partial_x n^\varepsilon\|_{W^{1,\infty}_{xy}}+\|C^b\|_{W^{1,\infty}_{xy}}\Big{]}ds\\
        &\leq \|\omega^\varepsilon_0\|_{W^{1,\infty}_{xy}}+C\int^t_0\Big{[}(\| u^\varepsilon\|_{W^{2,\infty}_{xy}}+\|\partial_yu^\varepsilon\|_{W^{1,\infty}_{xy}})(\| \omega^\varepsilon\|_{W^{1,\infty}_{xy}}+\|\omega^\varepsilon\|_{\widetilde{H}^{m_0+2}_{xy}})\\
        &\quad+\varepsilon\| \nabla\omega^\varepsilon\|_{\widetilde{H}^m_{xy}}+\|\partial_x n^\varepsilon\|_{W^{1,\infty}_{xy}}+\|\omega^\varepsilon\|^2_{\widetilde{H}^m_{xy}}\Big{]}ds
        \end{split}
		\end{equation*}
        for $m\geq m_0+3$ and $m_0\geq1$.
        By the Cauchy-Schwarz inequality, it can be inferred that $(\ref{omega maximum principle})_2$ holds.\\

        By Lemma $\ref{lem4.1}$-Lemma $\ref{lem4.5}$, we can prove Theorem  $\ref{th nvu regularity}$ as follows.

        \textbf{{Proof} of Theorem  $\ref{th nvu regularity}$.}
        Combining Lemma $\ref{lem4.1}$-Lemma $\ref{lem4.5}$, we choose $\nu=\frac{1-2C\delta}{8}$, then for sufficiently small $\delta$ there are smooth solutions $(n^\varepsilon,v^\varepsilon,u^\varepsilon)$ and $\omega^\varepsilon$ on $[0,T]$ with the following
        \begin{equation}\label{Equ(4.17)}
		\begin{split}{}             &\|n^\varepsilon\|^2_{L^\infty_TL^2_{xy}}+\|v^\varepsilon\|^2_{L^\infty_TL^2_{xy}}+\|u^\varepsilon\|^2_{L^\infty_TL^2_{xy}}+\|\nabla n^\varepsilon\|^2_{L^\infty_T\widetilde{H}^m_{xy}}+\|\nabla v^\varepsilon\|^2_{L^\infty_T\widetilde{H}^m_{xy}}+\|\nabla u^\varepsilon\|^2_{L^\infty_T\widetilde{H}^{m+1}_{xy}}\\
             &+\|\nabla^2 n^\varepsilon\|^2_{L^2_T\widetilde{H}^m_{xy}}+\varepsilon\|\nabla^2 v^\varepsilon\|^2_{L^2_T\widetilde{H}^m_{xy}}+\varepsilon\|\nabla^2 u^\varepsilon\|^2_{L^2_T\widetilde{H}^{m+1}_{xy}}+\|\omega^\varepsilon\|^2_{L^\infty_TL^\infty_{xy}}+\|\omega^\varepsilon\|^2_{L^\infty_TW^{1,\infty}_{xy}}\leq C.\\	
		\end{split}
	\end{equation}
        Therefore, we complete the proof of Theorem  $\ref{th nvu regularity}$.

	\section*{Acknowledgements}
L. Zhao was supported by the National Natural Science Foundation of China(NSFC)[grant number 12501305].
	\hspace*{\parindent}

\medskip
\noindent\textbf{Data Availability Statement:}
Data sharing is not applicable to this article as no data sets were generated or analyzed during the current study.

\noindent\textbf{Conflict of Interest:}
The authors declare that they have no conflict of interest.
	\hspace*{\parindent}


\begin{thebibliography}{00}
        \bibitem{AF} A. Adams, J.J.F. Fournier, Sobolev Space, 2nd edn. Pure Appl. Math. 140 (2009).
        \bibitem{CHW} J.A. Carrillo, G. Hong, Z. Wang, Convergence of boundary layers of chemotaxis models with physical boundary conditions I: degenerate initial data. SIAM J. Math. Anal. 56 (6) (2024) 7576-7643.
		\bibitem{CLW} J.A. Carrillo, J.Y. Li, and Z.A. Wang, Boundary spike-layer solutions of the singular Keller-Segel system: existence and stability. Proc. Lond. Math. Soc. 122 (3) (2021) 42-68.
        \bibitem{CLW1} H. Chen, J.M. Li, K. Wang, On the vanishing viscosity limit of a chemotaxis model. Discrete Contin. Dyn. Syst. Ser. A. 40 (3) (2020) 1963-1987.

        \bibitem{CLX} F. Cheng, W.X. Li, C.J. Xu, Vanishing viscosity limit of Navier-Stokes Equations in Gevrey class. Math. Methods Appl. Sci. 40 (14) (2017) 5161-5176.
		\bibitem{DCC}C. Dombrowski, L. Cisneros, S. Chatkaew, R. Goldstein, and J. Kessler, Selfconcentration and large-scale coherence in bacterial dynamics. Phys. Rev. Lett. 93 (2004) 098103.
        \bibitem{FGL1} M. Fei, C. Gao, Z. Lin, et al., Prandtl-Batchelor flows on an annulus. Adv. Math. 458 (2024) 109994.
        \bibitem{FGL2} M. Fei, C. Gao, Z. Lin, et al., Prandtl-Batchelor Flows on a Disk. Commun. Math. Phys. 397 (3) (2023) 1103-1161.
        \bibitem{FTZ} M. Fei, T. Tao, Z. Zhang, On the zero-viscosity limit of the Navier-Stokes equations in $R_+^3$ without analyticity. J. Math. Pures Appl. 112 (2018) 170-229.
		\bibitem{HL}Q. Hou, C.J. Liu, Y.G. Wang, Z. Wang, Stability of boundary layers for a viscous hyperbolic system arising from chemotaxis: one dimensional case. SIAM J. Math. Anal. 50 (3) (2018) 3058-3091.
        \bibitem{HO} Q. Hou, Boundary layer problem on the chemotaxis model with Robin boundary conditions. Discrete Contin. Dyn. Syst. 44 (2) (2024) 378-424.
        \bibitem{HO1} Q. Hou, Boundary layer effects induced by the fluid in a chemotaxis-Navier-Stokes system. arXiv preprint arXiv:2509.03028, 2025.
		\bibitem{HW}Q. Hou, Z. Wang, Convergence of boundary layers for the Keller-Segel system with singular sensitivity in the half-plane. J. Math. Pures Appl. 130 (2019) 251-287.
		\bibitem{HWZ}Q. Hou, Z. Wang, K. Zhao, Boundary layer problem on a hyperbolic system arising from chemotaxis. J. Differ. Equ. 261 (2016) 5035-5070.
		\bibitem{IP} D. Iftimie, G. Planas, Inviscid limits for the Navier-Stokes equations with Navier friction boundary conditions. Nonlinearity. 19 (4) (2006) 899-918.
		\bibitem{IS} D. Iftimie, F. Sueur, Viscous boundary layers for the Navier-Stokes equations with the Navier slip conditions. Arch. Ration. Mech. Anal. 199 (1) (2011) 145-175.
		\bibitem{JM} W. Jager, A. Mikeli'c, On the roughness-induced effective boundary conditions for an incompressible viscous flow. J. Differ. Equ. 170 (2001) 96-122
		\bibitem{KE} J.P. Kelliher, The strong vanishing viscosity limit with Dirichlet boundary conditions. Nonlinearity. 36 (5) (2023) 2708.
        \bibitem{LWY}C.C. Lee, Z.A. Wang, W. Yang, Boundary-layer profile of a singularly perturbed nonlocal semi-linear problem arising in chemotaxis. Nonlinearity, 33 (2020) 5111-5141.
        \bibitem{LS}H.A. Levine, B.D. Sleeman, A system of reaction diffusion equations arising in the theory of reinforced random walks. SIAM J. Appl. Math. 57 (1997) 683-730.
		\bibitem{LL}D. Li, T. Li, K. Zhao, On a hyperbolic-parabolic system modeling chemotaxis. Math. Models Methods Appl. Sci. 21 (2011) 1631-1650.
        \bibitem{LP} T. Li, R. Pan, K. Zhao, Global dynamics of a hyperbolic-parabolic model arising from chemotaxis. SIAM J. Appl. Math. 72 (1) (2012) 417-443.
         \bibitem{LSW} B. Li, F. Shi, W. Wang, Zero-viscosity limit of the chemotaxis-Navier-Stokes equations with the Navier-slip boundary condition. arXiv preprint arXiv:2605.02394, 2026.
		\bibitem{LW}T. Li, Z. Wang, Asymptotic nonlinear stability of traveling waves to conservation laws arising from chemotaxis. J. Differ. Equ. 250 (3) (2011) 1310-1333.


		\bibitem{LZ} H. Li, K. Zhao, Initial-boundary value problems for a system of hyperbolic balance laws arising from chemotaxis. J. Differ. Equ. 258 (2) (2015) 302-338.
        \bibitem{MR} N. Masmoudi and F. Rousset, Uniform regularity for the Navier-Stokes equation with Navier boundary condition. Arch. Ration. Mech. Anal. 203 (2012) 529-575.
       \bibitem{MW1} N. Masmoudi, T. Wong, Local-in-time existence and uniqueness of solutions to the Prandtl equations by energy methods. Comm. Pure Appl. Math. 68 (2015) 1683-1741.
       \bibitem{MXW} L. Meng, W.Q. Xu, S. Wang, On the vanishing viscosity limit for a 3-D system arising from the Keller-Segel model. Math. Methods Appl. Sci. 43 (2) (2020) 920-938.
       \bibitem{MXW1} L. Meng, W.Q. Xu, S. Wang, Boundary layer analysis for a 2-D Keller-Segel model. Open Math. 18 (1) (2020) 1895-1914.
	    \bibitem{OS} O.A. Oleinik, V.N. Samokhin, Mathematical Models in Boundary Layer Theory. Appl. Math. Comput. (1999) 528.
		\bibitem{PRA}L. Prandtl, \"{U}ber Fl\"{u}ssigkeitsbewegungen bei sehr kleiner Reibung. Verhandl. 3rd Int. Math. Kongr. Heidelberg (1904), Leipzig. (1905) 484-491.
		\bibitem{PW} H. Peng, Z.A. Wang, K. Zhao, et al, Boundary layers and stabilization of the singular Keller-Segel system. Kinet. Rel. Models. 11 (5) (2018) 1085-1123.

        \bibitem{RWW} L.G Rebholz, D. Wang, Z. Wang, et al. Initial boundary value problems for a system of parabolic conservation laws arising from chemotaxis in multi-dimensions. Discrete Contin. Dyn. Syst. Ser. A. 39 (7) (2019) 3789¨C3838.

        \bibitem{SC1} M. Sammartino, R.E. Caflisch, Zero viscosity limit for analytic solutions, of the Navier-Stokes equation on a half-space. I. Existence for Euler and Prandtl equations. Commun. Math. Phys. 192 (2) (1998) 433-461.
        \bibitem{SC2} M. Sammartino, R.E. Caflisch, Zero Viscosity Limit for analytic solutions of the Navier-Stokes equation on a half-space. II. Construction of the Navier-Stokes solution. Commun. Math. Phys. 192 (2) (1998) 463-491.

        \bibitem{SG} H. Schlichting, K. Gersten, K. Krause, H. Oertel Jr, C. Mayes, Boundary-Layer Theory, 8th edition, Springer, 2004.
        \bibitem{TAO} T. Tao, Vanishing vertical viscosity limit of anisotropic Navier¨CStokes equation with no-slip boundary condition. J. Differ. Equ. 265 (9) (2018) 4283-4310.
        \bibitem{TC} I. Tuval, L. Cisneros, C. Dombrowski, et al, Bacterial swimming and oxygen transport near contact lines. P Natl Acad Sci Usa. 102 (7) (2005) 2277-2282.
        \bibitem{twz}
        T. Tao, W.D. Wang, Z.F. Zhang, Zero-viscosity limit of the Navier-Stokes equations with the Navier friction boundary condition. SIAM J. Math. Anal. 52 (2020), no. 2, 1040-1095.
        \bibitem{WWX} X.P. Wang, Y.G. Wang, Z.P. Xin, Boundary layers in incompressible Navier-Stokes equations with Navier boundary conditions for the vanishing viscosity limit. Comm. Math. Sci. 8 (2010) 965-998.
        \bibitem{WWZ} H. Wang, W.D. Wang, L.L. Zhao, Boundary layer analysis for the 2D chemotaxis-Navier-Stokes system with logarithmic sensitivity, Part I: Well-posedness. arXiv preprint arXiv:2608.05940, 2026.
        \bibitem{WX}Y.G. Wang, Z.P. Xin, Zero-viscosity limit of the linearized compressible Navier-Stokes equations with highly oscillatory forces in the half-plane. SIAM J. Math. Anal. 37 (2005) 1256-1298
		\bibitem{WXZ}L. Wang, Z. Xin, A.b. Zang, Vanishing viscous limits for 3D Navier-Stokes equations with a Navier-slip boundary condition. J. Math. Fluid Mech. 14 (4) (2012) 791-825.
        \bibitem{WYZ1} C. Wang, J. Yue, Z. Zhang, The Navier-Stokes equations in $\mathbb R^ 2_+$ with point vortex initial data: construction of the solution. arXiv preprint arXiv:2604.05765, 2026.
        \bibitem{WYZ2} C. Wang, J. Yue, Z. Zhang, The Navier-Stokes equations in $\mathbb R^ 2_+$ with point vortex initial data: Zero-viscosity limit. arXiv preprint arXiv:2604.05787, 2026.
        \bibitem{WE} C. Weibull, Movement. In: The Bacteria. Academic Press, New York. 1 (1960) 153-205.  		
		\bibitem{XX} Y.L. Xiao, Z.P. Xin, Remarks on the vanishing viscosity limit for the Navier-Stokes equations with a slip boundary condition. Chin. Ann. Math. 32 (3) (2011) 321-332.
		\bibitem{XX1} Y.L. Xiao, Z.P. Xin, On the vanishing viscosity limit for the 3D Navier-Stokes equations with a slip boundary condition. Comm. Pure Appl. Math. 60 (2007) 1027-1055.
		\bibitem{XY}Z.P. Xin, T. Yanagisawa, Zero-viscosity limit of the linearized Navier-Stokes equations for a compressible viscous fluid in the half-plane. Comm. Pure Appl. Math. 52 (1999) 479-541.		











	\end{thebibliography}
\end{document}